\documentclass[a4paper,12pt]{article}
\usepackage{amsmath,amsthm,amssymb}
\usepackage[T2A]{fontenc}
\usepackage{amscd}

\usepackage{enumitem}

\usepackage{MnSymbol} %udots in matrices
\usepackage{multirow}

\usepackage[table]{xcolor} %colored cells

\usepackage{tikz-cd}

\usepackage{hyperref}

  \newskip\prethm \prethm3.0pt plus1.3pt minus.4pt
  \newskip\posthm \posthm2.7pt plus1.4pt minus.3pt
  \newtheoremstyle{STATEMENT}%
       {\prethm}{\posthm}{\itshape}{\parindent}{\scshape}%  {\bfseries}
       {.}{.6em plus.2em minus.1em}{}
  \newtheoremstyle{EXPLANATION}%
       {\prethm}{\posthm}{}{\parindent}{\scshape}%  {\bfseries}
       {.}{.6em plus.2em minus.1em}{}
\theoremstyle{STATEMENT}

\newtheorem{theorem}{Theorem}[section]
\newtheorem{proposition}{Proposition}[section]
\newtheorem{lemma}{Lemma}[section]
\newtheorem{assertion}{Assertion}[section]
\newtheorem{corollary}{Corollary}[section]
\newtheorem{claim}{Claim}[section]
\newtheorem{conjecture}{Conjecture}[section]
\newtheorem{question}{Question}[section]

\theoremstyle{EXPLANATION}
\newtheorem{definition}{Definition}[section]
\newtheorem{remark}{Remark}[section]
\newtheorem{problem}{Problem}[section]
\newtheorem{example}{Example}[section]

\medskip

\title{Geometry of bi-Lagrangian Grassmannian}
\author{I.\,K.~Kozlov\thanks{No Affiliation, Moscow, Russia. E-mail: {\tt ikozlov90@gmail.com} }}
\date{}

\begin{document}

\maketitle
\begin{abstract} This paper explores the structure of bi-Lagrangian Grassmanians for pencils of $2$-forms on real or complex vector spaces. We reduce the analysis to the pencils whose Jordan–Kronecker Canonical Form consists of Jordan blocks with the same eigenvalue. We demonstrate that this is equivalent to studying Lagrangian subspaces invariant under a nilpotent self-adjoint operator. We calculate the dimension of bi-Lagrangian Grassmanians and describe their open orbit under the automorphism group. We completely describe the automorphism orbits in the following three cases: for one Jordan block, for sums of equal Jordan blocks and for a sum of two distinct Jordan blocks. \end{abstract}


\tableofcontents

\section{Introduction} \label{S:IntroSec}

Let $V$  be a complex finite-dimensional vector space. Consider a pencil $\mathcal{P}$ generated by two skew-symmetric bilinear forms $A$ and $B$ on $V$:  \[\mathcal{P} = \left\{A_{\lambda} =  A + \lambda B \,\, \bigr| \, \, \lambda \in \bar{\mathbb{C}} \right\}.\] Here $\bar{\mathbb{C}} = \mathbb{C}  \cup \left\{\infty \right\} $ and $A_{\infty}= B$. A subspace $L \subset V$ is a \textbf{bi-Lagrangian} subspace if it satisfies the following two conditions:
\begin{enumerate}

\item $L$ is bi-isotropic, i.e. it is isotropic w.r.t. both forms $A$ and $B$. 

\item $L$ has maximal possible dimension, i.e. \[\operatorname{dim} L = \dim V - \frac{1}{2} \max_{\lambda \in \bar{\mathbb{C}}} \operatorname{rk} (A + \lambda B)\]

\end{enumerate}
In other words, $L$ is Lagrangian (i.e. maximal isotropic) w.r.t. almost all forms $A_{\lambda} \in \mathcal{P}$. In this paper we study the properties of the collection of all bi-Lagrangian subspaces for a given pencil $\mathcal{P}$, which we call a \textbf{bi-Lagrangian Grassmanian} and denote by $\operatorname{BLG}\left(V, \mathcal{P}\right)$.

\subsection{Structure of the paper}

In this section, we provide a road map of the paper, highlighting the key topics covered in each section. Sections~\ref{S:IntroSec}--\ref{S:BiLagraGr} serve as an introduction, providing essential foundational results.

\begin{enumerate}

\item Section~\ref{S:IntroSec} outlines the paper's structure. Optional Section~\ref{SubS:Motivation} discusses our motivation rooted in integrable systems theory.

\item Section~\ref{S:Basic} introduces the basic definitions we will utilize throughout this paper. In particular, in Section~\ref{S:JKTheorem} we describe the canonical form\footnote{These canonical forms are skew-symmetric counterparts of Kronecker Canonical Forms (KCFs) for matrix pencils under strict equivalence \cite{Gantmacher88}} for a pencil of $2$-forms $\mathcal{P}= \left\{A + \lambda B \right\}$ given by \textbf{the Jordan--Kronecker theorem}. This form expresses $\mathcal{P}$ as a sum of Jordan and Kronecker blocks. The corresponding \textbf{Jordan-Kronecker decomposition} of  $(V,\mathcal{P})$ is expressed as \begin{equation} \label{Eq:JKDecomIntro}
 (V, \mathcal{P}) = \bigoplus_{j=1}^{S}\left(\bigoplus_{k=1}^{N_j} \mathcal{J}_{\lambda_j, 2n_{j,k}} \right) \oplus  \bigoplus_{j=1}^q \mathcal{K}_{2k_j+1},
\end{equation} where $\mathcal{J}_{\lambda_j, 2n_{j,k}}$ represent Jordan blocks (with eigenvalues $\lambda_j$) and $\mathcal{K}_{2k_j+1}$ represent Kronecker blocks.  Obviously, the structure of the bi-Lagrangian Grassmanian $\operatorname{BLG}(V, \mathcal{P})$ depends on the decomposition \eqref{Eq:JKDecomIntro}. 

\item Section~\ref{S:BiLagraGr} establishes some fundamental properties of bi-Lagrangian Grassmanians:

\begin{enumerate}

\item  $\operatorname{BLG}(V, \mathcal{P})$ 
 is a projective subvariety of the corresponding Grassmanian (Lemma~\ref{L:BiLagrProjVar}). $\operatorname{BLG}(V, \mathcal{P})$ is always non-empty\footnote{We also prove in Theorem~\ref{T:BLGConnected} that $\operatorname{BLG}(V, \mathcal{P})$ is connected.} (Assertion~\ref{A:ExistBiLagr}). 

\item In the decomposition~\eqref{Eq:JKDecomIntro} we can group Jordan blocks with the same eigenvalue and combine Kronecker blocks: \begin{equation} \label{Eq:JKDecomIntro2}
 (V, \mathcal{P}) = \bigoplus_{j=1}^{S}\mathcal{J}_{\lambda_j}  \oplus  \mathcal{K}
\end{equation} By Theorems~\ref{T:BiLagrKronPart} and \ref{T:JordaMultEigen} any bi-Lagrangian subspace $L$ inherits this structure and decomposes into a direct product: \[ L =  \bigoplus_{j=1}^{S}\left(L \cap \mathcal{J}_{\lambda_j}\right)  \oplus  \left( L \cap \mathcal{K}\right).\] Hence, the bi-Lagrangian Grassmanian is isomorphic to the direct product: \begin{equation} \label{Eq:DecomIntroBLG} \operatorname{BLG}(V, \mathcal{P}) \approx \prod_{j=1}^S \operatorname{BLG} \left(\mathcal{J}_{\lambda_j}\right)  \times \operatorname{BLG} \left( \mathcal{K} \right)\end{equation} 

\item By Theorem~\ref{T:BiLagrKronPart} the ``Kronecker part" $\operatorname{BLG} \left(\mathcal{K} \right)$ consists of one element,  the core subspace (introduced in Section~\ref{SubS:CoreMantle}): \[\operatorname{BLG} \left(\mathcal{K} \right) = \left\{K \right\}.\] In particular, in the Kronecker case the core subspace $K$ is the only bi-Lagrangian subspace.

\item In the Jordan case, if $\operatorname{Ker} B =0$, then we can replace pencil $\mathcal{P}$ with symplectic form $B$ and self-adjoint operator $P = B^{-1}A$. By Lemma~\ref{L:NonGenDescBiLagr}, \[ L \subset (V, \mathcal{P}) \quad \text{is bi-Lagrangian} \Leftrightarrow L \subset (V, B) \quad \text{is $P$-invariant Lagrangian}.\]

\item Decomposition~\eqref{Eq:DecomIntroBLG} lets us analyze only pencils $\mathcal{P}$ that are sums of Jordan blocks with a common eigenvalue $\lambda$. Theorem~\ref{T:NonDependEigen} lets us assume $\lambda =0$. 

\item In Section~\ref{S:OneJordBlock} we describe the bi-Lagrangian Grassmanian for one Jordan block $\operatorname{BLG}(\mathcal{J}_{0, 2n})$.

 \begin{enumerate}

\item By Theorem~\ref{T:BiLagr_One_Jordan_Canonical_Form}, for any bi-Lagrangian subspace $L \subset \mathcal{J}_{0, 2n}$ there is a  standard basis $ e_1, \dots, e_n, f_1, \dots, f_ {n} $ from the Jordan--Kronecker theorem such that \[ L = \operatorname{Span}\left\{ e_{n-h+1}, e_{n-h}, \dots, e_n, f_1, \dots, f_{n-h}\right\}. \]

\item  Theorem~\ref{Th:OneJordBiLagrOrbits} proves that $\operatorname{BLG}(\mathcal{J}_{\lambda, 2n})$ consists of $\left[\frac{n+1}{2}\right]$ automorphism orbits that are diffeomorphic to \[\bigoplus_{n-1} T \mathbb{CP}^1, \qquad \bigoplus_{n-3} T \mathbb{CP}^1, \qquad \dots \qquad \mathbb{CP}^1 \quad \text{ or } \left\{ 0\right\}. \]

 \end{enumerate}

\end{enumerate}

\end{enumerate}

Next, we turn to the study of automorphism orbits\footnote{\cite[Problem 13]{BolsinovIzosimomKonyaevOshemkov12},  \cite[Problem in Section 2.2.2]{Rosemann15} and \cite[Problem 10]{BolsinovTsonev17} asked about partitioning of $\operatorname{BLG}(V,\mathcal{P})$ into automorphism orbits.} of $\operatorname{BLG}(V,\mathcal{P})$. Two preparatory steps (developed in Sections~\ref{S:BiPoisSect} and \ref{S:AutGroup}) are crucial before proceeding: bi-Poisson reduction and the structure of the automorphism group $\operatorname{Aut}(V,\mathcal{P})$. 

\begin{enumerate}
\setcounter{enumi}{3}

\item Section~\ref{S:BiPoisSect} introduces \textbf{bi-Poisson reduction}, generalizing symplectic reduction to a vector spaces equipped with a pencil of $2$-forms.

\begin{enumerate}
    \item  In Section~\ref{S:LinearRedDecr} we prove that for any admissible bi-isotropic subspace $U$ we can induce the pencil and bi-Lagrangian subspaces on $U^{\perp}/U$. ($U$ is \textbf{admissible} if $U^{\perp}$ is the same for almost all forms $A_{\lambda} \in \mathcal{P}$, see Section~\ref{S:Admissible}.)
    
    \item In Section~\ref{S:IsomAlgBiPRed} we show that the variety of bi-Lagrangian subspaces containing $U$ is isomorphic to $\operatorname{BLG}(U^{\perp}/ U)$ (see Lemma~\ref{L:AlgIsomBiisotrFactor}).
    
    \item  In Section~\ref{SubS:BiIsotr2BiLagr}  we answer the following question (see \cite[Problem 14]{BolsinovIzosimomKonyaevOshemkov12} and \cite[Problem 11]{BolsinovTsonev17}): \[\text{\textit{``When a bi-isotropic subspace extends to a bi-Lagrangian subspace?''}}\]
    
\end{enumerate}

\item In Section~\ref{S:AutGroup} we describe the group of bi-Poisson automorphsims $\operatorname{Aut}(V,\mathcal{P})$.

\end{enumerate}

Utilizing the tools developed earlier, in Section~\ref{S:DimBiLagrGrassm} we describe the open orbit of $ \operatorname{BLG}(V,\mathcal{P})$ and calculate its dimension. In Sections~\ref{S:InvSubspaces} we describe all invariant bi-isotropic and bi-Lagrangian subspaces. In Section~\ref{S:BiLagrContInv} we show that, roughy speaking, certain smaller orbits of $ \operatorname{BLG}(V,\mathcal{P})$ form smaller bi-Lagrangian Grassmanians.

\begin{enumerate}
\setcounter{enumi}{5}

\item In Section~\ref{S:DimBiLagrGrassm} we investigate generic bi-Lagrangian subspaces and calculate dimension of the bi-Lagrangian Grassmanian.

\begin{enumerate}
    \item $\operatorname{BLG}(V,\mathcal{P})$ has a unique open $\operatorname{Aut}(V,\mathcal{P})$-orbit $O_{\max}$ (Theorem~\ref{T:BiLagrGenDecomp}).
    \item Let $(V,\mathcal{P})= \bigoplus_{j=1}^N \mathcal{J}_{0, 2n_j}$. By Theorem~\ref{T:BiLagrGenDecomp}, for any bi-Lagrangian subspaces $L \in O_{\max}$ there is a standard basis 
$e^j_{k}, f^j_k$, where $j=1,\dots, N, k=1,\dots, n_j$, from the Jordan--Kronecker theorem such that \[ L = \operatorname{Span}\left\{e^j_k\right\}_{j=1,\dots, N, k=1,\dots, n_j}.\]

\item In Theorem~\ref{T:BiLagrJord} we prove that \[ \dim \operatorname{BLG}\left(\bigoplus_{j=1}^N \mathcal{J}_{0, 2n_j}\right) = \sum_{j=1}^N j \cdot  n_j, \] where $n_1 \geq \dots \geq n_N$.
\end{enumerate}

\item Section~\ref{S:InvSubspaces} analyzes automorphism orbits:

\begin{enumerate}

\item In Section~\ref{SubS:InvSubspaces} we describe all  $\operatorname{Aut}(V,\mathcal{P})$-invariant subspaces $U \subset (V,\mathcal{P})$. 

\item In Section~\ref{SubS:OrbitsVect} we describe all  $\operatorname{Aut}(V,\mathcal{P})$-orbits of vectors $v \in (V,\mathcal{P})$.

\item In Section~\ref{S:InvBiLagr} we describe all $\operatorname{Aut}(V,\mathcal{P})$-invariant bi-isotropic and bi-Lagrangian subspaces.

\end{enumerate}

\item In Section~\ref{S:BiLagrContInv} we analyze, for a fixed $\operatorname{Aut}(V,\mathcal{P})$-invariant bi-isotropic subspace $U$, how $\operatorname{BLG}(U^{\perp}/U)$ decomposes into   $\operatorname{Aut}(V,\mathcal{P})$-orbits  within  $\operatorname{BLG}(V,\mathcal{P})$. Our analysis yields two noteworthy results.

\begin{enumerate}

\item Section~\ref{SubS:InfOrb} provides an example of bi-Lagrangian Grassmanian $\operatorname{BLG}(V,\mathcal{P})$ with an infinite number of $\operatorname{Aut}(V,\mathcal{P})$-orbits.

\item Section~\ref{SubS:NonDecomp} shows that there exist indecomposable bi-Lagrangian subspaces $L \subset (V,\mathcal{P})$ (see Definition~\ref{Def:Decomp}).

\end{enumerate}

\end{enumerate}

Section~\ref{S:EqualJordBlocks}--\ref{S:TwoJordanBlocks} investigate the topology of $\operatorname{Aut}(V,\mathcal{P})$-orbits within the bi-Lagrangian Grassmannian $\operatorname{BLG}(V,\mathcal{P})$. We classify all orbits if $(V,\mathcal{P})$ is either a sum of equal Jordan blocks or a sum of two distinct Jordan blocks.  We also describe generic orbits for any space $(V,\mathcal{P})$.

\begin{enumerate}
\setcounter{enumi}{8}
\item  In Section~\ref{S:EqualJordBlocks} we study the bi-Lagrangian Grassmanian for a sum of $l$ equal Jordan blocks $(V,\mathcal{P}) = \bigoplus_{i=1}^l \mathcal{J}_{0, 2n}$.

\begin{enumerate}
    \item By Theorem~\ref{T:BiLagrEqJordBl}, for any bi-Lagrangian subspaces $L \in (V,\mathcal{P})$ there is a standard basis 
$e^i_j, f^i_j$, $i=1, \dots, l, j=1, \dots, n$, from the Jordan--Kronecker theorem such that   \[ L = \bigoplus_i  L_i, \qquad L_i = \operatorname{Span}\left\{e^i_{n-h_i +1}, \dots, e^i_n, f^i_1, \dots f^i_{n-h_i}\right\}. \]

    \item By Theorem~\ref{T:BiLagrEqTypesNum}, the heights $h_1 \geq h_2 \geq \dots \geq h_l \geq \frac{n}{2}$ are uniquely defined. Hence, the number of $\operatorname{Aut}(V,\mathcal{P})$-orbits is \[\binom{\left[ \frac{n}{2} \right] +  l}{l}.\] The dimension of the orbit is \[ \dim O_L = \sum_{1 \leq i \leq j \leq l} 2h_i - n.\]

    \item In Theorem~\ref{Th:EqualJordBiLagrMaxOrbit} we prove that the maximal orbit $O_{\max}$ is diffeomorphic to the jet space of the Lagrangian Grassmanian: \[ O_{\max} \approx J^{n-1}_0(\mathbb{C}, \Lambda(l)).\]
    
    \item In Theorem~\ref{T:EqualJordTopolOrb} we show that each $\operatorname{Aut}(V,\mathcal{P})$-orbit $O_L$ is a $\mathbb{C}^{M_L}$-fibre bundle over an isotropic flag variety  \[\pi: O_L  \xrightarrow{\mathbb{C}^{M_L}}  \operatorname{SF}(d_1, ..., d_k; 2 D).\]
\end{enumerate}

\item In Section~\ref{S:TopGenOrbit} we study the topology of the 
 maximal $\operatorname{Aut}(V,\mathcal{P})$-orbit $O_{\max}$ for a sum of a sum of Jordan blocks  \[\left(V, \mathcal{P}\right) = \bigoplus_{i=1}^t \left(\bigoplus_{j=1}^{l_i} \mathcal{J}_{0, 2n_i} \right),  \qquad n_1 > n_2 > \dots > n_t.\] Theorem~\ref{T:GenJordTopolMaxOrb} proves that $O_{\max}$ is a $\mathbb{C}^{M}$-fibre bundle over a product of Lagrangian Grassmanians:  \[  \pi: O_{\max}  \xrightarrow{\mathbb{C}^M}  \Lambda(l_1) \times \dots \times \Lambda(l_t). \]

\item In Section~\ref{S:TwoJordanBlocks} we describe the bi-Lagrangian Grassmanian for a sum of two different Jordan blocks $\mathcal{J}_{0, 2n_1} \oplus \mathcal{J}_{0, 2n_2}$. There are two distinct types of bi-Lagrangian subspaces: \textbf{semisimple} subspaces, constructed as direct sums of bi-Lagrangian subspaces within Jordan blocks, and \textbf{indecomposable} subspaces (see Definition~\ref{Def:Decomp}).

\begin{enumerate}

\item Section~\ref{SubS:TwoDistSemi} classifies semisimple bi-Lagrangian subspaces. By Theorem~\ref{T:BiLagrTwoDistSemi}, each semisimple bi-Lagrangian subspace admits a canonical form \[ L =  \operatorname{Span}\left\{f_1,\dots, f_{n_1 - h_1}, \quad e_{n_1 - h_1 + 1}, \dots, e_{n_1}, \quad \hat{f}_1,\dots, \hat{f}_{n_2 - h_2}, \quad \hat{e}_{n_2-h_2 +1}, \dots, \hat{e}_{n_2}\right\} \] in some standard basis \eqref{Eq:StandBasisTwoBlocks}. The heights $\displaystyle h_1 \geq \frac{n_1}{2}$, and $\displaystyle h_2 \geq \frac{n_2}{2}$ are uniquely defined. Hence, the number of semisimple $\operatorname{Aut}(V, \mathcal{P})$-orbits is
\[ \left( \left[\frac{n_1}{2}\right] + 1 \right) \left( \left[\frac{n_2}{2}\right] + 1 \right). \] The dimension and topology of the orbits are specified in Theorem~\ref{T:2DistJordTopolMaxOrbIndecTypeI}.

\item  In Section~\ref{SubS:TwoDistIndecom} we describe indecomposable bi-Lagrangian subspaces. According to Theorem~\ref{SubS:TwoDistIndecom}, any indecomposable subspace admits the form \[ \begin{gathered} L = \operatorname{Span} \left\{u, Pu, \dots, P^{r-1}u, \quad v, Pv, \dots, P^{r-1}v \right\} \oplus \\ \oplus \operatorname{Span} \left\{e_{n_1}, \dots, e_{n_1 - p_1 +1},\quad f_1, \dots, f_{q_1}, \quad \hat{e}_{n_2}, \dots, \hat{e}_{n_2 - p_2 +1},\quad \hat{f}_1, \dots, \hat{f}_{q_2}\right\}.   \end{gathered} \]
in some standard basis \eqref{Eq:StandBasisTwoBlocks}. Furthermore, the vectors $u$ and $v$ can be brought into one of the forms presented in Theorems~\ref{T:CanonIndecompTypeI} and \ref{T:CanonIndecompTypeII}. There is a finite number of indecomposable $\operatorname{Aut}(V, \mathcal{P})$-orbits. Their dimensions and topology are specified in Theorems~\ref{T:2DistJordTopolMaxOrbIndecTypeI},  \ref{T:Type2Special} and \ref{T:2DistJordTopolMaxOrbIndecTypeII}. Of particular interest are the orbits $O_r$ from Theorem~\ref{T:2DistJordTopolMaxOrbIndecTypeI}, which are diffeomorphic to  $(r-1)$-order jet space space of the $2$-dimensional special linear group: \[ O_r \approx J^{r-1}_0(\mathbb{K}, \operatorname{SL}(2,\mathbb{K})). \] \end{enumerate}

\end{enumerate}

A complete description of all bi-Lagrangian subspaces seems infeasible (cf. Section~\ref{S:TwoJordanBlocks}). However, Section~\ref{S:DecomposableSection} offers a silver lining: effective methods for describing specific types of these subspaces. 

\begin{enumerate}
\setcounter{enumi}{11}
\item Section~\ref{S:DecomposableSection} classifies \textbf{semisimple} bi-Lagrangian subspaces (sums of bi-Lagrangian subspaces within Jordan blocks) and computes their orbit dimensions. 
We also establish the equivalence of semisimple and marked bi-Lagrangian subspaces (Definition~\ref{Def:Marked}) and characterize spaces where all bi-Lagrangian subspaces are semisimple.
 
\end{enumerate}

The paper concludes with two sections offering additional content.

\begin{enumerate}
\setcounter{enumi}{12}
\item  In Section~\ref{S:RealCase} we describe the bi-Lagrangian Grassmanian in the real case. 

\item Section~\ref{S:OpenProblems} concludes the paper by posing several open problems for future investigation.
\end{enumerate}

\subsection{Motivation: integrable and bi-Hamiltonian systems} \label{SubS:Motivation}

Although bi-Lagrangian Grassmannians possess independent mathematical significance, our investigation is primarily driven by their relationship to bi-Hamiltonian systems. This section may be of interest to specialists in integrable systems, its content is not utilized in the subsequent analysis. Readers can choose to skip this section and proceed to the next section without any loss of continuity.

Since the pioneering work by Franco Magri \cite{Magri78} (which was futher developed in \cite{Gelfand79}, \cite{Magri84} and \cite{Reiman80}), it is well known that many  bi-Hamiltonian systems in mathematical physics, geometry and mechanics  are in fact integrable. However, an interesting asymmetry arises: these systems are Hamiltonian w.r.t. two compatible Poisson pencils, but integrability is typically achieved for only one of them. This naturally leads to the following question: 

\begin{question} \label{Q:BiIntegrable} Let $M$ be a (finite-dimensional) manifold and $\mathcal{A}, \mathcal{B}$ be compatible Poisson brackets on it. What bi-Hamiltonian systems $v = \mathcal{A} df = \mathcal{B} dg$ are bi-integrable? In other words, is there a (natural) way to construct a complete family of functions in involution w.r.t. both Poisson pencils $\mathcal{A}$ and $\mathcal{B}$? 
\end{question}

This question finds its roots in the study of integrable and bi-Hamiltonian systems on Lie algebras. While the details are beyond the scope of this paper, we can roughly sketch the path leading to this problem (for details see e.~g. \cite{BolsZhang}).   

\begin{itemize}

\item In 1978 A.\,S.~Mischenko  and  A.\,T.~Fomenko in \cite{MishchenkoFomenko78EulerEquations} suggested argument shift method  as  a  generalisation  of  S.\,V.~Manakov’s construction \cite{Manakov76}. For a Lie algebra $\mathfrak{g}$ and an element $a \in\mathfrak{g}^*$ they constructed a family of functions $\mathcal{F}_a$ in bi-involution w.r.t.  the standard linear Lie-Poisson bracket  $\left\{ , \right\}$  and a constant "frozen argument" bracket $\left\{ , \right\}_a$ on $\mathfrak{g}^*$. This family $\mathcal{F}_a$ is complete for any semisimple Lie algebra $\mathfrak{g}$ and any regular $a$.  

\item In 2004 S.\,T.~Sadetov in \cite{Sadetov2004} proved the Mischenko–Fomenko  conjecture that for any Lie algebra $\mathfrak{g}$ there exists a complete family of commuting polynomials w.r.t. the Lie-Poisson bracket  $\left\{ , \right\}$. 

\item In general, S.\,T.~Sadetov's family of function is not in bi-involution. Hence, in 2012 in A.\,V.~Bolsinov and P.~Zhang  in \cite{BolsZhang} proposed the following bi-Hamiltonian version of the Mischenko–Fomenko conjecture.

\begin{conjecture}[Generalised Argument Shift Conjecture, \cite{BolsZhang}] \label{Conj:GenArgConj} Let $\mathfrak{g}$ be a finite-dimensional Lie algebra. Then for every regular element $a \in \mathfrak{g}^*$,
there exists a complete family $\mathcal{G}_a$ of polynomials in bi-involution, i.e. in involution w.r.t. the
two brackets $\left\{ , \right\}$ and $\left\{ , \right\}_a$.\end{conjecture}

\end{itemize}

The Question~\ref{Q:BiIntegrable} is an analogue of the Generalised Argument Shift Conjecture  for arbitrary bi-Hamiltonian systems. Let us briefly explain how bi-integrable systems are related to bi-Lagrangian subspaces. If $f_1,\dots, f_k$ is a complete family of function in involution on a Poisson manifold $(M, \mathcal{A})$, then \[L_x = \operatorname{Span} \left\{ df_1(x), \dots, df_k(x) \right\} \subset (T_x^*M, \mathcal{A}_x) \] is a Lagrangian (i.e. maximal isotropic) subspace in the cotangent space. Similarly, if $\mathcal{P} = \left\{\mathcal{A} + \lambda \mathcal{B} \right\}$ is a pencil of compatible Poisson structures on a manifold $M$, then instead of complete family of functions in bi-involution we can consider integrable bi-Lagrangian distributions \[\mathcal{L} \subset (T^*M, \mathcal{P}).\] Simply speaking, in Question~\ref{Q:BiIntegrable} we search for a "natural way" to construct bi-Lagrangian subspaces in the cotangent spaces \[ L_x \subset \left( T_x^* M, \mathcal{P}_x\right), \qquad \mathcal{P}_x = \left\{\mathcal{A}_x + \lambda \mathcal{B}_x \right\}. \] By "natural" in this context, we imply that the bi-Lagrangian subspaces $L_x$ are independent of the choice of local coordinates on the manifold $M$. One straightforward approach to achieve this is to consider invariant subspaces within each cotangent bundle $\left(T_x^*M, \mathcal{P}_x\right)$. This leads us to the following question:

\begin{question} What bi-Lagrangian subspaces $L \subset (V, \left\{ A + \lambda B\right\})$ are $\operatorname{Aut}(V, \mathcal{P})$-invariant? \end{question}

We completely answer that question in Section~\ref{S:InvBiLagr}. We  present two contrasting examples to illustrate the potential and limitations of this approach. 

\begin{enumerate}

\item The Kronecker case, i.e. $\operatorname{rk} (A + \lambda B) = \operatorname{rk}B$ for all $\lambda \in \mathbb{C}$. In this case, as it is shown in \cite[Corollary 5]{BolsZhang} (see also Theorem~\ref{T:BiLagrKronPart}), there is only one bi-Lagrangian subspace, namely \[ L = \sum_{\lambda \in \bar{\mathbb{C}}} \operatorname{Ker} \left( A + \lambda B\right).\] In the context of bi-Hamiltonian systems, this result sheds light on A.S. Mishchenko and A.T. Fomenko's argument shift method \cite{MishchenkoFomenko78EulerEquations}: in the Kronecker case  the (local) Casimir functions of the Poisson brackets $\mathcal{A}_{\lambda} = \mathcal{A} + \lambda \mathcal{B}$ form a complete family of functions in bi-involution. 

\item The Jordan case, when also $\operatorname{Ker} B = 0$ and the spectrum of the recursion operator $P = B^{-1}A$ are $\frac{1}{2} \dim V$ pairs of distinct eigenvalues $\lambda_j$. There are no invariant bi-Lagrangian subspaces $L\subset (V, \left\{ A + \lambda B\right\})$.  Yet, for bi-Hamiltonian systems it is well-known (see e.g.
\cite[Corollary 9.10]{Kozlov23JKRealization}) that the eigenvalues $\lambda_j(x)$ are integrals of motion in bi-involution.  In this case, our approach falls short of establishing bi-integrability. This stems from the fact that the local automorphisms of a bi-Hamiltonian system may not realise the whole group $\operatorname{Aut}(V, \mathcal{P})$.

\end{enumerate}

To date, there is no known general integration method for bi-Hamiltonian systems in the Jordan case (i.e., with non-degenerate Poisson pencils) beyond utilizing the eigenvalues $\lambda_j(x)$ as integrals of motion. However, Theorem~\ref{T:InvarBilagr} demonstrates the existence of $\operatorname{Aut}(V, \mathcal{P})$-invariant subspaces in certain scenarios. These findings hold promise for the development of novel integration techniques for bi-Hamiltonian systems (see Section~\ref{SubS:BiIntProb}). However, a detailed exploration of this connection and its applications merits further investigation, which is a problem for another paper.

\subsection{Conventions and acknowledgements}

\textbf{Conventions.}  

\begin{enumerate}
    \item All vector spaces are finite-dimensional.

    \item Throughout this paper, we assume the field $\mathbb{K}$ to be either the complex numbers $\mathbb{C}$ or the real numbers $\mathbb{R}$, unless explicitly stated otherwise. Although, most of the results hold true for any algebraically closed field with characteristic zero.
    
    \item However, there's one exception for the real numbers:
    
    \begin{itemize}
        \item  General real case $\mathbb{K} = \mathbb{R}$ is studied in Section~\ref{S:RealCase}. Here no assumptions are made about the eigenvalues of operators. 
    
        \item In all other sections, when we consider the real case $\mathbb{K} = \mathbb{R}$, we require all eigenvalues $\lambda_i$ of operators to be real numbers.

    \end{itemize}

    \item $I_n$ denotes the $n\times n$ identity matrix.
    
    \item We use the following notation $\mathbb{KP}^1 = \bar{\mathbb{K}} = \mathbb{K} \cup \left\{ \infty \right\}$.

    \item For brevity, we'll use $\operatorname{Sp}(2n)$ to denote the symplectic group $\operatorname{Sp}(2n,\mathbb{K})$.

    \item We denote the Lagrangian Grassmanian for a $2n$-dimensional symplectic space as $\Lambda(n)$.

    \item We denote a bi-Lagrangian Grassmanian as $\operatorname{BLG}(V,\mathcal{P})$. For brevity, when considering a subspace  $U^{\perp}/U$ with its induced pencil, we write $\operatorname{BLG}(U^{\perp}/U)$.

\item Throughout this paper, "smooth" refers to:

\begin{itemize}
    \item \textbf{Real case}: $C^{\infty}$-smooth (infinitely differentiable). 
    \item \textbf{Complex case}:  complex-analytic (holomorphic).
\end{itemize}

\item Some property holds ``almost everywhere'' or ``at a generic point'' of a manifold $M$ if it holds on an open dense subset of $M$.
    
    \item We can encounter subsets of Grassmanians $\operatorname{Gr}(k, n)$ that are also algebraic subvarieties. Such subsets inherit two distinct topologies: the Zariski topology and the standard topology.
    
    \begin{itemize}
        \item When considering these subsets equipped with the standard topology, we employ the term ``\textbf{homeomorphic}'' (or ``\textbf{diffeomorphic}'' for smooth manifolds) to describe a bijective map preserving the basic topological structure. 
        \item  In contrast, the term ``\textbf{isomorphic}'' is used to establish one-to-one correspondences that preserve the algebraic structure of algebraic varieties.
    \end{itemize}

\item \textbf{JK} stands for Jordan-Kronecker (theorem/decomposition) throughout this paper.
    
\end{enumerate}

\par\medskip

\textbf{Acknowledgements.} The author would like to thank A.\,V.~Bolsinov and A.\,M.~Izosimov for useful comments. The study of bi-Lagrangian Grassmannians was initiated by A.\,V.~Bolsinov, who posed the fundamental questions regarding their structure.

\section{Basic definitions} \label{S:Basic}

 Let  $A, B$ be two skew-symmetric bilinear forms on a finite-dimensional vector space  $V$ over a field $\mathbb{K}$ (we are mostly interested in the cases $\mathbb{K} = \mathbb{C}$ and $\mathbb{R}$). We denote the corresponding pencil by \[\mathcal{P} = \left\{  A + \lambda B, \lambda \in \overline{\mathbb{K}} = \mathbb{K} \cup \{\infty\} \right\}.\] A form $A+\lambda B$ from the pencil $\mathcal{P}$ is also denoted by $A_{\lambda}$. We formally put $A_{\infty} = B$.

 \begin{definition} A pair $(V,\mathcal{P})$ consisting of a finite-dimensional vector space $V$ and a pencil of $2$-forms $\mathcal{P}$ on it will be called a \textbf{bi-Poisson vector space}. \end{definition} 

Two bi-Poisson vector spaces, $(V, \left\{A_{\lambda}\right\})$ and $(\hat{V}, \left\{\hat{A}_{\lambda}\right\})$ are \textbf{isomorphic} if there exists a linear isomorphism $f: V \to \hat{V}$ that preserves the pencils: \[\hat{A}_{\lambda}(fu, fv) = A_{\lambda}(u,v), \qquad \forall \lambda \in \bar{\mathbb{K}}, \quad \forall u, v\in V.\] In the same space $(V,\mathcal{P})$ a map $f: V \to V$ satisfying these conditions is called a \textbf{bi-Poisson automorphism}. The collection of all such automorphisms forms a group, denoted by $\operatorname{Aut}(V, \mathcal{P})$.

\subsection{Jordan--Kronecker theorem} \label{S:JKTheorem} 

First, let us recall the \textbf{Jordan–Kronecker Canonical Form} for a pair of skew-symmetric forms. The next theorem that describes it, which we call the Jordan--Kronecker theorem, is a classical result that goes back to Weierstrass and Kronecker. Its proof can be found in Thompson's paper \cite{Thompson91}, which builds upon the work of Gantmacher \cite{Gantmacher88}. This result has also been established in other sources, including \cite{Gurevich50} and the references cited in \cite{Lancaster05}.

\begin{theorem}[Jordan--Kronecker theorem]\label{T:Jordan-Kronecker_theorem}
Let $A$ and $B$ be skew-symmetric bilinear forms on a
finite-dimension vector space $V$ over an algebraically closed field $\mathbb{K}$ with $\textmd{char }  \mathbb{K} =0$. Then, there exists a basis for $V$ such that the matrices of both forms $A$ and $B$ are block-diagonal
matrices:

{\footnotesize
$$
A =
\begin{pmatrix}
A_1 &     &        &      \\
    & A_2 &        &      \\
    &     & \ddots &      \\
    &     &        & A_k  \\
\end{pmatrix}
\quad  B=
\begin{pmatrix}
B_1 &     &        &      \\
    & B_2 &        &      \\
    &     & \ddots &      \\
    &     &        & B_k  \\
\end{pmatrix}
$$
}

where each pair of corresponding blocks $A_i$ and $B_i$ is one of
the following:

\begin{itemize}

\item Jordan block with eigenvalue $\lambda_i \in \mathbb{K}$: {\scriptsize  \[A_i =\left(
\begin{array}{c|c}
  0 & \begin{matrix}
   \lambda_i &1&        & \\
      & \lambda_i & \ddots &     \\
      &        & \ddots & 1  \\
      &        &        & \lambda_i   \\
    \end{matrix} \\
  \hline
  \begin{matrix}
  \minus\lambda_i  &        &   & \\
  \minus1   & \minus\lambda_i &     &\\
      & \ddots & \ddots &  \\
      &        & \minus1   & \minus\lambda_i \\
  \end{matrix} & 0
 \end{array}
 \right)
\quad  B_i= \left(
\begin{array}{c|c}
  0 & \begin{matrix}
    1 & &        & \\
      & 1 &  &     \\
      &        & \ddots &   \\
      &        &        & 1   \\
    \end{matrix} \\
  \hline
  \begin{matrix}
  \minus1  &        &   & \\
     & \minus1 &     &\\
      &  & \ddots &  \\
      &        &    & \minus1 \\
  \end{matrix} & 0
 \end{array}
 \right)
\]} \item Jordan block with eigenvalue $\infty$ {\scriptsize \[
A_i = \left(
\begin{array}{c|c}
  0 & \begin{matrix}
   1 & &        & \\
      &1 &  &     \\
      &        & \ddots &   \\
      &        &        & 1   \\
    \end{matrix} \\
  \hline
  \begin{matrix}
  \minus1  &        &   & \\
     & \minus1 &     &\\
      &  & \ddots &  \\
      &        &    & \minus1 \\
  \end{matrix} & 0
 \end{array}
 \right)
\quad B_i = \left(
\begin{array}{c|c}
  0 & \begin{matrix}
    0 & 1&        & \\
      & 0 & \ddots &     \\
      &        & \ddots & 1  \\
      &        &        & 0   \\
    \end{matrix} \\
  \hline
  \begin{matrix}
  0  &        &   & \\
  \minus1   & 0 &     &\\
      & \ddots & \ddots &  \\
      &        & \minus1   & 0 \\
  \end{matrix} & 0
 \end{array}
 \right)
 \] } \item   Kronecker block {\scriptsize \[ A_i = \left(
\begin{array}{c|c}
  0 & \begin{matrix}
   1 & 0      &        &     \\
      & \ddots & \ddots &     \\
      &        & 1    &  0  \\
    \end{matrix} \\
  \hline
  \begin{matrix}
  \minus1  &        &    \\
  0   & \ddots &    \\
      & \ddots & \minus1 \\
      &        & 0  \\
  \end{matrix} & 0
 \end{array}
 \right) \quad  B_i= \left(
\begin{array}{c|c}
  0 & \begin{matrix}
    0 & 1      &        &     \\
      & \ddots & \ddots &     \\
      &        &   0    & 1  \\
    \end{matrix} \\
  \hline
  \begin{matrix}
  0  &        &    \\
  \minus1   & \ddots &    \\
      & \ddots & 0 \\
      &        & \minus1  \\
  \end{matrix} & 0
 \end{array}
 \right)
 \] }
 \end{itemize}

\end{theorem}

Each Kronecker block is a $(2k_i-1) \times (2k_i-1)$ block, where
$k_i \in \mathbb{N}$. If $k_i=1$, then the blocks are $1\times 1$
zero matrices \[A_i =
\begin{pmatrix}
0
\end{pmatrix}, \qquad B_i=
\begin{pmatrix}
0
\end{pmatrix}.\]

\begin{remark} Section~\ref{S:RealCase} describes the real analog of the JK theorem. \end{remark}

\subsubsection{Jordan--Kronecker decomposition}

\begin{definition}
We call a decomposition of $(V, \mathcal{P})$ into a sum of subspaces corresponding to the Jordan and Kronecker blocks in $\mathcal{P}$ a \textbf{Jordan-Kronecker decomposition}:  
\begin{equation} \label{Eq:JKDecomp} (V, \mathcal{P}) = \bigoplus_{j=1}^{S}\left(\bigoplus_{k=1}^{N_j} \mathcal{J}_{\lambda_j, 2n_{j,k}} \right) \oplus  \bigoplus_{i=1}^q \mathcal{K}_{2k_i+1}.\end{equation}
\end{definition} We use the following notation to represent specific types of bi-Poisson vector spaces:
\begin{itemize}
    \item $\mathcal{J}_{\lambda, 2n}$ corresponds to a Jordan block of dimension 2n and eigenvalue $\lambda$.

    \item $\mathcal{K}_{2k+1}$ corresponds to a Kronecker block of dimension $(2k+1)$. 
\end{itemize}

\begin{remark} The JK decomposition \eqref{Eq:JKDecomp} is not unique, just as the basis in the JK theorem. But the sizes and types of blocks in the JK theorem are uniquely defined. \end{remark} 

The sum \eqref{Eq:JKDecomp} is \textbf{bi-orthogonal}, meaning that it is a direct sum of subspaces that are pairwise orthogonal w.r..t all forms from the pencil $\mathcal{P}$.

\begin{definition}There are two special cases of bi-Poisson vector spaces and pencils based on the decomposition~\eqref{Eq:JKDecomp}):

\begin{itemize}

\item If the JK decomposition consists solely of Jordan blocks, then  $(V,\mathcal{P})$ is a \textbf{Jordan bi-Poisson space} (and  $\mathcal{P}$ is  a  \textbf{Jordan pencil}).

\item If the JK decomposition consists only of Kronecker blocks, then  $(V,\mathcal{P})$ is a \textbf{Kronecker bi-Poisson space} (and  $\mathcal{P}$ is  a  \textbf{Kronecker pencil}).

\end{itemize}
\end{definition}

\subsubsection{Standard basis}

The basis from the JK theorem will be called a standard basis. For instance: 
\begin{itemize}

\item A basis $e_1,\dots, e_n, f_1,\dots, f_n$ is considered the \textbf{standard basis for a Jordan block} of dimension $2n$ and eigenvalue $\lambda_0$ if the pencil P has the following form: \begin{equation} \label{Eq:StandJord} A + \lambda B = \left( \begin{matrix} 0 & J_{\lambda_0 + \lambda} \\ -J_{\lambda_0 + \lambda}^T & 0 \end{matrix}  \right), \qquad J_{\mu} = \left(\begin{matrix} \mu & 1 & &  \\  & \ddots & \ddots & & \\   &  & \ddots & 1 \\   & & & \mu \end{matrix}  \right)\end{equation}

\item A basis $e_1,\dots, e_n, f_0,\dots, f_n$ is considered the \textbf{standard basis for a Kronecker block} if the matrices of the forms in the pencil $\mathcal{P}$ take the form \begin{equation} \label{Eq:StandKron} A + \lambda B = \left( \begin{matrix} 0 & P \\ -P^T & 0 \end{matrix}  \right), \qquad P = \left(\begin{matrix} 1 & \lambda & & \\ & 1 & \ddots & & \\ & & \ddots & \ddots & \\ & & & 1 & \lambda  \end{matrix}  \right)\end{equation} 

\end{itemize}

\subsection{Rank and eigenvalues of a linear pencil}

\begin{definition} The \textbf{rank} of a pencil $\mathcal{P} = \left\{ A_{\lambda} = A + \lambda B\right\}$ is \[ \operatorname{rk} \mathcal{P} = \max_{\lambda \in \bar{\mathbb{K}}} \, \operatorname{rk} A_{\lambda}.\] A form $A_\lambda$ is \textbf{regular} if $\operatorname{rk}A_{\lambda} = \operatorname{rk} \mathcal{P}$. Non-regular forms $A_{\lambda}$ are called \textbf{singular}. \end{definition}

\begin{definition}  Consider JK decomposition~\eqref{Eq:JKDecomp}. The values $\lambda_j \in \bar{\mathbb{K}}$ are \textbf{eigenvalues} of $\mathcal{P}$. The \textbf{spectrum} of $\mathcal{P}$ consists of its eigenvlaues: \[ \sigma(\mathcal{P}) = \left\{\lambda_1, \dots, \lambda_S \right\}.\] \end{definition}

Due to our sign convention in the JK theorem, singular forms correspond to the eigenvalues \textbf{multiplied by $-1$}. Formally,
\begin{assertion} \label{A:RegForms} $A_{\lambda_0} \in \mathcal{P}$ is singular if and only if $-\lambda_0 \in \sigma(\mathcal{P})$. \end{assertion}

Below we use the following simple statement.

\begin{assertion} \label{A:DimVRnP} For any bi-Poisson space $(V, \mathcal{P})$ the number of Kronecker blocks in the JK decomposition is $\dim V - \operatorname{rk} \mathcal{P}$. Thus \[ \dim V - \frac{1}{2}\operatorname{rk} \mathcal{P} = \dim K + \frac{1}{2} \dim V_J.\]

\end{assertion}

\subsection{Admissible subspaces} \label{S:Admissible}

Let $\mathcal{P}= \left\{A_{\lambda} \right\}$ be a linear pencil on $V$. For a subspace $U\subset (V, \mathcal{P})$ we denote by $U^{\perp_{A_\lambda}}$ or $U^{\perp_{\lambda}}$ its skew-orthogonal complement w.r.t. the form $A_{\lambda}$: \[ U^{\perp_{\lambda}} = \left\{ v\in V \, \, \bigr| \, \, A_{\lambda}(v, U) = 0\right\}.\]

\begin{definition}  A subspace $U \subset  (V, \mathcal{L})$ is \textbf{admissible} if its skew-orthogonal complements $U^{\perp_{A_\lambda}}$ coincide for almost all forms $A_\lambda$ of the pencil $\mathcal{P}$. We denote this complement as $U^{\perp_{\mathcal{P}}}$ or $U^{\perp}$. \end{definition}

Recall that we can regard a bilinear form $A_{\lambda}$ as a linear map $A_{\lambda}: V \to V^*$. For a subspace of a bi-Poisson space $U \subset (V, \mathcal{P})$ we have \begin{equation} \label{Eq:PerpAnn} U^{\perp_{\lambda}} = \operatorname{Ann} \left( A_{\lambda} U \right).\end{equation}  Hence, a subspace $U \subset (V, \mathcal{P})$ is admissible if and only if the images $A_{\lambda}(U)$ coincide for a generic $\lambda \in \bar{\mathbb{C}}$. We denote that image as $\mathcal{P}(U)$. Below we need the following statement.

\begin{assertion} \label{A:AddmImagePen} For any admissible subspace $U\subset (V, \mathcal{P})$ and any $\lambda \in \bar{\mathbb{C}}$ the following holds: \begin{equation} \label{Eq:EquivAdm} A_{\lambda} (U) \subseteq \mathcal{P}(U) \qquad \Leftrightarrow \qquad U^{\perp_{\lambda}} \supseteq U^{\perp}. \end{equation}
\end{assertion}

\begin{proof}[Proof of Assertion~\ref{A:AddmImagePen}]   Without loss of generality $A(U) = B(U) = \mathcal{P}(U)$. Then it is obvious that $A_{\lambda}(U) \subset \mathcal{P}(U)$. The conditions in \eqref{Eq:EquivAdm}  are equivalent by \eqref{Eq:PerpAnn}. Assertion~\ref{A:AddmImagePen} is proved.  \end{proof}

In the Jordan case, where the regular forms are non-degenerate, admissible subspaces exhibit a particularly simple characterization.  Simply speaking, ``admissible = $P$-invariant''.   (The proof of the next statement is straightforward and omitted here.)

\begin{assertion} \label{A:AdmPInv}
Let $\mathcal{P} = \left\{ A + \lambda B\right\}$ be a linear pencil on $V$. Assume that $B$ is nondegenerate (i.e. $\operatorname{Ker} B = 0$) and let $P = B^{-1}A$ be the recursion operator. A subspace $U \subset (V, \mathcal{P})$ is admissible if and only if $U$ is $P$-invariant.
\end{assertion}

\subsection{Core and mantle subspaces} \label{SubS:CoreMantle} 

Consider a JK decomposition \eqref{Eq:JKDecomp}. Denote by $(V_J, \mathcal{P}_J)$ and $(V_K, \mathcal{P}_K)$ the sum of all Jordan and all Kronecker blocks respectively \[ (V_J, \mathcal{P}_J) = \bigoplus_{j=1}^{S}\left(\bigoplus_{k=1}^{N_j} \mathcal{J}_{\lambda_j, 2n_{j,k}} \right), \qquad (V_K, \mathcal{P}_K) = \bigoplus_{i=1}^q \mathcal{K}_{2k_i+1}.\] The decomposition \[ V = V_J \oplus V_K \] for a linear pencil $\mathcal{P}$ is not as natural as it seems. Informally, the Jordan part  $V_J$ is "sandwiched" between "larger" and "smaller halves" of the Kronecker blocks. To formalize this intuition and pave the way for further investigation, we introduce two crucial invariant subspaces.

\begin{definition} \label{Def:CoreMantle} Consider a pencil of skew-symmetric forms $\left\{ A_{\lambda} = A + \lambda B\right\}$.

\begin{enumerate}

\item The \textbf{core} subspace is  the sum of the kernels of all regular forms in the pencil:\[ K = \sum_{-\lambda \not \in \sigma(\mathcal{P})} \operatorname{Ker} A_{\lambda}. \] 

\item The \textbf{mantle} subspace is the skew-orthogonal complement to the core (w.r.t. any regular form  $A_{\lambda}$): \begin{equation} \label{Eq:Mantle} M = K^{\perp}. \end{equation}

\end{enumerate}

\end{definition} 

\begin{remark} The mantle $M$, given by \eqref{Eq:Mantle}, is well-defined, since by Corollary~\ref{Cor:CoreMantle} below the core $K$ is an admissible space. \end{remark}

Now fix any basis from the JK theorem and let us describe the core and mantle subspaces in it. First, let us describe $\operatorname{Ker} A_{\lambda}$ for each type of block. This statement is readily verified for matrices in the standard basis, where they take the form of \eqref{Eq:StandJord} or \eqref{Eq:StandKron}.

\begin{proposition} \label{Prop:KernelJordKronBlocks}

\begin{enumerate}

\item For one Jordan block with eigenvalue $\lambda_0$ we have $\operatorname{Ker} (A + \lambda B ) = 0$ except when $\lambda = - \lambda_0$. This holds true even if $\lambda_0 = \infty$. 

\item  In the standard basis  $e_1, \dots, e_n, f_0, \dots, f_n$ for one Kronecker block we have \[ \operatorname{Ker} \left( A + \lambda B \right) = \langle f_n - \lambda f_{n-1} + \lambda^2 f_{n-2} - \dots \rangle. \]

\end{enumerate}

\end{proposition}

As a consequence we get another easy description of the core and mantle.

\begin{corollary} \label{Cor:CoreMantle} For any JK decomposition we have the following. 

\begin{enumerate}

\item  The core subspace K is spanned by vectors corresponding to the down-right zero matrices of Kronecker blocks, like this one:  \begin{equation} \label{Eq:CoreKronBlock} A_i + \lambda B_i = \left(
\begin{array}{c|c}
  0 & \begin{matrix}
   1 & \lambda      &        &     \\
      & \ddots & \ddots &     \\
      &        & 1    & \lambda  \\
    \end{matrix} \\
  \hline
  \begin{matrix}
  \minus1  &        &    \\
  \minus \lambda   & \ddots &    \\
      & \ddots & \minus1 \\
      &        & \minus\lambda  \\
  \end{matrix} &\cellcolor{blue!25} 0 
 \end{array}
 \right).
 \end{equation} In particular, the core $K$ is an admissible space.

\item The mantle subspace is the core plus all Jordan blocks: \[ M = K \oplus V_J. \]

\end{enumerate}

\end{corollary}

Now, the analogy with the Earth in Fig.~\ref{Fig:CoreMantle} becomes obvious. The mantle $M$ ``wraps around'' the core $K$, the rest of Kronecker blocks is ``kind of a crust''. 

\begin{figure}[ht!]
  \centering
    \includegraphics[width=0.3\textwidth]{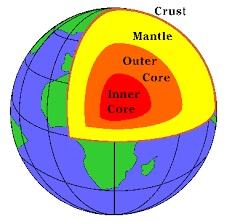}
      \caption{Earth structure.}
      \label{Fig:CoreMantle}
\end{figure}

\section{Bi-Lagrangian Grassmanian: definition and basic properties} \label{S:BiLagrLinear} \label{S:BiLagraGr}

\subsection{Bi-Lagrangian and bi-isotropic subspaces}

Let $B$ be a skew-symmetric form on a vector space $V$. Recall that a subspace $U \subset V$ is called 

\begin{itemize}

\item \textbf{isotropic} if $B(u,v) = 0$ for all $u,v \in V$.

\item \textbf{Lagrangian} or \textbf{maximal isotropic} if it is isotropic and \[ \dim U = \dim V - \frac{1}{2} \operatorname{rk} B.\]

\end{itemize} These well-known notions of isotropic and  Lagrangian subspaces can be naturally extend to the context of bi-Poisson vector spaces. However, we have to distinguish two types of bi-isotropic subspaces: ``inclusion-maximal'' (maximal bi-isotropic) and ``dimension-maximal'' (bi-Lagrangian). Let $\mathcal{P} = \left\{ A + \lambda B\right\}$ be a pencil of $2$-forms on a vector space $V$.

\begin{definition} A subspace $U \subset V$ of a bi-Poisson vector space $(V, \mathcal{P})$ is called 

\begin{itemize}

\item \textbf{bi-isotropic} if $A_{\lambda}(u, v) = 0$ for all $u, v\in V$ and all $A_{\lambda} \in \mathcal{P}$;

\item \textbf{maximal bi-isotropic} if it is bi-isotropic and is not contained in any larger bi-isotropic subspace;

\item  \textbf{bi-Lagrangian} if it is bi-isotropic and \[\dim U = \dim V - \frac{1}{2}  \operatorname{rk} \mathcal{P}.\]

\end{itemize}

 \end{definition}

Subspaces $U \subset (V, \mathcal{P})$ that are isotropic  w.r.t. forms $A$ and $B$ are also isotropic w.r.t. any their linear combination $\mu A + \lambda B$. So, in practice, it is easier to check that a subspace $U$ is bi-isotropic as being ``twice isotropic''.

 \begin{example} The core subspace $K$ (see Definition~\ref{Def:CoreMantle}) is bi-isotropic for any bi-Poisson vector space  $(V, \mathcal{P})$, since both forms $A$ and $B$ vanish on it. It is bi-Lagrangian if and only if $(V, \mathcal{P})$ is Kronecker. 
 \end{example} 

Bi-Lagrangian subspaces can be also characterized as subspaces that are Lagrangian w.r.t. all regular forms of the pencil $\mathcal{P}$. Indeed, $ \operatorname{rk} \mathcal{P} =  \operatorname{rk} A_{\lambda}$ for regular forms $A_{\lambda}$ and $\operatorname{rk} \mathcal{P} >  \operatorname{rk} A_{\mu}$ for singular forms $A_{\mu}$.  It is easy to see that the following holds. 

\begin{assertion} \label{A:BiLagrMaxIsotrRegular}  Let $(V, \mathcal{P})$ be a bi-Poisson vector space.
\begin{enumerate}

\item A bi-Lagrangian subspace $L \subset (V, \mathcal{P})$ is Lagrangian (i.e. maximal isotropic) w.r.t. a form $A_{\lambda} \in \mathcal{P}$ if and only if $A_{\lambda}$ is regular. 

\item If a subspace $U \subset (V, \mathcal{P})$ is Lagrangian w.r.t. any two forms $A_{\lambda}, A_{\mu} \in \mathcal{P}$, then $U$ is a bi-Lagrangian subspace.

\end{enumerate}

\end{assertion}

Since a bi-Lagrangian subspace $L$ has maximal possible dimension, it is maximal bi-isotropic. The next example shows that the inverse is not true.  

\begin{example} Let $e_1, \dots, e_n, f_0, \dots, f_n$ be a standard basis for a Kronecker block, i.e. the matrices of the forms  are \[ A + \lambda B = \left(
\begin{array}{c|c}
  \cellcolor{blue!25} 0 & \begin{matrix}
   1 & \lambda      &        &     \\
      & \ddots & \ddots &     \\
      &        & 1    & \lambda  \\
    \end{matrix} \\
  \hline
  \begin{matrix}
  \minus1  &        &    \\
  \minus \lambda   & \ddots &    \\
      & \ddots & \minus1 \\
      &        & \minus\lambda  \\
  \end{matrix} &0 
 \end{array}
 \right). 
 \] Then the subspace $\operatorname{Span}\langle e_1, \dots, e_n \rangle$, corresponding to the shaded block of $A + \lambda B$, is maximal bi-isotropic but not bi-Lagrangian.\end{example}
 
The following statement shows the difference between maximal bi-isotropic and bi-Lagrangian subspaces. Recall that we defined admissible subspaces in Section~\ref{S:Admissible}.

\begin{assertion} \label{A:BiLagrMaxBiIsotAdm}
A subspace $U\subset (V, \mathcal{P})$ is bi-Lagrangian if and only if it is maximal bi-isotropic and admissible.
\end{assertion}

\begin{proof}[Proof of Assertion~\ref{A:BiLagrMaxBiIsotAdm}] 

\begin{itemize}

\item[$(\Rightarrow)$] Let $U$ be a bi-Lagrangian subspace.  By Assertion~\ref{A:BiLagrMaxIsotrRegular}  $U$ is Lagrangian w.r.t. all regular forms $A_{\lambda}$ and hence $U^{\perp_{\lambda}} = U$. Thus, $U$ is admissible.

\item[$(\Leftarrow)$]  Let $U$ be maximal bi-isotropic and admissible. Then $U \subseteq U^{\perp}$. Since we can't extend $U$ to a bigger bi-isotropic subspace, $U^{\perp} = U$. Thus, $U$ is Lagrangian w.r.t. almost all forms $A_{\lambda} \in \mathcal{P}$. By Assertion~\ref{A:BiLagrMaxIsotrRegular} $U$ is bi-Lagrangian.

\end{itemize}

Assertion~\ref{A:BiLagrMaxBiIsotAdm} is proved.  \end{proof}

We discuss when a bi-isotropic subspace can be extended to a bi-Lagrangian subspace in Section~\ref{SubS:BiIsotr2BiLagr}. We are mostly interested in the structure of the set of bi-Lagrangian subspaces.

\begin{definition} The collection of all bi-Lagrangian subspaces of $(V, \mathcal{P})$ is called a \textbf{bi-Lagrangian Grassmannian} and denoted by $\operatorname{BLG}\left(V, \mathcal{P}\right)$. \end{definition}

\subsubsection{Existence of bi-Lagrangian subspaces}

First, we prove that  in each bi-Poisson space there is a bi-Lagrangian subspace. The next obvious statement shows that it suffices to find a bi-Lagrangian subspace in each Jordan and Kronecker block.

\begin{assertion} \label{A:DecompBiLagr}
Let $(V, \mathcal{P}) = \bigoplus_{i=1}^N \left(V_i, \mathcal{P}_i\right)$ be a sum of bi-orthogonal subspaces. Assume that $L \subset (V, \mathcal{P})$ is a direct sum of subspaces \[ L = \bigoplus_i L_i, \qquad L_i = L \cap (V_i, \mathcal{P}_i). \] The subspace $L$ is a bi-Lagrangian (bi-isotropic, maximal bi-isotropic) subset of $(V, \mathcal{P})$ if and only if each $L_i$ is a bi-Lagrangian (respectively, bi-isotropic, maximal bi-isotropic) subspace of $(V_i, \mathcal{P}_i)$.
\end{assertion}

Now it is easy to prove that $\operatorname{BLG}(V, \mathcal{P}) \not = \emptyset$. 

\begin{assertion} \label{A:ExistBiLagr} For any bi-Poisson space $(V, \mathcal{P})$ there exists a bi-Lagrangian subspace $L \subset (V, \mathcal{P})$.\end{assertion}

\begin{proof}[Proof of Assertion~\ref{A:ExistBiLagr}] Note that the matrices of each Jordan and Kronecker block have the form \[ \left( \begin{array}{c|c} 0 & X \\ \hline -X^T & \cellcolor{blue!25} 0 \end{array}  \right) \] for some matrix $X$. The subspace corresponding to to the shaded lower-right zero block is a bi-Lagrangian subspace for each block. By Assertion~\ref{A:DecompBiLagr} the sum of these subspaces is a bi-Lagrangian subspace of $(V, \mathcal{P})$. Assertion~\ref{A:ExistBiLagr} is proved. \end{proof}

We discuss the structure of $\operatorname{BLG}\left(V, \mathcal{P}\right)$  in the next sections.

\subsubsection{Bi-Lagrangian Grassmanian as an algebraic variety}

Bi-Lagrangian Grassmannians are subspaces of the corresponding Grassmanians $\operatorname{Gr}_k{(V)}$ and we can study them as topological subpaces with the induced topology. Unlike Lagrangian Grassmanians, which are smooth manifolds, a bi-Lagrangian Grassmanian has a natural structure of an algebraic variety. 

\begin{lemma} \label{L:BiLagrProjVar} A bi-Lagrangian Grassmannian $\operatorname{BLG}\left(V, \mathcal{P}\right)$ is a projective subvariety of $\operatorname{Gr}(n-k, n)$, where $n= \dim V$ and $k = \frac{1}{2} \operatorname{rk} \mathcal{P}$. \end{lemma}

\begin{corollary} Bi-Lagrangian Grassmanians $\operatorname{BLG}(V, \mathcal{P})$ are compact. \end{corollary}

We use the following simple statement.

\begin{assertion} \label{A:ProdGrass} Let $k_1+k_2 = k$ and $n_1 +n_2 =n$. Then $\operatorname{Gr}(k_1, n_1) \times \operatorname{Gr}(k_2, n_2)$ is a subvariety of $\operatorname{Gr}(k, n)$.  \end{assertion} 

\begin{proof}[Proof of Assertion~\ref{A:ProdGrass}] The subvariety of $k$-planes whose intersection with a fixed $n_1$-plane has dimension at least $k_1$ is the closure of a specific Schubert cell. The product of Grassmanians is the intersection of that subavariety with a similar one, the space of $k$-planes whose intersection with a fixed $n_2$-plane has dimension at least $k_2$. Assertion~\ref{A:ProdGrass} is proved. \end{proof}

\begin{proof}[Proof of Lemma~\ref{L:BiLagrProjVar}] It is well-known that a Lagrangian Grassmanian $\Lambda(k)$ is a smooth projective subvariety of the Grassmannian $\operatorname{Gr}(k, 2k)$ (see e.g.\cite{Kolhatkar04}). For a Poisson vector space $(V, B)$ any 
maximal isotropic subspace has the form $L \oplus \operatorname{Ker} B$, where $L$ is a Lagrangian subspace of a symplectic space. Thus, the set of maximal isotropic subspaces  (which we also call a Lagrangian Grassmanian) is a subvariety \[ \Lambda(k) \times \operatorname{Gr}(n-2k, n-2k) \subset \operatorname{Gr}(k, 2k)\times \operatorname{Gr}(n-2k, n-2k).\] By Assertion~\ref{A:ProdGrass} it is also a subvariety of $\operatorname{Gr}(n-k, n)$. By Assertion~\ref{A:BiLagrMaxIsotrRegular} the bi-Lagrangian Grassmanian $\operatorname{BLG}\left(V, \mathcal{P}\right)$ is the intersection of two such Lagrangian Grassmanians corresponding to any two regular forms $A_\lambda, A_\mu \in \mathcal{P}$, and thus it is also a projective subvariety of $\operatorname{Gr}(n-k, n)$. Lemma~\ref{L:BiLagrProjVar} is proved. \end{proof}

\subsection{Kronecker part of bi-Lagrangian subspaces}

In this section we show that a bi-Lagrangian Grassmannian does not depend on Kronecker blocks.

\begin{theorem} \label{T:BiLagrKronPart} A subspace $L \subset (V, \mathcal{P})$ is bi-Lagrangian if and only if \[ L = K \oplus L_J,\]  where $K$ is the core subspace and  $L_J$ is a bi-Lagrangian subspace of the sum of Jordan blocks $V_J$. In particular, the core $K$ is the only bi-Lagrangian subspace  of the sum of Kronecker blocks $V_K$. \end{theorem}

Before proving Theorem~\ref{T:BiLagrKronPart} at the end of this section, we present two corollaries that illustrate its applications. Denote by $\mathcal{P}_{K}$ and $\mathcal{P}_J$ the restrictions of $\mathcal{P}$ on the sums of Kronecker and Jordan blocks $V_K$ and $V_J$ respectively. 

\begin{corollary} \label{C:IsomJordPart} A bi-Lagrangian Grassmanian is isomorphic (as a projective variety) to the bi-Lagrangian Grassmanian for its Jordan blocks\[ \operatorname{BLG}\left( V, \mathcal{P}\right) 
\approx \operatorname{BLG}(V_J, \mathcal{P}_J).\] If there are no Jordan blocks, then we formally assume that $\operatorname{BLG}(V_J, \mathcal{P}_J)$ consists of one point. 
\end{corollary}

\begin{proof}[Proof of Corollary~\ref{C:IsomJordPart}] By Theorem~\ref{T:BiLagrKronPart} \[ \operatorname{BLG}\left(V, \mathcal{P}\right) \approx \operatorname{BLG}(V_K, \mathcal{P}_{K}) \times \operatorname{BLG}(V_J, \mathcal{P}_{J})\]  and $\operatorname{BLG}(V_K, \mathcal{P}_{K}) $ consists of one point (the core $K$). Corollary~\ref{C:IsomJordPart} is proved. \end{proof}

Theorem~\ref{T:BiLagrKronPart} also gives us a simple criterion when there is a unique bi-Lagrangian subspace. 

\begin{corollary} \label{C:UniqBiLagrKron} Let $(V, \mathcal{P})$ be a Poisson vector space. The following conditions are equivalent:

\begin{enumerate}

\item $\mathcal{P}$ is of Kronecker type.

\item The core subspace $K$ is the only bi-Lagrange subspace in $(V, \mathcal{P})$. 

\end{enumerate}
  
\end{corollary}

\begin{proof}[Proof of Corollary~\ref{C:UniqBiLagrKron}]

In the Kronecker case the core subspace $K$ is the only bi-Lagrangian subspace by Theorem~\ref{T:BiLagrKronPart}. If there is at least one Jordan block in the JK decomposition, then we can easily construct several bi-Lagrangian subspaces similarly to Assertion~\ref{A:ExistBiLagr}. For example, for a Jordan block with a standard basis $e_1, \dots, e_{n_i}, f_1, \dots, f_{n_i}$ we can take  either $\operatorname{Span}(f_1, \dots, f_{n_i})$ or $\operatorname{Span}(e_1, \dots, e_{n_i})$ as a bi-Lagrangian subspace. Corollary~\ref{C:UniqBiLagrKron} is proved. \end{proof}

In order to prove Theorem~\ref{T:BiLagrKronPart} we use the following simple fact. 

\begin{assertion} \label{A:MaxIsotrKer} Let  $\omega$ be a bilinear skew-symmetric form on a space $V$. Then any maximally isotropic subspace  $L \subset V$ contains the kernel of the form $\omega$:
\[\operatorname{Ker} \omega \subset L.\] \end{assertion}

Now we can prove that any bi-Lagrangian subpace L ``is between K and M''. 

\begin{lemma} \label{L:KLM}  Any bi-Lagrangian subspace $L \subset (V, \mathcal{P})$ contains the core subspace $K$ and is contained in the mantle subspace: \[ K \subset L \subset M.\] \end{lemma}

\begin{proof}[Proof of Lemma~\ref{L:KLM}] By Assertion~\ref{A:BiLagrMaxIsotrRegular} a bi-Lagrangian subspace is maximal isotropic w.r.t. all regular forms $A_{\lambda}$. By  Assertion~\ref{A:MaxIsotrKer} it contains the kernels of regular forms $\operatorname{Ker} A_{\lambda}$ and, thus, also  contains their sum, i.e. the core subspace $K$. Since $K \subset L$ and $L$ is bi-Lagrangian, for any regular form $A_{\mu}$ we have $L = L^{\perp_{\mu}} \subset K^{\perp_{\mu}} = M$. Lemma~\ref{L:KLM} is proved.  \end{proof}

\begin{proof}[Proof of Theorem~\ref{T:BiLagrKronPart}] 

\begin{itemize}

\item[$\left(\Rightarrow\right)$] Let $L_J \subset V_J$ be a bi-Lagrangian subspace. Then the sum $K \oplus V$ is bi-Lagrangian in $V$, since $K$ is bi-Lagrangian in $V_K$ and the terms in the sum $V = V_K \oplus V_J$ are bi-orthogonal. 

\item[$\left(\Leftarrow\right)$] Let $L$ be a bi-Lagrangian subspace in $V$. Recall that $M/K \approx V_J$ and thus by Lemma~\ref{L:KLM} we have $L = K\oplus L/K$ where $L/ K \subset M/K$. The restriction of all forms on $L$ is trivial, hence $L/ K$ is bi-isotropic. By Assertion~\ref{A:DimVRnP},  \[ \dim L = \dim V - \frac{1}{2}  \operatorname{rk} \mathcal{P} = \dim K + \frac{1}{2} \dim V_J.\] Thus, $\dim L/K = \frac{1}{2} \dim V_J$, and $L/K$ is a bi-Lagrangian subspace of $M/K \approx V_J$. 

\end{itemize}

Theorem~\ref{T:BiLagrKronPart} is proved. 
\end{proof}

\subsection{Eigendecomposition in the Jordan case}

Let us now consider the Jordan case. In this section we reduce it to the case of one eigenvalue. Namely, we prove the following.

\begin{theorem} \label{T:JordaMultEigen} Assume that a bi-Poisson space $(V, \mathcal{P})$ is a sum of Jordan blocks: \begin{equation} \label{Eq:GenJordJKDec}  (V, \mathcal{P}) = \bigoplus_{j=1}^{S}\left(\bigoplus_{k=1}^{N_j} \mathcal{J}_{\lambda_j, 2n_{j,k}} \right).\end{equation} Denote by $\mathcal{J}_{\lambda_j} =\bigoplus_{k=1}^{N_j} \mathcal{J}_{\lambda_j, 2n_{j,k}}$ the sum of all Jordan blocks with eigenvalue $\lambda_j$. Then any bi-Lagrangian subspace $L \subset (V, \mathcal{P})$ is a sum of bi-Lagrangian subspaces $L_j \subset \mathcal{J}_{\lambda_j} $: \[ L = \bigoplus_{j=1}^S L_{j}, \qquad L_{j} = L \cap \mathcal{J}_{\lambda_j}.\] Thus, the bi-Lagragnain Grassmanian is isomorphic to the direct product: \begin{equation}  \label{Eq:DecompJordBiLagr} \operatorname{BLG}\left( V, \mathcal{P}\right) \approx \prod_{j=1}^S  \operatorname{BLG}\left(\mathcal{J}_{\lambda_j} \right).\end{equation} \end{theorem}

\begin{remark} In the real case $\mathbb{K} = \mathbb{R}$ Theorem~\ref{T:JordaMultEigen} holds if all eigenvalues are real $\lambda_j \in \mathbb{R}$. The general real case is studied in Section~\ref{S:RealCase}. \end{remark}

Without loss of generality, we can assume that the form $B$ is non-degenerate (otherwise we can replace $B$ with a regular form $A_{\lambda}$). Consider the  recursion operator $P = B^{-1} A$. The recursion operator $P$ is self-adjoint w.r.t. $A$ and $B$: \[ A(Pu, v) = A(u, Pv), \qquad B(Pu, v) = B(u, Pv).\] Thus, instead of the pair of forms $(A, B)$ on $V$ we can study a symplectic space $(V, B)$ with a self-adjoint operator $P$. First, let us recall some simple facts about self-adoint operators that we use below.

\begin{assertion}\label{A:SelfAdjProp} Let $P$  be a self-adjoint operator on a symplectic space $(V, B)$. Then the following holds:

\begin{enumerate}

\item The skew-orthogonal complement of any $P$-invariant subspace $W \subset V$ is also $P$-invariant. In other words, \[ PW \subset W \qquad \Rightarrow PW^{\perp} \subset W^{\perp}\]

\item For any $v \in V$ the vectors $v, Pv, \dots, P^nv,\dots$ are pairwise orthogonal w.r.t. $B$. 

\end{enumerate}

\end{assertion}

Recall that by Assertion~\ref{A:AdmPInv} in the Jordan case a subspace $U$ is admissible if and only if it is $P$-invariant. We get the following description of bi-isotropic and bi-Lagrangian subspaces.

\begin{lemma} \label{L:NonGenDescBiLagr} Let $\mathcal{P} = \left\{ A + \lambda B \right\}$ be a linear pencil on $V$. Assume that $B$ is non-degenerate (i.e. $\operatorname{Ker} B = 0$) and let $P = B^{-1}A$ be the recursion operator. 

\begin{enumerate}

\item A subspace $U \subset (V, \mathcal{P})$ is bi-isotropic if and only if it is isotropic w.r.t. $B$ and $P$-invariant.

\item A subspace $L \subset (V, \mathcal{P})$ is bi-Lagrangian w.r.t. $B$ if and only if it is Lagrangian w.r.t. $B$ and $P$-invariant.

\end{enumerate}

\end{lemma}

\begin{proof}[Proof of Lemma~\ref{L:NonGenDescBiLagr}]  \begin{enumerate}

\item A subspace $U$ is bi-isotropic iff it is isotropic w.r.t. $A$ and $B$. Assume that $U$ be isotropic w.r.t. $B$. Since \[ A(u, v) = B(u, Pv) = 0, \qquad \forall u \in U^{\perp_B}, \quad \forall v \in U,\] the subspace $U$ is isotropic w.r.t. $A$ if and only if it is $P$-invariant.

\item Since $B$ is nondegenerate, a subspace $L$ is bi-Lagrangian iff it is both bi-isotropic and satisfies $\dim L = \frac{1}{2} \dim V$. Hence, bi-Lagrangian subspaces are those Lagrangian w.r.t. B and P-invariant (follows from the characterization of bi-isotropic subspaces established earlier).

\end{enumerate}

Lemma~\ref{L:NonGenDescBiLagr}  is proved. 
\end{proof}

Now, we are ready to prove the decomposition~\eqref{Eq:DecompJordBiLagr}.

\begin{proof}[Proof of Theorem~\ref{T:JordaMultEigen}] Without loss of generality the form $B$ is nondegenerate. Note that for one Jordan  block $\mathcal{J}_{\lambda_0, 2n}$ in the JK decompostion of $(V, \mathcal{P})$ the recursion operator $P = B^{-1}A$ consists of two $n \times n$ Jordan blocks with the same eigenvalue: \[ P_i = B_i^{-1} A_{i} = \left( \begin{matrix} J_{\lambda_0}^T & 0 \\ 0 & J_{\lambda_0} \end{matrix} \right), \qquad J_{\lambda_0} = \left(\begin{matrix} \lambda_0 & 1 & &  \\  & \ddots & \ddots & & \\   &  & \ddots & 1 \\   & & & \lambda_0 \end{matrix}  \right). \]  Thus the sums of Jordan blocks with same eigenvalue $\mathcal{J}_{\lambda_j}$ are the generalized eigenspaces of the recursion operator $P$. By Lemma~\ref{L:NonGenDescBiLagr} any bi-Lagrangian subspace is $P$-invariant. It is well-known fact that any invariant subspace is the direct sum of its intersection with generalized eigenspaces (see e.g. \cite[Section 7.5, Lemma on page 263]{HoffmanKunze}). Thus $\displaystyle L = \bigoplus_{j=1}^S \left(L \cap \mathcal{J}_{\lambda_j}\right)$ and all $L_{j} = L \cap \mathcal{J}_{\lambda_j}$ are $P$-invariant. Since $L$ is Lagrangian w.r.t. $B$, all $L_{j}$ are also Lagrangian. By Lemma~\ref{L:NonGenDescBiLagr}, all $L_{j}$ are bi-Lagrangian. Theorem~\ref{T:JordaMultEigen} is proved.
\end{proof}

\begin{remark} If the field $\mathbb{K}$ (with $\operatorname{char}\mathbb{K} \not = 2$) is not algebraically closed\footnote{We exclude the case $\operatorname{char}\mathbb{K} = 2$, given that skew-symmetric and symmetric matrices become synonymous in this scenario. Alternate matrices may be explored in this context.}, then invariant Lagrangian subspaces can be decomposed according to the factorization of the minimal polynomial. Let $P$ be a self-adjoint operator on a symplectic space $(V,B)$ over $\mathbb{K}$ and let \[p = p_1^{r_1} \cdots p_k^{r_k}\] be the prime factorization of its minimal polynomial. Then $V$ decomposes into a direct sum (see e.g. \cite{Malagon17}): \[V = W_1 \oplus \dots \oplus W_k, \qquad  W_j = \operatorname{Ker} p_j(P)^{r_j}\] Analogous to Theorem~\ref{T:JordaMultEigen}, any $P$-invariant bi-Lagrangian $L \subset (V, B)$ splits as  \[ L = \sum_{j=1}^k L \cap W_j.\] \end{remark}

\subsubsection{Eigenvalue independence}

A bi-Lagrangian Grassmannian $\operatorname{BLG}(V, \mathcal{P})$ is independent of the specific eigenvalues\footnote{A bundle of a matrix pencil $\mathcal{B}(\mathcal{P})$ is the union of all pencils with the same canonical form as $\mathcal{P}$ up to specific eigenvalues (distinct eigenvalue remain distinct). Hence, bi-Lagrangian Grassmanian $\operatorname{BLG}(V, \mathcal{P})$ depends only on the bundle $\mathcal{B}(\mathcal{P})$.} $\lambda_j$ of the pencil $\mathcal{P}$. The next statement is trivial.

\begin{theorem} \label{T:NonDependEigen} For any sum Jordan blocks with the same eigenvalue we have an isomorphism \[ \operatorname{BLG}(\bigoplus_{k=1}^{N} \mathcal{J}_{\lambda, 2n_k}) \approx \operatorname{BLG}(\bigoplus_{k=1}^{N} \mathcal{J}_{0, 2n_k}).\] Hence, the bi-Lagrangian Grassmanian for $(V, \mathcal{P}) = \bigoplus_{j=1}^{S}\left(\bigoplus_{k=1}^{N_j} \mathcal{J}_{\lambda_j, 2n_{j,k}}\right) $ is isomorphic to \[ \operatorname{BLG}(V, \mathcal{P}) \approx \prod_{j=1}^{S}\operatorname{BLG}\left(\bigoplus_{k=1}^{N_j} \mathcal{J}_{0, 2n_{j,k}}\right). \]
\end{theorem}

\begin{remark} Although not essential, we can reformulate our investigation in a purely algebraic framework by considering $(V, \mathcal{P}) =\bigoplus_j \mathcal{J}_{0, 2n_{j}}$ as an $R = \mathbb{K}[x]/x^n$-module and reinterpreting bi-Lagrangian subspaces as Lagrangian $R$-submodules.  \end{remark}

\subsection{One Jordan block} \label{S:OneJordBlock}

In this section we begin the study of the $\operatorname{Aut}\left(V, \mathcal{P}\right)$-orbits of $\operatorname{BLG}\left(V, \mathcal{P}\right)$. We start with the simplest Jordan case, when $(V, \mathcal{P})$ consists of only one Jordan block $\mathcal{J}_{\lambda_0, 2n}$. Without loss of generality, we can assume that $\lambda_0=0$ (otherwise we can change the generators of $\mathcal{P}$, replacing $A$ with $A_{-\lambda_0} = A - \lambda_0 B$). Consequently, the recursion operator  $P = B^{-1}A$ becomes nilpotent.

Simply speaking, we study Lagrangian subspaces $L$ of a $2n$-dimensional symplectic space $(V^{2n}, B)$ that are invariant w.r.t. a nilpotent self-adjoint operator $N$.

\begin{definition} If $N$ is a nilpotent operator on a vector space $V$, then the \textbf{height} of a vector $v\in V$ (w.r.t. $N$) is \[ \operatorname{height}(v) = \min \left\{ k \in \mathbb{N} \,\, \bigr|\,\, N^k v  = 0\right\}.\] The \textbf{height of a subspace} $U \subset V$ is \[\operatorname{height}(U) = \max_{u \in U} \operatorname{height}(u).\] The \textbf{height of the operator} $N$ is the height of $V$. \end{definition}

\subsubsection{Canonical form of bi-Lagrangian subspaces}

If $(V, \mathcal{P})$ is just one Jordan block, then there is a very simple canonical form for bi-Lagrangian subspaces.

\begin{theorem}\label{T:BiLagr_One_Jordan_Canonical_Form} Consider a $2n\times 2n$ Jordan block $\mathcal{J}_{0, 2n}$ associated with the pencil $\mathcal{P} = \left\{ A+\lambda B\right\}$. Then for any Lagrangian subspace $L \subset \mathcal{J}_{0, 2n}$ there exists a standard basis $ e_1, \dots, e_n, f_1, \dots, f_ {n} $, i.e. a basis such that the matrices of the forms $A$ and $B$ are as in the JK theorem:
\[
A =\left(\begin{matrix} 0 & J_n(0) \\ -J_n^T(0) & 0  \end{matrix} \right), \qquad B = \left( \begin{matrix} 0 & I_n \\ -I_n & 0 \end{matrix}\right),\] 
and the Lagrangian subspace $L$ has the form  \begin{equation} \label{Eq:BiLagr_One_Jordan_Canonical_Form} L = \operatorname{Span}\left\{ e_{n-h+1}, e_{n-h}, \dots, e_n, f_1, \dots, f_{n-h} \right\} \end{equation} for some $h \geq \frac{n}{2}$. The number 
$h$ is the height of $L$, i.e. \begin{equation} \label{Eq:CondSOneJord} L \not \subset \operatorname{Ker} P^{h-1}, \qquad L \subset \operatorname{Ker} P^h. \end{equation} \end{theorem}

If we arrange the basis vectors in the following $2\times n$ table  \begin{equation} \label{Eq:OneJordanBlock_VectorInTable} \begin{array}{|c|c|} \hline e_1 & f_n \\ \hline e_2 & f_{n-1} \\ \hline \vdots & \vdots \\ \hline e_{n-1} & f_{2}  \\ \hline e_n & f_{1} \\ \hline \end{array}, \end{equation} then the bases of subspaces \eqref{Eq:BiLagr_One_Jordan_Canonical_Form} for $h=n,\dots \lceil \frac{n}{2} \rceil $ can be visualized as the following subsets \[ \begin{array}{|c|} \hline e_1 \\ \hline e_2 \\ \hline e_3 \\ \hline \vdots \\ \hline e_{n-1}   \\ \hline e_n  \\ \hline \end{array} , \qquad \begin{array}{|c|c} \multicolumn{1}{c}{} & \\ \cline{1-1} e_2 & \\ \cline{1-1} e_3 & \\ \cline{1-1} \vdots & \\ \cline{1-1} e_{n-1} &   \\ \hline e_n  & \multicolumn{1}{|c|}{f_1} \\ \hline \end{array},\qquad \begin{array}{|c|c} \multicolumn{1}{c}{} & \\ \multicolumn{1}{c}{}  & \\ \cline{1-1} e_3 & \\ \cline{1-1} \vdots & \\ \hline e_{n-1} & \multicolumn{1}{|c|}{f_2}   \\ \hline e_n  & \multicolumn{1}{|c|}{f_1} \\ \hline \end{array}, \qquad \dots,  \] which go up to \[ \begin{array}{|c|c|} \cline{1-1} e_{\frac{n-1}{2}} & \multicolumn{1}{c}{} \\ \hline e_{\frac{n+1}{2}} & f_{\frac{n-1}{2}} \\ \hline \vdots & \vdots \\ \hline e_n  & \multicolumn{1}{|c|}{f_1} \\ \hline \end{array} , \qquad \text{or} \qquad \begin{array}{|c|c|}  \multicolumn{1}{c}{} & \multicolumn{1}{c}{}   \\ \hline  e_{\frac{n}{2}} & f_{\frac{n}{2}} \\ \hline \vdots & \vdots \\ \hline e_n  & \multicolumn{1}{|c|}{f_1} \\ \hline \end{array} \] The vectors $e_i, f_j$ are arranged in the table~\eqref{Eq:OneJordanBlock_VectorInTable} by their heights. Simply speaking,  the recursion operator $P$ ``pushes them down'': \[P e_i = e_{i+1}, \qquad  P f_j = f_{j-1}.\] Thus all $\operatorname{Im} P^i$ and $\operatorname{Ker} P^j$ are spanned by vectors from some bottom rows of the table~\eqref{Eq:OneJordanBlock_VectorInTable}: \begin{equation} \label{Eq:ImKerOneNilp}  \operatorname{Im} P^{n-k} = \operatorname{Ker} P^k = \langle e_{n-k+1}, \dots, e_n, f_1, \dots, f_k \rangle  \end{equation}

\begin{proof}[Proof of Theorem~\ref{T:BiLagr_One_Jordan_Canonical_Form}]  The proof is in several steps.

\begin{enumerate}

\item \textit{Construction of the standard basis $e_i, f_j$.} Take $v \in L$ of maximal height. There exists $e_1 \in L$ of height $n$ such that $P^{n-h} e_1 = v$. Extend it to the canonical basis $e_1, \dots, e_n, f_1, \dots, f_n$ as follows: put $e_i = P^{i-1}e_1$, take $f_n$ such that $B(e_i, f_n) = \delta^i_n $ and then put $f_{n-i} = P^i f_n$.

\item \textit{$L$ has the form \eqref{Eq:BiLagr_One_Jordan_Canonical_Form}}. By construction $v = e_{n-h+1} \in L$. First, by Lemma~\ref{L:NonGenDescBiLagr} $L$ is $P$-invariant and hence $e_{n-h+1}, \dots, e_n \in L$. Second, $e_{n-h+1}$ is a vector of $L$ of maximal height. Therefore, $L \subset \operatorname{Im} P^{n-h}$. Since $L$ is bi-Lagrangian, \[L = L^{\perp_B} \supset \left(\operatorname{Im} P^{n-h} \right)^{\perp_B} = \operatorname{Ker} P^{n-h}.\] The vectors $f_1, \dots, f_{n-h}$ lie in $L$, since $\operatorname{Ker} P^{n-h}$ has the form~\eqref{Eq:ImKerOneNilp}. Thus, $L$ contains \eqref{Eq:BiLagr_One_Jordan_Canonical_Form}. These subspaces coincide, since $\dim L = n$. 

\item \textit{$h$ is the height of $L$ and $h \geq \frac{n}{2}$.} By construction. 

\end{enumerate}

Theorem~\ref{T:BiLagr_One_Jordan_Canonical_Form} is proved. \end{proof}

From Theorem~\ref{T:BiLagr_One_Jordan_Canonical_Form} it is easy to see that any bi-Lagrangian subspace is ``generated by any its top vector''. 

\begin{corollary} \label{Cor:OneJordBiLagrSpannedByVector} Let $L \subset{\mathcal{J}}_{0, 2n}$ be a bi-Lagrangian subspace. Then its height $h \geq \frac{n}{2}$ and for any vector $v \in  L$ of maximal height we  have \[ L = \langle v, Pv, \dots, P^{h-1} v \rangle \oplus \operatorname{Ker} P^{n-h}. \] \end{corollary}

\subsubsection{Topology of \texorpdfstring{$\operatorname{BLG}\left(\mathcal{J}_{\lambda, 2n}\right)$}{BLG(J(lambda, 2n)}}

In this section we describe $\operatorname{Aut}(V, \mathcal{P})$-orbits for one Jordan block $(V, \mathcal{P}) = \mathcal{J}_{\lambda, 2n}$. We are mostly interested in the real or complex case.

\begin{theorem} \label{Th:OneJordBiLagrOrbits} Consider a $2n\times 2n$ Jordan block $\mathcal{J}_{\lambda, 2n}$ over the field $\mathbb{K}= \mathbb{C}$ or $\mathbb{R}$. Let  $\mathcal{P} = \left\{ A+\lambda B\right\}$ be the associated linear pencil and $P = B^{-1}A$ be the recursion operator.

\begin{enumerate}

\item \label{Item:OneJordNumOrbits} There are $\left[\frac{n}{2}+1\right]$ orbits of the $\operatorname{Aut}(\mathcal{J}_{\lambda, 2n})$-action on the 
 bi-Lagrangian Grassmannian $\operatorname{BLG}(\mathcal{J}_{\lambda, 2n})$. Each orbit consists of bi-Lagrangian subspaces of the same height (w.r.t. the operator $P-\lambda E$): \[ O_h = \left\{ L \in \operatorname{BLG}(\mathcal{J}_{\lambda, 2n}) \quad \bigr| \quad \operatorname{height} (L) = h\right\}. \]

\item \label{Item:OneJordOrbitsTopology} The orbits $O_h$ for $h > \frac{n+1}{2}$ are diffeomorphic to the sum of $(2h - n - 1)$ tangent bundles to $\mathbb{KP}^1 = \operatorname{\Lambda}(1)$. In particular, the maximal orbit is \begin{equation}\label{Eq:TopOrbitOneBlock}  O_n \approx \bigoplus_{n-1} T \mathbb{KP}^1.\end{equation} If $n$ is odd, then the minimal orbit $O_{\frac{n+1}{2}}$ is $\mathbb{KP}^1$. If $n$ is even, then the minimal orbit $O_{\frac{n}{2}}$ is a single point, namely, the subspace $\displaystyle \operatorname{Ker} (P - \lambda E)^{\frac{n}{2}}$.

\end{enumerate}

\end{theorem}

\begin{remark} Theorem~\ref{Th:OneJordBiLagrOrbits} is slightly misleading. Complex vector bundles over $\mathbb{CP}^1$ have a very simple structure: they are classified by one number, namely their first Chern class\footnote{It is also widely known that by the Birkhoff-Grothendick theorem every holomorphic vector bundle over ${\displaystyle \mathbb {CP} ^{1}}$ is a direct sum of holomorphic line bundles (see e.g. \cite{HM82}).} (see e.g. \cite{McLean}). As it is shown in Theorem~\ref{Th:EqualJordBiLagrMaxOrbit} for the sum of equal Jordan blocks $\bigoplus_{k=1}^l \mathcal{J}_{\lambda, 2n}$  the top orbit is the jet space of the Lagrangian Grassmanian: \begin{equation} \label{Eq:TopOrbitEqJord1} O_{max} \approx T^{n-1} \Lambda(l) =  J^{n-1}_0(\mathbb{K},  \Lambda(l)).\end{equation}It is better to view the sum of tangent bundles in \eqref{Eq:TopOrbitOneBlock} as a particular case of the jet space \eqref{Eq:TopOrbitEqJord1}. \end{remark}

The orbits of $\operatorname{BLG}(\mathcal{J}_{\lambda, 2n})$ have dimensions $n, n-2, \dots$ and are diffeomorphic to \[\bigoplus_{n-1} T \mathbb{KP}^1, \qquad \bigoplus_{n-3} T \mathbb{KP}^1, \qquad \dots \qquad \mathbb{KP}^1 \quad \text{ or } \left\{ 0\right\}. \] For instance, for a single $4\times 4$ real Jordan block the bi-Lagrangian Grassmanian is a pinched torus as in Fig.~\ref{Fig:PinchedTorus}. In general, the bi-Lagrangian Grassmanian of a single Jordan block $\operatorname{BLG}(\mathcal{J}_{\lambda, 2n})$ consists of a top orbit $O_h$ and a smaller bi-Lagrangian Grassmanian $\operatorname{BLG}(\mathcal{J}_{\lambda, 2n-2})$. Thus, we get a ``matryoshka of bi-Lagrangian Grassmanians'': \begin{equation} \label{Eq:MatrOneJord} \operatorname{BLG}(\mathcal{J}_{\lambda, 2n}) \supset \operatorname{BLG}(\mathcal{J}_{\lambda, 2n-2}) \supset \operatorname{BLG}(\mathcal{J}_{\lambda, 2n-4}) \supset \dots \end{equation} We discuss nested Bi-Lagrangian Grassmanians in Section~\ref{SubS:MatrBLG}.

\begin{figure}[ht!]
  \centering
    \includegraphics[width=0.25\textwidth]{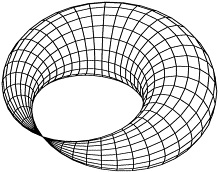}
      \caption{The real $\operatorname{BLG}(\mathcal{J}_{\lambda, 4})$ is a pinched torus}
      \label{Fig:PinchedTorus}
\end{figure}

\begin{proof}[Proof of Theorem~\ref{Th:OneJordBiLagrOrbits}] The bi-Lagrangian Grassmanian does not depend on the eigenvalue $\lambda$. Hence, without loss of generality $\lambda = 0$. Item~\ref{Item:OneJordNumOrbits} follows from Theorem~\ref{T:BiLagr_One_Jordan_Canonical_Form}. The proof of Item~\ref{Item:OneJordOrbitsTopology} is in several steps.

\begin{enumerate}

\item[Step 1.] \textit{The maximal orbit $O_n$ is diffeomorphic to $\bigoplus_{n-1} T \mathbb{KP}^1$.} First, fix an inner product on $\mathbb{K}^2$. Any element of $T \mathbb{KP}^1$ is given by two orthogonal vectors $u \perp v$, where $u \not = 0$. These pairs are equivalent under scalar multiplication by any non-zero number $\lambda \in \mathbb{K}^*$, meaning $(u, v) \sim (\lambda u, \lambda v)$. This is informally shown in Fig.~\ref{Fig:TangentCircle}.

\begin{figure}[ht!]
  \centering
    \includegraphics[width=0.25\textwidth]{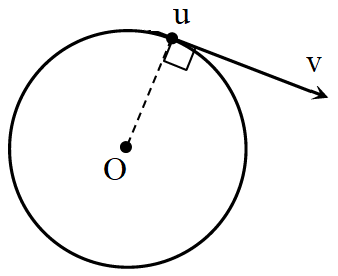}
      \caption{Tangent to $\mathbb{KP}^1$}
      \label{Fig:TangentCircle}
\end{figure} 

Now, fix a standard basis $e_1, \dots, e_n, f_1, \dots, f_n$ of a Jordan block. By Corollary~\ref{Cor:OneJordBiLagrSpannedByVector} any bi-Lagrangian subspace $L \in O_h$ has the  form \begin{equation} \label{Eq:VSpanL} L = \operatorname{Span} \left(v, Pv, \dots, P^{n-1} v \right). \end{equation} If the vector $v$ has the form \[ v= p_1 e_1 + \dots + p_n e_n + q_1 f_1 + \dots + q_n f_n,\] then $L$ is spanned by rows of the matrix \begin{equation} \label{Eq:SpanBiLagrOneJordMatrix}  \left( \begin{matrix} p_1 & \dots & p_n & q_1 & \dots & q_n \\ & \ddots & \vdots & \vdots & \udots & \\ & & p_1 & q_n & & \end{matrix} \right). \end{equation} The vector $v$ is defined up to proportionality and addition of a linear combination of $Pv, P^2v, \dots P^{n-1}v$. Since $\operatorname{height} (v) = n$ we have $\left(p_1, q_n\right) \not = 0$. Thus, replacing $v$ with a linear combination $\sum_i c_i P^iv$ we can make \begin{equation} \label{Eq:OrthogParts}  \left(p_i, q_{n-i+1}\right) \perp \left(p_1, q_n\right), \qquad i =2, \dots, n.\end{equation} A vector $v \in L$ such that \eqref{Eq:VSpanL} and \eqref{Eq:OrthogParts} hold is unique up to proportionality. Thus any bi-Lagrangian subspace $L \in O_n$ is uniquely defined by a point $(p_1: q_n) \in \mathbb{KP}^1$ and $(n-1)$ tangent vectors at these point \[\left( p_2, q_{n-1}\right), \quad \dots, \quad \left( p_n, q_{1}\right) \in T_{\left(p_1:q_n\right)}\mathbb{KP}^1. \]   We get a diffeomorphism $O_n\approx \bigoplus_{n-1} T \mathbb{KP}^1$.

\item[Step 2.] \textit{Non-maximal orbits $O_h$, $h <n$ are as in Theorem~\ref{Th:OneJordBiLagrOrbits}.} By Theorem~\ref{T:BiLagr_One_Jordan_Canonical_Form} bi-Lagrangian subspaces $L$ with $\operatorname{height}(L) < \operatorname{height}(V)$ contain $\operatorname{Ker} V$. Using Corollary~\ref{Cor:OneJordBiLagrSpannedByVector} it is easy to see that\footnote{This can also be proved using the bi-Poisson reduction, described in Lemma~\ref{L:AlgIsomBiisotrFactor}.} the orbit $O_{h}$ is diffeomorphic to the top orbit of $\operatorname{BLG}(\mathcal{J}_{0, 2n-4})$. Since we already know the structure of the maximal orbit $O_n$, the rest of the proof is by a simple induction. Note that the minimal orbit is either $\operatorname{BLG}(\mathcal{J}_{\lambda, 2}) \approx \mathbb{KP}^1$ for odd $n$, or, formally, the point $\operatorname{BLG}(\mathcal{J}_{\lambda, 0}) \approx \left\{ 0\right\}$ for even $n$.

  \end{enumerate}  
  
  Theorem~\ref{Th:OneJordBiLagrOrbits} is proved. \end{proof}

\begin{corollary} \label{Cor:ConnOneJord} The generic orbit $O_{\max}$ is dense (in the standard topology) in the bi-Lagrangian Grassmanian of a single Jordan block: \[ \bar{O}_{\max} = \operatorname{BLG}(\mathcal{J}_{0, 2n}).\] Since the automorphism group $\operatorname{Aut}(V,\mathcal{P})$ is connected (Theorem~\ref{T:AutConnected}), it preserves irreducible components of $\operatorname{BLG}(\mathcal{J}_{0, 2n})$. Therefore, $\operatorname{BLG}(\mathcal{J}_{0, 2n})$ coincides with the irreducible component containing the generic orbit $O_{\max}$. Thus, $\operatorname{BLG}(\mathcal{J}_{0, 2n})$  is an irreducible algebraic variety.  \end{corollary}

\section{Bi-Poisson reduction} \label{S:BiPoisSect}

In Section \ref{S:OneJordBlock} we showed that for one Jordan block $\mathcal{J}_{\lambda, 2n}$ the $\operatorname{Aut}(\mathcal{J}_{\lambda, 2n})$-orbits $O_h$ follow the sequence  \eqref{Eq:MatrOneJord}. It is easy to see that the collection of bi-Lagrangian subspaces containing the subspace $U = \operatorname{Ker}P^j$ is isomorphic to $\operatorname{BLG}(U^{\perp}/ U)$. In this section we demonstrate that this result is not limited to a single Jordan block case. 

\begin{itemize}
    \item  In Section~\ref{S:LinearRedDecr} we introduce a powerful technique called \textbf{bi-Poisson reduction}. Simply speaking, for any admissible bi-isotropic subspace $U$ we can induce the pencil and bi-Lagrangian subspaces on $U^{\perp}/U$.

    \item In Section~\ref{S:IsomAlgBiPRed} we show that the variety of bi-Lagrangian subspaces containing $U$ is isomorphic to $\operatorname{BLG}(U^{\perp}/ U)$ (see Lemma~\ref{L:AlgIsomBiisotrFactor}).
    
    \item  In Section~\ref{SubS:BiIsotr2BiLagr}  we answer the question ``When a bi-isotropic subspace extends to a bi-Lagrangian subspace?''
    
\end{itemize}

\subsection{Introducing bi-Poisson reduction} \label{S:LinearRedDecr}

Symplectic reduction, also known as the Marsden-Weinstein quotient, is a cornerstone technique in symplectic geometry. Linear symplectic reduction is described in detail in standard references like McDuff and Salamon's textbook \cite{McDuffSalamon}.  For Poisson vector spaces it takes the following form:

\begin{theorem}[Linear Symplectic Reduction] \label{T:SympReduction} Let $W$ be an isotropic subspace of a Poisson vector space  $(V, B)$. Then

\begin{enumerate}

\item The induced form $B'$ on $W^{\perp}/ W$ is well-defined and \[\operatorname{Ker} B' = \operatorname{Ker} B / (\operatorname{Ker} B \cap W).\] 

\item If $L$ is a Lagrangian (or isotropic) subspace of $(V, B)$, then \[ L' = \left( \left( L \cap W^{\perp}\right) + W \right) / W\] is a Lagrangian (respectively, isotropic) subspace of $W^{\perp}/W$. 

\end{enumerate}

\end{theorem}

Luckily for us, there is a similar bi-Poisson reduction for admissible bi-isotropic subspaces that we describe in the next theorem. (Recall that we defined admissible subspaces and described their basic properties in Section~\ref{S:Admissible}.)

\begin{theorem} \label{T:BiPoissReduction} Let $\mathcal{P} = \left\{A_{\lambda} \right\}$ be a linear pencil on $V$ and let $U\subset \left(V, \mathcal{P}\right)$ be an admissible bi-isotropic subspace. Then

\begin{enumerate}

\item The induced pencil $\mathcal{P}' = \left\{A'_{\lambda}\right\}$ on $U^{\perp}/ U$ is well-defined. 

\item If $L$ is a bi-Lagrangian (or bi-isotropic) subspace of $(V, \mathcal{P})$, then \[ L' = \left( \left( L \cap U^{\perp}\right) + U \right) / U\] is a bi-Lagrangian (respectively, bi-isotropic) subspace of $U^{\perp}/U$. 

\end{enumerate}

\end{theorem}

\begin{proof}[Proof of Theorem~\ref{T:BiPoissReduction}] By Assertion~\ref{A:AddmImagePen} $U^{\perp_{\lambda}} \supseteq U^{\perp}$ and the equality holds for almost all $\lambda \in \bar{\mathbb{C}}$. Thus we can perform the symplectic reduction w.r.t. almost all forms $A_{\lambda}$ using Theorem~\ref{T:SympReduction}. Consequently, $L'$ becomes Lagrangian (or isotropic) for a generic form $A'_{\lambda}$ and, by definition, bi-Lagrangian (bi-isotropic).

The remaining step is to demonstrate that the induced form $A'_{\mu}$ is well-defined if $U^{\perp_{\mu}} \not =  U^{\perp}$. In this case we can proceed by first inducing $A_{\mu}$ on $U^{\perp_{\mu}}/U$ via symplectic reduction. Then we can further restrict it to the subspace $U^{\perp}/U \subseteq U^{\perp_{\mu}}/U$ to obtain the desired induced form. Theorem~\ref{T:BiPoissReduction} is proved. \end{proof}

Note that $U \subset L$ for a bi-Lagrangian $L$ if and only if $L = L^{\perp} \subset U^{\perp}$. In that case Theorem~\ref{T:BiPoissReduction} takes the following simpler form.

\begin{corollary}\label{C:ContainBiisotAdm} Let $U \subset (V, \mathcal{P})$ be an admissible bi-isotropic subspace.  
 
\begin{enumerate}

\item If $L$ is bi-Lagrangian and $U\subset L$, then $L \subset U^{\perp}$.

\item A subspace $L$ such that $U \subset L \subset U^{\perp}$ is bi-Lagrangian if and only if $L/ U$ is bi-Lagrangian in $U^{\perp}/ U$.

\end{enumerate} 
\end{corollary}

\subsection{Isomorphism of algebraic varieties under reduction} \label{S:IsomAlgBiPRed}

Let $U \subset (V, \mathcal{P})$ be bi-isotropic and admissible. The set of bi-Lagrangian subspace $L$  such that $U \subset L$ is the projective variety $\operatorname{BLG}(V, \mathcal{P}) \cap \operatorname{Gr}_k (U^{\perp})$ for $k = \dim V - \frac{1}{2}  \operatorname{rk}  \mathcal{P}$. Note that $L \subset U^{\perp}$ by Corollary~\ref{C:ContainBiisotAdm}.  Furthermore, Corollary~\ref{C:ContainBiisotAdm} induces an isomorphism of this set with the corresponding bi-Lagrangian Grassmanian.

\begin{lemma} \label{L:AlgIsomBiisotrFactor}
Let $U\subset \left(V, \mathcal{P}\right)$ be bi-isotropic and admissible and $k = \dim V - \frac{1}{2}  \operatorname{rk} \mathcal{P}$. Denote by $\mathcal{P}_{U}$ the induced pencil on $U^{\perp}/U$. Then the variety of bi-Lagrangian subspaces containing $U$ is isomorphic to the bi-Lagrangian Grassmanian of $U^{\perp}/U$:
 \begin{equation} \label{Eq:BiPoissBLG} \operatorname{BLG}(V, \mathcal{P}) \cap \operatorname{Gr}_k (U^{\perp}) \approx \operatorname{BLG}(U^{\perp}/U, \mathcal{P}_U). \end{equation}
\end{lemma}

\begin{proof}[Proof of Lemma~\ref{L:AlgIsomBiisotrFactor}] It is easy to see that \begin{equation} \label{Eq:biPoissUperp}  \operatorname{BLG}(V, \mathcal{P}) \cap \operatorname{Gr}_k (U^{\perp}) =\operatorname{BLG}\left(U^{\perp}\right).  \end{equation}  Both sides of \eqref{Eq:biPoissUperp} consist of bi-isotropic subspaces of $U^{\perp}$ that have the same dimension $\dim V - \frac{1}{2}\operatorname{rk}\mathcal{P}$. Note that $U$ is a part of the core subspace of $U^{\perp}$. Therefore, by Corollary~\ref{C:IsomJordPart}, \[\operatorname{BLG}\left(U^{\perp}\right) \approx \operatorname{BLG}(U^{\perp}/U, \mathcal{P}_U). \] Lemma~\ref{L:AlgIsomBiisotrFactor} is proved. \end{proof}

\begin{remark} \label{Rem:RedDisc} Bi-Poisson reduction map \[ f_U: \operatorname{BLG}(V,\mathcal{P}) \to \operatorname{BLG}(U^{\perp}/U)\] is generally not continuous (in the standard topology). Let $(V,\mathcal{P}) = \mathcal{J}_{0, 6}$,  $P$ be the recursion operator and $U = \operatorname{Im} P^2$. In a standard basis $e_1,e_2, e_3, f_1, f_2, f_3$ the subspace $U = \operatorname{Span}\left\{e_3, f_1\right\}$. Consider limit of (row-spanned) bi-Lagrangian subspaces: \[ \left(\begin{array}{cccccc} \varepsilon & 1 & 0 & 0 & 1 & 0   \\ 0 & \varepsilon & 1 & 1 & 0 & 0 \\ 0 & 0 & \varepsilon & 0 & 0 & 0 \end{array}\right) \xrightarrow{\varepsilon \to 0 } \left(\begin{array}{cccccc} 0 & 1 & 0 & 0 & 1 & 0  \\ 0 & 0 & 1 & 1 & 0 & 0  \\ 0 & 0 & 1 & 0 & 0 & 0  \end{array} \right). \] Limit subspace after reduction differs from reduction of the limit subspace: \[ \left(\begin{array}{cc} \varepsilon & 0  \end{array} \right) \xrightarrow{\varepsilon \to 0 } \left(\begin{array}{cc} 0 & 0 \end{array} \right) \not = \left(\begin{array}{cc} 1 & 1 \end{array} \right).\] The map $f_U$ requires additional conditions (such as $\dim U \cap L = \operatorname{const}$) for continuity. \end{remark}

\subsection{Extending bi-isotropic subspaces to bi-Lagrangian} \label{SubS:BiIsotr2BiLagr} 

In this section we answer the question (see \cite[Problem 14]{BolsinovIzosimomKonyaevOshemkov12} and \cite[Problem 11]{BolsinovTsonev17}): 

\begin{itemize}

\item when a subspace $U$ extends to a bi-Lagrangian subspace $L$? 

\end{itemize} Obviously, $U$ must be bi-isotropic. Since the bi-Lagrangian Grassmanian is not an empty set by Assertion~\ref{A:ExistBiLagr}, the next statement immediately follows from Lemma~\ref{L:AlgIsomBiisotrFactor}. 

\begin{corollary} \label{Cor:ExtendsBiIsotJordan}
Any bi-isotropic and admissible subspace $U\subset \left(V, \mathcal{P}\right)$ extends to a bi-Lagrangian subspace $L \subset (V, \mathcal{P})$.
\end{corollary}

However, what if we consider a bi-isotropic subspace $U \subset (V, \mathcal{P})$ that is not admissible? To address this case, let's revisit the structure of bi-Lagrangian subspaces. Combining Theorem~\ref{T:BiLagrKronPart}  and Assertion~\ref{A:AdmPInv} we get the following:

\begin{assertion} \label{A:WhenBiLagr} Let $(V, \mathcal{P})$ be a bi-Poisson vector space, $K$ be the core subpsace, $M$ be the mantle subspace and $P$ be the induced recursion operator on $M/K$. A subspace $L \subset (V, \mathcal{P})$ is bi-Lagrangian if and only if the following holds:

\begin{enumerate}

\item $K \subset L \subset M$,

\item $L/K$ is bi-isotropic and admissible (i.e. $P$-invariant) in $M/K$.

\end{enumerate}

\end{assertion}

Simple extension criteria emerge for Kronecker and Jordan pencils.

 \begin{corollary} Consider a linear pencil $\mathcal{P} = \left\{ A+ \lambda B \right\} $ on a vector space $V$. 
 
 \begin{enumerate}
 
 \item If $\mathcal{P}$ is a Kronecker pencil, then a subspace $U\subset (V, \mathcal{P})$ extends to a bi-Lagrangian subspace $L$ if and only if it is contained in the core subspace $U \subset K$. 
 
 \item If $\mathcal{P}$ is a Jordan pencil and $P=B^{-1}A$, then a subspace $U\subset (V, \mathcal{P})$ extends to a bi-Lagrangian subspace $L$  if and only if the $P$-invariant subspace generated by $U$ is bi-isotropic.
 
 \end{enumerate}
 
 \end{corollary}

It is not hard to see that in the general case we have the following extension criteria.
 
 \begin{corollary}\label{Cor:BiIsotrExtendsBiLagr} Let $K$ and $M$ be the core and mantle subspaces of a bi-Poisson vector space $(V, \mathcal{P})$ and $P$ be the induced recursion operator on $M/K$. A subspace $U \subset (V, \mathcal{P})$ extends to a bi-Lagrangian subspace $L \supset U$ if and only if the following holds:
 
 \begin{enumerate}
 
 \item \label{I:Cond1BiIsot} $U \subset M$ (or, equivalently, $U$ is bi-orthogonal to $K$, since $M = K^{\perp}$),
  
 \item the $P$-invariant subspace generated by $\left(U + K\right)/ K$ in $M/K$ is bi-isotropic.

\end{enumerate}
 
 \end{corollary}

 \section{Group of automorphisms \texorpdfstring{$\operatorname{Aut}(V, \mathcal{P})$}{Aut(V, P)}} \label{S:AutGroup}

The Lie algebra for the group of bi-Poisson automorphisms $\operatorname{Aut}(V, \mathcal{P})$ was previously described in \cite{Pumei10}, see also \cite{Dmytryshyn2013}. In this section we describe this Lie algebra $\operatorname{aut}(V, \mathcal{P})$ in the Jordan case in a way that suits our needs.  Let $(V, \mathcal{P})$ be a sum of Jordan blocks with zero eigenvalue: \[(V, \mathcal{P}) = \bigoplus_{i=1}^t \left(\bigoplus_{j=1}^{l_i} \mathcal{J}_{0, 2n_i} \right),\] where $n_1
\geq n_2 \geq \dots \geq n_t$. First, we choose a convenient basis.

\begin{definition} We say that $e^{ij}_k, f^{ij}_k$ for $i=1, \dots, t, j=1,\dots, l_i, k=1, \dots, n_i$ is a \textbf{standard basis} of $(V, \mathcal{P})$, if for each fixed $i, j$ it is a standard basis for one of the Jordan blocks.  \end{definition}

It will be convenient to rearrange the basis $e^{ij}_k, f^{ij}_k$ as follows:

\begin{itemize}

\item Global Loop: Iterate over the index $i=1, \dots, t$;

\item Inner Loop (Fixed $i$): For a fixed value of $i$ we iterate over the index $k =1, \dots, n_i$;

\item Innermost Loop (Fixed $i$ and $k$): For a fixed pair $(i, k)$ first we take $e^{ij}_k$  for $j=1, \dots, l_i$. Then we take $f^{ij}_{n_i-k+1}$ for $j=1, \dots, l_i$.

\end{itemize}

Simply speaking, we take the following basis:

\begin{equation} \label{Eq:SeverJordBlocksGroupedBasisSt} \begin{gathered}
 e^{11}_1, \dots, e^{1 l_1}_{1}, \quad f^{11}_{n_1}, \dots, f^{1 l_1}_{n_1}, \quad \dots \quad  e^{11}_{n_1}, \dots, e^{1 l_1}_{n_1}, \quad f^{11}_{1}, \quad \dots, f^{1 l_1}_{1}, \quad \dots \\  e^{t1}_1, \dots, e^{t l_t}_{1}, \quad f^{t1}_{n_t}, \dots, f^{t l_t}_{n_t}, \quad \dots \quad  e^{t1}_{n_t}, \dots, e^{t l_t}_{n_t}, \quad f^{t1}_{1}, \dots, f^{t l_t}_{1}.
\end{gathered} \end{equation}

\begin{theorem}[P.~Zhang, \cite{Pumei10}] \label{T:BiSymp_General_Jordan_Case_Mega}
Consider a sum of Jordan blocks $(V, \left\{A + \lambda B\right\}) = \bigoplus_{i=1}^t \left(\bigoplus_{j=1}^{l_i} \mathcal{J}_{0, 2n_i} \right),$ where $n_1
\geq n_2 \geq \dots \geq n_t$. Then in the basis \eqref{Eq:SeverJordBlocksGroupedBasisSt} the matrices of the operator 
$P=B^{-1}A$ and the form $B$ are block-diagonal
\begin{equation}\label{E:Basis_Jordan_General_Case}
P = \begin{pmatrix} P_1 & & \\ & \ddots & \\ & & P_t  \end{pmatrix},
\qquad  B =  \begin{pmatrix} B_1 & & \\ & \ddots & \\ & & B_t
\end{pmatrix},
\end{equation}  where the blocks $P_i$ and $B_i$ are equal to \begin{equation} P_i = \left( \begin{array}{cccc} 0 &
& & \\ I_{2l_i} & \ddots & & \\ & \ddots & \ddots & \\
& & I_{2l_i} \end{array} \right) , \quad \text{and} \quad B_i = \left(
\begin{array}{cccc}  & & & Q_{2l_i} \\ & & \udots &  \\ & \udots &  &  \\
Q_{2l_i} & &  &  \end{array} \right),  \end{equation} where $Q_{2s} =
\begin{pmatrix} 0 & I_s \\ - I_s & 0 \end{pmatrix}$ and $I_s$ is the $s\times s$ identity matrix.  Then the Lie algebra $\textnormal{aut}(V, B, P)$ of the automorphism group consists of elements with the following structure: \begin{equation} \label{E:bsp_algebra_matrix}
C = \left( \begin{array}{ccccc|ccc|c} C^{1,1}_1 & & & & & &  & & \multirow{5}{*}{$\cdots$} \\
C^{1,1}_2 & C^{1,1}_1 & & & & & & &  \\ \vdots & \ddots & \ddots &  & & C^{1,2}_1 & & & \\
\vdots & \ddots & \ddots & \ddots & & \vdots & \ddots & &
\\ C^{1,1}_{n_1} & \cdots & \cdots & \cdots & C^{1,1}_1  & C^{1,2}_{n_2} & \cdots & C^{1,2}_1 &
\\ \hline C^{2,1}_1 &  & & & & C^{2,2}_1 & &  & \multirow{3}{*}{$\cdots$} \\
\vdots & \ddots & & & & \vdots & \ddots  & & \\ C^{2,1}_{n_2} & \cdots &
C^{2,1}_1 & & & C^{2,2}_{n_2} & \cdots & C^{2,2}_1 & \\ \hline
\multicolumn{5}{c|}{\cdots}   & \multicolumn{3}{c|}{\cdots} & \cdots
\end{array} \right).
\end{equation}
Each $C^{i,j}_s$, where $i,j =1, \dots, t$, $s = 1, \dots, n_{\max(i, j)}$, are  $2l_i \times 2l_j$ matrices satisfying the relations
\begin{equation} \label{E:Cond_on_BiSymp_Jordan_Case} (C^{j,i}_s)^T Q_{2l_j} + Q_{2l_i} C^{i, j}_s =0
\end{equation}
In particular, the elements of the diagonal blocks $C^{i,i}_s$ are elements of the symplectic Lie algebra $sp(2l_i)$ since \begin{equation} (C^{i,i}_s)^{T} Q_{2l_i} + Q_{2l_i} C^{i, i}_s =0
\end{equation} \end{theorem}

Here \eqref{E:bsp_algebra_matrix} means that $C$ is a block-matrix \begin{equation} \label{Eq:BiPoissJordAutBigBlocks} C = \left( \begin{matrix} X_{11} & \dots & X_{t1} \\ \vdots & \ddots &\vdots \\ X_{t1} & \dots & X_{tt} \end{matrix} \right),\end{equation} where each block $X_{ij}$ is a $2n_i l_i \times 2n_j l_j$ matrix. $X_{ij}$ has the form $\left(\begin{matrix} 0 \\ Y_{ij} \end{matrix} \right)$ for $i \geq j$ and the form $\left( \begin{matrix} Y_{ij} & 0 \end{matrix} \right)$ for $i < j$, where \begin{equation} \label{Eq:BiPoissJordAutBlocksSmall} Y_{ij} = \left(\begin{matrix} C^{i,j}_1 & & & \\ C^{i, j}_2 & C^{i,j}_1 &  & \\ \vdots & \ddots & \ddots  & \\ C^{i,j}_{n} & \dots & \dots  & C^{i,j}_1\\ \end{matrix}  \right), \qquad n = \max(n_i, n_j). \end{equation}

\begin{remark} We emphasize that Theorem~\ref{T:BiSymp_General_Jordan_Case_Mega} allows for both \textit{strict} and \textit{non-strict} inequalities in the sequence \[n_1 \geq n_2 \geq \dots \geq n_t.\] This makes Theorem~\ref{T:BiSymp_General_Jordan_Case_Mega}  more general and more flexible. If needed, we can group together all Jordan blocks with the same size and get strict inequalities \[n_1 > n_2 > \dots > n_q.\] However, in some cases it may be more convenient to allow equal block sizes $n_i = n_{j}$. 
\end{remark}

\begin{remark} Note that Theorem~\ref{T:BiSymp_General_Jordan_Case_Mega} is about the Lie algebra $\operatorname{aut}(V, \mathcal{P})$. Elements of the Lie group $\operatorname{Aut}(V, \mathcal{P})$ have the same form  \eqref{E:bsp_algebra_matrix}  but the conditions on blocks become much more convoluted (especially if the equal Jordan blocks  are not grouped together).
\end{remark}

\begin{proof}[Proof of Theorem~\ref{T:BiSymp_General_Jordan_Case_Mega}]
First of all, let us show that in the basis~\eqref{Eq:SeverJordBlocksGroupedBasisSt} the matrices of $P$ and $B$ are as required. It suffices to describe such basis for $l$ Jordan $2n \times 2n$ blocks. Let $e_1^{j}, \dots, e_n^{j}, f_1^{j}, \dots, f_{k}^{j}, \dots, f_{n}^{j}$, where $j =1, \dots l$, be the basis for these Jordan blocks from the JK theorem. That is, for each Jordan block the matrices of forms are {\scriptsize
\[
A_j =\left(
\begin{array}{c|c}
  0 & \begin{matrix}
   0 &1&        & \\
      & 0 & \ddots &     \\
      &        & \ddots & 1  \\
      &        &        & 0   \\
    \end{matrix} \\
  \hline
  \begin{matrix}
  0  &        &   & \\
  \minus1   & 0 &     &\\
      & \ddots & \ddots &  \\
      &        & \minus1   & 0 \\
  \end{matrix} & 0
 \end{array}
 \right)
\quad  B_j = \left(
\begin{array}{c|c}
  0 & \begin{matrix}
    1 & &        & \\
      & 1 &  &     \\
      &        & \ddots &   \\
      &        &        & 1   \\
    \end{matrix} \\
  \hline
  \begin{matrix}
  \minus1  &        &   & \\
     & \minus1 &     &\\
      &  & \ddots &  \\
      &        &    & \minus1 \\
  \end{matrix} & 0
 \end{array}
 \right)
\]
}
Then in the basis \[e_1^{1}, e_1^2, \dots, e^{l}_1, f_n^1,  f_n^2, \dots,
f^{l}_n, \quad \dots \quad e^{1}_{n}, \dots, e_{l}^{n}, f^{1}_1,
\dots, f^{l}_1\] the matrices of the operator $P$ and form $B$ are \begin{equation}\label{E:Basis_Same_Jordan_Blocks}  P = \left( \begin{matrix} 0 & & &  \\ I_{2l} & \ddots && \\ &
\ddots & \ddots & \\ & & I_{2l} & 0 \end{matrix} \right) \qquad B =
\left( \begin{matrix} & & & Q_{2l} \\ & &\udots & \\ &\udots & &  \\
Q_{2l}&  & & \end{matrix} \right). \end{equation}  

The rest of the proof is by direct calculation. Elements $ C \in \textnormal{aut}(V, B, P)$ have the form \eqref{E:bsp_algebra_matrix} because they commute with operator $P$. Indeed, $C$, given by \eqref{Eq:BiPoissJordAutBigBlocks}, commutes with $P$ if and only if for all $i,j=1, \dots, t$ the following holds: \[ P_i X_{ij} = X_{ij} P_j.\] Note that 
\begin{itemize}

\item $P_i X_{ij}$ consists of the elements of $X_{ij}$ ``shifted down by $l_i$ rows'';

\item $X_{ij} P_{j}$ consists of the elements of $X_{ij}$ ``shifted left by $l_j$ columns''.

\end{itemize}

It is easy to see that $X_{ij}$  has the form $\left(\begin{matrix} 0 \\ Y_{ij} \end{matrix} \right)$ for $i \geq j$ and the form $\left( \begin{matrix} Y_{ij} & 0 \end{matrix} \right)$ for $i < j$, where $Y_{ij}$ is given by \eqref{Eq:BiPoissJordAutBlocksSmall}, as required. Thus, any $C \in \operatorname{aut}(V, \mathcal{P})$ has the form \eqref{E:bsp_algebra_matrix}. Such $C$ preserves the form $B$, i.e. $C^T B + BC = 0$ if and only if the conditions \eqref{E:Cond_on_BiSymp_Jordan_Case} 
are satisfied. Theorem~\ref{T:BiSymp_General_Jordan_Case_Mega} is proved. \end{proof}

\begin{corollary}[P.~Zhang, \cite{Pumei10}] \label{C:DimAutJordanCase}  
Let $(V, \mathcal{P}) = \bigoplus_{j=1}^N \mathcal{J}_{\lambda, 2n_j}$, where $n_1
\geq n_2 \geq \dots \geq n_N$. Then \begin{equation} \label{Eq:DimAutJordanCase} \dim  \operatorname{Aut}(V, \mathcal{P}) = \sum_{j=1}^N (4j-1) n_j. \end{equation} 
\end{corollary}

\subsection{Connectedness of automorphism group}

\begin{theorem} \label{T:AutConnected} For any Jordan bi-Poisson space $(V,\mathcal{P})$ its group of automorphisms $\operatorname{Aut}(V,\mathcal{P})$ is connected.  \end{theorem}

\begin{proof}[Proof of Theorem~\ref{T:AutConnected}] The automorphism group admits a decomposition according to the eigenvalues. Without loss of generality, we may assume that the only eigenvalue is zero. The proof is by induction on $\dim V$. 

\begin{itemize}

\item \textit{Base case ($\dim V = 2$)}. Then $\operatorname{Aut}(V,\mathcal{P}) \approx \operatorname{Sp}(2,\mathbb{K})$, which is connected. 

\item \textit{Inductive step}. Let $(V,\mathcal{P}) = \oplus_{i=1}^N\mathcal{J}_{0, 2n_i}$, where $n_1 \geq \dots \geq n_N$. Fix a standard basis $e^{i}_j, f^{i}_j$, where $i=1,\dots, N, j=1,\dots, n_i$, of $(V,\mathcal{P})$. Consider the action of $\operatorname{Aut}(V,\mathcal{P})$ on the basis vector $e^{11}_1$.  Denote by $O_{e}$ its orbit and by $\operatorname{St}_e$ its stabilizer. We get a fiber bundle  \[ \pi: \operatorname{Aut}(V,\mathcal{P})  \xrightarrow{St_e} O_e. \] The orbit $O_e$ is connected, since it is diffeomorphic to $V - \operatorname{Ker}P^{n_1 - 1}$ (see Theorem~\ref{T:VectOrbits}). It remains to prove that the stabilizer group $\operatorname{St}_e$ is connected. Since elements of  $\operatorname{Aut}(V,\mathcal{P})$ preserve the recursion operator $P$, the stabilizer group $\operatorname{St}_e$ leaves invariant the  preserve the $P$-invariant subspace \[ U = \operatorname{Span} \left\{e^{11}_1, \dots, e^{11}_{n_1} \right\}.\] The action of $\operatorname{St}_e$ of $U^{\perp}/U$ induces a surjective Lie group homomorphism \[ f: \operatorname{St}_e  \to  \operatorname{Aut}(U^{\perp}/U). \] By the induction hypothesis, the group $\operatorname{Aut}(U^{\perp}/U)$ is connected. It remains to establish the connectedness of the subgroup $H=\operatorname{Ker}f$, consisting of elements $C \in  \operatorname{Aut}(V,\mathcal{P}) $ such that \[ C\bigr|_{U} = \operatorname{id}, \qquad C\bigr|_{U^{\perp}/U} = \operatorname{id}.\] Since $C$ preserves the recursion operator $P$, it is determined by the images of the ``top-height'' vectors in the Jordan chains: \[ \begin{gathered} Ce^{1}_1 = e^{1}_1, \qquad Cf^{1}_{n_1} = \sum_{i=1}^N \sum_{j=1}^{n_1} \left( a^i_j e^{i}_{j} + b^i_j f^{i}_j  \right)  \\ Ce^{i}_{1} = e^{i}_1 + \sum_{j=1}^{n_1} c^i_j e^1_j, \qquad  Cf^{i}_{n} = f^{i}_n + \sum_{j=1}^{n_1} d^i_j e^1_j.  \end{gathered} \] Since $C$ preserves the bilinear form $B$, the coefficients $b^1_j = \delta^N_j$ and other $a^i_j, b^i_j$ for $i > 1$ are uniquely determined by $c^i_j$ and $d^i_j$. It can be verified that the remaining coefficients $a^i_1, c^i_j$ and $d^i_j$ can assume arbitrary values, resulting in the connectedness of both $H=\operatorname{Ker}f$ and the full automorphism group $\operatorname{Aut}(V,\mathcal{P})$. \end{itemize}

Theorem~\ref{T:AutConnected} is proved. \end{proof}

\section{Dimension of bi-Lagrangian Grassmanian} \label{S:DimBiLagrGrassm}

In this section we calculate  the dimension of a bi-Lagrangian Grassmanian $\operatorname{BLG}(V,\mathcal{P})$. We do it as follows:

\begin{itemize}

\item In Section~\ref{SubS:ExtractOneJordSimple} we examine when it's possible to extract a maximal Jordan block with a generic bi-Lagrangian subspace in it.

\item In Section~\ref{SubS:GenBiLag} we show that $\operatorname{BLG}(V,\mathcal{P})$ has a unique open orbit $O_{\max}$ and describe a canonical form for the bi-Lagrangian subspaces $L \in O_{\max}$. 

\item  Finally, in Section~\ref{SubS:DimBLG} we calculate $\operatorname{dim} \operatorname{BLG}(V,\mathcal{P})$.

\end{itemize}

\subsection{Extraction of one Jordan block} \label{SubS:ExtractOneJordSimple}

Let $(V, \mathcal{P}) = \bigoplus_{i=1}^t \left(\bigoplus_{j=1}^{l_i} \mathcal{J}_{0, 2n_i} \right)$, where $n_1 > n_2 > \dots > n_t$. Denote by $P$ the associated nilpotent operator on $V$. There are two types of bi-Lagrangian subspaces $L$.

\begin{enumerate}

\item \textbf{Height of $L$ is less than $n_1$}, i.e. $L \subset \operatorname{Ker} P^{n_1 - 1}$. Taking the orthogonal complements, we get  $\operatorname{Im}P^{n_1 -1} \subset L^T = L$.  Therefore, for such bi-Lagrangian subspaces L, we can perform a bi-Poisson reduction w.r.t. $U = \operatorname{Im}P^{n_1 -1}$.

\item \textbf{Height of $L$ equals $n_1$.} The next statement shows that in that case we can extract one maximal Jordan blocks with a generic bi-Lagrangian subspace. 

\end{enumerate}

\begin{theorem} \label{T:ExtractMaxBlockGen} If $(V, \mathcal{P})$ is a sum of Jordan $0$-blocks and $L \subset (V, \mathcal{P})$ is a bi-Lagrangian subspace with \[\operatorname{height}(L) = \operatorname{height}(V) = n_1\] there exists a decomposition of $(V, \mathcal{P})$ and $L$: \begin{equation} \label{Eq:Decomp2CompBiLagr} (V, \mathcal{P}) = (V_1,\mathcal{P}_1) \oplus (V_2, \mathcal{P}_2), \quad L = L_1 \oplus L_2, \quad L_i = L \cap V_i, \end{equation} such that $(V_1,\mathcal{P}_1) =\mathcal{J}_{0, 2n_1}$ and $L_1$ is generic in $\mathcal{J}_{0, 2n_1}$, i.e. $\operatorname{height} (L_1) = n_1$.\end{theorem}

Imagine representing $(V, \mathcal{P})$ as a Young-like diagram, where each block $\mathcal{J}_{0, 2n_i}$ is depicted as a $2n\times 2$ rectangle similar to \eqref{Eq:OneJordanBlock_VectorInTable}. Each cell corresponds to a basic vector and the nilpotent operator $P$ "pushes" the blocks down the diagram. Then we have the following alternative for bi-Lagrangian subspaces $L$:

\begin{itemize}

    \item either $L$  contains $\operatorname{Im}P^{n_1-1}$, i.e. the "bottom row" of the largest $2n_1 \times 2$ rectangles, 

    \item or (in some basis) $L$ contains  a "vertical  column" in a largest  $2n_1 \times 2$ rectangle.
\end{itemize} In \eqref{Eq:Alternative} we illustrate these two possibilities for the sum $\mathcal{J}_{0, 8} \oplus\mathcal{J}_{0, 8} \oplus \mathcal{J}_{0, 4}$. The shaded region depicts the subpace that $L$ contains.

\begin{equation} \label{Eq:Alternative}
  \begin{tabular}{|c|c|c|c|c|c|} 
  \cline{1-4} & & &  & \multicolumn{1}{c}{} & \multicolumn{1}{c}{} \\
 \cline{1-4} & & &  & \multicolumn{1}{c}{} & \multicolumn{1}{c}{} \\
     \hline  & &  & & &  \\ 
   \hline {\cellcolor{gray!25} } & {\cellcolor{gray!25} }  & {\cellcolor{gray!25} } & {\cellcolor{gray!25} }  & &  \\ \hline
  \end{tabular} 
  %%%
  \qquad
  \begin{tabular}{|c|c|c|c|c|c|} 
  \cline{1-4} {\cellcolor{gray!25} } & & &  & \multicolumn{1}{c}{} & \multicolumn{1}{c}{} \\
 \cline{1-4} {\cellcolor{gray!25} } & & &  & \multicolumn{1}{c}{} & \multicolumn{1}{c}{} \\
     \hline {\cellcolor{gray!25} } & &  & & &  \\ 
   \hline {\cellcolor{gray!25} } &  &  &   & &  \\ \hline
  \end{tabular} 
\end{equation}

\begin{proof}[Proof of Theorem~\ref{T:ExtractMaxBlockGen}]  Take a vector $e_1 \in L$ with $\operatorname{height} e_1 = n_1$. Then take the canonical basis $e_1, \dots, e_{n_1}, f_1, \dots, f_{n_1}$ as in Theorem~\ref{T:BiLagr_One_Jordan_Canonical_Form}. Simply speaking, put $e_i = P^{i-1}e_1$, take $f_n$ such that $B(e_i, f_{n_1}) = \delta^i_{n_1} $ and put $f_{n_1-i} = P^i f_{n_1}$. We separated one block \[V_1  = \operatorname{Span}\left\{e_1, \dots, e_{n_1}, f_1, \dots, f_{n_1}\right\}. \] Then we take $V_2 = V_1^{\perp}$. Using the bi-Poisson reduction (Theorem~\ref{T:BiPoissReduction}) for $U = L_1$ we get that $L = L_1 \oplus L_2$, where $L_2$ is a bi-Lagrangian subspace of $(V_2, \mathcal{P}_2)$. We got the required decomposition \eqref{Eq:Decomp2CompBiLagr}. Theorem~\ref{T:ExtractMaxBlockGen}  is proved. \end{proof}

Obviously, we can apply Theorem~\ref{T:ExtractMaxBlockGen} multiple times and potentially extract several Jordan blocks. 

\begin{corollary} \label{Cor:JordExtr}
Let $(V, \mathcal{P}) = \bigoplus_{i=1}^t \left(\bigoplus_{j=1}^{l_i} \mathcal{J}_{0, 2n_i} \right)$, where $n_1  > n_2 >\dots > n_t$, $P$ be the recursion operator and $L$ be a bi-Lagrnagian subspace with \[\dim (L + \operatorname{Ker}P^{n_1 - 1} ) / \operatorname{Ker}P^{n_1 - 1}  = k.\]
Then there exists a decomposition of $(V, \mathcal{P})$ into a bi-orthogonal direct sum with the corresponding decomposition of $L$: \[(V, \mathcal{P}) = \bigoplus_{i=1}^k  \mathcal{J}_{0, 2n_1} \oplus \left(\hat{V}, \hat{\mathcal{P}}\right), \qquad L = \bigoplus_{i=1}^k L_i \oplus \hat{L},\] where $L_i$ are generic bi-Lagrangian subspaces of $\mathcal{J}_{0, 2n_1}$  and $\hat{L} $ is a bi-Lagrangian subspace of $\left(\hat{V}, \hat{\mathcal{P}}\right)$ that  satisfies \[ \operatorname{Im}P^{n_1 - 1} \subset \hat{L} \subset \operatorname{Ker} P^{n_1 - 1}.\] \end{corollary}

\subsection{Generic bi-Lagrangian subspaces} \label{SubS:GenBiLag}

In this section we prove that there is a unique maximal dimensional $\operatorname{Aut} \left(V, \mathcal{P}\right)$-orbit in $\operatorname{BLG}\left(V, \mathcal{P}\right)$ and describe all bi-Lagrangian subspaces in it.

\begin{theorem} \label{T:BiLagrGenDecomp} Let $(V, \mathcal{P}) = \bigoplus_{i=1}^t \left(\bigoplus_{j=1}^{l_i} \mathcal{J}_{0, 2n_i} \right)$, where $n_1  > n_2 >\dots > n_t$, $P$ be the associated nilpotent recursion operator  and $L \subset (V, \mathcal{P})$ be a bi-Lagrangian subspace. Then the following conditions are equivalent. 

\begin{enumerate}

\item $L$ belongs to the unique top-dimensional $\operatorname{Aut} \left(V, \mathcal{P}\right)$-orbit $O_{\max}$.

\item For each $j = 1,\dots, t$ we have  \begin{equation} \label{Eq:DimTopPart} \dim \left( L \cap \operatorname{Ker} P^{n_j} / \left(\operatorname{Ker} P^{n_j} \cap  \operatorname{Im} P + \operatorname{Ker} P^{n_j-1} \right)  \right)  = l_j.\end{equation}

\item There is a standard basis  $e^{ij}_k, f^{ij}_k$, where $i=1, \dots, t, j=1,\dots, l_i, k=1, \dots, n_i$, from the JK theorem such that \begin{equation} \label{Eq:ExDecomptop} L = \operatorname{Span}\left\{e^{ij}_k\right\}_{i=1, \dots, t, j=1,\dots, l_i, k=1, \dots, n_i}.\end{equation}

\end{enumerate}

\end{theorem}

\begin{definition} Bi-Lagrangian subspaces $L$ as in Theorem~\ref{T:BiLagrGenDecomp} are called \textbf{generic}.
\end{definition}

Simply speaking, $L$ decomposes into a sum $L = \oplus_{i=1}^N L_i$ of top-dimensional subspaces $L_i \subset \mathcal{J}_{0, 2n_i}$ (for some JK decomposition  $(V, \mathcal{P}) =  \bigoplus_{i=1}^N \mathcal{J}_{0, 2n_i}$). In \eqref{Eq:TopOrbit} we realize $\mathcal{J}_{0, 8} \oplus\mathcal{J}_{0, 8} \oplus \mathcal{J}_{0, 4}$ as a Young-like diagram. The shaded area corresponds to a generic bi-Lagrangian subspace. 
\begin{equation} \label{Eq:TopOrbit}
  \begin{tabular}{|c|c|c|c|c|c|} 
  \cline{1-4} {\cellcolor{gray!25} } & &  {\cellcolor{gray!25} }  &  & \multicolumn{1}{c}{} & \multicolumn{1}{c}{} \\
 \cline{1-4} {\cellcolor{gray!25} } & &  {\cellcolor{gray!25} }  &  & \multicolumn{1}{c}{} & \multicolumn{1}{c}{} \\
     \hline {\cellcolor{gray!25} } & &  {\cellcolor{gray!25} }  & &  {\cellcolor{gray!25} }  &  \\ 
   \hline {\cellcolor{gray!25} } &  &  {\cellcolor{gray!25} }  &   &  {\cellcolor{gray!25} }  &  \\ \hline
  \end{tabular} 
\end{equation}

\subsubsection{Symplectic structure on top parts of Jordan blocks}

Before proving Theorem~\ref{T:BiLagrGenDecomp} let us explain the condition~\eqref{Eq:DimTopPart}. Consider a subspace \[ U_i = \operatorname{Ker} P^{n_i} /\left(\operatorname{Ker} P^{n_i} \cap \operatorname{Im} P  +  \operatorname{Ker} P^{n_i-1}\right). \] If we represent $(V, \mathcal{P})$ as a Young-like diagram similar to \eqref{Eq:Alternative}, then $U_i$ is the ``top row'' of Jordan blocks with height $n_i$. For example, in \eqref{Eq:TopRows} for $\mathcal{J}_{0, 8} \oplus\mathcal{J}_{0, 6} \oplus \mathcal{J}_{0, 4}$ shaded cells represent $\left(\operatorname{Ker} P^{n_i} \cap \operatorname{Im} P\right)  +  \operatorname{Ker} P^{n_i-1}$ for $i=2$ and ``x'' marks the basis of $U_2$.  \begin{equation} \label{Eq:TopRows}
  \begin{tabular}{|c|c|c|c|c|c|} 
  \cline{1-2}  &  &  \multicolumn{1}{c}{} & \multicolumn{1}{c}{} & \multicolumn{1}{c}{} & \multicolumn{1}{c}{} \\
 \cline{1-4} {\cellcolor{gray!25} } &  {\cellcolor{gray!25} }   &  X  & X & \multicolumn{1}{c}{} & \multicolumn{1}{c}{} \\
     \hline {\cellcolor{gray!25} } & {\cellcolor{gray!25} }&  {\cellcolor{gray!25} }  & {\cellcolor{gray!25} }  &  {\cellcolor{gray!25} }  &  {\cellcolor{gray!25} } \\ 
   \hline {\cellcolor{gray!25} } &  {\cellcolor{gray!25} } &  {\cellcolor{gray!25} }  &  {\cellcolor{gray!25} }  &  {\cellcolor{gray!25} }  &  {\cellcolor{gray!25} }  \\ \hline
  \end{tabular} 
\end{equation}

First, note that a pair of forms $A$ and $B$ actually define a sequence of $2$-forms $A_n$ on $V$. The next statement is trivial. 

\begin{assertion} \label{A:SeqForms} Let $\mathcal{P} = \left\{A + \lambda B\right\}$ be a Jordan pencil on $V$ and $P = B^{-1}A$. Consider the  $2$-forms \begin{equation} \label{Eq:AnForms} A_0 = B, \quad A_1 = A = B \circ P, \quad \dots, \quad A_k = B \circ P^k, \quad  \dots \end{equation} Then any bi-Lagrangian subspace $L$ is isotropic w.r.t. all forms $A_k$, $k \geq 0$. \end{assertion}

Now it is easy to prove the following. 

\begin{assertion} \label{A:InducedImP} Let $(V, \mathcal{P}) = \bigoplus_{i=1}^t \left(\bigoplus_{j=1}^{l_i} \mathcal{J}_{0, 2n_i} \right)$, where $n_1 > \dots > n_t$. The form $A_{n_i-1}$ given by \eqref{Eq:AnForms} induces a symplectic structure on each  \[ U_i = \operatorname{Ker} P^{n_i} /\left(\operatorname{Ker} P^{n_i} \cap \operatorname{Im} P  +  \operatorname{Ker} P^{n_i-1}\right)  \]  for $i=1,\dots, t$. Moreover, for any bi-Lagrangian subspace $L \subset (V, \mathcal{P})$ the subspace \begin{equation} \label{Eq:QuotTopJord} \hat{L}_j = (L \cap \operatorname{Ker} P^{n_j} + W_j)/W_j,  \end{equation} where \[W_j = \left(\operatorname{Ker} P^{n_j} \cap  \operatorname{Im} P + \operatorname{Ker} P^{n_j-1} \right)\] is isotropic w.r.t. the induced symplectic form on $U_i$. \end{assertion}

\begin{proof}[Proof of Assertion~\ref{A:InducedImP}] It is easy to see that \[  \operatorname{Ker} A_{n_i-1} \bigr|_{\operatorname{Ker} P^{n_i} } = \operatorname{Ker} P^{n_i} \cap \operatorname{Im} P  +  \operatorname{Ker} P^{n_i-1}.\] Therefore, $A_{n_i-1}$ induces the symplectic structure on $U_i$. By Assertion~\ref{A:SeqForms} $L$ isotropic w.r.t. $A_{n_i-1}$ and hence the subspace \eqref{Eq:QuotTopJord} is isotropic in $U_i$. Assertion~\ref{A:InducedImP} is proved. \end{proof}

\subsubsection{Proof of Theorem~\ref{T:BiLagrGenDecomp}}

\begin{itemize}

\item[$(1 \Rightarrow 2)$.] The dimension \[d = \dim \left( L + \operatorname{Ker} P^{n_1 -1}   / \operatorname{Ker} P^{n_1 -1} \right) \] is a lower semicontinuous function on $\operatorname{BLG}(V, \mathcal{P})$ (i.e. it can only increase if we slightly change $L$). By Assertion~\ref{A:InducedImP} $d \leq l_1$ and  the equality hods for the subspace \eqref{Eq:ExDecomptop} (for any JK decomposition). Hence $d = l_1$. Using Corollary~\ref{Cor:JordExtr} we can extract $l_1$ maximal Jordan blocks. Therefore, a generic bi-Lagrangian subspace $L$ has the form \[ L = L_1 \oplus L_2,\] where $L_1$ is the sum  of $l_1$ generic bi-Lagrangian subspaces in $\mathcal{J}_{0, 2n_i}$ and $L_2$ is a bi-Lagrangian subspace in the sum of smaller Jordan blocks $ \bigoplus_{i=2}^t \left(\bigoplus_{j=1}^{l_i} \mathcal{J}_{0, 2n_i} \right)$. By repeating the same logic for $L_2$, we get \eqref{Eq:DimTopPart}.

\item[$(2 \Rightarrow 3)$.] Decompose $L$ into a sum of bi-Lagrangian subspaces of each Jordan block using Corollary~\ref{Cor:JordExtr} and then take the basis as in Theorem~\ref{T:BiLagr_One_Jordan_Canonical_Form}.

\item[$(3 \Rightarrow 2)$.] Trivial.

\item[$(2 \Rightarrow 1)$.] We've already established that \eqref{Eq:DimTopPart} holds for the maximal orbit $O_{\max}$. Also, any bi-Lagrangian subspace $L$ satisfying \eqref{Eq:DimTopPart} belong the orbit of the same subspace \eqref{Eq:ExDecomptop}. Therefore, these subspaces $L$ form the maximal orbit $O_{\max}$.
 \end{itemize}

Theorem~\ref{T:BiLagrGenDecomp} is proved. 

\subsubsection{Jordan normal forms for bi-Lagrangian subspaces}

As a $P$-invariant subspace, any bi-Lagrangian subspace $L \subset (V,\mathcal{P})$ has the following obvious invariants:
\begin{itemize}
\item Jordan normal form for the operators $P\bigr|_{L}$ and $P\bigr|_{V/L}$.
\end{itemize}

The next statement shows that a bi-Lagrangian space $L$ is generic if and only the Jordan normal form of $P\bigr|_{L}$ is a ``half'' of $P$'s Jordan form. We denote the $n\times n$ Jordan block (from the Jordan normal form) as $ J(n)$.

\begin{lemma} \label{L:JNFGeneric} Let $(V, \mathcal{P}) = \bigoplus_{i=1}^N  \mathcal{J}_{0, 2n_i}$, where $n_1 \geq \dots \geq n_N$, $P$ be the nilpotent recursion operator and $L \subset (V,\mathcal{P})$ be a bi-Lagrangian subspace.

\begin{enumerate}
    \item \label{Item:LemmaInvForm} If  $P\bigr|_{L} \sim  J(h_1) \oplus \dots \oplus J(h_M)$, where $h_1 \geq \dots \geq h_M$, then $M \leq 2N$ and \begin{equation} \label{Eq:Inequality} (h_1, \dots, h_M,0, \dots, 0) \leq (n_1, n_1, \dots, n_N, n_N) \end{equation} component-wise. 
    
    \item $L$ is generic if and only if $M = N$, in which case $P\bigr|_{L} \sim  J(n_1) \oplus \dots J(n_M)$.
\end{enumerate}  

\end{lemma}

\begin{proof}[Proof of Lemma~\ref{L:JNFGeneric}]

\begin{enumerate}

\item It is standard knowledge that the Jordan structure of Jordan normal form of $P\bigr|_{L}$ is ``contained'' within the Jordan normal form of $P$  for any P-invariant subspace L. \eqref{Eq:Inequality} follows from counting Jordan blocks of size at least $p$ via  $\dim \operatorname{Ker}P^p - \dim \operatorname{Ker}P^{p-1}$.

\item For generic L, the claim follows from Theorem~\ref{T:BiLagrGenDecomp}. If $M = N$, let $l_1$  be the multiplicity of $n_1$ in $\left\{n_i\right\}$. As $\sum h_i = \sum n_i$, there are at least $l_1$ values $h_i = n_1$.  Corollary~\ref{Cor:JordExtr} yields remaining $h_j \leq n_2$, and iterating gives $h_i = n_i$ for all $i$. \end{enumerate}

Lemma~\ref{L:JNFGeneric} is proved. \end{proof}

\begin{remark} The result in Item~\ref{Item:LemmaInvForm} of Lemma~\ref{L:JNFGeneric} can be derived from a theorem due to  P.\,R.~Halmos (\cite{Halmos71}, simplified in \cite{Faouzi01}, \cite{Domanov10}): for any $P$-invariant subspace $W$ within a finite-dimensional complex vector space $V$ there exists an operator $B$ which commutes with $A$ and such that $W = B(V)$. Operators $B$ commuting with $A$ have block-matrix structure \eqref{E:bsp_algebra_matrix}  (with arbitrary matrices $C^{i,j}_s$).  \end{remark}

\subsubsection{Complementary bi-Lagrangian subspaces}

Any Lagrangian subspace of a symplectic space has a complementary Lagrangian subspace. For bi-Lagrangians subspaces this is only true in the Jordan case for generic bi-Lagrangian subspaces.

\begin{theorem} \label{T:Complem} Let $(V, \mathcal{P}) = \bigoplus_{i=1}^N  \mathcal{J}_{0, 2n_i}$, where $n_1  \geq n_2 \geq \dots \geq n_N$, and $L \subset (V, \mathcal{P})$ be a bi-Lagrangian subspace.

\begin{enumerate} 

\item There is a complementary bi-Lagrangian subspace $L'$, i.e.\begin{equation} \label{Eq:ComplPair} (V, \mathcal{P})  = L \oplus L',\end{equation} if and only if $L$ is generic.

\item Moreover, for any complementary pair of bi-Lagrangian subspaces $L, L'$ there is a standard basis $e^i_{j}, f^i_j$, where $i=1,\dots, N, j=1,\dots, n_i$ such that \begin{equation} \label{Eq:ExDecomptop2} L = \operatorname{Span}\left\{e^i_j\right\}_{i=1,\dots, N, j=1,\dots, n_i}, \qquad L' = \operatorname{Span}\left\{f^i_j\right\}_{i=1,\dots, N, j=1,\dots, n_i}.\end{equation}

\end{enumerate}

\end{theorem}

\begin{proof}[Proof of Theorem~\ref{T:Complem}] For any generic $L \in O_{\max}$ we can take the complementary subspace \eqref{Eq:ExDecomptop2} in the coordinates from Theorem~\ref{T:BiLagrGenDecomp}. Next, consider a complementary pair \eqref{Eq:ComplPair}. Let $P$ be the recursion operator on $(V, \mathcal{P})$. Since $L$ is $P$-invariant, we can choose its basis  \[L =\operatorname{Span}\left\{\hat{e}^i_j\right\}_{i=1,\dots, M, j=1,\dots, q_i} \] such that \[ P\hat{e}^i_j =\hat{e}^i_{j+1}, \qquad P\hat{e}^i_{q_i} = 0. \] Since $(V, B)$ is a symplectic space, we can choose the complementary basis in $L'$: \[L' =\operatorname{Span}\left\{\hat{e}^i_j\right\}_{i=1,\dots, M, j=1,\dots, q_i}, \qquad B(\hat{e}^i_j, \hat{f}^k_l) = \delta^i_k \delta^j_l.\] Recall that  $L'$ is $P$-invariant and $P$ is self-adjoint w.r.t. $B$. Hence, we can easily find $P\hat{f}^k_l$ from the equalities \[ B(\hat{e}^i_j, P\hat{f}^k_l) = B(P\hat{e}^i_j, \hat{f}^k_l).\] We get that \[P\hat{f}^k_l = \hat{f}^k_{l-1},\qquad P\hat{f}^k_1 = 0.\] It is easy to see that $\hat{e}^i_j, \hat{f}^i_j$ is a standard basis for the sum of $M$ Jordan blocks $\mathcal{J}_{0, 2q_i}, i=1,\dots, M$. The sizes of Jordan blocks in the JK theorem are uniquely defined, hence $M=N$ and the sizes $n_i$ match the sizes $q_i$ (except for possibly a different order). Therefore, \eqref{Eq:ExDecomptop2} holds and $L$ belongs to the top-dimensional $\operatorname{Aut}(V, \mathcal{P})$-orbit $O_{\max}$. Theorem~\ref{T:Complem} is proved. 
\end{proof}

\begin{corollary} Let $(V,\mathcal{P})  = \bigoplus_{i=1}^N  \mathcal{J}_{0, 2n_i}$. The following conditions are equivalent:

\begin{enumerate}

\item \label{Item:Complem} For any bi-Lagrangian subspace $L \subset (V,\mathcal{P})$ there is a complementary bi-Lagrangian subspace $L'$.

\item \label{Item:Jord22} All Jordan blocks are $2\times 2$, i.e. $n_i = 2$ for all $i$.

\end{enumerate}

\end{corollary}

It was shown in \cite[Theorem 2.4]{RanRodman93} that Condition~\ref{Item:Jord22} implies  Condition~\ref{Item:Complem}.

\subsection{Dimension of bi-Lagrangian Grassmanian} \label{SubS:DimBLG}

\begin{theorem} \label{T:BiLagrJord} Let $(V, \mathcal{P}) = \bigoplus_{j=1}^N \mathcal{J}_{\lambda, 2n_j}$, where $n_1 \geq \dots \geq n_N$. Then \begin{equation} \label{Eq:DimBiLagrJord} \dim \operatorname{BLG}(V, \mathcal{P}) = \sum_{j=1}^N j \cdot  n_j. \end{equation} 
\end{theorem}

\begin{remark} Let us group all equal values $n_i$ from Theorem~\ref{T:BiLagrJord} together. Assume that there are $l_1$ values equal to $\hat{n}_1$, $l_2$ values equal to $\hat{n}_2$, etc., $l_t$ values equal to $\hat{n}_t$. Then  \begin{equation} \label{Eq:DimBLGComb} \dim \operatorname{BLG}(V, \mathcal{P}) = \sum_{j=1}^t \frac{l_j(l_j + 1)}{2} \hat{n}_j  + \sum_{1 \leq i < j \leq t}  l_i l_j\hat{n}_j.\end{equation} \end{remark}

\begin{proof}[Proof of Theorem~\ref{T:BiLagrJord} ] The dimension of $\operatorname{BLG}(V, \mathcal{P})$ is equal to the dimension of the open orbit $O_{\max}$ of the $\operatorname{Aut}(V, \mathcal{P})$-action on it. By 
Theorem~\ref{T:BiLagrGenDecomp} the subspace $L = \operatorname{Span}\left\{e^i_j\right\}$  lies in the open orbit $O$.  It suffices to find the dimension of the stabilizer in the automorphism Lie algebra \[\operatorname{St}_L = \left\{ C \in \operatorname{aut}(V, \mathcal{P}) \quad \bigr| \quad CL = L\right\},\] since \[\dim O_{\max} = \dim \operatorname{Aut}(V, \mathcal{P}) - \dim \operatorname{St}_L.\]  
Let us describe the Lie algebra of $\operatorname{St}_L$ By Theorem~\ref{T:BiSymp_General_Jordan_Case_Mega} all elements of the Lie algebra $C \in \operatorname{aut}(V, \mathcal{P})$ have the form 
\eqref{E:bsp_algebra_matrix}, where the blocks $C^{i, j}_s$ satisfy \eqref{E:Cond_on_BiSymp_Jordan_Case}. Such an element  $C$ preserves $L = \operatorname{Span}\left\{e^i_j\right\}$ if and only if each $2l_i \times 2l_j$ block $C^{i,j}_s$ has the form \begin{equation} \label{Eq:StabBiLagrGenBlock} C^{i, j}_s = \left( \begin{matrix} X^{i,j}_s & Y^{i,j}_s \\ 0 & Z^{i,j}_s\end{matrix}\right). \end{equation}

Now, let us count $\dim \operatorname{St}_L $.

\begin{itemize}

\item For $i <j$ all $C^{i,j}_s$ are arbitrary matrices of the form~\eqref{Eq:StabBiLagrGenBlock} and $C^{j,i}_s$ are uniquely determined by \eqref{E:Cond_on_BiSymp_Jordan_Case}. Each matrix $C^{i,j}_s$  has $3l_i l_j$ coefficients.

\item For $i = j$ the matrices $C^{i,i}_s$ are $2l_i \times 2l_i$ matrices of the form \[  C^{i, j}_s = \left( \begin{matrix} X^{i,i}_s & Y^{i,i}_s \\ 0 & -\left(X^{i,i}_s\right)^T\end{matrix}\right),\] where $Y^{i,i}_s  = \left(Y^{i,i}_s \right)^T$. Each matrix $C^{i,i}_s$  has $\displaystyle \frac{l_i \left( 3l_i + 1\right)}{2}$ coefficients.

\end{itemize}
In total, \[\dim \operatorname{St}_L = \sum_{j=1}^t \hat{n}_j l_j \left(\frac{3l_j + 1}{2} + 3 \sum_{i=1}^{j-1} l_i \right). \]  We know $\dim \operatorname{Aut}(V, \mathcal{P})$ from Corollary~\ref{C:DimAutJordanCase}. Thus, \[ \dim \operatorname{BLG}(V, \mathcal{P}) = \dim \operatorname{Aut}(V, \mathcal{P}) - \dim \operatorname{St}_L  =\sum_{j=1}^t \hat{n}_j l_j \left(\frac{l_j + 1}{2} +  \sum_{i=1}^{j-1} l_i \right) = \sum_{j=1}^s j \cdot  n_j .\] Theorem~\ref{T:BiLagrJord}  is proved. \end{proof} 

\begin{remark} We generalize Theorem~\ref{T:BiLagrJord} in Section~\ref{S:DecomposableSection} below.  In that section we show that all generic bi-Lagrangian subspaces are semisimple. In Theorem~\ref{T:DimDecomp} we calculate the dimension of any $\operatorname{Aut}(V,\mathcal{P})$-orbit for an arbitrary semisimple bi-Lagrangian subspace $L \in (V,\mathcal{P})$. \end{remark}

\section{Invariant subspaces} \label{S:InvSubspaces}

In this section we study $\operatorname{Aut}(V,\mathcal{P})$-orbits of vectors and $\operatorname{Aut}(V,\mathcal{P})$-invariant subspaces (for a Jordan pencil $\mathcal{P})$). 

\begin{itemize}

\item In Section~\ref{SubS:InvSubspaces} we describe all  $\operatorname{Aut}(V,\mathcal{P})$-invariant subspaces $U \subset (V,\mathcal{P})$. Also, in Section~\ref{SubS:OrbitsVect} we describe all  $\operatorname{Aut}(V,\mathcal{P})$-orbits of vectors $v \in (V,\mathcal{P})$.

\item In Section~\ref{S:InvBiLagr} we describe all $\operatorname{Aut}(V,\mathcal{P})$-invariant bi-isotropic and bi-Lagrangian subspaces.

\end{itemize}

\subsection{Structure of invariant subspaces} \label{SubS:InvSubspaces}

Let $\mathcal{P} = \left\{ A + \lambda B\right\}$ be a Jordan pencil on $V$. What subspaces $U \subset (V, \mathcal{P})$ are $\operatorname{Aut}(V, \mathcal{P})$-invariant? In other words, what subspaces $U$ are invariantly defined and do not depend on a JK decomposition of $(V, \mathcal{P})$? We'll simply refer to such subspaces as \textbf{invariant}. For the sum of equal Jordan blocks $\bigoplus_{i=1}^l \mathcal{J}_{0, 2n}$ there is only $n+1$ invariant subspace, characterized by the maximum height of the vectors they contain.

\begin{assertion} \label{A:InvSubsJord_OneJord}
The invariant subspaces of $(V, \mathcal{P}) = \bigoplus_{i=1}^l \mathcal{J}_{0, 2n}$ are \[\textnormal{Ker} P^i = \textnormal{Im} P^{n-i}, \qquad i = 0, \dots, n.\]
\end{assertion}

\begin{proof}[Proof of Assertion~\ref{A:InvSubsJord_OneJord}] Obviously, the subspaces $\operatorname{Ker} P^i$ are invariant. It remains to prove that there are no other invariant subspaces. A subspace $U \subset (V, \mathcal{P})$ is  $\operatorname{Aut}(V, \mathcal{P})$-invariant if for any element of the Lie algebra $C \in \operatorname{aut}(V, \mathcal{P})$ we have $C U \subset U$. By Theorem~\ref{T:BiSymp_General_Jordan_Case_Mega} the elements of $\operatorname{aut}(V, \mathcal{P})$  have the form \[C=\left( \begin{matrix} A_1 & & \\ \vdots & \ddots & \\ A_n & \dots & A_{1}\end{matrix} \right), \] where $A_i \in \operatorname{sp}(2)$. For any vector $v\in (V, \mathcal{P})$ with height $h$ it is easy to see that the vectors $Cv$ span the subspace $\operatorname{Ker} P^h$. Thus, if $U$ is invariant, then it is one of the kernels $\operatorname{Ker} P^i$. Assertion~\ref{A:InvSubsJord_OneJord} is proved. \end{proof}

In general case, for a sum of Jordan blocks $(V, \mathcal{P}) = \bigoplus_{i=1}^N \mathcal{J}_{0, 2n_i}$, all the subspaces $\operatorname{Ker} P^k$ and $\operatorname{Im} P^l$ are obviously invariant. And any their sums and intersections are also invariant. It turns out, as we prove below, that there are no other invariant subspaces. A priori the number of such subspaces \[W = \bigoplus_{i=1}^M (\textnormal{Ker} P^{k_i}
\cap \textnormal{Im} P^{l_i}),\] could be infinite, but we also prove that any invariant subspace $U\subset \bigoplus_{i=1}^t \mathcal{J}_{0, 2n_i}$ is a sum of invariant subspaces of the Jordan blocks \[\mathcal{J}_{0, 2n_i} \cap \operatorname{Ker} P^{h_i}\] with some simple conditions on the heights $h_i$. Thus, in the Jordan case, there is a finite number of invariant subspaces $U \subset  \bigoplus_{i=1}^N \mathcal{J}_{0, 2n_i}$.

\begin{theorem} \label{T:JordInvSubs}
Let $\mathcal{P} = \left\{ A + \lambda B\right\}$ be a Jordan pencil on $V$, the recursion operator $P = B^{-1}A$ be nilpotent and \[ (V, \mathcal{P}) = \bigoplus_{i=1}^N \mathcal{J}_{0, 2n_i} \] be a Jordan--Kronecker decomposition, where $n_1 \geq n_2 \geq \dots \geq n_N$. Let $U \subset (V, \mathcal{P})$ be a subspace. The following conditions are equivalent:

\begin{enumerate}

\item $U$ is $\operatorname{Aut}(V, \mathcal{P})$-invariant.

\item $U$ has the form   \begin{equation}
\label{Eq:InvSubsJord_ImKer} W = \bigoplus_{i=1}^M (\textnormal{Ker} P^{k_i}
\cap \textnormal{Im} P^{l_i}), \end{equation} for some $M \in \mathbb{N}$ and 
$k_i, l_i \geq 0$, $i=1, \dots, M$. 

\item There exists $h_1,\dots, h_N$ such that \begin{equation} \label{Eq:InvHeightCond} 0 \leq h_{i} - h_{i+1} \leq n_{i} - n_{i+1}, \qquad i = 1, \dots, N-1,\end{equation} and $U$ has the form \begin{equation} \label{Eq:JordInvSumHeight} U =  \bigoplus_{i=1}^N \mathcal{J}_{0, 2n_i}^{\leq h_i},\end{equation}  where \begin{equation} \label{Eq:JordSubBlockHeightLeqH} \mathcal{J}_{0, 2n}^{\leq h} = \mathcal{J}_{0, 2n} \cap \operatorname{Ker} P^{h}.  \end{equation}

\end{enumerate}

\end{theorem}

\begin{corollary} Let $(V,\mathcal{P})$ be a Jordan bi-Poisson vector space and $P$ be the recursion operator. All $\operatorname{Aut}(V,\mathcal{P})$-subspaces  are $P$-invariant (and, hence, admissible). 
\end{corollary}

If we put $\Delta_N = h_N$ and $\Delta_i = h_i - h_{i+1}$ for $i=1, \dots, N-1$, then \[ h_i = \sum_{j =i}^{N} \Delta_j\] and the condition~\eqref{Eq:InvHeightCond} takes the form \[ 0 \leq \Delta_i \leq n_i - n_{i+1},\] where we formally assume that $n_{N+1} = 0$. Now we can easily count the number of possible $\Delta_i$ (and $h_i$). 

\begin{corollary}  For a bi-Poisson subspace $ (V, \mathcal{P}) = \bigoplus_{i=1}^N \mathcal{J}_{0, 2n_i}$, where $n_1 \geq n_2 \geq \dots \geq n_N$, the number of $\operatorname{Aut}(V, \mathcal{P})$-invariant subspaces is \begin{equation} \label{Eq:NumInvSubs}  \prod_{i=1}^{N} \left( n_{i} - n_{i+1} + 1 \right). \end{equation}  Here we formally put $n_{N+1} = 0$.
\end{corollary}

\begin{remark} The resulting count corresponds to the number of hyperinvariant subspaces outlined in  \cite[Theorem 9.4.2]{Gohberg86}
(cf. Remark~\ref{Rem:HyperInv} below).
\end{remark}

Note that if $n_{i+1} =n_i$, then $h_{i+1} = h_i$. In particular, if all Jordan blocks are $2n \times 2n$, then we get $n+1$ invariant subspace, as in Assertion~\ref{A:InvSubsJord_OneJord}. For a sum of two Jordan blocks $\mathcal{J}_{0, 2n_1} \oplus \mathcal{J}_{0, 2n_2}$ for the difference of heights $h_2 - h_1$ we have the following:

\begin{itemize}

\item the difference is minimal $ h_1 - h_2 = 0$ for $\operatorname{Ker} P^{h_2}$, where $h_2 \leq n_2$.

\item the difference is maximal $ h_1 - h_2 = n_1 - n_2$ for  $\operatorname{Im} P^{h_1}$, where $h_1 \leq n_1 - n_2$.

\end{itemize}

For example, the proper invariant subspaces of $\mathcal{J}_{0, 6} \oplus \mathcal{J}_{0, 4}$ are informally visualized in \eqref{Eq:SumBlocksInvariant}. Here $\mathcal{J}_{0, 6}$ and $\mathcal{J}_{0, 4}$ are represented as $3 \times 2$ and $2 \times 2$ rectangles. The shaded cells represent an invariant subspace.

\begin{equation} \label{Eq:SumBlocksInvariant}
  \begin{tabular}{|c|c|c|c|} 
 \cline{1-2} &  & \multicolumn{1}{c}{} & \multicolumn{1}{c}{} \\
     \hline   & & &  \\ 
   \hline   {\cellcolor{gray!25} } & {\cellcolor{gray!25} } & &  \\ \hline
  \end{tabular} 
  %%%
  \qquad
  \begin{tabular}{|c|c|c|c|} 
 \cline{1-2} &  & \multicolumn{1}{c}{} & \multicolumn{1}{c}{} \\
     \hline   & & &  \\ 
   \hline   {\cellcolor{gray!25} } & {\cellcolor{gray!25} } &   {\cellcolor{gray!25} } & {\cellcolor{gray!25} }   \\ \hline
  \end{tabular} 
  %%%
\qquad
  %%%
 \begin{tabular}{|c|c|c|c|} 
 \cline{1-2} &  & \multicolumn{1}{c}{} & \multicolumn{1}{c}{} \\
     \hline     {\cellcolor{gray!25} } & {\cellcolor{gray!25} }  & &  \\ 
   \hline   {\cellcolor{gray!25} } & {\cellcolor{gray!25} } &   {\cellcolor{gray!25} } & {\cellcolor{gray!25} }  \\ \hline
  \end{tabular} 
  %%%
  \qquad
   \begin{tabular}{|c|c|c|c|} 
 \cline{1-2} &  & \multicolumn{1}{c}{} & \multicolumn{1}{c}{} \\
     \hline     {\cellcolor{gray!25} } & {\cellcolor{gray!25} }  &   {\cellcolor{gray!25} } & {\cellcolor{gray!25} }  \\ 
   \hline   {\cellcolor{gray!25} } & {\cellcolor{gray!25} } &  {\cellcolor{gray!25} } & {\cellcolor{gray!25} }   \\ \hline
  \end{tabular} 
\end{equation}

In the general case, the invariant subspaces have a similar structure.

\begin{proof}[Proof of Theorem~\ref{T:JordInvSubs}] The proof is in several steps.

\begin{enumerate}

\item[$2) \Rightarrow 1)$.] The subspaces $\operatorname{Ker} P^k$ and $\operatorname{Im} P^l$ are obviously $\operatorname{Aut}(V, \mathcal{P})$-invariant. And any sums and intersections of $\operatorname{Aut}(V, \mathcal{P})$-invariant subspaces is also $\operatorname{Aut}(V, \mathcal{P})$-invariant.

\item[$3) \Rightarrow 2)$.] It is easy to see that $\operatorname{Ker} P^k$ and $\operatorname{Im} P^l$ have the form \eqref{Eq:JordInvSumHeight} and that the set of subspaces \eqref{Eq:JordInvSumHeight} is closed under the operations of sum and intersection of subspaces. 

\item[$1) \Rightarrow 3)$.] We split these part of the proof into two statements.

\begin{assertion}\label{A:InvDecomposes}
Any invariant subspace $U \subset \bigoplus_{i=1}^N \mathcal{J}_{0, 2n_i}$ has the from \eqref{Eq:JordInvSumHeight}, i.e. it is a sum of invariant subspaces of each Jordan block: $U = \bigoplus_{i=1}^N \mathcal{J}_{0, 2n_i}^{\leq h_i}$. 
\end{assertion}

\begin{proof}[Proof of Assertion~\ref{A:InvDecomposes}] 
The group of bi-Poisson automorphisms $\operatorname{Aut}\left(\bigoplus_{i=1}^N \mathcal{J}_{0, 2n_i} \right)$ obviously contains the direct product of bi-Poisson automorphisms for each Jordan block: \begin{equation} \label{Eq:ProdSubGroupAut} \prod_{i=1}^N \operatorname{Aut}\left(\mathcal{J}_{0, 2n_i} \right) \subset \operatorname{Aut}\left(\bigoplus_{i=1}^N \mathcal{J}_{0, 2n_i}\right).\end{equation} Thus, if $U\subset \bigoplus_{i=1}^N \mathcal{J}_{0, 2n_i}$  is invariant, then the intersection with each Jordan block $U \cap \mathcal{J}_{0, 2n_i}$ is also an invariant subspace of this block. It remains to prove that \[ U = \bigoplus_{i=1}^N U \cap \mathcal{J}_{0, 2n_i}.\] Take any vector $v \in U$, it has the form $v = \sum_{i=1}^N v_i$, where $v_i \in \mathcal{J}_{0, 2n_i}$. We want to prove that $v_i \in U \cap \mathcal{J}_{0, 2n_i}$. By Assertion~\ref{A:InvSubsJord_OneJord} $U \cap \mathcal{J}_{0, 2n_i} = \mathcal{J}_{0, 2n_i}^{\leq h_i}$.  Thus, it suffices to find a vector $u_i \in U \cap \mathcal{J}_{0, 2n_i}$ with the same height as $v_i$. Since $U$ is $\operatorname{Aut}\left(\bigoplus_{i=1}^N \mathcal{J}_{0, 2n_i}\right)$-invariant and we act on $v$ by a direct product subgroup \eqref{Eq:ProdSubGroupAut}, \[ \left(g_i - \operatorname{Id} \right) v_i \in U, \qquad \forall g_i \in  \operatorname{Aut}\left(\mathcal{J}_{0, 2n_i} \right).\] By Theorem~\ref{T:BiSymp_General_Jordan_Case_Mega} there are non-degenerate elements $g_i - \operatorname{Id}$, and then the height of $\left(g_i - \operatorname{Id} \right) v_i$ equals the height of $v_i$. Thus, all $v_i \in U$ and $U$ decomposes into a sum of invariant subspaces.  Assertion~\ref{A:InvDecomposes} is proved.
\end{proof}

It remains to prove the inequalities~\eqref{Eq:InvHeightCond}. It suffices to consider a case of two Jordan blocks.

\begin{assertion}\label{A:InvTwoBlocks}
Let $U \subset \mathcal{J}_{0, 2n_1} \oplus \mathcal{J}_{0, 2n_2}$ be an invariant subspace, where $n_1 \geq n_2$.

\begin{enumerate}

\item \label{I:InvTwoBlocksI1} If $\mathcal{J}_{0, 2n_1}^{\leq h_1} \subset U$, where $h_1 \geq n_1 - n_2$, then $\mathcal{J}_{0, 2n_2}^{\leq h_1 - (n_1 -n_2)} \subset U$. 

\item \label{I:InvTwoBlocksI2} If $\mathcal{J}_{0, 2n_2}^{\leq h_2} \subset U$, where $h_2 \leq n_2$, then $\mathcal{J}_{0, 2n_1}^{\leq h_2}  \subset U$.

\end{enumerate}

\end{assertion}

\begin{proof}[Proof of Assertion~\ref{A:InvTwoBlocks}] If  $U$ is  $\operatorname{Aut}(V, \mathcal{P})$-invariant, then it is also invariant w.r.t. the corresponding Lie algebra $\operatorname{aut}(V, \mathcal{P})$. By Theorem~\ref{T:BiSymp_General_Jordan_Case_Mega} this Lie algebra contains the elements of the form \[
C = \left( \begin{array}{ccccc|ccc} 0 & & & & & &  &  
\\& 0 & & & & & & 
 \\ &  & \ddots &  & & C^{1,2}_1 &  & 
\\  &  & & \ddots & & \vdots & \ddots & 
\\  & &  & & 0  & C^{1,2}_{n_2} & \cdots & C^{1,2}_1 
\\ \hline C^{2,1}_1 &  & & & & 0 & &   
\\ \vdots & \ddots & & & &  & \ddots  &  
\\ C^{2,1}_{n_2} & \cdots & C^{2,1}_1 & & & &  & 0
\end{array} \right),
\] where \eqref{E:Cond_on_BiSymp_Jordan_Case} holds. We can solve \eqref{E:Cond_on_BiSymp_Jordan_Case}  for any matrices $C^{1,2}_k$ (or any matrices $C^{2,1}_k$). Thus, for any $v \in U \cap \mathcal{J}_{0, 2n_1}$ with height $h_1 \geq n_1 - n_2$ we can easily find $C \in \operatorname{aut}(V, \mathcal{P})$ such that the vector $C v \in U \cap \mathcal{J}_{0, 2n_2}$ has the height $h_1 - (n_1-n_2)$. By Assertion~\ref{A:InvDecomposes}  $\mathcal{J}_{0, 2n_2}^{\leq h_1 - (n_1 -n_2)} \subset U$. This proves the first part of Assertion~\ref{A:InvTwoBlocks}. The second part is proved similarly. Assertion~\ref{A:InvTwoBlocks} is proved. 
\end{proof}

\end{enumerate}

Theorem~\ref{T:JordInvSubs} is proved.  \end{proof}

\begin{remark} \label{Rem:HyperInv} $\operatorname{Aut}(V,\mathcal{P})$-invariant subspaces characterized in Theorem~\ref{T:JordInvSubs} coincide with the lattice of $P$-hyperinvariant subspaces described in \cite[Theorem 9.4.2]{Gohberg86}. A subspace $U \subset V$ is hyperinvariant if it is invariant w.r.t. any transformation that commutes with $P$. Note that operators commuting with $P$ are block-matrices \eqref{E:bsp_algebra_matrix}, with arbitrary matrices $C^{i,j}_s$.
\end{remark}

\begin{remark} \label{Rem:GenJordInv} Consider a general Jordan case, where  $(V, \mathcal{P})$ has the form \eqref{Eq:GenJordJKDec}. Denote by  $\mathcal{J}_{\lambda_j}$ the sum of all Jordan blocks with eigenvalue $\lambda_j$. Then any invariant  subspace $U \subset (V, \mathcal{P})$ is a sum of invariant subspaces $U_j  = U \cap \mathcal{J}_{\lambda_j} $. \end{remark}

\subsubsection{Properties of invariant subspaces}

Finding orthogonal complements of invariant subspaces is straightforward.

\begin{assertion}
Let $\mathcal{P} = \left\{ A + \lambda B\right\}$ be a Jordan pencil on $V$. Then the following holds.

\begin{enumerate}

\item Any $\operatorname{Aut}(V, \mathcal{P})$-invariant subspace $U$ is admissible. 

\item If $U = \bigoplus_{i=1}^N \mathcal{J}_{\lambda_i, 2n_i}^{\leq h_i}$, then $U^{\perp} =  \bigoplus_{i=1}^N \mathcal{J}_{\lambda_i, 2n_i}^{\leq n_i - h_i}$. 

\end{enumerate}
\end{assertion}

In particular, we have the following.

\begin{corollary} 
Let $\mathcal{P} = \left\{ A + \lambda B\right\}$ be a Jordan pencil on $V$ with the recursion operator $P$. Then subspaces $\operatorname{Im} P^i$ and $\operatorname{Ker} P^j$ are admissible and \begin{equation} \label{Eq:OrthKerIm} \left(\operatorname{Im} P^i\right)^{\perp} = \operatorname{Ker} P^i, \qquad \left(\operatorname{Ker} P^j\right)^{\perp} = \operatorname{Im} P^j. \end{equation}
\end{corollary} 

The next statement describes  unique minimal and maximal (proper) invariant subspaces. 

\begin{corollary} For  $(V, \mathcal{P}) =  \bigoplus_{i=1}^N \mathcal{J}_{0, 2n_i}$, where $n_1 \geq n_2 \geq \dots \geq n_N$ and $P$ is the recursion operator, any proper $\operatorname{Aut}(V,\mathcal{P})$-invariant subspace $U$ is contained between the minimal invariant subspace  $\operatorname{Im}P^{n_1 - 1}$ and the maximal invariant subspace $\operatorname{Ker}P^{n_1 - 1}$: \[\operatorname{Im}P^{n_1 - 1} \subset U \subset \operatorname{Ker}P^{n_1 - 1}.\] \end{corollary}

\subsubsection{Orbits of vectors under \texorpdfstring{$\operatorname{Aut}(V, \mathcal{P})$}{Aut(V, P)}-action} \label{SubS:OrbitsVect}

After we described $\operatorname{Aut}(V, \mathcal{P})$-invariant subspaces of $(V, \mathcal{P})$ we can easily describe the orbits of vectors under the action of $\operatorname{Aut}(V, \mathcal{P})$ action. These invariant subspaces partition $V$ into distinct sets, each one uniquely corresponding to a specific $\operatorname{Aut}(V, \mathcal{P})$-orbit. 

\begin{theorem} \label{T:VectOrbits} Let $(V, \mathcal{P}) = \bigoplus_{i=1}^t \left(\bigoplus_{j=1}^{l_i} \mathcal{J}_{0, 2n_i} \right)$ be a bi-Poisson space with $n_1 > n_2 > \dots > n_t$. For the action of the Lie group $\operatorname{Aut}(V, \mathcal{P})$ we have:

\begin{enumerate}
 
\item \textbf{Orbits of vectors}. Let $U_{\alpha}, \alpha \in A$ be the collection of all $\operatorname{Aut}(V, \mathcal{P})$-orbits. Given a vector $v \in  (V, \mathcal{P})$, define:

\begin{itemize}
    \item $W_v = \bigcap_{v \in U_{\alpha}} U_{\alpha}$ (minimal invariant subspace  containing $v$);
    \item $Z = \bigcup_{\alpha \in A, W_v\not \subset U_{\alpha}} \left( U_{\alpha} \cap W_v \right)$ (union of invariant subspaces strictly contained in $W_v$).
\end{itemize} Then the $\operatorname{Aut}(V, \mathcal{P})$-orbit of $v$ is the difference \begin{equation} \label{Eq:OvDiff} O_v = W_v - Z.\end{equation}

\item \textbf{Orbit Representation}. Fix any standard basis $e^{ij}_k, f^{ij}_k,$ where $i=1, \dots, t, j=1,\dots, l_i, k=1, \dots, n_i$. Any orbit $O_v$ contains a vector of the form  \begin{equation} \label{Eq:FormAutAct} w = \sum_{j=1}^t f^{j1}_{h_j}, \end{equation} where the heights $h_j$ satisfy \begin{equation} \label{Eq:autOrbHcond} 0 \leq h_j \leq n_j, \qquad 0 \leq h_j - h_{j+1} \leq n_j - n_{j+1}.\end{equation} We formally put $n_{t+1} = 0$ and $f^j_0  = 0$. The orbit $O_v$ is uniquely determined by the set of heights $\left\{h_j\right\}_{j=1,\dots, t}$.
\end{enumerate}

\end{theorem}

Since orbits of vectors and invariant subspaces are in one-to-one correspondence, the number of vector orbits is given by \eqref{Eq:NumInvSubs}. 

In Figure~\eqref{Eq:VecOrbit} we visualise an example of an orbit representative within the space $\mathcal{J}_{0, 6} \bigoplus \mathcal{J}_{0, 4} \oplus \mathcal{J}_{0, 4} \oplus \mathcal{J}_{0, 2}$. The representative element $f^{11}_2 + f^{21}_1 + f^{31}_1$ is depicted using shaded cells. Alternatively, the representative can be visualized as a polyline with vertices specified by coordinates $(j, h_j), j=1,\dots, t$, as shown on the right side of Figure~\eqref{Eq:VecOrbit}.

\begin{equation} \label{Eq:VecOrbit}
\begin{minipage}{0.45\textwidth}
  \begin{tabular}{|c|c|c|c|c|c|c|c|c} 
 \cline{1-2}  & & \multicolumn{1}{c}{}  & \multicolumn{1}{c}{} & \multicolumn{1}{c}{} & \multicolumn{1}{c}{}  & \multicolumn{1}{c}{} & \multicolumn{1}{c}{} \\
   \cline{1-6}  & {\cellcolor{gray!25} } &   & &  &   & \multicolumn{1}{c}{} & \multicolumn{1}{c}{} \\ 
   \hline &  &   & {\cellcolor{gray!25} }  &    &   &  & {\cellcolor{gray!25} }  \\ \hline
  \end{tabular} 
\end{minipage}
\hfill
\begin{minipage}{0.45\textwidth}
\begin{tikzpicture}
\draw (1,2) node[left] {(1, 2)} -- (2,1) node[below] {(2, 1)} -- (3,1) node[right] {(3, 1)};
\end{tikzpicture}
\end{minipage}
\end{equation}

\begin{remark} We described the orbits for one eigenvalue $\lambda = 0$. As in Remark~\ref{Rem:GenJordInv} in the general case the orbits  decompose into a product $O_v = \prod_j O_{v_j}$ corresponding to different eigenvalues $\lambda_j$.  \end{remark}

\begin{proof}[Proof of Theorem~\ref{T:VectOrbits}] The subspaces $O_v$, given by \eqref{Eq:OvDiff}, are invariant under the action of $\operatorname{Aut}(V, \mathcal{P})$. It remains to prove that each subset $O_v$ contains exactly one orbit. It suffices to prove that each vector $v$ can be brought to the form \eqref{Eq:FormAutAct} that satisfies \eqref{Eq:autOrbHcond}. This is done similarly to 
Assertion~\ref{A:InvDecomposes} and \ref{A:InvTwoBlocks}. Theorem~\ref{T:VectOrbits} is proved.\end{proof}

\begin{remark} Similar to Assertion~\ref{A:InvTwoBlocks} we can show that the orbit of $f^{j1}_{h_j}$ includes $f^{j-1,1}_{h_j}$ and $f^{j1}_{h_j - n_j + n_{j+1}}$. This allows us to:

\begin{itemize}

\item \textbf{Construct a Representative for a basis vector}. For any vector $e^{jl}_{n-h_j+1}$ or $f^{jl}_{h_j}$ its representative is \[ \sum_{i \leq j} f^{i1}_{h_j} + \sum_{i > j} f^{i1}_{h_j - \delta_i}, \qquad \delta_i = \sum_{k= j}^{i-1} n_k - n_{k+1}, \] where $f_m = 0$ for $m\leq 0$.

\item \textbf{Simplify the Representative}. We can reduce the number of terms in \eqref{Eq:FormAutAct}. Specifically, we can remove $f^{j1}_{h_j}$ if either $h_j = h_{j+1}$ or $h_{j} = h_{j-1} - n_{j-1} + n_j$.  This optimized representative takes the form \[ w' = \sum_{k=1}^S f^{j_k 1}_{h_{j_k}}, \qquad 1 \leq j_1 < \dots < j_S \leq t,\] where $0 < h_{j_k} - h_{j_{k+1}} < n_{j_k} - n_{j_{k+1}}$.

\end{itemize}

 \end{remark}

In Figure~\eqref{Eq:SimpVecOrbit} we demonstrate that for the space $\mathcal{J}_{0, 8} \bigoplus \mathcal{J}_{0, 6} \oplus \mathcal{J}_{0, 6} \oplus \mathcal{J}_{0, 4}$ the orbit representative for the basis vector $f^{21}_2$  is  $f^{11}_2 + f^{21}_2 + f^{31}_1$.

\begin{equation} \label{Eq:SimpVecOrbit}
\begin{minipage}{0.45\textwidth}
  \begin{tabular}{|c|c|c|c|c|c|c|c|c} 
 \cline{1-2}  & & \multicolumn{1}{c}{}  & \multicolumn{1}{c}{} & \multicolumn{1}{c}{} & \multicolumn{1}{c}{}  & \multicolumn{1}{c}{} & \multicolumn{1}{c}{} \\
   \cline{1-6}  &  &   & &  &   & \multicolumn{1}{c}{} & \multicolumn{1}{c}{} \\ 
   \hline &  {\cellcolor{gray!25} }  & & {\cellcolor{gray!25} }  &    &   &  &  \\    \hline &  &   &   &    &   &  & {\cellcolor{gray!25} }  \\ \hline
  \end{tabular} 
\end{minipage}
\begin{minipage}{0.45\textwidth}
  \begin{tabular}{|c|c|c|c|c|c|c|c|c} 
 \cline{1-2}  & & \multicolumn{1}{c}{}  & \multicolumn{1}{c}{} & \multicolumn{1}{c}{} & \multicolumn{1}{c}{}  & \multicolumn{1}{c}{} & \multicolumn{1}{c}{} \\
   \cline{1-6}  &  &   & &  &   & \multicolumn{1}{c}{} & \multicolumn{1}{c}{} \\ 
   \hline &   & & {\cellcolor{gray!25} }  &    &   &  &  \\    \hline &  &   &   &    &   &  &   \\ \hline
  \end{tabular} 
\end{minipage}
\end{equation}

\begin{corollary} The biggest vector orbit $\operatorname{Aut}(V, \mathcal{P})$-orbit is the set of vectors with the maximal height $V - \operatorname{Ker} P^{n_1}$. \end{corollary}

\subsection{Invariant bi-isotropic and bi-Lagrangian subspaces} \label{S:InvBiLagr}

Theorem~\ref{T:JordInvSubs} gives us the full list of $\operatorname{Aut}(V, \mathcal{P})$-invariant subspaces for a sum of Jordan blocks. We can easily list all of them that are bi-isotropic. 

\begin{lemma} \label{L:InvSubBiIsotr} An invariant subspace $ U =  \bigoplus_{i=1}^N \mathcal{J}_{0, 2n_i}^{\leq h_i}$ of  $(V, \mathcal{P}) =  \bigoplus_{i=1}^N \mathcal{J}_{0, 2n_i}$ is bi-isotropic if and only if $h_i \leq \frac{n_i}{2}$ for all $i=1,\dots, N$.
\end{lemma}

\begin{remark} Lemma~\ref{L:InvSubBiIsotr} addresses the single-eigenvalue case. In the general case, let $\mathcal{J}_{\lambda_j}$ be the sum of all Jordan blocks with eigenvalue $\lambda_j$. By Remark~\ref{Rem:GenJordInv} any bi-isotropic invariant  subspace $U \subset (V, \mathcal{P})$ decomposes into the direct sum $U_j  = U \cap \mathcal{J}_{\lambda_j} $.
\end{remark}

Combining Theorem~\ref{T:BiLagrKronPart} and Lemma~\ref{L:InvSubBiIsotr} we get all $\operatorname{Aut}(V, \mathcal{P})$-invariant bi-Lagrangian subspaces.

\begin{theorem} \label{T:InvarBilagr} Let $(V, \mathcal{P})$ be a bi-Poisson space. The following statements are equivalent:
\begin{enumerate}

\item  $(V, \mathcal{P})$ contains an $\operatorname{Aut}(V, \mathcal{P})$-invariant bi-Lagrangian subspace. 
\item In a JK decomposition of $(V, \mathcal{P})$, the size of each Jordan block is divisible by 4.
\end{enumerate}

Furthermore, if the JK decomposition has the form:
\[ (V, \mathcal{P}) = \left( \bigoplus_{i=1}^q \mathcal{K}_{2k_i + 1} \right) \oplus \left( \bigoplus_{j=1}^N \mathcal{J}_{\lambda_j, 4m_j}\right),\] then the only  $\operatorname{Aut}(V, \mathcal{P})$-invariant bi-Lagrangian subspace is \[ L = K  \oplus \left( \bigoplus_{j=1}^N \mathcal{J}^{\leq m_j}_{\lambda_j, 4m_j}\right).\] \end{theorem}

For one Jordan block $\mathcal{J}_{\lambda, 4n}$ in the standard basis $e_1, \dots, e_{2m}, f_1, \dots, f_{2m}$ the $\operatorname{Aut}(V, \mathcal{P})$-invariant bi-Lagrangian subspace is \begin{equation} \label{Eq:OneJordInvBiLagr} L = \operatorname{Span}\left\{ e_{m+1}, \dots, e_{2m}, f_1, \dots, f_m\right\}.\end{equation} For a sum of several Jordan blocks, the invariant bi-Lagrangian subspace is the direct sum of the subspaces given in \eqref{Eq:OneJordInvBiLagr} for each individual block. For instance, consider the sum $\mathcal{J}_{0, 8} \oplus \mathcal{J}_{0, 4}$. Similar to the visualization in equation~\eqref{Eq:SumBlocksInvariant}, the corresponding invariant bi-Lagrangian subspace can be depicted as:
 
\begin{equation} \label{Eq:BiLagrSum}
 \begin{tabular}{|c|c|c|c|} 
  \cline{1-2} &  & \multicolumn{1}{c}{} & \multicolumn{1}{c}{} \\
 \cline{1-2} &  & \multicolumn{1}{c}{} & \multicolumn{1}{c}{} \\
     \hline     {\cellcolor{gray!25} } & {\cellcolor{gray!25} }  & &  \\ 
   \hline   {\cellcolor{gray!25} } & {\cellcolor{gray!25} } &   {\cellcolor{gray!25} } & {\cellcolor{gray!25} }  \\ \hline
  \end{tabular} 
\end{equation}

\section{Bi-Lagrangian subspaces containing invariant subspaces} \label{S:BiLagrContInv}

Let $\mathcal{P}$ be a Jordan pencil on $V$. In this section we combine our two previous results:

\begin{itemize}
    \item In Section~\ref{S:BiPoisSect} we described bi-Poisson reduction. We showed that for any admissible bi-isotropic subspace $U \subset (V,\mathcal{P})$ there is an inclusion of bi-Lagrangian Grassmanians \begin{equation} \label{Eq:IncBLG} \operatorname{BLG}(U^{\perp}/U) \subset \operatorname{BLG}(V, \mathcal{P})\end{equation} given by \eqref{Eq:BiPoissBLG} 

\item In Section~\ref{S:InvSubspaces} we described all bi-isotropic $\operatorname{Aut}(V,\mathcal{P})$-invariant subspaces $U \subset (V,\mathcal{P})$ (and proved that they are all admissible). 
    
\end{itemize}

Leveraging our previous findings, we arrive at the following. Consider any $\operatorname{Aut}(V,\mathcal{P})$-invariant subspace $U \subset (V,\mathcal{P})$. Then the bi-Lagrangian subspaces $L\subset (V,\mathcal{P})$ that satisfy \[ U \subset L \subset U^{\perp}\]  form a smaller bi-Lagrangian Grassmanian $\operatorname{BLG}(U^{\perp}/U)$. It's important to note that the inclusion of Grassmanians \eqref{Eq:IncBLG} does not necessarily imply the inclusion of corresponding orbits under the automorphism group action. The automorphism group $\operatorname{Aut}(V,\mathcal{P})$ acts on the quotient $U^{\perp}/U$ through its image under the Lie group homomorphism \begin{equation} f: \label{Eq:HomoGroupAut} \operatorname{Aut}\left(V, \mathcal{P}\right) \to \operatorname{Aut} \left(U^{\perp} / U\right).\end{equation} In this section we study how $\operatorname{BLG}(U^{\perp}/U, \mathcal{P}_U)$ is broken into  $\operatorname{Aut}(V,\mathcal{P})$-orbits  in  $\operatorname{BLG}(V,\mathcal{P})$.

\begin{enumerate}

\item In Section~\ref{SubS:ActionOnQuotient} we describe the homomorphism \eqref{Eq:HomoGroupAut}. 

\item In Section~\ref{SubS:MatrBLG} focus on subspaces $U$ where the homomorphism \eqref{Eq:HomoGroupAut} is surjective (and hence the inclusion \eqref{Eq:IncBLG} preserves orbits).  We will show that this occurs only when $U =\operatorname{Ker}P^h$, where $P$ is the recursion operator and $2h$ is not bigger that the size of the smallest Jordan block. 

\item Then we explore two interesting cases where the homomorphism \eqref{Eq:HomoGroupAut} is not surjective.

\begin{itemize}

\item In Section~\ref{SubS:InfOrb} we show that $\operatorname{BLG}(V,\mathcal{P})$ may have an infinite number of $\operatorname{Aut}(V,\mathcal{P})$-orbits.

\item In Section~\ref{SubS:NonDecomp} we show that not all bi-Lagrangian subspaces $L \subset (V,\mathcal{P})\approx \bigoplus_{i=1}^N \mathcal{J}_{0, 2n_i}$ can be decomposed into a sum of bi-Lagrangian subspaces restricted to individual Jordan blocks $L_i \subset \mathcal{J}_{0, 2n_i}$. 

\end{itemize}

\end{enumerate}

\subsection{Action of \texorpdfstring{$\operatorname{Aut}(V,\mathcal{P})$}{Aut(V,P)} on \texorpdfstring{$U^{\perp}/U$}{Uperp/U}} \label{SubS:ActionOnQuotient}

Let $(V, \mathcal{P}\}) = \bigoplus_{i=1}^t \left(\bigoplus_{j=1}^{l_i} \mathcal{J}_{0, 2n_i} \right)$, where $n_1
\geq n_2 \geq \dots \geq n_t$, and $U =  \bigoplus_{i=1}^t \left(\bigoplus_{j=1}^{l_i} \mathcal{J}^{\leq h_{i}}_{0, 2n_i} \right)$ be an isotropic invariant subspace. Then the JK decomposition of $U^{\perp}/U$ of 
$\bigoplus_{i=1}^t \left(\bigoplus_{j=1}^{l_i} \mathcal{J}_{0, 2h_{i} - n_i} \right)$. The bi-Poisson reduction w.r.t. $U$ induces the Lie algebra homomorphim \[\varphi:  \operatorname{aut}(V,\mathcal{P}) \to \operatorname{aut}(U^{\perp}/U).\]  The following lemma describes the image $\varphi(C)$ of a an element $C$ belonging to the bi-Poisson automorphism Lie algebra  $\operatorname{aut}(V,\mathcal{P})$.

\begin{lemma} \label{L:ImageAut} Consider the basis of $(V,\mathcal{P})$ from Theorem~\ref{T:BiSymp_General_Jordan_Case_Mega}, in which an element $C$ of the bi-Poisson automorphism Lie algebra $\operatorname{aut}(V,\mathcal{P})$ has the form \[ C = \left( \begin{matrix} X_{11} & \dots & X_{t1} \\ \vdots & \ddots &\vdots \\ X_{t1} & \dots & X_{tt} \end{matrix} \right),\] given by \eqref{E:bsp_algebra_matrix}. Then its image $\varphi(C) \in \operatorname{aut}(U^{\perp}/U)$ has the form \[\varphi(C) = \left( \begin{matrix} \hat{X}_{11} & \dots & \hat{X}_{t1} \\ \vdots & \ddots &\vdots \\ \hat{X}_{t1} & \dots & \hat{X}_{tt} \end{matrix} \right), \] where $\hat{X}_{pq}$ is the $(n_{h_p} - 2h_p) \times (n_{h_q} - 2h_q)$ central matrix of $X_{pq}$. \end{lemma}
Roughly speaking, the image has the form 
\begin{equation} \label{Eq:ImageAut}
\varphi(C) = \left( \begin{array}{ccccc|cccc|c} C^{1,1}_1 & & & & & &  & & & \multirow{5}{*}{$\cdots$} \\ C^{1,1}_2 & C^{1,1}_1 & & &  & 0 & & & & \\ \vdots & \ddots & \ddots &  & & C^{1,2}_1 & \ddots & & & \\\vdots & \ddots & \ddots & \ddots & & \vdots & \ddots & \ddots & &\\ C^{1,1}_{n_1 - 2h_1} & \cdots & \cdots & \cdots & C^{1,1}_1  & C^{1,2}_{n_2 - h_1-h_2} & \cdots & C^{1,2}_1 & 0 &
\\ \hline 0 &  & & & & C^{2,2}_1 & &  & & \multirow{3}{*}{$\cdots$} \\ C^{2,1}_1 & \ddots &  & & & \vdots & \ddots &  & & \multirow{3}{*}{$\cdots$} \\ \vdots & \ddots & \ddots & & & \vdots & \ddots  & \ddots & &
\\ C^{2,1}_{n_2- h_1 - h_2} & \cdots &C^{2,1}_1 & 0 & & C^{2,2}_{n_2 -2h_2} & \cdots & \cdots & C^{2,2}_1 & \\ \hline\multicolumn{5}{c|}{\cdots}     & \multicolumn{4}{c|}{\cdots} & \cdots
\end{array} \right).
\end{equation} Equation~\eqref{Eq:ImageAut} resembles Equation~\eqref{E:bsp_algebra_matrix} in structure, where both represent automorphisms as block diagonal matrices. However, they differ in the following way:
\begin{itemize}

\item Within each block  $\hat{X}_{ij}$ of the image $\varphi(C)$,  we only keep a $\min(n_i, n_j) - h_j - h_i$ of non-zero diagonals of $X_{ij}$.

\item If $\min(n_i, n_j) - h_j - h_i<0$, then  the corresponding block   $\hat{X}_{ij}$  becomes a zero submatrix.

\end{itemize} 

The Lie algebra homomorphism \[ \varphi: \operatorname{aut}(V,\mathcal{P}) \to \operatorname{aut}\left(U^{\perp}/U\right)\] is  surjective if and only if $h_1 = h_2 = \dots = h_t$. Given the connectedness of the automorphism groups (see Theorem~\ref{T:AutConnected}), the corresponding Lie group homomorphism is also surjective. We get the following.

\begin{corollary} \label{Cor:ImageAutIso} Let $(V, \mathcal{P}) = \bigoplus_{i=1}^N \mathcal{J}_{0, 2n_i}$, where $n_1 \geq n_2 \geq \dots \geq n_N$, and $P$ be the recursion operator. For any $h \leq \frac{n_N}{2}$ we have the following.

\begin{enumerate}
\item $\operatorname{Ker} P^h$ is a bi-isotropic $\operatorname{Aut}(V,\mathcal{P})$-invariant subspace and $(\operatorname{Ker}P^h)^{\perp} = \operatorname{Im} P^h$.

\item The Lie algebra homomorphism \[ \operatorname{Aut}\left(V, \mathcal{P}\right) \to \operatorname{Aut} (\operatorname{Im} P^h / \operatorname{Ker} P^h) = \operatorname{Aut}\left(\bigoplus_{i=1}^N \mathcal{J}_{0, 2n-4h}\right)\] is surjective.

\end{enumerate}
\end{corollary}

\subsection{Nested Bi-Lagrangian Grassmanians} \label{SubS:MatrBLG}

If the homomorphism \eqref{Eq:HomoGroupAut} is surjective, then the inclusion \eqref{Eq:IncBLG} is orbit-preserving. Aligning with Corollary~\ref{Cor:ImageAutIso} we get the following.

\begin{theorem} Let $(V, \mathcal{P}) = \bigoplus_{i=1}^N \mathcal{J}_{0, 2n_i}$, where $n_1 \geq n_2 \geq \dots \geq n_N$, and $P$ be the recursion operator. For any $h \leq \frac{n_N}{2}$ and any bi-Lagrangian subspace $L \subset (V,\mathcal{P})$ such that \[ \operatorname{Ker}P^h \subset L \subset  \operatorname{Im} P^h \]  the following two orbits are isomorphis:

\begin{enumerate}
    \item the $\operatorname{Aut}\left(V, \mathcal{P}\right)$-orbit of $L$ in $\operatorname{BLG}(V,\mathcal{P})$
    \item and the $\operatorname{Aut} (\operatorname{Im} P^h / \operatorname{Ker} P^h)$-orbit of $L/\operatorname{Ker} P^h$ in $\operatorname{BLG} (\operatorname{Im} P^h / \operatorname{Ker} P^h)$.

\end{enumerate}  
    \end{theorem}

We get an orbit-preserving inclusion of bi-Lagrangian Grassmanians \begin{equation} \label{Eq:MatrEqualJordSuper} \operatorname{BLG}\left(\bigoplus_{i=1}^N \mathcal{J}_{0, 2n_i}\right) \supset \operatorname{BLG}\left(\bigoplus_{i=1}^N \mathcal{J}_{0, 2n_i-4}\right) \supset \dots \supset \operatorname{BLG}\left(\bigoplus_{i=1}^N \mathcal{J}_{0, 2n_i- 4 \left[ \frac{n_N}{2}\right]}\right)   \end{equation} Formula~\eqref{Eq:MatrEqualJordSuper} generalizes ``matryoshka of bi-Lagrangians'' \eqref{Eq:MatrOneJord}. 

\begin{remark} Consider $(V, \mathcal{P}) = \bigoplus_{i=1}^N \mathcal{J}_{0, 2n_i}$. Bi-isotropic invariant subspaces within $(V,\mathcal{P})$ have the form $U = \bigoplus_{i=1}^N \mathcal{J}^{\leq h_i}_{0, 2n_i}$, where $h_i \leq \frac{n_i}{2}$. For any such $U$ there is an inclusion \[ \operatorname{BLG}(U^{\perp}/ U) =  \operatorname{BLG}\left(\bigoplus_{i=1}^N \mathcal{J}_{0, 2n_i- 4h_i }\right)\subset  \operatorname{BLG}\left(\bigoplus_{i=1}^N \mathcal{J}_{0, 2n_i}\right)=  \operatorname{BLG}\left(V,\mathcal{P}\right). \] But if not all heights $h_1,\dots, h_t$ are equal, then $\operatorname{Aut}\left(V,\mathcal{P}\right)$ further divide the orbits of $\operatorname{BLG}\left(U^{\perp}/U\right)$ into smaller orbits within $\operatorname{BLG}\left(V,\mathcal{P}\right)$. 

\end{remark}

\subsection{Example with infinite number of \texorpdfstring{$\operatorname{Aut}(V, \mathcal{P})$}{Aut(V,P)}-orbits} \label{SubS:InfOrb}

\begin{lemma} \label{L:InfOrb} Consider the sum of Jordan blocks \[(V, \mathcal{P}) =\bigoplus_{i=1}^n \mathcal{J}_{0,4i-2} = \mathcal{J}_{0,2} \oplus \mathcal{J}_{0,6} \oplus \dots \oplus \mathcal{J}_{0,4n-2} .\] If $n>5$, then the bi-Lagrangian Grassmanian consists of an infinite number of orbits under the action of automorphism group $\operatorname{Aut}(V, \mathcal{P})$. \end{lemma}

\begin{proof}[Proof of Lemma~\ref{L:InfOrb}] Consider the $\operatorname{Aut}(V, \mathcal{P})$-invariant bi-isotropic subspace \[ 
U = \bigoplus_{i=1}^n \mathcal{J}^{\leq i-1}_{0, 4i-2}.\] 
We're interested in orbits of bi-Lagrangian subspaces $L$ such that \[ U \subset L \subset U^{\perp}.\] The situation is visualized in \eqref{Eq:BiLagrInfOrbEx}. The Jordan blocks are realized as Young-like diagram similar to \eqref{Eq:OneJordanBlock_VectorInTable}.  The colored cells represent the subspace $U$, while $U^{\perp}$ is spanned by both colored and marked cells \begin{equation} \label{Eq:BiLagrInfOrbEx}
\begin{tabular}{|c|c|}\hline
 & \\ \hline
 \vdots & \vdots \\ \hline
    & \\ \hline
   X & X  \\ \hline
       {\cellcolor{gray!25} } & {\cellcolor{gray!25} } \\ \hline
   \vdots & \vdots   \\ \hline
    {\cellcolor{gray!25} } & {\cellcolor{gray!25} } \\ \hline
  \end{tabular} \quad  \cdots \quad \begin{tabular}{|c|c|cc}
       \multicolumn{1}{c}{} & \multicolumn{1}{c}{}& & \\
       \multicolumn{1}{c}{} & \multicolumn{1}{c}{}& & \\
       \multicolumn{1}{c}{} & \multicolumn{1}{c}{}& & \\
       \multicolumn{1}{c}{} & \multicolumn{1}{c}{}& & \\    \cline{1-2}  
   & & & \\ \cline{1-2}   
   X & X & & \\ \hline
    {\cellcolor{gray!25} } & {\cellcolor{gray!25} }  & X &   \multicolumn{1}{|c|}{X}\\ \hline
  \end{tabular}
%\resizebox{0.3\textwidth}{!}{
%}
\end{equation}

\begin{itemize}

\item On one hand, $\dim \operatorname{BLG}(U^{\perp}/U)$ is rather high. More precisely, $U^{\perp}/U$ is isomorphic to the sum  $\bigoplus_{i=1}^N \mathcal{J}_{0,2}$. Therefore, $\operatorname{BLG}{\left(U^{\perp}/U\right)} \approx \Lambda(n)$ and \[\dim \operatorname{BLG}(U^{\perp}/U)  = \dim \Lambda(n)  = \frac{n(n+1)}{2}\]

\item On the other hand, the image of the homomorphism  $f: \operatorname{Aut}(V, \mathcal{P}) \to \operatorname{Aut}(U^{\perp}/U)$ is a rather small subgroup. As seen in \eqref{Eq:BiLagrInfOrbEx}, $U^{\perp}/U$ decomposes in $n$ Jordan blocks: \[U^{\perp}/U = \bigoplus_{i=1}^n H_j, \qquad H_j =  (U^{\perp}/U) \cap \operatorname{Im} P^{j-1} \cap \operatorname{Ker} P^j.  \] Since  $\operatorname{Aut}(V, \mathcal{P})$ preserves each block $H_j=\mathcal{J}_{0,2}$ , its image is contained in a product of $n$ symplectic groups\footnote{We could also use Lemma~\ref{L:ImageAut}}. Hence, 
  \[\dim f( \operatorname{Aut}(V, \mathcal{P})) \leq \dim \left(  \underbrace{\operatorname{Sp}(2)\times \dots \times  \operatorname{Sp}(2)}_n \right) = 3 n \]

\end{itemize}

A group of dimension at most $3n$ acts on a $\frac{n(n+1)}{2}$ manifold. This mismatch in sizes leads to an an infinite number of orbits. Lemma~\ref{L:InfOrb} is proved. \end{proof}

\subsection{Example of a indecomposable bi-Lagrangian subspace} \label{SubS:NonDecomp}

We say that a bi-Lagrangian subspace $L \subset (V,\mathcal{P})$ of a Jordan bi-Poisson space is \textbf{indecomposable} if it is not isomorphic to a direct sum of two non-trivial bi-Lagrangian subspaces (see Definition~\ref{Def:Decomp}). In the next statement we  construct an example of a indecomposable bi-Lagrangian subspace. This example shares similarities with the one discussed in Section~\ref{SubS:InfOrb}.

\begin{assertion} \label{A:NonDecomp} Let $(V, \mathcal{P}) =  \mathcal{J}_{0, 6}  \oplus \mathcal{J}_{0, 2}$, and $e_1,e_2, e_3, f_1, f_2, f_3, \hat{e}_1, \hat{f}_1$ be its standard basis  (i.e. the basis from the JK theorem).  The subspace \begin{equation} \label{Eq:BiLagr_NotDecomp} L = \operatorname{Span}\left\{ e_{3}, f_{1}, e_{2} + \hat{e}_{1}, f_{2} - \hat{f}_{1}. \right\} \end{equation}  is bi-Lagrangian and  indecomposable.
\end{assertion}

\eqref{Eq:BiLagr_NotDecomp_Fig} illustrates the concept.  The Jordan blocks are represented by Young-like diagrams similar to \eqref{Eq:OneJordanBlock_VectorInTable}. The basis of the indecomposable bi-Lagrangian subspace $L$ is visualized by two colored cells and two pairs of linked cells (each pair distinguished by the same letter P or M).

\begin{equation}  \label{Eq:BiLagr_NotDecomp_Fig} 
  \begin{tabular}{|c|c|cc}\cline{1-2}
   & & & \\ \cline{1-2}   
     P  &  M & & \\ \hline
    {\cellcolor{gray!25} } & {\cellcolor{gray!25} }  & P &   \multicolumn{1}{|c|}{M }\\ \hline
  \end{tabular}
\end{equation}

\begin{proof}[Proof of Assertion~\ref{A:NonDecomp}] The subspace $L$ satisfies \[\dim L +\operatorname{Ker} P/ \operatorname{Ker} P = 2.\] No decomposable bi-Lagrangian subspace of $ \mathcal{J}_{0, 6}  \oplus \mathcal{J}_{0, 2}$  can fulfill this property. Assertion~\ref{A:NonDecomp} is proved. \end{proof}

A more in-depth analysis of $\operatorname{BLG}( \mathcal{J}_{0, 6}  \oplus \mathcal{J}_{0, 2})$ is presented in Section~\ref{SubS:Exam}.

\section{Several equal Jordan blocks} \label{S:EqualJordBlocks}

In this section we study the bi-Lagrangian Grassmanian $\operatorname{BLG}(V, \mathcal{P})$ for the sum of $l$ Jordan $2n \times 2n$ blocks with the same eigenvalue $\lambda =0$, i.e. $(V, \mathcal{P})=\bigoplus_{i=1}^l \mathcal{J}_{0, 2n}$. Namely, we describe the $\operatorname{Aut}(V, \mathcal{P})$-orbits of $\operatorname{BLG}(V, \mathcal{P})$. As usual, $P$ denotes the recursion operator.

We say that a basis $e^i_j, f^i_j$, $i=1, \dots, l, j=1, \dots, n$ is a \textbf{standard basis} of $\bigoplus_{i=1}^l \mathcal{J}_{0, 2n}$ if for each $i=1,\dots, l$ the vectors $e^i_j, f^i_j$ form a standard basis of a $2n\times 2n$ Jordan block. Simply speaking, \[ e^1_1, \dots, e^1_n, f^1_1, \dots, f^1_n, \dots e^l_1, \dots, e^l_n, f^l_1, \dots, f^l_n,\] is a basis from the JK theorem~\ref{T:Jordan-Kronecker_theorem}.

\subsection{Dimension and types of \texorpdfstring{$\operatorname{Aut}(V, \mathcal{P})$}{Aut(V,P)}-orbits}

First, we describe a canonical form for a bi-Lagrangian subspace $L\subset \bigoplus_{i=1}^l \mathcal{J}_{0, 2n}$.

\begin{theorem} \label{T:BiLagrEqJordBl} Let $(V, \mathcal{P}) = \bigoplus_{i=1}^l \mathcal{J}_{0, 2n}$ a sum of $l$ equal Jordan blocks. For any bi-Lagrangian subspace $L \subset \bigoplus_{i=1}^l \mathcal{J}_{\lambda, 2n}$ there exist numbers $h_1, \dots, h_l$ such that \[ h_1 \geq h_2 \geq \dots \geq h_l \geq \frac{n}{2}\] and a standard basis $e^i_j, f^i_j$, $i=1, \dots, l, j=1, \dots, n$ such that \begin{equation} \label{Eq:CanonEqJord} L = \bigoplus_i  L_i, \qquad L_i = \operatorname{Span}\left\{e^i_{n-h_i +1}, \dots, e^i_n, f^i_1, \dots f^i_{n-h_i} \right\} .\end{equation} \end{theorem}

Simply speaking, any bi-Lagrangian subspace $L \subset \bigoplus_{i=1}^l \mathcal{J}_{\lambda, 2n}$ is semisimple (see Definition~\ref{Def:Decomp}).  For example, consider \eqref{Eq:EqJordOrb} where we realize $\mathcal{J}_{0, 8} \oplus\mathcal{J}_{0, 8} \oplus \mathcal{J}_{0, 8}$ as a Young-like diagram. The shaded area represents a bi-Lagrangian subspace with the corresponding heights $h_1 = 4, h_2 = 3$ and $h_3 = 2$. 

\begin{equation} \label{Eq:EqJordOrb}
  \begin{tabular}{|c|c||c|c||c|c|} 
  \hline {\cellcolor{gray!25} } & &    &  &  &  \\
 \hline {\cellcolor{gray!25} } & &  {\cellcolor{gray!25} }  &  & &  \\
     \hline {\cellcolor{gray!25} } & &  {\cellcolor{gray!25} }  & &  {\cellcolor{gray!25} }  & {\cellcolor{gray!25} }  \\ 
   \hline {\cellcolor{gray!25} } &  &  {\cellcolor{gray!25} }  &  {\cellcolor{gray!25} } &  {\cellcolor{gray!25} }  & {\cellcolor{gray!25} } \\ \hline
  \end{tabular} 
\end{equation}

\begin{proof}[Proof of Theorem~\ref{T:BiLagrEqJordBl}] The proof is by induction on $\dim L$. 

\begin{itemize}

\item The base case $(\dim L = 1)$ is trivial.

\item Induction Step: We consider two cases:

\begin{enumerate}
    \item \textit{Case $1$: $L$ does not contain $\operatorname{Ker} P$.} We can extract several Jordan blocks using Corollary~\ref{Cor:JordExtr}. Since the remaining subspace has lower dimension, the induction hypothesis applies.

    \item \textit{Case $2$: $\operatorname{Ker} P\subset L$.} By induction, $L/ \operatorname{Ker}P$ can be brought to the standard form~\eqref{Eq:CanonEqJord} in $\operatorname{Im}P / \operatorname{Ker}P$ using an automorphism $C \in \operatorname{Aut} (\operatorname{Im} P / \operatorname{Ker} P)$. Corollary~\ref{Cor:ImageAutIso}  guarantees that there exists a corresponding automorphism $C' \in \operatorname{Aut}\left(V, \mathcal{P}\right)$ that brings $L$ itself to the standard form~\eqref{Eq:CanonEqJord} within $(V, \mathcal{P})$.
\end{enumerate} 

\end{itemize}

Theorem~\ref{T:BiLagrEqJordBl} is proved.  \end{proof}

\begin{definition} For a bi-Lagrangian subspace $L \subset \bigoplus_{i=1}^l \mathcal{J}_{0, 2n}$ we define its \textbf{type} $H(L)$, as a collection of $l$ heights $h_i$ obtained after applying Theorem~\ref{T:BiLagrEqJordBl}:  \[ H(L) = \left\{ h_1, \dots, h_l\right\}\] \end{definition}  We always assume that $h_i \in H(L)$ are sorted in the descending order: \begin{equation} \label{Eq:IneqH} h_1 \geq h_2 \geq \dots \geq h_l \geq \frac{n}{2}.\end{equation}

\begin{theorem} \label{T:BiLagrEqTypesNum}  Let $(V, \mathcal{P}) = \bigoplus_{i=1}^l \mathcal{J}_{0, 2n}$ a sum of $l$ equal Jordan blocks.

\begin{enumerate}

\item Two bi-Lagrangian subspaces $L_1,L_2 \in \operatorname{BLG}
(V,\mathcal{P})$ belong to the same $\operatorname{Aut}(V, \mathcal{P})$-orbit if and only if they have the same type $H(L_1) = H(L_2)$. 

\item The total number of $\operatorname{Aut}(V, \mathcal{P})$-orbits in $\operatorname{BLG}(V,\mathcal{P})$ is \begin{equation} \label{Eq:NumOrbEqJord} \displaystyle\left(\!\! \binom{\left[\frac{n}{2}\right] + 1} {l}\!\!\right)= \binom{\left[ \frac{n}{2} \right] +  l}{l}.\end{equation}

\item For a bi-Lagrangian subspace $L$ with type $H(L) = \left\{h_1, \dots, h_l\right\}$  the dimension of its $\operatorname{Aut}(V, \mathcal{P})$-orbit $O_L$ is \begin{equation} \label{Eq:DimOEqualJord} \dim O_L = \sum_{1 \leq i \leq j \leq l} 2h_i - n.\end{equation}

\end{enumerate}

\end{theorem}

\begin{remark} For convenience, we can group together equal $h_i$  within the type $H(L)$. Assume that there are $l_1$ values of $h_i$ equal to $\hat{h}_1$,  $l_2$ values equal to $\hat{h}_2$, and so on,  $l_t$ values equal to $\hat{h}_t$, where $\hat{h}_1 \geq \hat{h}_2 \geq \dots \geq \hat{h}_t$ (descending order). We denote the type $H(L)$ as \[ H(L) = \left\{ (h_i, n) \times l_i\right\}_{i=1,\dots, l}.\] Formula~\eqref{Eq:DimOEqualJord} takes the form \begin{equation} \label{Eq:DimOLEqJord} \dim O_L = \sum_{i=1}^t \frac{l_i(l_i + 1)}{2} (2h_i -n) + \sum_{1 \leq i < j \leq t} l_i l_j (2h_i -n).\end{equation}
\end{remark}

\begin{proof}[Proof of Theorem~\ref{T:BiLagrEqTypesNum} ]

\begin{enumerate}

\item The heights $h_i$ are uniquely determined by the basis-independent dimensions $\dim (L \cap \operatorname{Ker}P^j)$.

\item \eqref{Eq:NumOrbEqJord} is the number of possible heights $h_i$ that satisfy \eqref{Eq:IneqH}.

\item  It is proved similar to Theorem~\ref{T:BiLagrJord} (also, below we prove a more general statement in Corollary~\ref{Cor:NumTypes}).

\end{enumerate}

Theorem~\ref{T:BiLagrEqTypesNum} is proved. \end{proof}

Similar to Corollary~\ref{Cor:ConnOneJord} we get the following.

\begin{corollary} The generic orbit $O_{\max}$ is dense (in the standard topology) in the bi-Lagrangian Grassmanian for a sum of equal Jordan blocks: \[ \bar{O}_{\max} = \operatorname{BLG}(\bigoplus_{i=1}^l\mathcal{J}_{0, 2n}).\]  Thus, $\operatorname{BLG}(\bigoplus_{i=1}^l\mathcal{J}_{0, 2n})$  is an irreducible algebraic variety. \end{corollary}

\subsection{Maximal \texorpdfstring{$\operatorname{Aut}(\bigoplus_{i=1}^l \mathcal{J}_{\lambda, 2n})$}{Aut(l sum J2n)}-orbit is a jet space to Lagrangian Grassmanian}

Now, let us describe the topology of the maximal orbit of $\operatorname{BLG}(\bigoplus_{i=1}^l \mathcal{J}_{0, 2n})$. Recall that the $k$-th-order jet space\footnote{Sometimes $T^kM$ is also called $k$-th-order tangent bundle, although this may lead to confusion. For that reason we prefer the notation $J^k_0(\mathbb{K}, M)$.}, denoted by $T^kM$ or $J^k_0(\mathbb{K}, M)$, consists of the k-th jets (at zero) of curves $\gamma: \mathbb{K} \to M$. The jet space $J^k_0(\mathbb{K}, M)$ is a fibre bundle over $M$. For $k=1$ it is the tangent bundle $TM$ and for $k > 1$ the canonical projection $J^k_0(\mathbb{K}, M) \to J^{k-1}_0(\mathbb{K}, M)$ is an affine bundle. For a more detailed discussion of jet spaces see e.g. \cite{Kolar93}. In the holomorphic case, the constructions are straightforward adaptations of the real ones, see e.g. \cite{LewisJet}.

\begin{theorem} \label{Th:EqualJordBiLagrMaxOrbit} There is a unique maximal $\operatorname{Aut}\left(\bigoplus_{i=1}^l \mathcal{J}_{0, 2n}\right)$-orbit $O_{max}$ in $\operatorname{BLG}\left(\bigoplus_{i=1}^l \mathcal{J}_{0, 2n}\right)$, its dimension is \[ \dim O_{max} = n \frac{l(l+1)}{2}.\] If the field $\mathbb{K} = \mathbb{C}$ or $\mathbb{R}$, then this orbit is diffeomorphic to the $(n-1)$-th order jet space of the Lagrangian Grassmannian $\Lambda(l)$: \[ O_{max} \approx T^{n-1} \Lambda(l) =  J^{n-1}_0(\mathbb{K}, \Lambda(l)).\] \end{theorem}

For example, if $n =1$, then the form $A$ is proportional to $B$ and \[ \operatorname{BLG}\left(\bigoplus_{i=1}^l \mathcal{J}_{0, 2}\right) = \Lambda(l).\] If $n=2$, then the maximal orbit of $\operatorname{BLG}(\bigoplus_{i=1}^l \mathcal{J}_{0, 4})$ is the tangent bundle to the Lagrangian Grassmanian $T\Lambda(l)$.

\begin{proof}[Proof of Theorem~\ref{Th:EqualJordBiLagrMaxOrbit}] The proof is in several steps. 

\begin{enumerate}

    \item \textit{Fix a standard basis}. Let $e^i_j, f^i_j$, $i=1, \dots, l, j= 1, \dots, n$
be a standard basis from the JK theorem. This time it is more convenient to rearrange the basis as follows: \[e_1^{1}, e^2_1, \dots, e^l_1, e^1_2, \dots, e^l_n, f_1^1, f^2_1, \dots, f^l_n.\] Then the matrix of $B$ is the symplectic matrix \[ B = \left( \begin{matrix} 0 & I_{nl} \\ -I_{nl} & 0 \end{matrix} \right)\] and the matrix of the recursion operator $P$ is \[ P = \left( \begin{matrix} J^T & 0 \\ 0 & J \end{matrix} \right), \qquad J = \left( \begin{matrix} 0_l & & & \\ I_{l} &\ddots &   & \\ & \ddots & \ddots & \\ & & I_{l} & 0_l  \end{matrix} \right). \]

\item \textit{Maximal orbit $O_{max}$ as a homogeneous space $X/G$.} By Theorem~\ref{T:BiLagrEqJordBl} any  bi-Lagrangian subspace $L \in O_{max}$  is generated by rows of a matrix   \begin{equation} \label{Eq:SpanBiLagrEqualJordMatrix} M = \left( \begin{matrix} P_1 & \dots & P_n & Q_1 & \dots & Q_n \\ & \ddots & \vdots & \vdots & \udots & \\ & & P_1 & Q_n & & \end{matrix} \right), \end{equation} where $P_i$ and $Q_j$ are $l\times l$ matrices and \begin{equation} \label{Eq:RankLMaxOrb} \operatorname{rk} \left( \begin{matrix} P_1 & Q_n \end{matrix} \right) = l.\end{equation} The subspace $L$ generated by $M$ is $P$-invariant. This is because the remaining rows of $M$ are are formed by applying $P$ to the first $l$ rows. The subspace $L$ is bi-Lagrangian if and only if \begin{equation} \label{Eq:TopOrbLagCond} M B M^T = M \left(\begin{matrix} 0 & I_{nl} \\  - I_{nl} & 0 \end{matrix}\right) M^T = 0.\end{equation} The matrix $M$ is defined up to left multiplication by a matrix \begin{equation} \label{Eq:TopOrbMatrixC} C = \left( \begin{matrix} C_1 & \dots & C_n \\ & \ddots & \vdots   \\ & & C_1 \end{matrix} \right), \qquad \det C_1 \not = 0. \end{equation} This specific form of $C$ ensures that the later rows of $CM$ result from applying $P$ to the first $l$ rows. We get that $O_{max}$ is diffeomorphic to the homogeneous space $X/G$, where 
\begin{itemize}

    \item $X \subset \operatorname{Gr}(nl, 2nl)$ is the space of matrices $M$, given by \eqref{Eq:SpanBiLagrEqualJordMatrix}, \eqref{Eq:RankLMaxOrb} and  \eqref{Eq:TopOrbLagCond}. 

    \item $G \subset \operatorname{GL}(nl)$ is the subgroup of matrices $C$, given by \eqref{Eq:TopOrbMatrixC}.
\end{itemize}

\item \textit{Rephrasing condition \eqref{Eq:RankLMaxOrb} in terms of the matrices $P_j, Q_k$.}  \eqref{Eq:RankLMaxOrb} is equivalent to \begin{equation} \label{Eq:CondEqJordBiLagr} \begin{gathered} P_1 Q_n^T = \left(P_1 Q_n\right)^T \\ P_1 Q_{n-1}^T + P_2 Q_n^T = \left(P_1 Q_{n-1}^T + P_2 Q_n^T\right)^T  \\ \dots \\ P_1 Q_{1}^T + \dots +  P_n Q_n^T = \left( P_1 Q_{1}^T + \dots +  P_n Q_n^T \right)^T. \end{gathered} \end{equation} Recall the following trivial fact.

\begin{assertion} \label{A:LagrStand}
Let $(V, \omega)$ be a symplectic space, $\omega = \left(\begin{matrix} 0 & I \\ -I & 0 \end{matrix} \right)$. A subspace spanned by the rows of a matrix $\left( \begin{matrix} P & Q \end{matrix} \right)$, where $P$ and $Q$ are $n\times n$ matrices, is Lagrangian iff \[PQ^T = \left(PQ^T \right)^T.\]

\end{assertion} 
 Put \begin{equation} \label{Eq:PQPolynom} P(\lambda) = P_1 + \lambda P_2 + \dots + \lambda^{n-1} P_n, \qquad Q(\lambda) = Q_n + \lambda Q_{n-1} + \dots + \lambda^{n-1} Q_1.\end{equation} It is easy to see that \eqref{Eq:RankLMaxOrb} is equivalent to \begin{equation} \label{Eq:PlambdaQlambdaUpto} P(\lambda)Q(\lambda)^T = \left(P(\lambda)Q(\lambda)^T\right)^T \mod \lambda^{n}.\end{equation} 

\item \textit{The orbit $O_{max}$ is a fiber bundle over the Lagrangian Grassmanian $\operatorname{\Lambda}(l))$. Moreover, locally it is isomorphic to the $(n-1)$-order jet space space $J^{n-1}_0(\mathbb{K}, \Lambda(l))$.} The projection $\pi: O_{max} \to \Lambda(l))$ is given by \begin{equation} \label{Eq:TopOrbProjection} \left( \begin{matrix} P_1 & \dots & P_n & Q_1 & \dots & Q_n \\ & \ddots & \vdots & \vdots & \udots & \\ & & P_1 & Q_n & & \end{matrix} \right) \to \left(\begin{matrix} P_1& Q_n\end{matrix} \right). \end{equation}  Assertion~\ref{A:LagrStand} and \eqref{Eq:PlambdaQlambdaUpto}   guarantee that $\left(\begin{matrix} P_1 & Q_n\end{matrix} \right) \in \operatorname{\Lambda}(l)$. To prove  $\pi$ defines a fiber bundle, we consider the standard atlas on $\Lambda(l)$ (see e.g.\cite[Section 3.3]{Arnold67}). We'll analyze a specific chart where  $\det P_1 \not = 0$ (similar analysis for other charts). In this chart elements of $\Lambda(l)$ take the simpler form: \[\left(\begin{matrix} E & Q \end{matrix}\right), \qquad  Q = Q^T. \] Since $\det P_1 \not = 0$, there is a unique matrix $C$  given by \eqref{Eq:TopOrbMatrixC}, such that \begin{equation} \label{Eq:BringLagrToGood1} CM =\left( \begin{matrix} I_l &  & & Q_1 & \dots & Q_n \\ & \ddots &  & \vdots & \udots & \\ & & I_l & Q_n & & \end{matrix} \right).  \end{equation} By \eqref{Eq:PlambdaQlambdaUpto}, all $Q_j$ are symmetric. This establishes $O_{max}$ as the fiber bundle over $\Lambda(l)$ with fibers $\mathbb{C}^N, N =  (n-1) \frac{l(l-1)}{2}$. Furthermore, by considering polynomial curves of the form: \[\gamma(\lambda) = \left( \begin{matrix} I_l &  Q(\lambda)\end{matrix} \right), \qquad Q(\lambda) = Q_n + \lambda Q_{n-1} + \dots + \lambda^{n-1} Q_1. \] we establish a local isomorphism between $O_{max}$ and the jet space $J^{n-1}_0(\mathbb{K}, \Lambda(l))$.

\item \textit{The orbit $O_{max}$ is diffeomorphic to  $(n-1)$-order jet space space $J^{n-1}_0(\mathbb{K}, \operatorname{\Lambda}(l))$.} This requires demonstrating that elements in different charts define the same jet. For simplicity sake we consider two charts of $\Lambda(l)$:  \[U_p =\left\{\left(\begin{matrix} P & I_{l} \end{matrix}\right) \, \,\bigr| \, \, P = P^T \right\},\qquad U_q = \left\{\left(\begin{matrix} I_l & Q \end{matrix}\right) \, \,\bigr| \, \, Q = Q^T \right\}, \] other charts are considered similarly. If for the chart $U_q$ an element of $O_{max}$ has the form \[ M_q = \left(\begin{matrix} I_{nl} & Q \end{matrix}\right), \qquad Q = \left(\begin{matrix}  Q_1 & \dots & Q_n \\  \vdots & \udots & \\  Q_n & & \end{matrix} \right),\] then for the chart $U_q$ it is takes the form \[ M_p = \left(\begin{matrix} P & \hat{I} \end{matrix}\right), \qquad \hat{I} = \left(\begin{matrix}  &  & I_{nl} \\   & \udots & \\  I_{nl} & & \end{matrix} \right) \quad P = \left( \begin{matrix} P_1 & \dots & P_n  \\ & \ddots &  \vdots \\ & & P_1  \end{matrix} \right)= \left( \begin{matrix} Q_1 & \dots & Q_n  \\ & \ddots & \vdots \\ & & Q_1  \end{matrix} \right)^{-1}.\] The corresponding elements of the jet space are given by the curves \[ \gamma_p = \left( \begin{matrix}P(\lambda) & I_l   \end{matrix} \right), \qquad P(\lambda) = P_{1} + \lambda P_{2} + \dots + \lambda^{n-1} P_n\]  and \[ \gamma_q(\lambda) =  \left( \begin{matrix} I_l &  Q(\lambda)\end{matrix} \right), \qquad Q(\lambda) = Q_n + \lambda Q_{n-1} + \dots + \lambda^{n-1} Q_1.\] Since $Q(\lambda)$ is nondegenerate  for small $\lambda$ we can define another curve \[ \hat{\gamma}_q (\lambda) = \left( \begin{matrix}  Q(\lambda)^{-1} & I_l  \end{matrix} \right).\] Both $\gamma_q(\lambda)$ and $\hat{\gamma}_q(\lambda)$ represent the same curve in the Grassmanian, leading to the same jet. Furthermore, the Taylor expansions of $ \hat{\gamma}_q (\lambda)$  and $\gamma_p(\lambda)$ agree up to order $\lambda^n$, ensuring they define the same $(n-1)$-order jet. Since the jet representation is independent of the chart, the top orbit $O_{max}$ is isomorphic to $J^{n-1}_0(\mathbb{K}, \Lambda(l))$.

\end{enumerate}

 Theorem~\ref{Th:EqualJordBiLagrMaxOrbit} is proved. \end{proof}

\subsection{Topology of orbits}

We outline the topology of non-maximal $\operatorname{Aut}(V, \mathcal{P})$-orbits for the sum of equal Jordan blocks.
\begin{definition} Let $(W^{2n}, \omega)$ be a $2n$-dimensional symplectic vector space and $1 \leq d_1 < \dots < d_k \leq n$. The \textbf{partial isotropic flag variety} in $(W^{2n}, \omega)$ with signature $(d_1, \dots, d_k)$, which we denote by  $\operatorname{SF}(d_1, ..., d_k; 2n)$, is the space of all flags \[\left\{0\right\} = V_0 \subset V_1 \subset \dots \subset V_k \subset W^{2n}, \] where each $V_i$ is an isotropic subspace and $\dim V_i = d_i$ and for all $i = 1,\dots, k$. \end{definition}

  $\operatorname{SF}(d_1, ..., d_k; 2n)$ is a partial flag variety $\operatorname{Sp}(V)/P$ (see e.g. \cite[S23.3]{Fulton91}). Its dimension is
\begin{equation} \label{Eq:DimsSF} \dim \operatorname{SF}(d_1, ..., d_k; 2n) = 2n d_k - \frac{3d_k^2 - d_k}{2} + \sum_{j=1}^{k-1} d_j (d_{j+1} - d_j)\end{equation} (see e.g. \cite{Coskun2014}). $\operatorname{SF}(d; 2n)$ is the \textbf{symplectic isotropic Grassmannian} parameterizing $d$-dimensional isotropic subspaces of $(W^{2n}, \omega)$. In the case $n=d$ it coincides with the Lagrangian Grassmanian $\operatorname{SF}(n; 2n) = \Lambda(n)$. 

For any bi-Lagrangian subspace $L \subset \bigoplus_{i=1}^t \left(\bigoplus_{j=1}^{l_i} \mathcal{J}_{0, 2n} \right)$ we can construct a flag of isotropic subspaces in a $2n$-dimensional symplectic space as follows:

\begin{itemize}

\item First, define the symplectic structure on $\operatorname{Ker}P$. Consider the symplectic structure $\omega$ on $V/ \operatorname{Im}P$ given by Assertion~\ref{A:InducedImP} and then induce it on $\operatorname{Ker}P$ using the isomorphism: \begin{equation} \label{Eq:Pnminus1} P^{n-1}: (V/ \operatorname{Im}P, \omega) \to \operatorname{Ker} P. \end{equation}

\item For a bi-Lagrangian subspace $L$ consider the subspaces \begin{equation} \label{Eq:ImageLagr} \left\{0\right\} = W_0 \subset W_1 \subset \dots \subset W_{\left[ \frac{n+1}{2} \right]} \subset \operatorname{Ker}P, \qquad W_j = P^{n-j}(L) \cap \operatorname{Ker} P.  \end{equation}  
\end{itemize}

\begin{assertion} \label{A:SubsWIsotr} The subspaces $W_j$, given by \eqref{Eq:ImageLagr}, are isotropic w.r.t. the symplectic structure given by \eqref{Eq:Pnminus1}.  \end{assertion}

\begin{proof}[Proof of Assertion~\ref{A:SubsWIsotr}] In the standard basis $e^i_j, f^i_j$ from Theorem~\ref{T:BiLagrEqJordBl}, the vectors $e^i_n, f^i_1, i=1,\dots, l$ form a symplectic basis of $\operatorname{Ker}P$. Furthermore, all subspaces $W_j$ are contained within the Lagrangian subspace $\operatorname{Span}\left\{e^i_n\right\}_{i=1,\dots, l}$. Assertion~\ref{A:SubsWIsotr} is proved. \end{proof}

For example, consider \eqref{Eq:ProjEqJord} where we realize $\mathcal{J}_{0, 8} \oplus\mathcal{J}_{0, 8} \oplus \mathcal{J}_{0, 8}$ as a Young-like diagram. The shaded area and cells with numbers represents a bi-Lagrangian subspace with the corresponding heights $h_1 = 4, h_2 = 3$ and $h_3 = 2$. The right-hand side shows the subspaces $W_j \subset \operatorname{Ker} P$. A cell with number $j$ correspond to vectors that belong to all subspaces $W_i$ for $i \leq j$.

\begin{equation} \label{Eq:ProjEqJord}
  \begin{tabular}{|c|c||c|c||c|c|} 
  \hline 1 & &    &  &  &  \\
 \hline {\cellcolor{gray!25} } & &  2  &  & &  \\
     \hline {\cellcolor{gray!25} } & &  {\cellcolor{gray!25} }  & &  {\cellcolor{gray!25} }  & {\cellcolor{gray!25} }  \\ 
   \hline {\cellcolor{gray!25} } &  &  {\cellcolor{gray!25} }  &  {\cellcolor{gray!25} } &  {\cellcolor{gray!25} }  & {\cellcolor{gray!25} } \\ \hline
  \end{tabular} \qquad \to \qquad \begin{tabular}{|c|c||c|c||c|c|} 
   \multicolumn{1}{c}{}  & \multicolumn{1}{c}{}  &  \multicolumn{1}{c}{}   &  \multicolumn{1}{c}{}  & \multicolumn{1}{c}{}  & \multicolumn{1}{c}{}  \\
 \multicolumn{1}{c}{}  & \multicolumn{1}{c}{}  &  \multicolumn{1}{c}{}   &  \multicolumn{1}{c}{}  & \multicolumn{1}{c}{}  & \multicolumn{1}{c}{}  \\
     \multicolumn{1}{c}{}  & \multicolumn{1}{c}{}  &  \multicolumn{1}{c}{}   &  \multicolumn{1}{c}{}  & \multicolumn{1}{c}{}  & \multicolumn{1}{c}{}  \\
   \hline 1 &  &  2  &   &   &  \\ \hline
  \end{tabular}
\end{equation}

\begin{theorem} \label{T:EqualJordTopolOrb} Consider a sum of equal Jordan blocks $(V, \mathcal{P}) = \bigoplus_{i=1}^t \left(\bigoplus_{j=1}^{l_i} \mathcal{J}_{0, 2n} \right)$. Let $L \subset (V, \mathcal{P})$ be  a bi-Lagrangian subspace with type $\left\{(h_i, n) \times l_i \right\}_{i=1,\dots, t}$, where $h_1 > h_2 > \dots > h_t$. Put $d_i = \sum_{j=1}^{i} l_i$. Then 
the $\operatorname{Aut}(V, \mathcal{P})$-orbit $O_L$ is a $\mathbb{K}^N$-fibre bundle over a partial isotropic flag variety: \begin{equation} \label{Eq:ProjOrbEqJord} \pi: O_L  \xrightarrow{\mathbb{K}^N}  \operatorname{SF}(d_1, ..., d_k; 2 D), \quad D = \sum_{i=1}^t l_i, \quad k = \begin{cases} t, \quad &\mbox{if } 2h_t > n, \\ t-1, \quad &\mbox{if } 2h_t = n.\end{cases}.\end{equation} In the special case $t=1$ and $2h_1 = n$, the orbit $O_{L}$ consists of one point. The projection $\pi$ is given by \eqref{Eq:ImageLagr}, where we consider only one representative of the equal subspaces $W_j$.    \end{theorem}

\begin{remark} The dimension of the fiber $N$ can be easily obtained by subtracting the dimension of the base (given by \eqref{Eq:DimsSF}) from  the orbit dimension (given by \eqref{Eq:DimOLEqJord}). We get \begin{equation} \label{Eq:DimNEqJord} N = \sum_{i=1}^t \left( \frac{l_i(l_i + 1)}{2} (2h_i -n) \right) + \sum_{1 \leq i < j \leq t} \left(l_i l_j (2h_i -n - 1)\right) - \frac{D(D+1)}{2} + \delta,\end{equation} where \[D = \sum_{i=1}^t l_i, \qquad  \delta = \begin{cases} 0, \quad &\mbox{if } 2h_t > n, \\ \frac{l_t(l_t + 1)}{2}, \quad &\mbox{if } 2h_t = n.\end{cases}\] \end{remark}

\begin{proof}[Proof of Theorem~\ref{T:EqualJordTopolOrb}] First, we need to ensure the projection $\pi$ is well-defined and surjective. When applying Equation~\eqref{Eq:ImageLagr}, we consider only one representative for any equal subspaces $W_j$. This ensures that the resulting dimensions of isotropic subspaces are  $d_1,\dots, d_k$. It can be verified using the basis from Theorem~\ref{T:BiLagrEqJordBl} that \begin{equation} \label{Eq:DimIm} \dim P^{h_i -1}(L) = d_i =  \sum_{j=1}^i l_i, \end{equation} and that $\pi$ is surjective. Assuming  \[ h_1 = n\] (otherwise  perform bi-Poisson reduction w.r.t. $\operatorname{Ker} P$), we proceed by induction on  $n$. 

\begin{itemize}

\item \textbf{Base cases ($n=1$ and $2$)}. If $n=1$, then there is only one orbit diffeomorphic to the Lagrangian Grassmanian $O_L = \operatorname{\Lambda}(l)$. Now, consider the case $n=2$. Let a bi-Lagrangian subspace $L \subset \left(\bigoplus_{j=1}^{l_1} \mathcal{J}_{0, 2n} \right) \oplus \left(\bigoplus_{j=1}^{l_2} \mathcal{J}_{0, 2n} \right)$ have type \[H(L) = \left\{(2,2) \times l_1, (1,2) \times l_2\right\}.\] Simply speaking, $\dim P(L) = l_1$. Let's select a standard basis as follows: \[ \vec{e}_1, \vec{e}_2, \vec{f}_1, \vec{f}_2, \vec{e}'_1, \vec{e}'_2, \vec{f}'_1, \vec{f}'_2,\] Each vector in this basis represents a block of elements from the full basis.  The first four vectors $\vec{e}_1, \vec{e}_2, \vec{f}_1, \vec{f}_2$ correspond to $l_1$ Jordan block, and the other four vectors correspond to $l_2$ Jordan blocks. For example, $\vec{e}_1 = (e^1_1, \dots, e^{l_1}_1)$. We choose this basis such that the image \[P(L) = \operatorname{Span}\left\{\vec{e}_2\right\} = \operatorname{Span}\left\{e_2^1,\dots, e_2^{l_1}\right\}.\] The projection of the orbit has the form \[\pi(L) = W_1 = P(L) \subset \operatorname{SF}(l_1, 2n).\] Any such bi-Lagrangian subspace has the form \[ L = \operatorname{Span}\left\{ \vec{e}_1 + X\vec{f}_1, \vec{e}_2, \vec{e}'_1, \vec{f}'_2\right\}, \qquad X^T =X \in \operatorname{Mat}_{l_1}(\mathbb{K}). \] It is easy to see that the projection $\pi$ defines a fiber bundle with $\mathbb{K}^{l_1(l_1+1)/2}$-fibers (given by symmetric matrices).

\item \textbf{Induction step}. Consider the projection \eqref{Eq:ProjOrbEqJord}. Let \[ \hat{L} = (L \cap \operatorname{Im}P + \operatorname{Ker}P)/ \operatorname{Ker}P\] be the image of $L$ in $\operatorname{BLG}(\operatorname{Im}P/\operatorname{Ker}P)$ after the bi-Poisson reduction w.r.t. $U =\operatorname{Ker}P $. Denote its $\operatorname{Aut}(\operatorname{Im}P/\operatorname{Ker}P)$-orbit in $\operatorname{BLG}(\operatorname{Im}P/\operatorname{Ker}P)$ as $\hat{O}_{\hat{L}}$. Let the projection  \eqref{Eq:ProjOrbEqJord} have the form \[ \pi(L)= \left\{ U_1\subset \dots \subset U_k\right\} \in \operatorname{SF}(d_1, ..., d_k; 2 D).\] The similar projection of $\hat{O}_{\hat{L}}$ has the form \[ \hat{\pi}(\hat{L})= \begin{cases} \left\{ U_1\subset \dots \subset U_k\right\} \in \operatorname{SF}(d_1, ..., d_k; 2 D), \quad &\mbox{if }  h_1 > h_2 + 1,\\ \left\{ U_2\subset \dots \subset U_k\right\} \in \operatorname{SF}(d_2, ..., d_k; 2 D), \quad &\mbox{if }  h_1 = h_2 + 1.
\end{cases} \] Consider the commutative diagram:
\[
\begin{tikzcd}
    O_L \arrow[dr, "\pi"] \arrow[d, "p_1"']\\
    (\hat{O}_{\hat{L}}, U_1) \arrow[r,"p_2"'] & \operatorname{SF}(d_1, ..., d_k; 2 D)
\end{tikzcd}
\] The map $p_1$ is given by $p_1(L) = (\hat{L}, U_1)$. If $h_1 = h_2 +1$, then the map $p_2$ adds $U_1$ to the beginning of the flag $\hat{\pi}(\hat{L})$. Otherwise, if $h_1 > h_2 + 1$, the map $p_2$ is the projection $\pi_2$.

\begin{assertion} \label{A:ProjFiberBundles} The map $p_1$ is a $\mathbb{K}^{M_1}$-fiber bundle, where \[ M_1 =  \begin{cases} l_1(2D-l_1 + 1), \quad &\mbox{if }  h_1 > h_2 + 1,\\  l_1(2D-l_1 + 1) - l_1 l_2, \quad &\mbox{if }  h_1 = h_2 + 1.
\end{cases} \] The map $p_2$ is a 
$\mathbb{K}^{M_2}$-fiber bundle, where $M_1 + M_2 = N$, given by \eqref{Eq:DimNEqJord}. \end{assertion}

\begin{proof}[Proof of Assertion~\ref{A:ProjFiberBundles}] We start by examining the case $h_1 > h_2 + 1$. Then, we'll discuss the necessary adjustments to the formulas when $h_1 = h_2 + 1$. 

\begin{itemize}
    \item \textbf{Case $h_1 > h_2 +1$}. then $p_2$ is the projection $\hat{\pi}$ for the orbit of the type \[\hat{H}(\hat{L}) = \left\{(\hat{h}_j, \hat{n}) \times l_j\right\}_{j=1,\dots, t}, \qquad \hat{h}_j = \begin{cases} n - 2, \quad &\mbox{if }  j = 1,\\ h_j - 1, \quad &\mbox{if }  j > 1,\end{cases} \qquad \hat{n} = n-2.\] Hence, $p_2$ is a $\mathbb{K}^{M_2}$-fiber bundle, where $M_2$ can be found by formula \eqref{Eq:DimNEqJord}. Next, consider the map $p_1$. Using Theorem~\ref{T:BiLagrEqJordBl} and Corollary~\ref{Cor:ImageAutIso} we can find a canonical basis $\vec{e}^i_j,\vec{f}^i_j, i=1,\dots, t, j=1,\dots, n$ in $V$ that brings $\hat{L}$ to its canonical form: \[ \hat{L}  = \operatorname{Span}\left\{\vec{f}^i_2,\dots, \vec{f}^i_{n-\hat{h}_j - 1},  \vec{e}^i_{n-\hat{h}_j}, \dots, \vec{e}^i_{n-1}\right\}_{i=1,\dots, t}\] Here $\vec{e}^i_j$ and $\vec{f}^i_j$ are short notations for $l_i$ basis vectors. For instance, $\vec{e}^i_j = \left(e^{i,1}_j, \dots, e^{i,l_i}_j \right)$. The subspace $U = P(\hat{L})$ of $V$ is well-defined and has the form \[  U  = \operatorname{Span}\left\{\vec{f}^i_1,\dots, \vec{f}^i_{n-\hat{h}_j-2},  \vec{e}^i_{n-\hat{h}_j+1}, \dots, \vec{e}^i_{n}\right\}_{i=1,\dots, t}.\] By construction, \[ U \subset L \subset U^{\perp},\] where \[ U^{\perp} = U \oplus \operatorname{Span}\left\{\vec{f}^i_{n-\hat{h}_j-1}, \vec{f}^i_{n-\hat{h}_j},  \vec{e}^i_{n-\hat{h}_j}, \vec{e}^i_{n-\hat{h}_j-1}\right\}_{i=1,\dots, t}.\] Theorem~\ref{T:BiLagrEqJordBl} guarantees a specific structure for the bi-Lagrangian space $L$: \begin{equation} \label{Eq:TopGenStrL} L = Z + P(Z) + P^2 (Z) + \dots, \qquad Z = \operatorname{Span} \left\{\vec{u}_1, \vec{v}_1, \dots, \vec{u}_t, \vec{v}_t \right\}. \end{equation} The subspace $L$ has the type $\left\{(h_i, n) \times l_i \right\}_{i=1,\dots, t}$ and we know the images $P^{h_i -1}(L)$ from the projection $\pi(L)$. Therefore, the ``top generators'' have the form \[  \begin{gathered} \vec{u}_1 = \vec{e}^1_1 +  \text{lower terms}, \qquad \vec{v}_1 = 0, \\ \vec{u}_i =  \vec{e}^i_{n-h_i+1} +  \text{lower terms}, \qquad \vec{v}_i = \vec{f}^i_{n-h_i} +  \text{lower terms}. \end{gathered}   \] We can bring these vectors $\vec{u}_i$ and $\vec{v}_i$ to the following form: \begin{equation} \label{Eq:TopGen} \begin{gathered} \vec{u}_1 = \vec{e}^1_1 + X_1 \vec{f}^1_1 + X_2 \vec{f}^1_2 + \sum_{j=2}^n \left( Y_j \vec{e}^j_{n-h_j} + Z_j \vec{f}^j_{n-h_j +1}\right), \qquad \vec{v}_1 = 0, \\ \vec{u}_i =  \vec{e}^i_{n-h_i+1} + Q_i \vec{f}^1_1, \qquad \vec{v}_i = \vec{f}^i_{n-h_i} + R_i \vec{f}_1^1.   \end{gathered}\end{equation} The subspace $L$ given by \eqref{Eq:TopGenStrL} is bi-Lagrangian if and only if \[ X_1 = X_1^T, \quad X_2 = X_2^T, \quad Q_i = Z_j^T, \quad R_i = - Y_j^T.\]  Since the map $p_1$ is defined by these specific matrices, it becomes a $\mathbb{K}^{l_1(2D-l_1 + 1)}$-fiber bundle.

\item \textbf{The case $h_1 = h_2 + 1$}. The type of orbit $\hat{O}_{\hat{L}}$ changes: \[\hat{H}(\hat{L}) = \left\{(\hat{h}_j, \hat{n}) \times \hat{l}_j\right\}_{j=2,\dots, t}, \quad \hat{h}_j = h_j - 1, \quad \hat{l}_j = \begin{cases} l_1 + l_2 , \quad &\mbox{if }  j = 2,\\ l_j, \quad &\mbox{if }  j > 2,\end{cases} \qquad \hat{n} = n-2.\] 
Fiber dimension for $p_2$ gains $l_1 l_2$ due to choosing $U_1$ within $U_2$. This gain is "compensated" by a loss of $l_1 l_2$ in the fibers of $p_1$. Fixing  $U_1 = P^{n-1}(L)$ introduces an additional constraint: $Y_2 = 0$ in the formulas \eqref{Eq:TopGenStrL} and \eqref{Eq:TopGen}.  Overall, the sum of fiber dimensions for $p_1$ and $p_2$ remains the same.

\end{itemize}

Assertion~\ref{A:ProjFiberBundles} is proved.  \end{proof}

By Assertion~\ref{A:ProjFiberBundles} the maps $p_i$ define $\mathbb{K}^{M_i}$-fiber bundle. Since composite fiber bundles are themselves fiber bundles (see e.g. \cite{Poor81}), $\pi$ defines a  $\mathbb{K}^{M_1+M_2}$-fiber bundle, as required. \end{itemize}

Theorem~\ref{T:EqualJordTopolOrb} is proved. \end{proof}

\section{Topology of generic orbits} \label{S:TopGenOrbit}

In  Section~\ref{SubS:GenBiLag} we identified the maximal orbit of the automorphism group $\operatorname{Aut} \left(V, \mathcal{P}\right)$-orbit in $\operatorname{BLG}\left(V, \mathcal{P}\right)$. In this section we analyze the topology of this maximal orbit $O_{\max}$ similar to Theorem~\ref{T:EqualJordTopolOrb}. 

Let $\left(V, \mathcal{P}\right) = \bigoplus_{i=1}^t \left(\bigoplus_{j=1}^{l_i} \mathcal{J}_{0, 2n_i} \right)$, where $n_1 > n_2 > \dots > n_t$. For any generic bi-Lagrangian subspace $L \in O_{\max}$ we can construct $t$ Lagrangian subspaces \begin{equation} \label{Eq:ProdLag} (L_1, \dots, L_t) \in \Lambda(l_1) \times \dots \times \Lambda(l_t)\end{equation} as follows:

\begin{itemize}

\item First, using Assertion~\ref{A:InducedImP} we can induce a symplectic structure $\omega$ on $\operatorname{Ker}P$. Moreover, it allows us to construct a flag of symplectic subspaces within   \begin{equation} \label{Eq:FlagSympGen} \left\{0\right\} \subset S_1 \subset \dots \subset S_t, \qquad S_j = \operatorname{Im}P^{n_j -1} \cap \operatorname{Ker} P. \end{equation} Furthermore, we consider the quotient spaces of these subspaces, which also inherit the symplectic structure: \[ \hat{S}_j = S_j / S_{j-1}, \qquad j=1,\dots, t. \] Here $\dim \hat{S}_j = l_j$ and we formally put $S_{0} = \left\{0\right\}, \hat{S}_1 = S_1$.

\item By Theorem~\ref{T:BiLagrGenDecomp} for any generic bi-Lagrangian subspace $L$ there are $t$ Lagrangian subspace \[ \hat{L}_j \subset S_j, \qquad \hat{L}_j =  P^{n_j -1} \cap \operatorname{Ker} P. \] We define the elements $L_j$ in \eqref{Eq:ProdLag} as \begin{equation} \label{Eq:ProjGenJord} L_j = \hat{L}_j / \hat{L}_{j-1}, \qquad j=1,\dots, t.\end{equation} We set $L_{0} = \left\{0\right\}, L_1 = \hat{L}_1$.

\end{itemize}

For example, consider \eqref{Eq:ProjGenJordTopol} where we realize $\mathcal{J}_{0, 8} \oplus\mathcal{J}_{0, 6} \oplus \mathcal{J}_{0, 4}$ as a Young-like diagram. The shaded area and cells with numbers represent a generic bi-Lagrangian subspace $L$. The right-hand side depicts the Lagrangian subspaces $\hat{L}_j \subset \operatorname{Ker} P$. A cell with number $j$ correspond to vectors that belong to all subspaces $\hat{L}_i$ for $i \leq j$.

\begin{equation} \label{Eq:ProjGenJordTopol}
  \begin{tabular}{|c|c||c|c||c|c|} 
  \cline{1-2} 1 & &   \multicolumn{1}{c}{} & \multicolumn{1}{c}{}  & \multicolumn{1}{c}{} & \multicolumn{1}{c}{} \\
  \cline{1-4} {\cellcolor{gray!25} } & &  2  &  &\multicolumn{1}{c}{} &\multicolumn{1}{c}{}  \\
     \hline {\cellcolor{gray!25} } & &  {\cellcolor{gray!25} }  & &  3  &   \\ 
   \hline {\cellcolor{gray!25} } &  &  {\cellcolor{gray!25} }  & &  {\cellcolor{gray!25} }  &  \\ \hline
  \end{tabular} \qquad \to \qquad \begin{tabular}{|c|c||c|c||c|c|} 
   \multicolumn{1}{c}{}  & \multicolumn{1}{c}{}  &  \multicolumn{1}{c}{}   &  \multicolumn{1}{c}{}  & \multicolumn{1}{c}{}  & \multicolumn{1}{c}{}  \\
 \multicolumn{1}{c}{}  & \multicolumn{1}{c}{}  &  \multicolumn{1}{c}{}   &  \multicolumn{1}{c}{}  & \multicolumn{1}{c}{}  & \multicolumn{1}{c}{}  \\
     \multicolumn{1}{c}{}  & \multicolumn{1}{c}{}  &  \multicolumn{1}{c}{}   &  \multicolumn{1}{c}{}  & \multicolumn{1}{c}{}  & \multicolumn{1}{c}{}  \\
   \hline 1 &  &  2  &   &  3 &  \\ \hline
  \end{tabular}
\end{equation}

\begin{theorem} \label{T:GenJordTopolMaxOrb}  Let $(V, \mathcal{P}) = \bigoplus_{i=1}^t \left(\bigoplus_{j=1}^{l_i} \mathcal{J}_{0, 2n_i} \right)$, where $n_1 > n_2 > \dots > n_t$, and $O_{\max}$ be the maximal 
 $\operatorname{Aut} \left(V, \mathcal{P}\right)$-orbit in $\operatorname{BLG}\left(V, \mathcal{P}\right)$. This orbit has a  structure of a $\mathbb{K}^{N}$-fibre bundle over a product of Lagrangian Grassmanians: \begin{equation} \label{Eq:ProjOrbGenJord} \pi: O_{\max}  \xrightarrow{\mathbb{K}^N}  \Lambda(l_1) \times \dots \times \Lambda(l_t). \end{equation} The projection $\pi$ is given by \eqref{Eq:ProjGenJord}.    \end{theorem}

\begin{remark} The dimension of the fiber $N$ can be easily obtained by subtracting the dimension of the base from  the orbit dimension (given by \eqref{Eq:DimBLGComb}). We get \begin{equation} \label{Eq:DimNGenJord} N = \sum_{i=1}^t \left( \frac{l_i(l_i + 1)}{2} (2h_i -n - 1) \right) + \sum_{1 \leq i < j \leq t} \left(l_i l_j (2h_i -n )\right). \end{equation} \end{remark}

\begin{proof}[Proof of Theorem~\ref{T:GenJordTopolMaxOrb}] The proof follows a similar approach as in Theorem~\ref{T:EqualJordTopolOrb}. If all Jordan blocks have equal size, i.e. for $t=1$, then the statement follows from Theorem~\ref{Th:EqualJordBiLagrMaxOrbit}. Hence, we focus on the case $t \geq 2$. The projection $\pi$ is well-defined and surjective by Theorem~\ref{T:BiLagrGenDecomp}. For the remainder of the proof, we'll proceed by induction on  $n$. 

\begin{itemize}

\item \textbf{Base cases ($n=1$)}. If $n=1$, then there is only one orbit diffeomorphic to the Lagrangian Grassmanian $O_L = \operatorname{\Lambda}(l_1)$. The projection $\pi$ is the identity map.

\item \textbf{Induction step}.Let $L \in \operatorname{BLG}(V,\mathcal{P})$ be a generic bi-Lagrangian subspace and \[ \hat{L} = (L \cap \operatorname{Im}P + \operatorname{Ker}P)/ \operatorname{Ker}P\] be the image of $L$ in $\operatorname{BLG}(\operatorname{Im}P/\operatorname{Ker}P)$ after the bi-Poisson reduction w.r.t. $U =\operatorname{Ker}P $. Denote the 
maximal orbit in $\operatorname{BLG}(\operatorname{Im}P/\operatorname{Ker}P)$ as $\hat{O}_{\max}$. Let the projection  \eqref{Eq:ProjOrbGenJord} have the form \[ \pi(L)= (L_1, \dots, L_t) \in \Lambda(l_1) \times \dots \times \Lambda(l_t).\] The similar projection of $\hat{O}_{\hat{L}}$ has the form \[ \hat{\pi}(\hat{L})= \begin{cases} (L_1, \dots, L_t) \in \Lambda(l_1) \times \dots \times \Lambda(l_t), \quad &\mbox{if }  n_t > 2,\\ (L_1, \dots, L_{t-1}) \in \Lambda(l_1) \times \dots \times \Lambda(l_{t-1}), \quad &\mbox{if }  n_{t-1} > 2, n_{t} \leq 2, \\ (L_1, \dots, L_{t-2}) \in \Lambda(l_1) \times \dots \times \Lambda(l_{t-2}), \quad &\mbox{if }  n_{t-1} =  2. 
\end{cases} \] Consider the commutative diagram:
\[
\begin{tikzcd}
    O_{\max} \arrow[dr, "\pi"] \arrow[d, "p_1"']\\
    (\hat{O}_{\max}, L_{t-1}, L_t) \arrow[r,"p_2"'] &  \Lambda(l_1) \times \dots \times \Lambda(l_t)
\end{tikzcd}
\] The map $p_1$ is given by $p_1(L) = (\hat{L}, L_{t-1}, L_t)$. In the map $p_2$ we add the subspace $L_{t-1}$ and $L_t$ to $\hat{\pi}(\hat{L})$, if needed.

\begin{assertion} \label{A:ProjFiberBundlesGenJord} The map $p_1$ is a $\mathbb{K}^{M_1}$-fiber bundle, where \[ M_1 = \sum_{i=1}^{t} l_i(l_i + 1) + \sum_{1 \leq i < j \leq t} 2l_i l_j - \delta,\]and \[\delta= \begin{cases} 0, \quad &\mbox{if }  n_t >2,\\   l_t(l_t + 1) + \sum_{i=1}^{t-1} l_i l_t, \quad &\mbox{if }  n_{t-1}>2, n_t =1 \\  \frac{1}{2}l_t(l_t + 1), \quad &\mbox{if }  n_{t-1}>2, n_t =2 \\ \frac{1}{2}l_{t-1}(l_{t-1} + 1) + l_t(l_t + 1) + \sum_{i=1}^{t-1} l_i l_t, \quad &\mbox{if }  n_{t-1}=2, n_t =1.
\end{cases} \] The map $p_2$ is a 
$\mathbb{K}^{M_2}$-fiber bundle, where $M_1 + M_2 = N$, given by \eqref{Eq:DimNGenJord} . \end{assertion}

\begin{proof}[Proof of Assertion~\ref{A:ProjFiberBundlesGenJord}] Similar to Assertion~\ref{A:ProjFiberBundles}. $p_2$ is a $\mathbb{K}^{M_2}$-fiber bundle by induction hypothesis.  Using Theorem~\ref{T:BiLagrGenDecomp} and Corollary~\ref{Cor:ImageAutIso} we can find a canonical basis $\vec{e}^i_j,\vec{f}^i_j, i=1,\dots, t, j=1,\dots, n$ in $V$ that brings $\hat{L}$ to its canonical form: \[ \hat{L}  = \operatorname{Span}\left\{ \vec{e}^i_{2}, \dots, \vec{e}^i_{n_i-1}\right\}_{i=1,\dots, t}\]  Here $\vec{e}^i_j$ and $\vec{f}^i_j$ are short notations for $l_i$ basis vectors. For instance, $\vec{e}^i_j = \left(e^{i,1}_j, \dots, e^{i,l_i}_j \right)$. The subspace $U = P(\hat{L})$ of $V$ is well-defined and has the form \[  U  = \operatorname{Span}\left\{\vec{e}^i_{3}, \dots, \vec{e}^i_{n_i}\right\}_{i=1,\dots, t}.\] By construction, \[ U \subset L \subset U^{\perp},\] where \[ U^{\perp} = U \oplus \operatorname{Span}\left\{\vec{f}^i_{1}, \vec{f}^i_{2},  \vec{e}^i_{1}, \vec{e}^i_{2}\right\}_{i=1,\dots, t}.\]

Theorem~\ref{T:BiLagrGenDecomp} guarantees a specific structure for the bi-Lagrangian space $L$: \[ L = Z + P(Z) + P^2 (Z) + \dots, \qquad Z = \left\{\vec{u}_1, \dots, \vec{u}_t\right\}. \] The Lagrangian subspaces $L_j = \operatorname{Im} P^{n_j-1} (L) \cap \operatorname{Ker} P$ have the form \begin{equation} \label{Eq:LjImGen} L_j = \operatorname{Span}\left\{ \vec{e}^i_{n}\right\}_{i=1,\dots, j}\end{equation} (if $n_{t-1} \leq 2$ or $n_t \leq 2$ we obtain \eqref{Eq:LjImGen} for the lower-dimensional Jordan blocks by choosing a suitable canonical basis in them). Therefore, the ``top generators'' have the form \[   \vec{u}_1 = \vec{e}^1_1 +  \text{lower terms}.  \] 
Through a suitable basis change in $L$, the vectors $\vec{u}_i$ can be transformed into the following form: \begin{equation} \label{Eq:TopGenGenJord} \vec{u}_i = \vec{e}^i_1 + \sum_{j=1}^t \left( X^i_j \vec{f}^j_{1} + Y^i_j \vec{f}^j_{2}\right), \qquad i = 1,\dots, t. \end{equation} The subspace $L$ given by \eqref{Eq:TopGenStrL} is bi-Lagrangian if and only if \[ (X^i_j)^T = X_j^i, \quad (Y^i_j)^T = Y^i_j.\]  If $n_j = 1$, then there are is no $Y^j_j$ for \eqref{Eq:TopGenStrL} to be fulfilled. Similarly, if $n_t = 1$, then there are no $X^t_t, Y^t_t$ and $Y^i_t$. Since the map $p_1$ is defined by these specific matrices, it becomes a $\mathbb{K}^{M_1}$-fiber bundle. In special cases $n_t \leq 2$ the fiber dimensions of $p_2$ increase due to choosing additional Lagrangian subspaces $L_{t-1}$ or $L_t$. Overall, the sum of fiber dimensions for $p_1$ and $p_2$ remains the same. Assertion~\ref{A:ProjFiberBundlesGenJord} is proved.  \end{proof}

By Assertion~\ref{A:ProjFiberBundlesGenJord} $\pi$ defines a  $\mathbb{K}^{M_1+M_2}$-fiber bundle, as required. \end{itemize}

Theorem~\ref{T:GenJordTopolMaxOrb} is proved. \end{proof}

\subsection{Connectedness of Bi-Lagrangian Grassmannian}

\begin{theorem} \label{T:BLGConnected} All bi-Lagrangian Grassmanians $\operatorname{BLG}(V,\mathcal{P})$ are path connected.  \end{theorem}

\begin{proof}[Proof of Theorem~\ref{T:BLGConnected}] The proof is by induction on $\dim V$. The base ($\operatorname{dim} V =2$) is trivial. The induction step. Let $L\subset (V,\mathcal{P})$ be a bi-Lagrangian subspace. If $(V,\mathcal{P})$ consists of one Jordan block, then $\operatorname{BLG}(V,\mathcal{P})$ is path connected by Theorem~\ref{Th:OneJordBiLagrOrbits}. Otherwise, we can connect $L$ with a generic bi-Lagrangian subspace as follows. Consider two cases as in Section~\ref{SubS:ExtractOneJordSimple}:

\begin{enumerate}
    \item $\operatorname{height}(L) = \operatorname{height}(V)$. By Theorem~\ref{T:ExtractMaxBlockGen} we can extract one Jordan block: \[(V, \mathcal{P}) = (V_1,\mathcal{P}_1) \oplus (V_2, \mathcal{P}_2), \quad L = L_1 \oplus L_2, \quad L_i = L \cap V_i, \quad (V_1,\mathcal{P}_1) =\mathcal{J}_{0, 2n_1}.\] Since the bi-Lagrangian Grassmannians $\operatorname{BLG}(V_i, \mathcal{P}_i)$ are path connected by the induction hypothesis, we can construct paths connecting the subspaces $L_i$ with generic bi-Lagrangian subspaces $L_i' \subset (V_i, \mathcal{P}_i)$. Their sum $L_1' \oplus L_2'$ is generic in $(V,\mathcal{P})$.

    \item $\operatorname{height}(L) \not = \operatorname{height}(V)$. Then $L/U \subset \operatorname{BLG}(U^{\perp}/U)$, where $U= \operatorname{Im}P^{n_1 -1}$. By induction hypothesis,  $\operatorname{BLG}(U^{\perp}/U)$ is path connected. Hence, we can connect $L$ with a sum of bi-Lagrangian subspaces \[ L' = \oplus_{i=1}^N L_i, \qquad L'_i \subset \mathcal{J}_{0, 2n_i}. \] We can connect $L'$ with a generic $L''$ in $(V,\mathcal{P})$, since by Theorem~\ref{Th:OneJordBiLagrOrbits} $\operatorname{BLG}(\mathcal{J}_{0,2n_i})$ are path connected.

\end{enumerate}

  By Theorem~\ref{T:GenJordTopolMaxOrb}, generic orbit of $\operatorname{BLG}(V,\mathcal{P})$ is path connected, implying the path connectedness of the entire space.
Theorem~\ref{T:BLGConnected} is proved. \end{proof}

\section{Two different Jordan blocks} \label{S:TwoJordanBlocks}

This section describes the $\operatorname{Aut}(V, \mathcal{P})$-orbits of $\operatorname{BLG}(V, \mathcal{P})$ for the sum of $2$ distinct Jordan blocks with the same zero eigenvalue: \begin{equation} \label{Eq:JKDecompTwoDistinct} (V, \mathcal{P})= \mathcal{J}_{0, 2n_1} \oplus \mathcal{J}_{0, 2n_2}, \qquad n_1 > n_2.\end{equation} As usual, $P$ denotes the recursion operator. A basis \begin{equation} \label{Eq:StandBasisTwoBlocks} e_1,\dots, e_{n_1} f_1,\dots, f_{n_1}, \hat{e}_1,\dots, \hat{e}_{n_2}, \hat{f}_1,\dots, \hat{f}_{n_2}\end{equation} from the JK theorem~\ref{T:Jordan-Kronecker_theorem} is a \textbf{standard basis} for $ (V, \mathcal{P})$. 

For the sum of distinct Jordan blocks there are two types of bi-Lagrangian subspaces. A bi-Lagrangian subspace $L \subset (V,\mathcal{P})$ is \textbf{semisimple} if there exists a JK decomposition \eqref{Eq:JKDecompTwoDistinct} such that \[ L = L_1 \bigoplus L_2, \qquad L_i =  (L \cap \mathcal{J}_{0, 2n_i}), \quad i=1,2. \] Otherwise $L$ is \textbf{indecomposable} (see Definition~\ref{Def:Decomp}). Respective $\operatorname{Aut}(V, \mathcal{P})$-orbits inherit these terms. 

\begin{enumerate}

\item In Section~\ref{SubS:TwoDistSemi} we describe semisimple bi-Lagrangian subspaces.

\item In Section~\ref{SubS:TwoDistIndecom} we describe indecomposable bi-Lagrangian subspaces. 

\item To provide an illustrative example, Section~\ref{SubS:Exam} examines bi-Lagrangian Grassmanian for the direct sum of the Jordan blocks $\mathcal{J}_{0,6}$ and $\mathcal{J}_{0,2}$.
\end{enumerate}

\subsection{Semisimple bi-Lagrangian subspaces} \label{SubS:TwoDistSemi}

Since we know the structure of bi-Lagrangian subspaces for one Jordan block (see Section~\ref{S:OneJordBlock}), we can easily describe semisimple bi-Lagrangian subspaces. 

\begin{theorem} \label{T:BiLagrTwoDistSemi} Let $(V, \mathcal{P})= \mathcal{J}_{0, 2n_1} \oplus \mathcal{J}_{0, 2n_2}$.

\begin{enumerate}

\item Any semisimple bi-Lagrangian subspace $L \subset (V, \mathcal{P})$ is uniquely characterized by integers $h_1, h_2$ such that \[ h_i \geq \frac{n_i}{2}, \qquad i=1,2,\] and it admits a canonical form \begin{equation} \label{Eq:CanonEqJordSemiSimpleTwo} \begin{gathered} L =  \operatorname{Span}\left\{f_1,\dots, f_{n_1 - h_1}, e_{n_1 - h_1 + 1}, \dots, e_{n_1}\right\} \oplus \\ \oplus \operatorname{Span}\left\{\hat{f}_1,\dots, \hat{f}_{n_2 - h_2}, \hat{e}_{n_2-h_2 +1}, \dots, \hat{e}_{n_2}\right\} \end{gathered} \end{equation} for some standard basis \eqref{Eq:StandBasisTwoBlocks}. The \textbf{type} of $L$ is \begin{equation} \label{Eq:TypeTwoBLocks} H(L) = \left\{ (h_1, n_1), (h_2, n_2) \right\}.\end{equation}

\item The number of semisimple $\operatorname{Aut}(V, \mathcal{P})$-orbits is \begin{equation} \label{Eq:NumOrbDistSemiSimple} \left( \left[\frac{n_1}{2}\right] + 1 \right) \left( \left[\frac{n_2}{2}\right] + 1 \right) .\end{equation}

\item The dimension of the $\operatorname{Aut}(V, \mathcal{P})$-orbit $O_L$ of $L$ with type \eqref{Eq:TypeTwoBLocks} is \begin{equation} \label{Eq:DimODist2Jord} \dim O_L = (2h_1 - n_1) + (2h_2 - n_2) + \Delta_{12},\end{equation}  where  where \begin{equation} \label{Eq:DeltaDimDec2Dist} \Delta_{12} = \max\left(0, h_2 - \left(n_1 - h_1\right)\right) + \max\left(0, h_2 - h_1, \left(n_2 - h_2\right) - \left(n_1 - h_1 \right)\right).\end{equation}  

\end{enumerate}

\end{theorem}

\begin{remark} In Theorem~\ref{T:DimDecomp} there are four possible values for $\Delta_{12}$, i.e.  \begin{equation} \label{Eq:FourDelta} \Delta_{12} = \begin{cases} 0, \qquad & h_1 \geq n_1 -h_1 \geq h_2 \geq n_2 - h_2 \\ h_2 - (n_1 - h_1), \qquad & h_1 \geq h_2 \geq n_1 - h_1 \geq n_2 - h_2, 
\\ n_2 - 2 (n_1 - h_1), \qquad &h_1 \geq h_2 \geq n_2 - h_2 \geq n_1 - h_1, 
\\ 2h_2 - n_1 \qquad & h_2 \geq h_1 \geq n_1-h_1 \geq n_2 - h_2. \end{cases} \end{equation}  \end{remark}

\begin{proof}[Proof of Theorem~\ref{T:BiLagrTwoDistSemi}] Canonical forms and orbit counts follow from Theorem~\ref{T:BiLagr_One_Jordan_Canonical_Form}. Orbit dimensions are computed as in Theorem~\ref{T:BiLagrJord} (see also Corollary~\ref{Cor:NumTypes} below). Theorem~\ref{T:BiLagrTwoDistSemi} is proved. \end{proof}

The four cases of \eqref{Eq:FourDelta} are visualized in Figure~\eqref{Eq:Proj2BlocksSemi}. Jordan blocks are depicted as Young-like diagrams, with shaded areas representing semisimple bi-Lagrangian subspaces.

\begin{equation} \label{Eq:Proj2BlocksSemi}
 \begin{tabular}{|c|c||c|c|} 
     \cline{1-2} & &   \multicolumn{1}{c}{} & \multicolumn{1}{c}{}   \\
     \cline{1-2}  & &   \multicolumn{1}{c}{} & \multicolumn{1}{c}{}   \\
     \cline{1-2} & &   \multicolumn{1}{c}{} & \multicolumn{1}{c}{}   \\
  \cline{1-2} {\cellcolor{gray!25} } & &   \multicolumn{1}{c}{} & \multicolumn{1}{c}{}   \\
  \cline{1-2} {\cellcolor{gray!25} } &  {\cellcolor{gray!25} } &   \multicolumn{1}{c}{} & \multicolumn{1}{c}{}   \\
     \hline {\cellcolor{gray!25} }  & {\cellcolor{gray!25} } &  {\cellcolor{gray!25} }  &  \\ 
   \hline {\cellcolor{gray!25} } &  {\cellcolor{gray!25} }    &  {\cellcolor{gray!25} }  &    \\ \hline
  \end{tabular}  \qquad  \begin{tabular}{|c|c||c|c|} 
     \cline{1-2} & &   \multicolumn{1}{c}{} & \multicolumn{1}{c}{}   \\
     \cline{1-2} {\cellcolor{gray!25} } & &   \multicolumn{1}{c}{} & \multicolumn{1}{c}{}   \\
     \cline{1-2} {\cellcolor{gray!25} } & &   \multicolumn{1}{c}{} & \multicolumn{1}{c}{}   \\
  \cline{1-4} {\cellcolor{gray!25} } &  & {\cellcolor{gray!25} }  &    \\
  \cline{1-4} {\cellcolor{gray!25} } &  &  {\cellcolor{gray!25} }  &    \\
     \hline {\cellcolor{gray!25} }  &  &  {\cellcolor{gray!25} }  &  \\ 
   \hline {\cellcolor{gray!25} } &  {\cellcolor{gray!25} }    &  {\cellcolor{gray!25} }  &    \\ \hline
  \end{tabular} \qquad \begin{tabular}{|c|c||c|c|} 
     \cline{1-2} {\cellcolor{gray!25} } & &   \multicolumn{1}{c}{} & \multicolumn{1}{c}{}   \\
     \cline{1-2} {\cellcolor{gray!25} } & &   \multicolumn{1}{c}{} & \multicolumn{1}{c}{}   \\
     \cline{1-2} {\cellcolor{gray!25} } & &   \multicolumn{1}{c}{} & \multicolumn{1}{c}{}   \\
  \cline{1-4} {\cellcolor{gray!25} } &  &  &    \\
  \cline{1-4} {\cellcolor{gray!25} } &  &  {\cellcolor{gray!25} }  &    \\
     \hline {\cellcolor{gray!25} }  &  &  {\cellcolor{gray!25} }  &  \\ 
   \hline {\cellcolor{gray!25} } &     &  {\cellcolor{gray!25} }  &   {\cellcolor{gray!25} }  \\ \hline
  \end{tabular} \qquad 
  \begin{tabular}{|c|c||c|c|} 
     \cline{1-2}  & &   \multicolumn{1}{c}{} & \multicolumn{1}{c}{}   \\
 \cline{1-4} &  &  {\cellcolor{gray!25} } &    \\
 \cline{1-4} &  &  {\cellcolor{gray!25} } &    \\
  \cline{1-4} {\cellcolor{gray!25} } &  &  {\cellcolor{gray!25} }  &    \\
  \cline{1-4} {\cellcolor{gray!25} } &  {\cellcolor{gray!25} } &  {\cellcolor{gray!25} }  &    \\
     \hline {\cellcolor{gray!25} }  & {\cellcolor{gray!25} }  &  {\cellcolor{gray!25} }  &  \\ 
   \hline {\cellcolor{gray!25} } &   {\cellcolor{gray!25} }  &  {\cellcolor{gray!25} }  &   \\ \hline
  \end{tabular} 
\end{equation}

\subsubsection{Topology of semisimple orbits} \label{S:Top2BlocksSemi}

Let $L \subset \left(V, \mathcal{P}\right) = \mathcal{J}_{0, 2n_1} \oplus \mathcal{J}_{0, 2n_2}$ be a semisimple bi-Lagrangian subspace with type \eqref{Eq:TypeTwoBLocks}. Let $D$ count the $h_j$ exceeding $n_j/2$, i.e. \[ D =  I_{2h_1 >  n_1} + I_{2h_2 >  n_2}.\] We can construct $D$ one-dimensional Lagrangian subspaces \begin{equation} \label{Eq:ProdLag2Dist} L_k \in \Lambda(1) = \mathbb{KP}^1, \qquad k = 1, \dots, D, \end{equation} as follows:

\begin{itemize}

\item As in Section~\ref{S:TopGenOrbit} there is a symplectic flag \[\left\{0\right\} \subset S_1 \subset S_2 = \operatorname{Ker}P\] given by \eqref{Eq:FlagSympGen}.  We take a pair of $2$-dimensional symplectic spaces \[\hat{S}_1 = S_1, \qquad \hat{S}_2 = S_2 / S_1. \]

\item Define \[ \hat{L}_j = P^{h_j-1}(L) \cap \operatorname{Ker} P, \qquad j=1,2. \] If both $h_j > n_j/2$, set \[ L_1 = \hat{L}_1 \cap S_1, \qquad L_2 = \left( \hat{L}_2  + S_1 \right) / S_1.\] Otherwise, consider only $D$ subspaces $L_j$ with $h_j > n_j/2$. \end{itemize}

For example, consider \eqref{Eq:Proj2DistJordTopol} where we realize $\mathcal{J}_{0, 14} \oplus\mathcal{J}_{0, 12}$ as a Young-like diagram. The shaded area and cells with numbers represent a generic bi-Lagrangian subspace $L$. Numbered right-hand cells correspond to Lagrangian subspaces $L_j \subset S_j$. 

\begin{equation} \label{Eq:Proj2DistJordTopol}
  \begin{tabular}{|c|c||c|c|} 
     \cline{1-2}  & &   \multicolumn{1}{c}{} & \multicolumn{1}{c}{}   \\
 \cline{1-4} &  &  2 &    \\
 \cline{1-4} 1 &  &  {\cellcolor{gray!25} } &    \\
  \cline{1-4} {\cellcolor{gray!25} }  &  &  {\cellcolor{gray!25} }  &    \\
  \cline{1-4} {\cellcolor{gray!25} } & &  {\cellcolor{gray!25} }  &    \\
     \hline {\cellcolor{gray!25} }  & {\cellcolor{gray!25} }  &  {\cellcolor{gray!25} }  &  \\ 
   \hline {\cellcolor{gray!25} } &   {\cellcolor{gray!25} }  &  {\cellcolor{gray!25} }  &   \\ \hline
  \end{tabular}  \qquad \to \qquad \begin{tabular}{|c|c||c|c|} 
   \multicolumn{1}{c}{}  & \multicolumn{1}{c}{}  &  \multicolumn{1}{c}{}   &  \multicolumn{1}{c}{}    \\
    \multicolumn{1}{c}{}  & \multicolumn{1}{c}{}  &  \multicolumn{1}{c}{}   &  \multicolumn{1}{c}{}    \\
  \multicolumn{1}{c}{}  & \multicolumn{1}{c}{}  &  \multicolumn{1}{c}{}   &  \multicolumn{1}{c}{}    \\
   \multicolumn{1}{c}{}  & \multicolumn{1}{c}{}  &  \multicolumn{1}{c}{}   &  \multicolumn{1}{c}{}    \\
 \multicolumn{1}{c}{}  & \multicolumn{1}{c}{}  &  \multicolumn{1}{c}{}   &  \multicolumn{1}{c}{} \\
     \multicolumn{1}{c}{}  & \multicolumn{1}{c}{}  &  \multicolumn{1}{c}{}   &  \multicolumn{1}{c}{}   \\
   \hline 1 &  &  2  & \\ \hline
  \end{tabular}
\end{equation}

\begin{theorem} \label{T:2DistJordTopolMaxOrb} Let $\left(V, \mathcal{P}\right) = \mathcal{J}_{0, 2n_1} \oplus \mathcal{J}_{0, 2n_2}$, where $n_1 > n_2$ and 
Let $L \subset \left(V, \mathcal{P}\right)$ be a semisimple bi-Lagrangian subspace with type \eqref{Eq:TypeTwoBLocks}.  Then its 
 $\operatorname{Aut} \left(V, \mathcal{P}\right)$-orbit  has a  structure of a $\mathbb{K}^{N}$-fibre bundle over a product of $D$ Lagrangian Grassmanians: \begin{equation} \label{Eq:ProjOrb2DistJord} \pi: O_{\max}  \xrightarrow{\mathbb{K}^N} \prod_{k=1}^D \Lambda(1) = \prod_{k=1}^D \mathbb{KP}^1. \end{equation} Here $\pi$ is given by \eqref{Eq:ProdLag2Dist} and $D$ counts the $h_j$ exceeding $n_j/2$.  If $D=0$, then $O_L$ is a point.  \end{theorem}

\begin{remark} We obtain a bundle over  $X_1 \times X_2$, where each $X_j\in \left\{\operatorname{pt}, \Lambda(1) \right\}$, reflecting the two types of isotropic flags in a 2-dimensional symplectic space: $\left\{0 \right\}$ and  $\left\{0 \right\} \subset U$, where $\dim U =1$. 
\end{remark}

\begin{proof}[Proof of Theorem~\ref{T:2DistJordTopolMaxOrb}] While a proof analogous to Theorems~\ref{T:EqualJordTopolOrb} and \ref{T:GenJordTopolMaxOrb} exists, we outline an alternative approach. Details are omitted for conciseness. To simplify, assume $h_1 = n_1$ or $h_2 = n_2$ (otherwise, apply bi-Poisson reduction w.r.t. $\operatorname{Ker} P$). If both $h_j = n_j$, apply Theorem~\ref{T:GenJordTopolMaxOrb}. We only consider the case $D=2$ (other cases are similar). A semisimple bi-Lagrangian subspace has the form \begin{equation} \label{Eq:GenSpanL2Dist} \begin{gathered} L = \operatorname{Span} \left(u_1, \dots, P^{h_1 - 1} u_1, \quad v_1, \dots, P^{n_1 - h_1 - 1} v_1, \right\} \oplus \\ \oplus \operatorname{Span} \left\{ u_2, \dots, P^{h_2 - 1} u_2, \quad v_2, \dots, P^{n_2 - h_2 - 1} v_2 \right) \end{gathered} \end{equation} If $h_j = n_j$, then $v_j =0$. If $h_j \not = n_j$, then  \[ \operatorname{Span}\left\{v_j,\dots, P^{n_2 - h_2 - 1}v_j \right\} = \operatorname{Ker} P^{n_j - h_j} \cap \operatorname{Span}\left\{ u_1, \dots, P^{h_1 - 1} u_1 , u_2, \dots, P^{h_2 - 1} u_2\right\}^{\perp}. \]
In short, $L$ is determined by $u_1$ and $u_2$. The orbit $O_L$ is a homogeneous space and the induced action on the base $\Lambda(1) \times \Lambda(1)$ is transitive. Hence, the projection $\pi$, given by \eqref{Eq:ProjOrb2DistJord}, is surjective and defines a fiber bundle. It remains to proof that the fibers $F\approx \mathbb{K}^N$. We show it by describing possible $u_1$ and $u_2$.

Fix a standard basis \eqref{Eq:StandBasisTwoBlocks} such that \[P^{h_1-1}u_1 = e_{n_1}, \qquad P^{h_2-1}u_2 = e_{n_2}. \] The vectors $u_1$ and $u_2$ must belong the $\operatorname{Aut}(V,\mathcal{P})$-orbits of $e_{n_1 -h_1 + 1}$ and $\hat{e}_{n_2 - h_2 + 1}$ respectively. By Theorem~\ref{T:VectOrbits}, they have the form \[ u_1 = e_{n_1 -h_1 +1} + w_1 + \hat{w}_1, \qquad u_2 = \hat{e}_{n_2 -h_2 +2} + w_2 + \hat{w}_2,\] where $w_j$ and $\hat{w}_j$ are combinations of $e_i, f_j$ and $\hat{e}_i, \hat{f}_j$ respectively such that  \[ \operatorname{height}(w_1) < h_1, \quad \operatorname{height}(\hat{w}_1) \leq n_2 - n_1 + h_1, \quad \operatorname{height}(w_2) \leq h_2, \quad \operatorname{height}(\hat{w}_2) < h_2. \] Four cases arise from \eqref{Eq:FourDelta}.

\begin{enumerate} 

\item If $h_1 \geq n_1 -h_1 \geq h_2 \geq n_2 - h_2$ (first case in \eqref{Eq:FourDelta}), perform bi-Poisson reduction w.r.t. $\operatorname{Im} P^{h_1}$. By  Lemma~\ref{L:ImageAut} after reduction the automorphism group $\operatorname{Aut} \left(V, \mathcal{P}\right)$ acts independently on each Jordan block, leading to a product orbit structure. 

\item If $h_1 \geq h_2 \geq n_2 - h_2 \geq n_1 - h_1$ (third case in \eqref{Eq:FourDelta}), then $h_1 = n_1$. The generators have the form \[ u_1 = e_1 + \dots, \quad u_2 = \hat{e}_{n_2-h_2+1} + \dots \quad v_2 = \hat{f}_{n_2 - h_2} + \dots\] By choosing a basis in $L$ we can make \[ u_1 = e_{1} + \sum_{j < h_1} b_j f_j +  \sum_{1 \leq j \leq n_2 - h_2} c_j \hat{e}_j + \sum_{n_2 - h_2 < j \leq n_2} d_j \hat{f}_j. \] We get $h_1 + n_2 - 1$ parameters for $u_1$. Since the vectors $P^i u_j$ are all orthogonal, we can bring $u_2$ to the form \[ u_2 = \hat{e}_{n_2 - h_2 + 1} - P^{n_2 - h_2}\left( \sum_{j < h_1} b_j f_j \right) + \sum_{n_2 - h_2 < j < h_2} \hat{d}_j \hat{f}_j. \] We get additional $2h_2 - n_2 - 1$ parameters for $u_2$. Thus, the fiber $F \approx \mathbb{K}^N$, where \[ N= (2h_1 - n_1) + (2h_2 - n_2) +  (n_2 - 2 (n_1 - h_1)) -2 = h_1 + 2h_2 -2, \] as required.

\item If $h_1 \geq h_2 \geq n_1 - h_1 \geq n_2 - h_2$ (second case in \eqref{Eq:FourDelta}), then $h_2 = n_2$ and  the generators have the form \[ u_1 = e_{n_1-h_1 +1} + \dots, \quad v_1 = f_{n_1 - h_1} + \dots, \quad u_2 = \hat{e}_{1} + \dots\] We can choose \[ u_2 = \hat{e}_{1} + \sum_{n_1 -h_1 < j \leq h_2} \hat{b}_j f_j +  \sum_{j < n_2} d_j \hat{f}_j, \] which gives us $h_1 + 3h_2 - n_1 - n_2 -1$ parameters (note that $h_2 = n_2$). The vector $u_1$ serves as a ``generator'' for a $h_1$-height bi-Lagrangian subspace within the $2n_1$-dimensional Jordan block $\operatorname{Span}\left\{P^j u_2\right\}^{\perp}$. That gives us additional $2h_1 - n_1 - 1$ parameters.  Here $F \approx \mathbb{K}^N$, where \[ N= (2h_1 - n_1) + (2h_2 - n_2) +  (h_2 - (n_1 - h_1)) -2 = 3h_1 + 3h_2 -2n_1 - n_2 -2, \] as required.
  
\item  The remaining case $h_2 \geq h_1 \geq n_1-h_1 \geq n_2 - h_2$ (fourth case in \eqref{Eq:FourDelta}) can be treated similarly.  Again, $n_2 = h_2$ but $u_2$ involves $e_j$ due $h_2 \geq h_1$. Hence, \[ u_2 = \hat{e}_{1} + \sum_{n_1 -h_1 < j \leq \leq n_1-h_2+1} \hat{a}_j e_j  + \sum_{n_1 -h_1 < j \leq h_2} \hat{b}_j f_j +  \sum_{j < n_2} d_j \hat{f}_j \] and we get additional $h_2 - h_1$ parameters. In this case  we get $F \approx \mathbb{K}^N$, where \[ N= (2h_1 - n_1) + (2h_2 - n_2) +  (2h_2 - n_1) -2 = 2h_1 + 4h_2 -2n_1 - n_2 -2, \] as required. \end{enumerate}

Theorem~\ref{T:2DistJordTopolMaxOrb} is proved. \end{proof}

\subsection{Indecomposable bi-Lagrangian subspaces} \label{SubS:TwoDistIndecom}
 
We now investigate bi-Lagrangian subspaces in $\mathcal{J}_{0, 2n_1} \oplus \mathcal{J}_{0, 2n_2}$, which cannot be decomposed into independent bi-Lagrangian subspaces within each Jordan block. To begin, we provide their general description.

\begin{theorem} \label{T:CanonIndecomp2Jord}   Let $(V, \mathcal{P})\approx \mathcal{J}_{0, 2n_1} \oplus \mathcal{J}_{0, 2n_2}$, where $n_1 > n_2$, and $P$ be the nilpotent recursion operator.  Any indecomposable bi-Lagrangian subspace $L \subset (V, \mathcal{P})$ admits a basis consisting of vectors $u, v \in V$,  together with an equal number of their images under $P$ and a subset of some standard basis \eqref{Eq:StandBasisTwoBlocks}: \begin{equation} \label{Eq:CanonFormIndecompGen} \begin{gathered} L = \operatorname{Span} \left\{u, Pu, \dots, P^{r-1}u, \quad v, Pv, \dots, P^{r-1}v \right\} \oplus \\ \oplus \operatorname{Span} \left\{e_{n_1}, \dots, e_{n_1 - p_1 +1},\quad f_1, \dots, f_{q_1}, \quad \hat{e}_{n_2}, \dots, \hat{e}_{n_2 - p_2 +1},\quad \hat{f}_1, \dots, \hat{f}_{q_2}\right\}.   \end{gathered} \end{equation} \end{theorem}

The canonical form from Theorem~\ref{T:CanonIndecomp2Jord} can be further refined by specifying the choices for vectors $u, v$ and the parameters $r, p_1, q_1, p_2, q_2$. We describe them in Theorems~\ref{T:CanonIndecompTypeI} and \ref{T:CanonIndecompTypeII} below. Consider two numbers \begin{equation} \label{Eq:HeightUp}   h = \max_{0 \leq m \leq n_1} \left\{L \subseteq \operatorname{Ker} P^m\right\}, \qquad d = \max_{0 \leq m \leq n_1} \left\{ L \subseteq \operatorname{Im} P^m\right\}. \end{equation} We divide indecomposable $L$ into $2$ types\footnote{Alternative divisions exist. For instance, in Section~\ref{SubS:TypeIIStop} we introduce subtype "Type II-S".  We chose a simpler classification scheme for exposition.}.

\begin{definition} \label{D:TypesIndecomp} Indecomposable bi-Lagrangian subspaces $L \subset \mathcal{J}_{0, 2n_1} \oplus \mathcal{J}_{0, 2n_2}$ can be categorized into two distinct types: 

\begin{itemize}

\item \textbf{Type I}. There is $v \in L - \operatorname{Ker}P^{h-1}$ such that $v \in \operatorname{Im}P^{d+1}$.

\item \textbf{Type II}. Any  $v \in L - \operatorname{Ker}P^{h-1}$ satisfies $v \not \in \operatorname{Im}P^{d+1}$.

\end{itemize}

\end{definition}

Obviously, $h, d$ are invariants of $L$ and bi-Lagrangian subspaces from distinct types are non-isomorphic. We begin by establishing a useful result concerning Jordan block extraction in Section~\ref{SubS:ExtractOneJordVer2}, followed by a description of canonical forms and orbit topology for Types I and II in subsequent sections.

\subsubsection{Extracting a Jordan block} \label{SubS:ExtractOneJordVer2}

The next statement generalizes Theorem~\ref{T:ExtractMaxBlockGen}. Recall that we described  $\operatorname{Aut}(V,\mathcal{P})$-orbits of vectors in Section~\ref{SubS:OrbitsVect}.

\begin{theorem} \label{T:ExtractOneJord} Consider a bi-Poisson vector subspace $(V, \mathcal{P}) = \bigoplus_{i=1}^N \mathcal{J}_{0, 2n_i}$, where $n_1 \geq n_2 \geq \dots \geq n_N$. Let $e^i_j, f^i_j$ be its standard basis, where $i=1,\dots, N, j=1,\dots, n_{i}$, and $L \subset (V, \mathcal{P})$ be a bi-Lagrangian subspace. Assume that there exists a vector $v$ such that 

\begin{enumerate}

\item $v$ has maximal possible height: \[ \operatorname{height}(v) = \operatorname{height}(L) = h.\]

\item The $\operatorname{Aut}(V,\mathcal{P})$-orbit of $v$ contains a basis vector $e^{i_0}_{j_0}$.

\end{enumerate}

Then there exists a decomposition of $(V, \mathcal{P})$ and $L$: \begin{equation} \label{Eq:Decomp2CompBiLagrVer2} (V, \mathcal{P}) = (V_1,\mathcal{P}_1) \oplus (V_2, \mathcal{P}_2), \quad L = L_1 \oplus L_2, \quad L_i = L \cap V_i, \end{equation} such that $(V_1,\mathcal{P}_1) =\mathcal{J}_{0, 2n_{i_0}}$ and $\operatorname{height}(L_1) = h$.\end{theorem}

In Figure~\eqref{Eq:ExtrEqVer2} we illustrate Theorem~\ref{T:ExtractOneJord} for $\mathcal{J}_{0, 8} \oplus\mathcal{J}_{0, 6} \oplus \mathcal{J}_{0, 4}$. Cells labeled N indicate $\operatorname{height}(L) = 2$ the shaded cell represents the vector $v = e^{2}_2$.

\begin{equation} \label{Eq:ExtrEqVer2}
  \begin{tabular}{|c|c||c|c||c|c|} 
  \cline{1-2} N & N & \multicolumn{1}{c}{}& \multicolumn{1}{c}{} & \multicolumn{1}{c}{} & \multicolumn{1}{c}{} \\
 \cline{1-4} N & N & N & N & \multicolumn{1}{c}{} & \multicolumn{1}{c}{} \\
     \hline  &  & {\cellcolor{gray!25} }  & & &  \\ 
   \hline  &   &  &  & &  \\ \hline
  \end{tabular} 
  %%%
  \qquad \to \qquad 
   \begin{tabular}{|c|c|} 
  \multicolumn{1}{c}{}& \multicolumn{1}{c}{}  \\
 \hline N & N \\
     \hline  {\cellcolor{gray!25} }  &  \\ 
   \hline   {\cellcolor{gray!25} }  & {\cellcolor{gray!25} }   \\ \hline
  \end{tabular} \oplus  \begin{tabular}{|c|c||c|c|} 
  \cline{1-2} N & N & \multicolumn{1}{c}{} & \multicolumn{1}{c}{} \\
 \cline{1-2} N & N &  \multicolumn{1}{c}{} & \multicolumn{1}{c}{} \\
     \hline  &   & &  \\ 
   \hline  &   & &  \\ \hline
  \end{tabular} 
\end{equation}

\begin{proof}[Proof of Theorem~\ref{T:ExtractOneJord}]  Consider a standard basis, where $v = e^{i_0}_{j_0}$. Define $V_1$ as the $i_0$-th Jordan block and $V_2$ as the sum of all other Jordan blocks in that basis. Since $\operatorname{height}(L) = h$ we have \[ \bigoplus_{i=1}^N \mathcal{J}_{0, 2n_i}^{ \leq n_i -h} \subset L \subset \bigoplus_{i=1}^N \mathcal{J}_{0, 2n_i}^{\leq h}\] and $j_0 = n_{i_0} - h +1$.  Since $L$ is $P$-invariant, $e^{i_0}_j\in L$, for all $j \geq j_0$. Therefore, $L$ contains the subspace \[ U =\operatorname{Span} \left\{e^{i_0}_{n_{i_0} - h +1}, \dots, e^{i_0}_{n_{i_0}}, f^{i_0}_{1}, \dots, f^{i_0}_{n_{i_0} - h} \right\}.\] Using the bi-Poisson reduction (Theorem~\ref{T:BiPoissReduction}) for $U = L_1$ we get that $L = L_1 \oplus L_2$, where $L_2$ is a bi-Lagrangian subspace of $(V_2, \mathcal{P}_2)$. We got the required decomposition \eqref{Eq:Decomp2CompBiLagrVer2}. Theorem~\ref{T:ExtractOneJord}  is proved. \end{proof}

\subsubsection{Canonical form. Type I} 

\begin{theorem} \label{T:CanonIndecompTypeI} For any indecomposable bi-Lagrangian subspaces $L \subset \mathcal{J}_{0, 2n_1} \oplus \mathcal{J}_{0, 2n_2}$ of type I there exists a standard basis \eqref{Eq:StandBasisTwoBlocks} such that $L$ has the form \eqref{Eq:CanonFormIndecompGen}, where \begin{equation} \label{Eq:TypeIParam} \begin{gathered}  u = e_{n_1 - h + 1} + \hat{e}_{n_2 - d - r +1}, \qquad v =  f_{n_1 - h + r} - \hat{f}_{n_2-d}, \\  p_1 = h - r, \qquad q_1 = n_1 - h, \qquad p_2 =  d \qquad q_2 = n_2 - d - r. \end{gathered} \end{equation}  The parameters $h, d$ and $r$  are uniquely defined, with the restrictions $r>0$ and \begin{equation} \label{Eq:CondIndecTypeI} p_1 \geq q_1 > q_2 > p_2.\end{equation} \end{theorem}

\begin{remark} We can rewrite \eqref{Eq:CondIndecTypeI} as \[ \max\left(0, (n_2 - d) - (n_1 - h)\right) < r \leq \min\left(2h - n_1, n_2 - 2d - 1\right).\] The parameter $r$ exists only if \[\frac{n_1}{2} < h \leq n_1, \qquad 0 \leq d < \frac{n_2 - 1}{2}, \qquad n_2 <  h+ d < n_1 - 1.\]  \end{remark}

\begin{remark} The subspace $L$ from Theorem~\ref{T:CanonIndecompTypeI} can also be expressed as \begin{equation} \label{Eq:TypeISimpleForm} L = \operatorname{Span} \left\{u, \dots P^{h-1}u, \quad v, \dots, P^{n_1 - h + r-1}v \right\} + \operatorname{Im}P^{h} + \operatorname{Ker} P^d.\end{equation} \end{remark}

An example of a Type I indecomposable subspace of $\mathcal{J}_{0, 18} \oplus \mathcal{J}_{0, 10}$ is visualized in Figure~\eqref{Eq:TypeIIndecomp}. The parameters are \[ p_1 = 4, \quad q_1 = 3, \quad p_2 = 1, \quad q_2 = 2, \quad r = 2.\] Jordan blocks are depicted as Young-like diagrams, a bi-Lagrangian subspace is depicted using shaded and labeled cells:

\begin{itemize}
    \item Bottom shaded column heights: $p_1, q_1, p_2, q_2$ (left to right).
    
    \item The labeled cells $P1, P2, M1$, and $M2$ represent specific linear combinations of basis vectors. There are a total of $4r$ labeled cell groups, with the $Pj$ cells corresponding to the vectors $u, \dots, P^{r-1}u$, and the $Mj$ cells corresponding to the vectors $v, \dots, P^{r-1}v$. 
\end{itemize}  

\begin{equation} \label{Eq:TypeIIndecomp}
 \begin{tabular}{|c|c||c|c|} 
   \cline{1-2}  & &   \multicolumn{1}{c}{} & \multicolumn{1}{c}{}   \\
     \cline{1-2}  & &   \multicolumn{1}{c}{} & \multicolumn{1}{c}{}   \\
     \cline{1-2} & &   \multicolumn{1}{c}{} & \multicolumn{1}{c}{}   \\
\cline{1-2} P1 & &   \multicolumn{1}{c}{} & \multicolumn{1}{c}{}   \\
  \hline P2 & M1 &  &    \\
   \hline {\cellcolor{gray!25} }  & M2 &    &  M1 \\ 
   \hline {\cellcolor{gray!25} } &  {\cellcolor{gray!25} }    &   P1 & M2  \\
     \hline {\cellcolor{gray!25} }  & {\cellcolor{gray!25} } &   P2  &  {\cellcolor{gray!25} }  \\ 
   \hline {\cellcolor{gray!25} } &  {\cellcolor{gray!25} }    &   {\cellcolor{gray!25} }  & {\cellcolor{gray!25} }   \\ \hline
  \end{tabular}
\end{equation}

\begin{proof}[Proof of Theorem~\ref{T:CanonIndecompTypeI}] The proof strategy is centered on constructing a Jordan basis within $L$ by successively selecting Jordan chains and applying the bi-Poisson reduction. The proof is in several steps:

\begin{enumerate}
    \item \textit{We can assume that $d=0$}, otherwise \[ \operatorname{Ker} P^d \subset L \subset \operatorname{Im} P^d \] and we can perform bi-Poisson reduction w.r.t. $\operatorname{Ker}P^{d}$. We can bring the original subspace to the canonical form using Corollary~\ref{Cor:ImageAutIso}. 

    \item Any vector $u \in L$ can be expressed as \begin{equation} \label{Eq:FormUEq1} u = e_{n_1 - x + 1} + \hat{e}_{n_2 - r +1}\end{equation} in a suitable basis. \textit{We claim that for any $u \in L$ with \begin{equation} \label{Eq:HeightEq} \operatorname{height}(u) = \operatorname{height}(L) = h\end{equation} the parameter}  \[ \label{} x =h > r. \] Otherwise $r = h \geq x$ and by Theorem~\ref{T:VectOrbits} the automorphism orbit of $u$ contains $\hat{e}_{n_2 - r +1}$. Theorem~\ref{T:ExtractOneJord} implies that $L$ is decomposable, leading to a contradiction.

    \item \label{Step:SmallR} Choose a vector $u \in L$ from the previous step \textit{such that the index $r$ is as small as possible.} Definition of $h$ (Equation~\eqref{Eq:HeightUp}) implies \begin{equation} \label{Eq:ImLKerh}\operatorname{Im} P^h \subset L \subset \operatorname{Ker} P^h.\end{equation} Hence, $L$ contains the subspace \[ U = \operatorname{Span} \left\{ e_{n_1 - h + 1} + \hat{e}_{n_2 - r +1}, \dots,  e_{n_1 - h + r} + \hat{e}_{n_2} , \quad e_{n_1 - h + r +1}, \dots, e_{n_1}, \quad  f_{1}, \dots, f_{n_1 - h}\right\}  \] due to $P$-invariance. Moreover, $U \subset L \subset U^{\perp}$, where \[ U^{\perp} = U \oplus \operatorname{Span} \left\{\hat{e}_1, \dots, \hat{e}_{n_2}, w, Pw, \dots, P^{n_2 - 1}w\right\}, \qquad w = f_{n_1 - h + r} - \hat{f}_{n_2}. \]

    \item \label{Step:RLessN2} \textit{For type I subspaces $r < n_2$}. The vectors that satisfy \eqref{Eq:HeightEq} form the set $L - \operatorname{Ker}P^{h-1}$. If all these vectors satisfy $r = n_2$, then none of them belongs to $\operatorname{Im}P$, which is contradictory with Definition~\ref{D:TypesIndecomp}.
    
    \item \textit{The subspace $L/ U$ in $U^{\perp}/U$ has height $n_2$}. We have $h < n_1$, otherwise $L$ is decomposable by Theorem~\ref{T:ExtractOneJord}. By \eqref{Eq:ImLKerh}, \[ \operatorname{Im}P^{n_1 -1} = \operatorname{Span}\left\{ e_{n_1}, f_1\right\} \subset L.\] Similarly, if $\operatorname{height} (L/U) < n_2$, then \[\operatorname{Span}\left\{\hat{e}_{n_2}, \hat{f}_1 - f_{n_1 - h - n_2 + r +1} \right\} \subset L/U.\] Here we put $f_m =0$ for $m < 0$. By Step~\ref{Step:RLessN2}, $r < n_2$. Thus, $f_{n_1 - h - n_2 + r +1} \in U \subset L$, as its height does not exceed $n_1 - h$. We get that \[ \operatorname{Span} \left\{e_{n_1}, f_1, \hat{e}_{n_2}, \hat{f}_1 \right\}  =\operatorname{Ker} P  \subset L.\] Hence, $L \subset \operatorname{Im}P$ and  $d > 0$, leading to a contradiction.

    \item By the preceding step and Theorem~\ref{T:BiLagr_One_Jordan_Canonical_Form},  $L/ U$ has the form \[ \operatorname{Span} \left\{v, Pv,\dots, P^{n_2-1}v \right\}, \qquad v = \sum_{j=1}^{n_2} a_j \hat{e}_j + \sum_{j=1}^{n_2} b_{n_2 - j + 1} P^{j-1} w, \quad (a_{1}, b_{n_2}) \not = (0,0). \] \textit{We claim that $b_{n_2} \not = 0$}. Otherwise the vector \[ u - \frac{1}{a_1} P^{n_2 - r} v \in L \] has the form \eqref{Eq:FormUEq1} with a smaller value of the parameter $r$. We get a contradiction.

    \item \label{Step:ChangeFTypeI} \textit{The subspace $L$ can be reduced to the canonical form, given by    \eqref{Eq:CanonFormIndecompGen} and \eqref{Eq:TypeIParam}.} We need to eliminate all $a_j$ and $b_j$ except for $b_{n_2}$. By the preceding step, $b_{n_2} \not = 0$.  Substituting $v$ with a linear combination of $P^j v$ we get $b_j = 0$ for all $j < n_2$. We can make $a_j = 0$ by changing the basis vectors $\hat{f}_j$ as follows: \[\hat{f}_{n_2}' = \hat{f}_{n_2} + \sum_{j=1}^{n_2} a_j \hat{e}_j, \qquad \hat{f}_{n_2 -j}' = P^j \left(\hat{f}_{n_2}'\right). \] The new basis is also standard. In the new basis \[ L = U \oplus \operatorname{Span} \left\{v, Pv,\dots, P^{n_2-1}v \right\}, \qquad v =  f_{n_1-h +r} + \hat{f}_{n_2}'. \] It is easy to check that \eqref{Eq:CanonFormIndecompGen} and \eqref{Eq:TypeIParam} are satisfied. 

\item \textit{The parameters $h, d$ and $r$ satisfy \eqref{Eq:CondIndecTypeI}.} 

\begin{itemize}

\item $p_1 \geq q_1$ since $\operatorname{height}(u) \geq \operatorname{height}(v)$ by \eqref{Eq:HeightEq}. 

\item $q_1 > q_2$, otherwise $\operatorname{height}(f_{n_1 - h+r}) \leq n_2$  and by Theorem~\ref{T:ExtractOneJord} $L$ is decomposable. 

\item $q_2 > p_1$, since $p_1 = u = 0$ and $r < n_2$.

\end{itemize}

\item \textit{All $h, d, r$ satisfying \eqref{Eq:CondIndecTypeI} are possible.} Easily verified by computation.

\item \textit{The parameters $h, d$ and $r$  are uniquely defined}. For $h$ and $d$ it is obvious. $r$ was uniquely determined in Step~\ref{Step:SmallR}. 
    
\end{enumerate}

Theorem~\ref{T:CanonIndecompTypeI} is proved. \end{proof}

\begin{remark} The Jordan normal form of the restriction $P$ to $L$ is \[ P\bigr|_{L} \sim J(h)\oplus J(n_1 - h + r) \oplus J(n_2 - d - r) \oplus J(d),\]  with the parameter $r$ uniquely specified by this form.
 Perform bi-Poisson reduction w.r.t. \[ U = \mathcal{J}_{0, 2n_1}^{\leq n_1 - h} \oplus \mathcal{J}_{0, 2n_2}^{\leq d}. \] The Jordan normal form of the restriction $P$ to $L/(L\cap U)$ is \[ P\bigr|_{L/(L\cap U)} \sim J(2h - n_1)\oplus J(n_2 - 2d).  \] The type of the indecomposable subspace is determined by this Jordan normal form (cf. proof of Theorem~\ref{T:UniqTypeII}). \end{remark}

\subsubsection{Dimension and topology of orbits. Type I}

Let $L \subset \left(V, \mathcal{P}\right) = \mathcal{J}_{0, 2n_1} \oplus \mathcal{J}_{0, 2n_2}$ be an indecomposable bi-Lagrangian subspace with type I (see Definition~\ref{D:TypesIndecomp}). The orbit of $L$ will be described in a manner analogous to Section~\ref{S:Top2BlocksSemi}. We start with a trivial statement about $P^j(L) \cap \operatorname{Ker}P$. 

\begin{assertion} For the subspace $L$ from Theorem~\ref{T:CanonIndecompTypeI} \[ P^j\left(L\right) \cap \operatorname{Ker}P = \begin{cases}\operatorname{Span}\left\{e_{n_1}, f_1, \hat{e}_{n_2}, \hat{f}_1\right\}  \qquad & 0 \leq j < p_2,\\  \operatorname{Span}\left\{e_{n_1}, f_1, \hat{f}_1\right\}  \qquad & p_2 \leq j < q_2,\\  \operatorname{Span}\left\{e_{n_1}, f_1, \right\}  \qquad & q_2 \leq j < q_1 + r,\\ \operatorname{Span}\left\{e_{n_1}\right\}  \qquad & q_1 + r \leq j < h, \\ \left\{0 \right\} \qquad & h \leq j.  \end{cases} \] \end{assertion}

Consider subspaces $S_j$ as in Section~\ref{S:Top2BlocksSemi}. Put \[\begin{gathered} \hat{L}_1 =  P^{h-1}(L \cap \operatorname{Im}P^{d+1}) \cap \operatorname{Ker} P, \qquad   \hat{L}_2 =  P^{p_2}(L) \cap \operatorname{Ker} P, \\ L_1 = \hat{L}_1 \cap S_1, \qquad L_2 = \hat{L}_2/ S_1.\end{gathered} \]  If $p_1 > q_1$, then $\hat{L}_1 =  P^{h-1}(L) \cap \operatorname{Ker} P$. For $p_1 = q_1$ we replace $L$ with $L \cap \operatorname{Im}P^{d+1}$ to ensure that the resulting subspace $L_1$ is one-dimensional.  The process yields a pair of one-dimensional (Lagrangian) subspaces $L_j \subset S_j$.  Define \begin{equation} \label{Eq:PiIndecTypeI} \pi(L) = (L_1, L_2) \in \Lambda(1) \times \Lambda(1). \end{equation} For instance, in \eqref{Eq:Proj2DistJordTopolTypeI} we visualise  subspaces $L_j \subset S_j$ for $L$ from \eqref{Eq:TypeIIndecomp}.

\begin{equation} \label{Eq:Proj2DistJordTopolTypeI}
  \begin{tabular}{|c|c||c|c|} 
   \cline{1-2}  & &   \multicolumn{1}{c}{} & \multicolumn{1}{c}{}   \\
     \cline{1-2}  & &   \multicolumn{1}{c}{} & \multicolumn{1}{c}{}   \\
     \cline{1-2} & &   \multicolumn{1}{c}{} & \multicolumn{1}{c}{}   \\
\cline{1-2} P1 & &   \multicolumn{1}{c}{} & \multicolumn{1}{c}{}   \\
  \hline P2 & M1 &  &    \\
   \hline {\cellcolor{gray!25} }  & M2 &    &  M1 \\ 
   \hline {\cellcolor{gray!25} } &  {\cellcolor{gray!25} }    &   P1 & M2  \\
     \hline {\cellcolor{gray!25} }  & {\cellcolor{gray!25} } &   P2  &  {\cellcolor{gray!25} }  \\ 
   \hline {\cellcolor{gray!25} } &  {\cellcolor{gray!25} }    &   {\cellcolor{gray!25} }  & {\cellcolor{gray!25} }   \\ \hline
  \end{tabular}  \qquad \to \qquad \begin{tabular}{|c|c||c|c|} 
     \multicolumn{1}{c}{}  & \multicolumn{1}{c}{}  &  \multicolumn{1}{c}{}   &  \multicolumn{1}{c}{}    \\
     \multicolumn{1}{c}{}  & \multicolumn{1}{c}{}  &  \multicolumn{1}{c}{}   &  \multicolumn{1}{c}{}    \\
   \multicolumn{1}{c}{}  & \multicolumn{1}{c}{}  &  \multicolumn{1}{c}{}   &  \multicolumn{1}{c}{}    \\
    \multicolumn{1}{c}{}  & \multicolumn{1}{c}{}  &  \multicolumn{1}{c}{}   &  \multicolumn{1}{c}{}    \\
  \multicolumn{1}{c}{}  & \multicolumn{1}{c}{}  &  \multicolumn{1}{c}{}   &  \multicolumn{1}{c}{}    \\
   \multicolumn{1}{c}{}  & \multicolumn{1}{c}{}  &  \multicolumn{1}{c}{}   &  \multicolumn{1}{c}{}    \\
 \multicolumn{1}{c}{}  & \multicolumn{1}{c}{}  &  \multicolumn{1}{c}{}   &  \multicolumn{1}{c}{} \\
     \multicolumn{1}{c}{}  & \multicolumn{1}{c}{}  &  \multicolumn{1}{c}{}   &  \multicolumn{1}{c}{}   \\
   \hline 1 &  &   & 2 \\ \hline
  \end{tabular}
\end{equation}

\begin{theorem} \label{T:2DistJordTopolMaxOrbIndecTypeI} Let $\left(V, \mathcal{P}\right) = \mathcal{J}_{0, 2n_1} \oplus \mathcal{J}_{0, 2n_2}$, where $n_1 > n_2$ and $L \subset \left(V, \mathcal{P}\right)$ be a Type I indecomposable bi-Lagrangian subspace from Theorem~\ref{T:CanonIndecompTypeI}. Then its 
 $\operatorname{Aut} \left(V, \mathcal{P}\right)$-orbit  has a  structure of a $\mathbb{K}^{N}\times \mathbb{K}^*$-fibre bundle over a product of Lagrangian Grassmanians: \begin{equation} \label{Eq:ProjOrb2DistJordTypeI} \pi: O_{\max}  \xrightarrow{\mathbb{K}^{N}\times \mathbb{K}^*} \Lambda(1) \times \Lambda(1) = \mathbb{KP}^1 \times \mathbb{KP}^1. \end{equation} Here $\pi$ is given by \eqref{Eq:PiIndecTypeI} and \[ N =  2h - n_1 + n_2 - 2d + r - 3.\]  \end{theorem}

\begin{proof}[Proof of Theorem~\ref{T:2DistJordTopolMaxOrbIndecTypeI}] The proof is analogous to Theorem~\ref{T:2DistJordTopolMaxOrb} (it can also be proved similar to Theorems~\ref{T:EqualJordTopolOrb} and \ref{T:GenJordTopolMaxOrb}). To simplify, assume $d=0$ (otherwise, perform bi-Poisson reduction w.r.t. $\operatorname{Ker} P^d$). The orbit $O_L$ is a homogeneous space and by Theorem~\ref{T:CanonIndecompTypeI} the induced action on the base $\Lambda(1) \times \Lambda(1)$ is transitive. Hence, the projection $\pi$, given by \eqref{Eq:PiIndecTypeI}, is surjective and defines a fiber bundle. It remains to proof that the fibers $F\approx \mathbb{K}^{N}\times \mathbb{K}^*$. We show it by describing possible $u$ and $v$ ($L$ is determined by $u$ and $v$ according to \eqref{Eq:TypeISimpleForm}). Fix a standard basis \eqref{Eq:StandBasisTwoBlocks} such that \[P^{h_1-1}u = e_{n_1}, \qquad P^{n_2 -1}v = f_{1} + w, \] where $w$ is a combination of $e_i, f_j$. The vectors $u$ and $v$ must belong the $\operatorname{Aut}(V,\mathcal{P})$-orbits of $u = e_{n_1 - h + 1} + \hat{e}_{n_2 - r +1}$ and $v =  f_{n_1 - h + r} - \hat{f}_{n_2}$ respectively. By Theorem~\ref{T:VectOrbits}, they have the form \[ u = e_{n_1 -h +1} + w_1 + \hat{w}_1, \qquad v = \hat{f}_{n_2 } + w_2 + \hat{w}_2,\] where $w_j$ and $\hat{w}_j$ are combinations of $e_i, f_j$ and $\hat{e}_i, \hat{f}_j$ respectively such that  \[ \operatorname{height}(w_1) < h, \quad \operatorname{height}(\hat{w}_1) \leq  r, \quad \operatorname{height}(w_2) \leq n_1 - h + r, \quad \operatorname{height}(\hat{w}_2) < n_2. \] By replacing $u$ and $v$ with suitable linear combinations of $P^j u, P^j v$ and $\operatorname{Im} P^{h}$ we make \[ \begin{gathered} u = e_{n_1 - h + 1} + \sum_{j = n_1 - h +1}^{h - 1} b_j f_j +  \sum_{j=1}^{r} c_{n_2 - j + 1} \hat{e}_{n_2 - j +1}, \\ v = \hat{f}_{n_2} + \sum_{j = n_2 - h+ 1}^{n_2 - h+ r} \hat{b}_j f_j +  \sum_{j=1}^{n_2-1} \hat{c}_{n_2 - j + 1} \hat{e}_{n_2 - j +1}. \end{gathered} \] where $c_{n_2 - r +1} \not = 0$. Thus, for the vector $u$ we get $2h - n_1 + r - 1$ parameters: one parameter $c_{n_2 - r + 1} \in \mathbb{K}^*$ and all other $b_i, c_{n_2 - j + 1} \in \mathbb{K}$. Since all $P^j u$ and $P^j v$ are orthogonal, \[ \hat{b}_{n_2 - h + x} = c_{n_2 - x + 1}, \qquad x = 1, \dots r. \] We get another  $n_2- 1$ parameters $\hat{c}_{n_2 - j + 1} \in \mathbb{K}$ for $v$. Hence, $F \approx \mathbb{K}^N \times \mathbb{K}^*$, where \[ N= 2h - n_1 + n_2  + r - 3, \] as required. Theorem~\ref{T:2DistJordTopolMaxOrbIndecTypeI} is proved. \end{proof}

\subsubsection{Canonical form. Type II}

\begin{theorem}  \label{T:CanonIndecompTypeII} For any indecomposable bi-Lagrangian subspaces $L \subset \mathcal{J}_{0, 2n_1} \oplus \mathcal{J}_{0, 2n_2}$ of type II there exists a standard basis \eqref{Eq:StandBasisTwoBlocks} such that $L$ has the form \eqref{Eq:CanonFormIndecompGen}, where \begin{equation} \label{Eq:TypeIIParam} \begin{gathered}  u = e_{n_1 - h + 1} + \hat{e}_{d +1}, \qquad v =  f_{n_1 - h + r} - \hat{f}_{d+r} + \delta \cdot \hat{e}_{n_2 - z + 1}, \\  p_1 = h - r, \qquad q_1 = n_1 - h, \qquad p_2 =   n_2 - d - r \qquad q_2 = d. \end{gathered} \end{equation} We require that $r > 0$ and \begin{equation} \label{Eq:TypeIICondpq1} p_1 \geq q_1 > q_2, \qquad p_1 > p_2 \geq q_2 \end{equation} with an additional condition  \begin{equation} \label{Eq:TypeIICondpq2}  p_1 = q_1, \quad \Rightarrow \quad p_2 = q_2.\end{equation} The parameter $ \delta \in \left\{0, 1 \right\}$. If $\delta = 1$, then the parameter $z$ satisfies \begin{equation} \label{Eq:TypeIICondZ} \max(p_2 - r, q_2) < z-r < \min(p_2, q_1).\end{equation} \end{theorem}

\begin{remark} In Theorem~\ref{T:CanonIndecompTypeII} we can rewrite restriction on $r$  as \[ 0 < r \leq \min\left(2h - n_1, n_2 - 2d\right),\] and the restrictions on $h,d$ as \[\frac{n_1}{2} < h \leq n_1, \qquad 0 \leq d < \frac{n_2}{2}, \qquad n_2 <  h+ d < n_1.\] The condition on $z$ is \begin{equation} \label{Eq:CondZInt} z \in (q_1 + r, q_2 +r) \cap (p_2, p_2 +r). \end{equation} \end{remark}

\begin{remark} The subspace $L$ from Theorem~\ref{T:CanonIndecompTypeII} can also be expressed as \begin{equation} \label{Eq:TypeIISimpleForm} L = \operatorname{Span} \left\{u, \dots P^{h-1}u, \quad v, \dots, P^{n_1 - h + r-1}v \right\} + \operatorname{Im}P^{h} + \operatorname{Ker} P^d.\end{equation} \end{remark}

Two examples of a Type II indecomposable subspace of $\mathcal{J}_{0, 18} \oplus \mathcal{J}_{0, 12}$ is visualized in Figure~\eqref{Eq:TypeIIndecomp}. The parameters associated with the left-hand side example are \[ p_1 = 5, \quad q_1 = 2, \quad p_2 = 3, \quad q_2 = 1, \quad r = 2, \quad \delta = 0.\] For the right-hand side the parameters are \[ p_1 = 4, \quad q_1 = 3, \quad p_2 = 3, \quad q_2 = 1, \quad r = 2, \quad \delta = 1.\] Jordan blocks are depicted as Young-like diagrams, a bi-Lagrangian subspace is depicted using shaded and labeled cells:

\begin{itemize}
    \item Bottom shaded column heights: $p_1, q_1, p_2, q_2$ (left to right).
    
    \item The labeled cells $P1, P2, M1$, and $M2$ represent specific linear combinations of basis vectors. There are a total of $4r$ labeled cell groups, with the $Pj$ cells corresponding to the vectors $u, \dots, P^{r-1}u$, and the $Mj$ cells corresponding to the vectors $v, \dots, P^{r-1}v$. 
    
\item Due to \eqref{Eq:CondZInt}, the M1 entry in the third column must be located within one of the $Pj$ cells and its height must fall between the heights of other $M1$ cells.

\end{itemize}  

\begin{equation} \label{Eq:TypeIIIndecomp} \begin{tabular}{|c|c||c|c|} 
   \cline{1-2}  & &   \multicolumn{1}{c}{} & \multicolumn{1}{c}{}   \\
     \cline{1-2}  & &   \multicolumn{1}{c}{} & \multicolumn{1}{c}{}   \\
     \cline{1-2} P1 & &   \multicolumn{1}{c}{} & \multicolumn{1}{c}{}   \\
\hline P2 & & &    \\
  \hline {\cellcolor{gray!25} } &  & P1 &    \\
   \hline {\cellcolor{gray!25} }  & M1 &  P2  &   \\ 
   \hline {\cellcolor{gray!25} } &  M2   &    {\cellcolor{gray!25} }  & M1  \\
     \hline {\cellcolor{gray!25} }  & {\cellcolor{gray!25} } &   {\cellcolor{gray!25} }   &  M2  \\ 
   \hline {\cellcolor{gray!25} } &  {\cellcolor{gray!25} }    &   {\cellcolor{gray!25} }  & {\cellcolor{gray!25} }   \\ \hline
  \end{tabular} \qquad 
 \begin{tabular}{|c|c||c|c|} 
   \cline{1-2}  & &   \multicolumn{1}{c}{} & \multicolumn{1}{c}{}   \\
     \cline{1-2}  & &   \multicolumn{1}{c}{} & \multicolumn{1}{c}{}   \\
     \cline{1-2} & &   \multicolumn{1}{c}{} & \multicolumn{1}{c}{}   \\
\hline P1 & & &    \\
  \hline P2 & M1 & P1 &    \\
   \hline {\cellcolor{gray!25} }  & M2 &  P2 M1  &   \\ 
   \hline {\cellcolor{gray!25} } &  {\cellcolor{gray!25} }    &    {\cellcolor{gray!25} }  & M1  \\
     \hline {\cellcolor{gray!25} }  & {\cellcolor{gray!25} } &   {\cellcolor{gray!25} }   &  M2  \\ 
   \hline {\cellcolor{gray!25} } &  {\cellcolor{gray!25} }    &   {\cellcolor{gray!25} }  & {\cellcolor{gray!25} }   \\ \hline
  \end{tabular}
\end{equation}

\begin{proof}[Proof of Theorem~\ref{T:CanonIndecompTypeII}] Proof follows Theorem~\ref{T:CanonIndecompTypeI}.  up to Step~\ref{Step:RLessN2}. For Type II subspaces, by Definition~\ref{D:TypesIndecomp}, Formula~\eqref{Eq:FormUEq1} takes the form \[u = e_{n_1 - h + 1} + \hat{e}_{1}.\] We get that $U \subset L \subset U^{\perp}$,  where \[ \begin{gathered} U = \operatorname{Span} \left\{ e_{n_1 - h + 1} + \hat{e}_{1}, \dots,  e_{n_1 - h + n_2} + \hat{e}_{n_2} , \quad e_{n_1 - h + n_2 +1}, \dots, e_{n_1}, \quad  f_{1}, \dots, f_{n_1 - h}\right\}, \\ U^{\perp} = U \oplus \operatorname{Span} \left\{\hat{e}_1, \dots, \hat{e}_{n_2}, w, Pw, \dots, P^{n_2 - 1}w\right\}, \qquad w = f_{n_1 - h + n_2} - \hat{f}_{n_2}. \end{gathered} \] By Theorem~\ref{T:BiLagr_One_Jordan_Canonical_Form}, if $L/ U$ has height $z$, then it has the form \[ \begin{gathered} \operatorname{Span} \left\{v, Pv,\dots, P^{z-1}v \right\} + \operatorname{Ker} \left(P\bigr|_{U^{\perp}/U}\right)^{n_2 - z}, \\ v = \sum_{j=n_2 - z + 1}^{z} a_j \hat{e}_j + \sum_{j=n_2 - z + 1}^{z} b_{n_2 - j + 1} P^{j-1} w, \quad (a_{n_2 - z + 1}, b_{z}) \not = (0,0). \end{gathered} \] There are several cases:

\begin{enumerate}[label=(\alph{enumi})]
    
    \item Assume that all $b_j = 0$. Then \[ L/ U = \operatorname{Span} \left\{\hat{e}_{n_2 - z +1}, \dots, \hat{e}_{n_2}, \quad \tilde{v}, \dots, P^{n_2 - z- 1}\tilde{v} \right\}, \qquad \tilde{v} = f_{n_1 - h + n_2 - z} - \hat{f}_{n_2 - z}\] and $L$ takes the form \eqref{Eq:TypeIIParam} with $r = n_2 - z$ and $\delta = 0$.

    \item Assume that $b_z \not = 0$. Similar to Step~\ref{Step:ChangeFTypeI} the coefficients $b_j, j < z$, can be eliminated by replacing $v$ with a linear combination of $P^jv$. Subsequently, we can make all $a_j = 0$ by a change of basis:  \[\hat{f}_{n_2}' = \hat{f}_{n_2} + \sum_{j=1}^{2z-n_2} a_{n_2 - z + j} \hat{e}_j, \qquad \hat{f}_{n_2 -j}' = P^j \left(\hat{f}_{n_2}'\right). \] The new basis is also standard. In the new basis \[ L = U \oplus \operatorname{Span} \left\{ \hat{e}_{z + 1}, \dots, \hat{e}_{n_2} \right\} \oplus  \operatorname{Span} \left\{v, Pv,\dots, P^{z-1}v \right\}, \quad v =  f_{n_1-h +z} - \hat{f}_{z}'. \] $L$ takes the form \eqref{Eq:TypeIIParam} with $r = z$ and $\delta = 0$.

    \item \label{Case:CTypeII} Assume that $b_z,\dots, b_{y-1} = 0$ and $b_{y} \not = 0$. In this case $a_{n_2 - z +1} \not = 0$, since $b_z = 0$. Similar to the previous step, the coefficients $a_{j}, j \geq n_2 - y +1$ and $b_j, j < y$, can be eliminated  by replacing $v$ with a linear combination of $P^jv$. Subsequently, we can make the remaining $a_j = 0$ by a change of basis: \[\hat{f}_{n_2}' = \hat{f}_{n_2} + \sum_{j=1}^{z + y -n_2} a_{n_2 - y + j} \hat{e}_j, \qquad \hat{f}_{n_2 -j}' = P^j \left(\hat{f}_{n_2}'\right). \] The new basis is also standard. In the new basis \[ L = U \oplus \operatorname{Span} \left\{ \hat{e}_{z + 1}, \dots, \hat{e}_{n_2} \right\} \oplus  \operatorname{Span} \left\{v, Pv,\dots, P^{y-1}v \right\}, \quad v =  f_{n_1-h +y} - \hat{f}_{y}' + C \cdot \hat{e}_{n_2 - z+1}. \] If $C = 0$, then $L$ takes the form \eqref{Eq:TypeIIParam} with $r = y$ and $\delta = 0$. If $C \not = 0$, then replace \[ u \to \sqrt{C}u = \sqrt{C}e_{n_1 - h + 1} + \sqrt{C}\hat{e}_{1}, \qquad v \to \frac{1}{\sqrt{C}} v =  \frac{1}{\sqrt{C}}f_{n_1-h +y} - \frac{1}{\sqrt{C}}\hat{f}_{y}' + \sqrt{C} \cdot \hat{e}_{n_2 - z+1}.\] The change of basis \[ e_i, f_i, \hat{e}_j \hat{f}_j \to \sqrt{C} e_i, \frac{1}{\sqrt{C} } f_i, \sqrt{C} \hat{e}_j \frac{1}{\sqrt{C} }\hat{f}_j  \] transforms $L$ into the form \eqref{Eq:TypeIIParam} with $r = y$ and $\delta = 1$.
    
\end{enumerate}

The remainder of the proof is divided into several steps.

\begin{enumerate}

\item \textit{The parameters $h, d$ and $r$ satisfy $r> 0$, \eqref{Eq:TypeIICondpq1} and \eqref{Eq:TypeIICondpq2}}. 

\begin{itemize}

\item $r> 0$, otherwise $L$ is decomposable. 

\item $p_1 \geq q_1$ since $\operatorname{height}(u) \geq \operatorname{height}(v)$ by \eqref{Eq:HeightEq}. 

\item $p_1 > p_2$, otherwise by Theorem~\ref{T:ExtractOneJord} $L$ is decomposable. 

\item $q_1 > q_2$, otherwise $\operatorname{height}(f_{n_1-h +z}) \leq z$  and by Theorem~\ref{T:ExtractOneJord} $L$ is decomposable. 

\item $p_2 \geq q_2$, otherwise $u \in \operatorname{Im}P^{d+1}$ and $L$ is not Type II.

\item If $p_1 = q_1$, then $p_2 =q_2$ (i.e. \eqref{Eq:TypeIICondpq2} holds), otherwise there exists $v \in L - \operatorname{Ker}P^{h-1}$ such that $v \in \operatorname{Im}P^{d+1}$.

\end{itemize}

\item  \textit{If $\delta \not = 0$, then the parameter $z$ satisfies \eqref{Eq:TypeIICondZ}}.
\begin{itemize}
    \item $z - r < q_2$, otherwise $v$ belong to the automorphism orbit of $\hat{e}_{n_2 - z+1}$ and $L$ is decomposable. 
    \item $z - r < p_2$, otherwise a linear combination of u and v can be constructed that belongs to both $L - \operatorname{Ker}P^{h-1}$ and $ \operatorname{Im}P^{d+1}$, which contradicts the assertion that L is a Type II subspace.
    
    \item  $z > q_2 + r$ by construction (see Case~\ref{Case:CTypeII}).

    \item $z > p_2$, otherwise $\hat{e}_{n_2 - z+1} \in L$ and we can simplify the formulas.
\end{itemize}  

\item \textit{All $h, d, r, z$ and $\delta$ satisfying the conditions of 
Theorem~\ref{T:CanonIndecompTypeII} are possible.} Easily verified by computation.

\end{enumerate}

Theorem~\ref{T:CanonIndecompTypeII} is proved. \end{proof}

\begin{theorem} \label{T:UniqTypeII} The parameters $h, d, r, \delta$ and $z$ in Theorem~\ref{T:CanonIndecompTypeII} are uniquely defined.
\end{theorem}
\begin{proof}[Proof of Theorem~\ref{T:UniqTypeII}]

\begin{enumerate}

\item $h$ and $d$ are given by \eqref{Eq:HeightUp}.

\item By Theorem~\ref{T:CanonIndecompTypeII}, \[  \begin{gathered} L = \operatorname{Span} \left\{u, Pu, \dots, P^{h-1}u, \quad v, Pv, \dots, P^{n_1 - h + r -1}v \right\} \oplus \\ \oplus \operatorname{Span} \left\{\hat{e}_{d + r +1}, \dots, \hat{e}_{n_2}, ,\quad \hat{f}_{1}, \dots, \hat{f}_d\right\}.   \end{gathered} \] The Jordan normal form of the restriction $P$ to $L$ is \[ P\bigr|_{L} \sim J(h)\oplus J(n_1 - h + r) \oplus J(n_2 - d - r) \oplus J(d),\] with the parameter $r$ uniquely specified by this form.

\item Perform bi-Poisson reduction w.r.t. \[ U = \mathcal{J}_{0, 2n_1}^{\leq n_1 - h} \oplus \mathcal{J}_{0, 2n_2}^{\leq d}. \] The Jordan normal form of the restriction $P$ to $L/(L\cap U)$ is \[ P\bigr|_{L/(L\cap U)} \sim J(2h-n_1)\oplus J(x) \oplus J(n_2 - 2d - x), \quad x = \begin{cases} r, \quad &\mbox{if } \delta = 0, \\ z-d, \quad &\mbox{if } \delta = 1. \end{cases} \] Note that in the case $\delta =0$ and $n_2 - 2d = r$ this Jordan normal form becomes \[P\bigr|_{L/(L\cap U)} \sim J(2h-n_1)\oplus J(n_2 - 2d). \] From this form, the values of the parameters $\delta$ and $z$ can be inferred.
\end{enumerate}

Theorem~\ref{T:UniqTypeII} is proved. \end{proof}

\subsubsection{Dimension and topology of orbits. Type II. Special case} \label{SubS:TypeIIStop}

We begin by examining the orbit topology for specific case of Type II indecomposable subspaces, with the general case to be explored in Section~\ref{SubS:TypeIItopGen}.

\begin{definition} \label{Def:TypeIIS}  A Type II indecomposable bi-Lagrangian subspace  $L \subset \left(V, \mathcal{P}\right) = \mathcal{J}_{0, 2n_1} \oplus \mathcal{J}_{0, 2n_2}$ is categorized as \textbf{Type II-S} if the parameters introduced in Theorem~\ref{T:CanonIndecompTypeII} satisfy \[ p_1 = q_1, \qquad p_2 = q_2.\]
\end{definition}
 
In other words, \[ r = 2h - n_1 = n_2 - 2d, \qquad \delta = 0.\] Note that it is possible only if the half-sizes of Jordan blocks share the same parity: \[ n_1 \equiv n_2 \pmod{2}. \] By Theorem~\ref{T:CanonIndecompTypeII} there exist a canonical basis such that  \begin{equation} \label{Eq:CanonFormTypeIISpecial} \begin{gathered}  L = \operatorname{Span} \left\{e_{\frac{n_1 - r}{2} +1} + \hat{e}_{\frac{n_2 - r}{2} +1}, \dots,  e_{\frac{n_1 + r}{2}} + \hat{e}_{\frac{n_2 + r}{2}}\right\} \oplus \\ \oplus \operatorname{Span} \left\{f_{\frac{n_1 + r}{2}} - \hat{f}_{\frac{n_2 + r}{2}},\dots, f_{\frac{n_1 - r}{2} +1} + \hat{f}_{\frac{n_2 - r}{2} +1} \right\} \oplus \\ \oplus \operatorname{Span} \left\{e_{n_1}, \dots, e_{\frac{n_1 + r}{2} +1},\quad f_1, \dots, f_{\frac{n_1 - r}{2}}, \quad \hat{e}_{n_2}, \dots, \hat{e}_{\frac{n_2 + r}{2} +1},\quad \hat{f}_1, \dots, \hat{f}_{\frac{n_2 - r}{2}}\right\}.   \end{gathered} \end{equation} An example of such subspace for $(V,\mathcal{P}) = \mathcal{J}_{0, 12} \oplus \mathcal{J}_{0, 8}$ is visualized in Figure~\ref{Eq:TypeIIIndecompSpecial}.  
\begin{equation} \label{Eq:TypeIIIndecompSpecial} \begin{tabular}{|c|c||c|c|} 
   \cline{1-2}  & &   \multicolumn{1}{c}{} & \multicolumn{1}{c}{}   \\
     \cline{1-2}  & &   \multicolumn{1}{c}{} & \multicolumn{1}{c}{}   \\
 \hline  P1  & M1 &    &   \\ 
 \hline   P2 &  M2   &    P1  & M1  \\
     \hline {\cellcolor{gray!25} }  & {\cellcolor{gray!25} } &   P2   &  M2  \\ 
   \hline {\cellcolor{gray!25} } &  {\cellcolor{gray!25} }    &   {\cellcolor{gray!25} }  & {\cellcolor{gray!25} }   \\ \hline
  \end{tabular} 
\end{equation}

\begin{theorem} \label{T:Type2Special} Let $\left(V, \mathcal{P}\right) = \mathcal{J}_{0, 2n_1} \oplus \mathcal{J}_{0, 2n_2}$, where $n_1 > n_2$ and $n_1 \equiv n_2 \pmod{2}$. Type II-S indecomposable bi-Lagrangian subspaces $L \subset \left(V, \mathcal{P}\right)$ determined  by the parameter $r$ from Theorem~\ref{T:CanonIndecompTypeII},  subject to the conditions   \[ 1 \leq r \leq n_2, \qquad r \equiv n_2 \pmod{2}.\] The total number of distinct $\operatorname{Aut}(V,\mathcal{P})$-orbits of Type II-S is \[\left[ \frac{n_2}{2}\right].\] The orbit $O_r$ associated with the parameter $r$  has dimension $3r$ and it is diffeomorphic to  $(r-1)$-order jet space space of the $2$-dimensional special linear group: \[ O_r \approx J^{r-1}_0(\mathbb{K}, \operatorname{SL}(2,\mathbb{K})). \]\end{theorem}

First, let us give another description for the jet space $J^{n}_0\left(\mathbb{K}, \operatorname{SL}\left(2,\mathbb{K}\right)\right)$. Consider the ring $R = \mathbb{K}[x]/x^{n+1}$. The special linear group $\operatorname{SL}(2, R)$ consists of polynomial matrices \[M(\lambda) = M_0 + \lambda M_1 + \dots + \lambda^{n} M_n \] that satisfy \begin{equation} \label{Eq:DetMatUpN1} \det M(\lambda) = 1 \pmod{\lambda^{n+1}}.\end{equation} We consider $\operatorname{SL}(2, R)$ as a subspace of the space of $n+1$ matrices $(M_0, \dots, M_n)$ with the induced topology. 

\begin{assertion} \label{A:JetSL2} The jet space $J^{n}_0\left(\mathbb{K}, \operatorname{SL}\left(2,\mathbb{K}\right)\right)$ of $n$-jets of curves in the special linear group of degree $2$ is diffeomorphic to the special linear group $\operatorname{SL}(2, R)$ over the ring $R = \mathbb{K}[x]/x^{n+1}$.\end{assertion}

\begin{proof}[Assertion~\ref{A:JetSL2}] For any smooth curve $\gamma(\lambda)$ within $\operatorname{SL}(2,\mathbb{K})$, a corresponding element of $\operatorname{SL}(2, R)$ can be obtained by taking the Taylor expansion to order $n$: \[ \gamma(\lambda) = M_0 + \lambda M_1 + \dots + \lambda^n M_n + o(\lambda^n). \] Conversely, any element  $M(\lambda) \in \operatorname{SL}(2, R)$  can be extended to a smooth curve $\gamma(\lambda)$ within $\operatorname{SL}(2,\mathbb{K})$ by expressing one coordinate in terms of the other three. For instance, if \[ M(\lambda) = \left( \begin{matrix} a(\lambda) & b(\lambda) \\ c(\lambda) & d(\lambda) \end{matrix} \right)\] and $a(0) \not = 0$, then we can consider the curve \[ \gamma(\lambda) =  \left( \begin{matrix} \hat{a}(\lambda) & b(\lambda) \\ c(\lambda) & d(\lambda) \end{matrix} \right), \qquad \hat{a}(\lambda) = \frac{1 - b(\lambda) c(\lambda)}{d(\lambda)}. \] The n-jet of this curve $\gamma(\lambda)$ coincides with $M(\lambda)$ and is independent of the specific coordinate choice. Constructed maps yield diffeomorphism 
 $J^{n}_0\left(\mathbb{K}, \operatorname{SL}\left(2,\mathbb{K}\right)\right) \approx \operatorname{SL}(2, R)$. Assertion~\ref{A:JetSL2} is proved.  \end{proof}

 \begin{remark} The jet bundle $J^k G$ of $k$-jets of curves in a Lie group $G$ has a natural Lie group structure, see e.g. \cite{Vizman13}.
 \end{remark}

\begin{proof}[Proof of Theorem~\ref{T:Type2Special}] Fix a standard basis of $(V,\mathcal{P})$. There are unique $u, v \in L$ form Theorem~\ref{T:CanonIndecomp2Jord} such that \[ u= e_{\frac{n_1 - r}{2} + 1} + \sum_{j=1}^r \left( a_j \hat{e}_{\frac{n_2 - r}{2}+j} + b_j  \hat{f}_{\frac{n_2 - r}{2}+j} \right), \quad  v= -f_{\frac{n_1+ r}{2}} + \sum_{j=1}^r \left( c_j \hat{e}_{\frac{n_2 - r}{2}+j} + d_j  \hat{f}_{\frac{n_2 - r}{2}+j} \right). \]  It is easy to check that the corresponding subspace $L$, given by \eqref{Eq:CanonFormIndecompGen}, is bi-Lagrangian if and only if the matrix  \[ \begin{gathered} M(\lambda) = \left( \begin{matrix} a(\lambda) & b(\lambda) \\ c(\lambda) & d(\lambda) \end{matrix} \right), \quad a(\lambda) = a_1 + \dots + \lambda^{r-1}a_r, \quad  b(\lambda) = b_r + \dots + \lambda^{r-1}b_1, \\ c(\lambda) = c_1 + \dots + \lambda^{r-1}c_r, \quad d(\lambda) = d_r + \dots + \lambda^{r-1}d_1.\end{gathered}\] is an element of $\operatorname{SL}(2,  \mathbb{K}[x]/x^{n+1})$, i.e. $M(\lambda)$ satisfies \eqref{Eq:DetMatUpN1}. By Assertion~\ref{A:JetSL2} the orbit $O_r$ is diffeomorphic to $J^{r-1}_0(\mathbb{K}, \operatorname{SL}(2,\mathbb{K}))$. Theorem~\ref{T:Type2Special} is proved. \end{proof}

\subsubsection{Dimension and topology of orbits. Type II. General case} \label{SubS:TypeIItopGen}

Let $L \subset \left(V, \mathcal{P}\right) = \mathcal{J}_{0, 2n_1} \oplus \mathcal{J}_{0, 2n_2}$ be an indecomposable bi-Lagrangian subspace  that is categorized as type II (see Definition~\ref{D:TypesIndecomp}) and does not belong to the Type II-S category (see Definition~\ref{Def:TypeIIS}). The orbit of $L$ will be described in a manner analogous to Section~\ref{S:Top2BlocksSemi}. Consider subspaces $S_j$ as in Section~\ref{S:Top2BlocksSemi}. Put \[ L_1 =  P^{h-1}(L) \cap S_1, \qquad   L_2 =  P^{n_2 - d -r-1}(L) \cap \operatorname{Ker} P / S_1\] We get a pair of one-dimensional\footnote{Type II-S cases were excluded as the subspaces $L_i$ are two-dimensional in these cases.} (Lagrangian) subspaces $L_j \subset S_j$. Define \begin{equation} \label{Eq:PiIndecTypeII} \pi(L) = (L_1, L_2) \in \Lambda(1) \times \Lambda(1). \end{equation} For instance, in Figure~\eqref{Eq:Proj2DistJordTopolTypeII} we visualise  subspaces $L_j \subset S_j$ for one of the subspaces $L$ from \eqref{Eq:TypeIIIndecomp}.

\begin{equation} \label{Eq:Proj2DistJordTopolTypeII}
  \begin{tabular}{|c|c||c|c|} 
   \cline{1-2}  & &   \multicolumn{1}{c}{} & \multicolumn{1}{c}{}   \\
     \cline{1-2}  & &   \multicolumn{1}{c}{} & \multicolumn{1}{c}{}   \\
     \cline{1-2} & &   \multicolumn{1}{c}{} & \multicolumn{1}{c}{}   \\
\hline P1 & & &    \\
  \hline P2 & M1 & P1 &    \\
   \hline {\cellcolor{gray!25} }  & M2 &  P2 M1  &   \\ 
   \hline {\cellcolor{gray!25} } &  {\cellcolor{gray!25} }    &    {\cellcolor{gray!25} }  & M1  \\
     \hline {\cellcolor{gray!25} }  & {\cellcolor{gray!25} } &   {\cellcolor{gray!25} }   &  M2  \\ 
   \hline {\cellcolor{gray!25} } &  {\cellcolor{gray!25} }    &   {\cellcolor{gray!25} }  & {\cellcolor{gray!25} }   \\ \hline
  \end{tabular}  \qquad \to \qquad \begin{tabular}{|c|c||c|c|} 
     \multicolumn{1}{c}{}  & \multicolumn{1}{c}{}  &  \multicolumn{1}{c}{}   &  \multicolumn{1}{c}{}    \\
     \multicolumn{1}{c}{}  & \multicolumn{1}{c}{}  &  \multicolumn{1}{c}{}   &  \multicolumn{1}{c}{}    \\
   \multicolumn{1}{c}{}  & \multicolumn{1}{c}{}  &  \multicolumn{1}{c}{}   &  \multicolumn{1}{c}{}    \\
    \multicolumn{1}{c}{}  & \multicolumn{1}{c}{}  &  \multicolumn{1}{c}{}   &  \multicolumn{1}{c}{}    \\
  \multicolumn{1}{c}{}  & \multicolumn{1}{c}{}  &  \multicolumn{1}{c}{}   &  \multicolumn{1}{c}{}    \\
   \multicolumn{1}{c}{}  & \multicolumn{1}{c}{}  &  \multicolumn{1}{c}{}   &  \multicolumn{1}{c}{}    \\
 \multicolumn{1}{c}{}  & \multicolumn{1}{c}{}  &  \multicolumn{1}{c}{}   &  \multicolumn{1}{c}{} \\
     \multicolumn{1}{c}{}  & \multicolumn{1}{c}{}  &  \multicolumn{1}{c}{}   &  \multicolumn{1}{c}{}   \\
   \hline 1 &  & 2  &  \\ \hline
  \end{tabular}
\end{equation}

\begin{theorem} \label{T:2DistJordTopolMaxOrbIndecTypeII} Let $\left(V, \mathcal{P}\right) = \mathcal{J}_{0, 2n_1} \oplus \mathcal{J}_{0, 2n_2}$, where $n_1 > n_2$ and $L \subset \left(V, \mathcal{P}\right)$ be an indecomposable bi-Lagrangian subspace of Type $II$ such that its parameters from Theorem~\ref{T:CanonIndecompTypeII} satisfy $p_1 > q_1$. The $\operatorname{Aut} \left(V, \mathcal{P}\right)$-orbit of $L$ has a  structure of a $F$-fibre bundle over a product of two Lagrangian Grassmanians: \begin{equation} \label{Eq:ProjOrb2DistJordTypeII} \pi: O_{\max}  \xrightarrow{F} \Lambda(1) \times \Lambda(1) = \mathbb{KP}^1 \times \mathbb{KP}^1, \end{equation} where $\pi$ is given by \eqref{Eq:PiIndecTypeII} and the fiber $F \approx \left(\mathbb{K}^*\right)^M \times \mathbb{K}^N$, where \[ N = 2h - n_1 + n_2 - 2d + \max{(0, 2r-n_2 +2d)} -2 + \delta - M, \qquad M = \begin{cases} 1 &\mbox{if } r = 1, \\  2, \quad &\mbox{if } r > 1. \end{cases} \]. \end{theorem}

\begin{proof}[Proof of Theorem~\ref{T:2DistJordTopolMaxOrbIndecTypeII}] The proof is analogous to Theorem~\ref{T:2DistJordTopolMaxOrb} (it can also be proved similar to Theorems~\ref{T:EqualJordTopolOrb} and \ref{T:GenJordTopolMaxOrb}). To simplify, assume $d=0$ (otherwise, perform bi-Poisson reduction w.r.t. $\operatorname{Ker} P^d$). The orbit $O_L$ is a homogeneous space and by Theorem~\ref{T:CanonIndecompTypeI} the induced action on the base $\Lambda(1) \times \Lambda(1)$ is transitive. Hence, the projection $\pi$, given by \eqref{Eq:PiIndecTypeII}, is surjective and defines a fiber bundle. It remains to demonstrate that the fibers $F$ are diffeomorphic to the specified spaces. We show it by describing possible $u$ and $v$ ($L$ is determined by $u$ and $v$ according to \eqref{Eq:TypeIISimpleForm}). Since $u$ belongs to the automorphism orbit of $e_{n_1 - h+1} + \hat{e}_1$ there is a standard basis \eqref{Eq:StandBasisTwoBlocks} such that \[ \begin{gathered} u = e_{n_1 - h+1} + c_1 \hat{e}_1 + w_1 + \hat{w}_1, \\ w_1 \in \operatorname{Span}\left\{ e_i, f_j\right\}, \quad \operatorname{height}(w_1) < h, \qquad \hat{w}_1\in \operatorname{Span}\left\{ e_i, f_j\right\}, \quad \operatorname{height}(\hat{w}_1) < n_2, \end{gathered} \] where $c_1\not = 0$.  It is possible to eliminate only one of the coefficients preceding $e_{n_1 - h+1}$ and $\hat{e}_1$  through multiplication of $u$ by a constant factor. The projection $\pi$ takes the form: \[ L_1 = \operatorname{Span}\left\{ e_1\right\}, \qquad L_2 =\operatorname{Span} \left\{ \hat{e}_1\right\}.\] Since $L$ contains the vectors $P^ju$ and the subspace $\operatorname{Im}P^h$ we can replace $u$ with \[ u = e_{n_1 - h+1} + \sum_{j=n_1 - h+1}^{h-1} b_j f_j + c_1 \hat{e}_1 + \sum_{j=1}^{n_2 - 1} \left( c_{n_2 - j+1} \hat{e}_{n_2 - j+1} + d_j \hat{f}_j\right).\] Since $v$ belongs to the  $\operatorname{Aut}(V,\mathcal{P})$-orbit of $f_{n_1 - h + r} - \hat{f}_{r} + \delta \hat{e}_{n_2 - z + 1}$, by Theorem~\ref{T:VectOrbits},  it has the form \[ \begin{gathered} v =  w_2 + \hat{w}_2, \\ w_2 \in \operatorname{Span}\left\{ e_i, f_j\right\}, \quad \operatorname{height}(w_2) \leq n_1 - h + r, \qquad \hat{w}_2\in \operatorname{Span}\left\{ \hat{e}_i, \hat{f}_j\right\}, \qquad \operatorname{height}(\hat{w}_2) \leq \hat{z}, \end{gathered} \] where  we put \[ \hat{z} = \begin{cases} z, \quad &\mbox{if } \delta = 1, \\  r, \quad &\mbox{if } \delta = 0. \end{cases} \]  By substituting $v$ with an appropriate linear combination of $P^ju$ and vectors from $\operatorname{Im}P^h$, the vector takes the following form: \[v =  \sum_{j = n_1 - h +1}^{n_1 - h+r} \hat{b}_j f_j + \sum_{j=1}^{\hat{z}}\left(\hat{c}_{n_2 - j+1} \hat{e}_{n_2 - j+1} + \hat{d}_j \hat{f}_j \right). \] Here $\hat{d}_r \not = 0$, and in the case $\delta = 1$ also $\hat{c}_{n_2 - z +1} \not = 0$ (see proof of Theorem~\ref{T:CanonIndecompTypeII}). Replacing $u$ and $v$ with suitable linear combinations with $P^jv$ we get \[ \begin{gathered}  v =  \sum_{j = n_1 - h +1}^{n_1 - h+r} \hat{b}_j f_j + \delta \cdot \hat{c}_{n_2 - z + 1}\hat{e}_{n_2 - z + 1} +  \sum_{j=1}^r\hat{c}_{n_2 - j+1} \hat{e}_{n_2 - j+1} +  \hat{f}_r, \\ u = e_{n_1 - h+1} + \sum_{j=n_1 - h+1}^{h-1} b_j f_j + c_1  \hat{e}_1 + \sum_{j=1}^{n_2 - 1} c_{n_2 - j+1} \hat{e}_{n_2 - j+1} + \sum_{j=r+1}^{n_2 - 1} d_j \hat{f}_j.  \end{gathered} \] In the basis from Theorem~\ref{T:CanonIndecompTypeII} the subspace $L$ includes the subspace $\operatorname{Span}\left\{\hat{e}_{r+1},\dots, \hat{e}_{n_2}\right\}$. Therefore, $L$ contains images under $P$ of the component of v associated with the smaller Jordan block: \[ P^{j} \left(c_1 \hat{e}_1 + \sum_{j=1}^{n_2 - 1} c_{n_2 - j+1} \hat{e}_{n_2 - j+1} + \sum_{j=r+1}^{n_2 - 1} d_j \hat{f}_j\right) = c_1 \hat{e}_{j+1} + \text{lower vectors}, \qquad j \geq r. \] Subtracting linear combinations of these vectors we can bring $u$ and $v$ to the form 
 \begin{equation} \label{Eq:FormTypeIIuv} \begin{gathered}  u = e_{n_1 - h+1} + \sum_{j=n_1 - h+1}^{h-1} b_j f_j + c_1  \hat{e}_1 + \sum_{j=n_2 - r + 1}^{n_2 - 1} c_{n_2 - j+1} \hat{e}_{n_2 - j+1} + \sum_{j=r+1}^{n_2 - 1} d_j \hat{f}_j, \\ v =  \sum_{j = n_1 - h +1}^{n_1 - h+r} \hat{b}_j f_j + \delta \cdot \hat{c}_{n_2 - z + 1}\hat{e}_{n_2 - z + 1} + \sum_{j=n_2  -r +1 }^{\max(r, n_2 - r)}\hat{c}_{n_2 - j+1} \hat{e}_{n_2 - j+1} +  \hat{f}_r. \end{gathered} \end{equation} The sum $\displaystyle \sum_{j=n_2  -r +1 }^{\max(r, n_2 - r)} \hat{c}_{n_2 - j+1} \hat{e}_{n_2 - j+1}$ contains no elements when $r \leq n_2 -r$. Since vectors $P^j u$ and $P^j v$ are orthogonal, we get the following conditions: \begin{equation} \label{Eq:OrthogTypeII} \hat{b}_{n_1 - h + x} + c_{r -x+1} - \sum_{j=1}^{n_2} d_{j + x-1} \hat{c}_j = 0, \qquad x = 1,\dots, r. \end{equation}  Here the coefficients $d_{j+x}$ or $\hat{c}_j$  that do not appear in Equation \eqref{Eq:FormTypeIIuv} are assigned the value zero. Also for $\delta = 0$ we put $\hat{c}_{n_2 - z+1} =0$. The coefficients $\hat{b}_{n_1 - h+x}$ can be found by from \eqref{Eq:OrthogTypeII}.  The orbit of $L$ is determined by the remaining parameters from \eqref{Eq:FormTypeIIuv}. The number of these parameters is \[2h - n_1 + n_2 + \max{(0, 2r - n_2)} - 2 +\delta.\] Let us recall restrictions on the parameters. We previously assumed that $c_1 \not = 0$. Additionally, $\hat{b}_{n_1 - h+1} \not = 0$, since $v$ belongs to the orbit of $f_{n_1 - h + r} - \hat{f}_{r}$. By \eqref{Eq:OrthogTypeII} for $x=1$, \[ \hat{b}_{n_1 - h+1} + c_{r} =0. \] There are two variants:
 
 \begin{enumerate}
     \item  If $r=1$, then $\delta = 0$ and $\hat{b}_{n_1 - h+1} = - c_1 \not = 0$. In this case, the parameter $c_1 \in \mathbb{K}^*$ and all other parameters  belong to the field $\mathbb{K}$. 
     \item If $r>1$, then the parameters $c_1, c_{r} \in \mathbb{K}^*$, while the remaining parameters belong to $\mathbb{K}$.
 \end{enumerate} We got the required restrictions on the parameters. Theorem~\ref{T:2DistJordTopolMaxOrbIndecTypeII} is proved. \end{proof}

\subsection{Example: \texorpdfstring{$\operatorname{BLG}\left(\mathcal{J}_{0, 6} \oplus \mathcal{J}_{0, 2}\right)$}{BLG(J06+J02}} \label{SubS:Exam}

Let us consider a simple example when $(V,\mathcal{P}) = \mathcal{J}_{0,6} \oplus \mathcal{J}_{0,2}$. By Theorem~\ref{T:BiLagrJord}  \[\dim \operatorname{BLG}(\mathcal{J}_{0,6} \oplus \mathcal{J}_{0,2}) = 5.\] The bi-Lagrangian Grassmanian consists of $3$ $\operatorname{Aut}(V,\mathcal{P})$-orbits:

\begin{enumerate}

\item The orbit of maximal dimension, $O_{\max}$, consisting of generic bi-Lagrangian subspaces. By Theorem~\ref{T:GenJordTopolMaxOrb} it is a fiber bundle over a product of two Lagrangian Grassmanians:  \[ \pi: O_{\max}  \xrightarrow{\mathbb{K}^3}  \Lambda(1) \times  \Lambda(1), \qquad  \dim O_{\max} = 5. \]

\item The orbit $O_{s}$, which contains non-generic semisimple bi-Lagrangian subspacess. By Theorem~\ref{T:2DistJordTopolMaxOrb} it is diffeomorphic to the product of two Lagrangian Grassmanians:  \[ O_{s} \approx  \Lambda(1) \times  \Lambda(1) = \mathbb{KP}^1 \times \mathbb{KP}^1, \qquad  \dim O_{s} = 2. \]

\item The orbit $O_{ind}$, consisting of indecomposable bi-Lagrangian subspaces (of Type II-S with parameter $r = 1$). By Theorem~\ref{T:Type2Special} it is diffeomorphic to the special linear group of degree $2$:  \[ O_{ind} \approx  \operatorname{SL}(2,\mathbb{K}), \qquad  \dim O_{ind} = 3. \]

\end{enumerate}

We visualise bi-Lagrangian subspaces from these $3$ orbits in Figure~\eqref{Eq:BiLagrExampleJ6J2} .
\begin{equation}  \label{Eq:BiLagrExampleJ6J2}  \begin{tabular}{|c|c|cc}\cline{1-2}
  {\cellcolor{gray!25} }  & & & \\ \cline{1-2}   
    {\cellcolor{gray!25} }   &   & & \\ \hline
    {\cellcolor{gray!25} } &  & {\cellcolor{gray!25} } &   \multicolumn{1}{|c|}{ }\\ \hline
  \end{tabular} \qquad  \begin{tabular}{|c|c|cc}\cline{1-2}
    & & & \\ \cline{1-2}   
    {\cellcolor{gray!25} }   &   & & \\ \hline
    {\cellcolor{gray!25} } & {\cellcolor{gray!25} } & {\cellcolor{gray!25} } &   \multicolumn{1}{|c|}{ }\\ \hline
  \end{tabular} \qquad 
  \begin{tabular}{|c|c|cc}\cline{1-2}
   & & & \\ \cline{1-2}   
    P1  & P2  & & \\ \hline
    {\cellcolor{gray!25} } & {\cellcolor{gray!25} }  & M1 &   \multicolumn{1}{|c|}{ M2}\\ \hline
  \end{tabular}
\end{equation}

Performing bi-Poisson reduction w.r.t. the subspace $U =\operatorname{Im}P^2$ and using Lemma~\ref{L:AlgIsomBiisotrFactor}, we find that the union of non-maximal orbits is isomorphic to
Lagrangian Grassmanian of a two-dimensional space: \[O_{s} \cup O_{ind} \approx \operatorname{BLG}(U^{\perp}/U) \approx \Lambda(2).\]

\begin{lemma} \label{L:ClosureTop} The closure (in the standard topology) of the largest orbit $O_{max}$ within $ \operatorname{BLG}(\mathcal{J}_{0,6} \oplus \mathcal{J}_{0,2})$ is its union with the smallest orbit: \begin{equation} \label{Eq:ClosureTop} \bar{O}_{\max} = O_{\max} \cup O_{s}. \end{equation} \end{lemma}

\begin{proof}[Proof of Lemma~\ref{L:ClosureTop}] The closure $\bar{O}_{\max}$ is closed and $\operatorname{Aut}(V,\mathcal{P})$-subset. Hence it contains $O_{\max}$ and $O_{s}$. It suffices to prove that some indecomposable subspace $L \in O_{ind}$ does not belong the closure $\bar{O}_{\max}$. Consider the subspace $L$, given by \eqref{Eq:BiLagr_NotDecomp}. In the basis from Assertion~\ref{A:NonDecomp} $L$ is generated by the rows of the matrix \begin{equation} \label{Eq:SpanLEx1} \left(\begin{array}{cccccc|cc} 0 & 1 & 0 & 0 & 0 & 0 & 1 & 0 \\ 0 & 0 & 1 & 0 & 0 & 0 & 0 & 0 \\ 0 & 0 & 0 & 1 & 0 & 0 & 0 & 0 \\ \hline 0 & 0 & 0 & 0 & 1  & 0 & 0 & -1 \\   \end{array} \right). \end{equation} Whereas, by Theorem~\ref{T:BiLagrGenDecomp}, a generic bi-Lagrangian subspace $\hat{L} \in O_{\max}$ is spanned by the rows of a matrix \begin{equation} \label{Eq:SpanLEx2} \left(\begin{array}{cccccc|cc} a_1 & a_2 & a_3 & b_1 & b_2 & b_3 & c_1 & d_1 \\ 0 & a_1 & a_2 & b_2 & b_3 & 0 & 0 & 0 \\ 0 & 0 & a_1 & b_3 & 0 & 0 & 0 & 0 \\ \hline 0 & 0 & 0 & 0 & 0  & 0 & c_2 & d_2 \\   \end{array} \right). \end{equation} The minor formed by columns $2$-$5$ is zero for the matrix~\eqref{Eq:SpanLEx2} and nonzero for the matrix~\eqref{Eq:SpanLEx1}. Therefore, $L \not \in \bar{O}_{\max}$ and \eqref{Eq:ClosureTop} holds. Lemma~\ref{L:ClosureTop} is proved. \end{proof}

\begin{remark} \label{Rem:ClosureOrbitsJ2} By the closed orbit lemma (see e.g. \cite[Section 8.3]{Humphreysl75}) an orbit under an algebraic group action over an algebraically closed field is a constructible set. Thus, orbit closure under the Zariski topology coincides with orbit closure under the standard topology. Consider $\operatorname{BLG}(V,\mathcal{P})$ as an algebraic variety. Its  irreducible components are invariant under the action of $\operatorname{Aut}(V,\mathcal{P})$ and hence are unions of
$\operatorname{Aut}(V,\mathcal{P})$-orbits.  Consequently, bi-Lagrangian Grassmanian $ \operatorname{BLG}(\mathcal{J}_{0,6} \oplus \mathcal{J}_{0,2})$ consists of two irreducible components: \[ \bar{O}_{\max} = O_{\max} \cup O_{s}, \qquad \bar{O}_{ind} = O_{ind} \cup O_{s}. \] \end{remark}

\section{Semisimple bi-Lagrangian subspaces} \label{S:DecomposableSection}

Denote a bi-Lagrangian subspace as $(V,\mathcal{P}, L)$, which is trivial if $\dim V = 0$. 

\begin{itemize}
\item The \textbf{direct sum} of bi-Lagrangian subspaces $(V_1, \mathcal{P}_1, L_1)$ and $(V_2, \mathcal{P}_2, L_2)$ is $\left(V_1 \oplus V_2, \mathcal{P}_1 + \mathcal{P}_2, L_1 \oplus L_2\right)$. 

\item Two Lagrangian subspaces are \textbf{isomorphic} $(V_1, \mathcal{P}_1, L_1) \approx (V_2, \mathcal{P}_2, L_2)$ if there is an a linear isomorphism \[ f:(V_1, \mathcal{P}_1, L_1) \to (V_2, \mathcal{P}_2, L_2) \] such that $f^*(\mathcal{P}_2)=\mathcal{P}_1$ and $f(L_1) = L_2$. 
\end{itemize}

\begin{definition} \label{Def:Decomp} A bi-Lagrangian subspace $L \subset (V,\mathcal{P})$  is called

\begin{enumerate}

\item \textbf{Indecomposable} if it is not isomorphic to a direct sum of two non-trivial bi-Lagrangian subspaces. 

\item \textbf{Simple} if contained within a single Jordan block $L \subset \mathcal{J}_{0, 2n}$.

\item \textbf{Semisimple} if it is isomorphic to a direct sum of simple bi-Lagrangian subspaces: \[ (V,\mathcal{P}, L) \approx \bigoplus_{i=1}^N \left( \mathcal{J}_{\lambda_i, 2n_i}, L_i\right).\]

\end{enumerate}

\end{definition}

In Section~\ref{SubS:SemiMark} we prove that a bi-Lagrnangian subspace is semisimple if and only if it is marked and characterize spaces where all bi-Lagrangian subspaces are semisimple.  Then we classify semisimple bi-Lagrangian subspaces (Section~\ref{SubS:TypesSemiSimple}) and compute the dimension of their automorphism orbits (Section~\ref{SubS:Dimension of orbits}).

\subsection{Semisimple subspaces are marked} \label{SubS:SemiMark}

Marked invariant subspaces were introduced in \cite{Gohberg86}.

\begin{definition} \label{Def:Marked} Let $C: V \to V$ be a linear operator.  A $C$-invariant subspace $W \subset V$ is called \textbf{marked} if there is a Jordan basis of $W$ which can be extended to a Jordan basis of $V$. \end{definition}

\begin{theorem} \label{T:Marked} Let $(V,\mathcal{P})$ be a sum of Jordan blocks and $P$ be the recursion operator. A bi-Lagrangian subspace $L \subset (V,\mathcal{P})$ is semisimple if and only if it is marked (as a $P$-invariant subspace). \end{theorem}
\begin{proof}[Proof of Theorem~\ref{T:Marked}] It is evident that semisimple subspaces are marked. To prove the converse,  it suffices to extract one (skew-symmetric) Jordan block. The proof is in several steps.

\begin{enumerate}
    \item \textit{Without loss of generality, $P$ is nilpotent.} Both semisimple and marked subspaces admit an eigendecomposition.

    \item \textit{Extend a Jordan basis of  $L$ to a Jordan basis of $V$.} Let $ e_{n-h+1},\dots, e_n$ be the longest Jordan chain within the subspace $L$, assuming that this chain is embedded within a Jordan chain $e_1, \dots, e_n$ in the space V. 

\item \textit{Extract one Jordan block.} As the Jordan chain $e_1,\dots, e_n$ is part of a Jordan basis, Theorem~\ref{T:VectOrbits} guarantees the existence of a decomposition \begin{equation}  (V, \mathcal{P}) = \mathcal{J}_{0, 2n} \oplus (V', \mathcal{P}'),\end{equation}such that the canonical basis of $\mathcal{J}_{0,2n}$ is $e_1,\dots, e_n, f_1,\dots, f_n$. 

\item \textit{We claim that $L$ also decomposes:} \begin{equation} \label{Eq:DecomLSemiMark} L = L_1 \oplus L_2, \qquad L_1 = L \cap \mathcal{J}_{0, 2n}, \quad L_2 = L \cap (V', \mathcal{P}'). \end{equation} Given that the Jordan chain $ e_{n-h+1},\dots, e_n$ 
  was the longest within $L$,  $\operatorname{height}(L) = h$ and, thus, \[ \operatorname{Im} P^h \subset L \subset \operatorname{Ker}P^h. \] Consequently, $L$ contains the subspace \[ L_1 = \operatorname{Span} \left\{e_{n-h+1},\dots, e_n, f_1,\dots, f_h \right\}.\] Applying Theorem~\ref{T:BiPoissReduction}, we get decomposition~\ref{Eq:DecomLSemiMark}.

\end{enumerate}

By applying a similar decomposition process, we can express $L$ as a direct sum of simple bi-Lagrangian subspaces. Thus, $L$ is semisimple. Theorem~\ref{T:Marked} is proved. \end{proof}

Next statement characterizes the precise conditions under which all bi-Lagrangian subspaces are decomposable.

\begin{theorem}[R.~Bru et al., \cite{Bru91}] Let $C$ be an n X n matrix. Then every $C$-invariant subspace is marked if and only if for every eigenvalue $\lambda_0$ of $C$ the difference between the biggest and the smallest multiplicity of $C$ corresponding to $\lambda_0$
does not exceed $1$. \end{theorem}

\subsection{Types of semisimple bi-Lagrangian subspaces} \label{SubS:TypesSemiSimple}

First, note that by Theorem~\ref{T:BiLagrKronPart} we can always ``decompose'' all Kronecker blocks (if $\left(V_i, \mathcal{P}_i\right)$ is a Kronecker block, then $L_i$ is its core). Then by Theorem~\ref{T:JordaMultEigen} $L$ decomposes into a sum of $L_{\lambda}$ corresponding to each eigenvalue $\lambda$. Thus, without loss of generality we can assume that $(V, \mathcal{P})$ is a sum Jordan blocks with eigenvalue $0$. Let us describe a canonical form for a semisimple bi-Lagrangian subspace $L \subset (V, \mathcal{P})$.

\begin{definition}
Let $(V, \mathcal{P})$ be a sum of Jordan blocks with zero eigenvalue. A semisimple bi-Lagrangian $L \subset (V, \mathcal{P})$ \textbf{has type} \begin{equation} \label{Eq:TypeDecBiLagr} H(L) = \left\{ \left( h_i, n_i\right)\right\}_{i=1, \dots, t}\end{equation} if there exists a JK decomposition \[ (V, \mathcal{P}) = \bigoplus_{i=1}^t \mathcal{J}_{0, 2n_i}, \qquad L = \bigoplus_{i} L_i, \qquad L_i = L \cap  \mathcal{J}_{0, 2n_i}\] such that $\operatorname{height}(L_i) = h_i$ for $i=1, \dots, t$. Here $h_i \geq \frac{n_i}{2}$. We always assume that $n_1 \geq n_2 \geq \dots \geq n_t$ and if $n_i = n_{i+1}$, then $h_i \geq h_{i+1}$. 
\end{definition}

For example,  the following bi-Lagrangian subspace $L \subset \mathcal{J}_{0,8} \oplus \mathcal{J}_{0, 4} \oplus  \mathcal{J}_{0, 4}$ has the type $\left\{ (3, 4), (2, 2), (1, 2) \right\}$:

\begin{equation} \label{Eq:SumBlocksDecompEx1}
 \begin{tabular}{|c|c|} 
  \cline{1-2}   &   \\
 \cline{1-2}  {\cellcolor{gray!25} }  &   \\
     \hline     {\cellcolor{gray!25} } &   \\ 
   \hline   {\cellcolor{gray!25} } & {\cellcolor{gray!25} }\\ \hline
  \end{tabular} \,
   \begin{tabular}{|c|c|} 
 \multicolumn{1}{c}{} & \multicolumn{1}{c}{}  \\
 \multicolumn{1}{c}{} & \multicolumn{1}{c}{}   \\
     \hline     {\cellcolor{gray!25} } &  \\ 
   \hline    {\cellcolor{gray!25} } &   \\ \hline
  \end{tabular} \,  
   \begin{tabular}{|c|c|} 
   \multicolumn{1}{c}{} & \multicolumn{1}{c}{} \\
   \multicolumn{1}{c}{} & \multicolumn{1}{c}{} \\
     \hline      &  \\ 
   \hline   {\cellcolor{gray!25} } &  {\cellcolor{gray!25} } \\ \hline
  \end{tabular} 
\end{equation}  
Each Jordan block $\mathcal{J}_{0, 2n}$ is visualized as a rectangle with dimensions $n\times 2$, similar to the one depicted in \eqref{Eq:OneJordanBlock_VectorInTable}. Here, the shaded cells within these rectangles correspond to the basis vectors $e_i, f_j$ of the Jordan blocks that together span the semisimple bi-Lagrangian subspace $L$.

Sometimes it will be convenient to group some equal pairs $(h_i, n_i)$ together.  Let  $(V, \mathcal{P})$ be a sum of $l_i$ Jordan
$2n_i \times 2n_i$ blocks with eigenvalue $\lambda = 0$, where $i=1, \dots, t$ and $n_1
\geq n_2 \geq \dots \geq n_t$, i.e. \[ \left(V, \mathcal{P} \right) = \bigoplus_{i=1}^t \left( \bigoplus_{j=1}^{l_i} \mathcal{J}_{0, 2n_i} \right). \]  Let $L \subset (V, \mathcal{P})$ be a semisimple bi-Lagrangian subspace such that each corresponding component $L_i \subset \mathcal{J}_{0, 2n_i}$ have the type $(h_i, n_i)$ for $i=1, \dots, t$. Then the type of $L$ can be formally denote it as \[ H(L) = \left\{ \underbrace{\left( h_i, n_i\right), \dots, \left( h_i, n_i\right)}_{l_i} \right\}_{i=1, \dots, t}\] or, for short, as \begin{equation} \label{Eq:TypeDecBiLagrGrouped} H(L) = \left\{ \left( h_i, n_i\right) \times l_i\right\}_{i=1, \dots, t}\end{equation} 

\begin{remark} Below it will be convenient to work with type $H(L)$ denoted by \eqref{Eq:TypeDecBiLagrGrouped}, where not all equal values are grouped together. In other words, it is possible that $(h_i, n_i) = (h_j, n_j)$ for $i \not = j$. Note that \eqref{Eq:TypeDecBiLagrGrouped} is just a convenient notation. The type $H(L)$ is a set of pairs $(h_i, n_i)$, as in \eqref{Eq:TypeDecBiLagr}, i.e. one pair for each Jordan block. \end{remark}

Using Theorem~\ref{T:BiLagr_One_Jordan_Canonical_Form} we get the following canonical form a semisimple bi-Lagrangian subspace.

\begin{assertion} \label{A:CanonDecompBiLagr} Let $L \subset (V, \mathcal{P}) = \bigoplus_{i=1}^t \left( \bigoplus_{j=1}^{l_i} \mathcal{J}_{0, 2n_i} \right)$ be a bi-Lagrangian subspace with type $H(L) = \left\{ \left( h_i, n_i\right) \times l_i\right\}_{i=1, \dots, t}$. Then there exists a standard basis $e^{ij}_k, f^{ij}_k$ for $i=1, \dots, t, j=1,\dots, l_i, k=1, \dots, n_i$ such that \begin{equation} \label{Eq:CanonBasis} L = \bigoplus_{i,j} L_{ij}, \qquad   L_{ij} = \operatorname{Span} \left( e^{ij}_{n_i -h_i + 1},\quad \dots \quad e^{ij}_{n_i}, f^{ij}_1, \quad \dots \quad, f^{ij}_{n_i - h_i} \right). \end{equation}
\end{assertion}

Now, let us prove that the type of a bi-Lagrangian subspace is uniquely defined. Recall that we described $\operatorname{Aut}(V, \mathcal{P})$-invariant subspaces in Section~\ref{S:InvSubspaces}.

\begin{lemma}  \label{L:DecomTypeDetermIntersect}
Consider a bi-Poisson space $(V, \mathcal{P}) = \bigoplus_{i=1}^N \mathcal{J}_{0, 2n_i}$. Let $L \subset (V, \mathcal{P})$ be a semisimple bi-Lagrangian subspace. Its type $H(L) = \left\{ \left( h_i, n_i\right)\right\}_{i=1, \dots, t}$ is uniquely defined by the dimensions of its intersections with $\operatorname{Aut}(V, \mathcal{P})$-invariant subspaces: \[ \dim \left( L \cap \left(\bigoplus_{i=1}^{N}  \mathcal{J}_{0, 2n_i}^{\leq k_i} \right)\right),\]  where $0 \leq k_{i} - k_{i+1} \leq n_{i} - n_{i+1}$.
\end{lemma}

\begin{proof}[Proof of Lemma~\ref{L:DecomTypeDetermIntersect}] Note that \[ \dim \left( L \cap \left(\bigoplus_{i=1}^{N} \mathcal{J}_{0, 2n_i}^{\leq k_i} \right)\right) = \sum_{i=1}^N \min(h_i, k_i) + \min(n_i -h_i, k_i).\] Let there be $d$ distinct values of $n_i$, occurring in groups of sizes $l_1, \dots, l_d$. The number of $n_i$ in the first $j$ groups is $M_j = \sum_{s=1}^j l_j$.  First, we can determine all pairs $(n_i, n_i)$ in the type $H(L)$. Indeed, the number of pairs $(n_i, n_i)$ for the first $j$ distinct values of $n_j$ is \[ \dim L - \dim \left( L  \cap \left(\bigoplus_{i=1}^{M_j} \mathcal{J}_{0, 2n_i}^{\leq n_i -1} \right) \oplus  \bigoplus_{i=M_j+1}^{t} \mathcal{J}_{0, 2n_i}^{\leq n_i} \right).\]  On the next step we can determine all the pairs $(n_i-1, n_i)$ through the numbers \[ \dim L - \dim \left( L  \cap \left(\bigoplus_{i=M_1}^{j} \mathcal{J}_{0, 2n_i}^{\leq n_i -2}\right) \oplus  \bigoplus_{i=M_j+1}^{t} \mathcal{J}_{0, 2n_i}^{\leq n_i-1} \right).\] Then we can determine all other pairs in a similar fashion. Lemma~\ref{L:DecomTypeDetermIntersect} is proved. \end{proof}

\begin{remark} A $P$-invariant subspace $W \subset V$ is called \textbf{marked} if there is a Jordan basis of $W$ which can be extended to a Jordan basis of $V$ (see \cite{Gohberg86}). Semisimple bi-Lagrangian subspaces are marked. By \cite{Ferrer96}, marked subspaces are determined by $\dim L \cap E_d^h$, where $E_d^h = \operatorname{Ker}P^h \cap \operatorname{Im} P^d$, offering an alternative proof for Lemma~\ref{L:DecomTypeDetermIntersect}.
\end{remark}

\begin{corollary} \label{Cor:NumTypes} Consider a bi-Poisson subspace \[\left(V, \mathcal{P} \right) = \bigoplus_{i=1}^t \left( \bigoplus_{j=1}^{l_i} \mathcal{J}_{0, 2n_i} \right),\]  where $n_1 > n_2 > \dots > n_t$. The  number of different types of semisimple bi-Lagrangian subspaces $L\subset (V, \mathcal{P})$ is \[ \prod_{i=1}^T \binom{\left[ \frac{n_i}{2} \right] +  l_i}{l_i}. \]\end{corollary}

\begin{proof}[Proof of Corollary~\ref{Cor:NumTypes}]
For each $i=1, \dots, t$ the $l$ numbers $h_i$ that take one of the $\left[\frac{n_i}{2}\right] + 1$ values, from $\lceil \frac{n_i}{2} \rceil$ to $n_i$. We can easily get the number of types using the stars and bars formula. Corollary~\ref{Cor:NumTypes} is proved. \end{proof}

\subsection{Dimension of orbits} \label{SubS:Dimension of orbits}

\begin{theorem} \label{T:DimDecomp}
Let $L \subset \bigoplus_{i=1}^t \left( \bigoplus_{j=1}^{l_i} \mathcal{J}_{0, 2n_i} \right)$ be a semisimple bi-Lagrangian subspace of type \[ H(L) = \left\{ (h_i, n_i) \times l_i \right\}_{i=1,\dots, t},\]  where $n_i \geq n_{i+1}$ and if $n_i = n_{i+1}$, then $h_i \geq h_{i+1}$. Let $O_L$ be the $\operatorname{Aut}\bigoplus_{i=1}^t \left( \bigoplus_{j=1}^{l_i} \mathcal{J}_{0, 2n_i} \right)$-orbit of $L$. Then its dimension is \begin{equation} \label{Eq:DimDec} \dim O_{L} = \sum_{i=1}^t \left( \frac{l_i (l_i + 1)}{2} (2h_i -n_i) + \sum_{j=i+1}^{t} l_i l_j \Delta_{ij}\right),\end{equation} where \begin{equation} \label{Eq:DeltaDimDec} \Delta_{ij} = \max\left(0, h_j - \left(n_i - h_i\right)\right) + \max\left(0, h_j - h_i, \left(n_j - h_j\right) - \left(n_i - h_i \right)\right).\end{equation}  \end{theorem}

\begin{remark} In Theorem~\ref{T:DimDecomp} there are four possible values for $\Delta_{ij}$, depending on the numbers $h_i, h_j, n_i-h_i$ and $n_j - h_j$. For $i=1, j=2$ they are given by \eqref{Eq:FourDelta}. The other variants are not possible, since $n_1 \geq n_2$ and if $n_1 =n_2$, then $h_1 \geq h_2$. Note that if $n_1 = n_2 = n$ and $h_1 \geq h_2$, then \[ \Delta_{12} = 2h_1 - n.\]  \end{remark} 

\begin{remark} Theorem~\ref{T:DimDecomp} also applies when all $l_i = 1$ (i.e., no grouping of equal blocks). In other words, if $L \subset \bigoplus_{i=1}^N\left( \mathcal{J}_{0, 2n_i} \right)$ has type $H(L) = \left\{ (h_i, n_i\right)_{i=1,\dots, N}$, where $n_i \geq n_{i+1}$ and $h_1 \geq h_2$ if $n_1 =n_2$. Then the formula~\ref{Eq:DeltaDimDec}  takes the form \begin{equation} \label{Eq:DimOSimpDelta} \dim O_L = \sum_{1 \leq i \leq j \leq N} \Delta_{ij},\end{equation} where $\Delta_{ij}$ is given by \eqref{Eq:DeltaDimDec}. Note that \eqref{Eq:DimOSimpDelta} contains summands $\Delta_{ii} = 2h_i - n_i$. \end{remark}

The proof of Theorem~\ref{T:DimDecomp} is straightforward: we calculate the stabilizer $\operatorname{St}_L$ of $L$ in the Lie algebra of bi-Poisson automorphisms, i.e. we find all  $C \in \operatorname{aut}(V, \mathcal{P})$ such that $CL \subset L$. The Lie algebra $\operatorname{aut}(V, \mathcal{P})$ is described in Theorem~\ref{T:BiSymp_General_Jordan_Case_Mega}. The stabilizer $\operatorname{St}_L$  is its linear subspace, so finding it is a relatively tedious exercise on Linear Algebra.

\subsubsection{Stabilizers of semisimple bi-Lagrangian subspaces} 

We begin by identifying all linear operators in the endomorphism algebra $\operatorname{End}(V)$ that leave the bi-Lagrangian subspace $L$ from Theorem~\ref{T:DimDecomp} invariant. The next statement is proved by direct calculation.

\begin{assertion} \label{A:PreserDecompBiLagr} Let a bi-Lagrangian subspace $L \subset (V, \mathcal{P}) = \bigoplus_{i=1}^t \left( \bigoplus_{j=1}^{l_i} \mathcal{J}_{0, 2n_i} \right)$ have type $H(L) = \left\{ \left( h_i, n_i\right) \times l_i\right\}_{i=1, \dots, t}$. If $L$ is given by \eqref{Eq:CanonBasis}, then a linear map $C \in \operatorname{End}(V)$ preserves $L$, i.e. $CL \subset L$, if and only if it has the following form. In the basis \eqref{Eq:SeverJordBlocksGroupedBasisSt} the matrix of $C$ is \[ C = \left(\begin{matrix} C_{11} & \dots & C_{1t} \\ \vdots & \ddots & \vdots \\ C_{t1} & \dots & C_{tt} \\ \end{matrix}  \right). \] Each $C_{pq}$ is a $2n_p l_p  \times 2n_q l_q$ matrix that has the form \begin{equation} \label{Eq:FormCpq} C_{pq} = \left(\begin{matrix} D_{n_p 1} & \dots & D_{n_p n_q} \\ \vdots & \ddots & \vdots \\ D_{11} & \dots & D_{1n_q} \\ \end{matrix}  \right). \end{equation} The blocks $D_{\alpha \beta}$ are $2l_p \times 2l_q$ matrices that have the following form:

\begin{enumerate}

\item If $\alpha \leq n_1 -h_1$ or $\beta\leq n_2 - h_2$, then $D_{\alpha \beta}$ are arbitrary  $2l_p \times 2l_q$ matrices;

\item If $\alpha > h_1$ and  $\beta > h_2$, then $D_{\alpha \beta} = \left( \begin{matrix} 0 & 0 \\ 0 & 0 \end{matrix} \right)$;

\item If $n_1 -h_1 < \alpha \leq h_1$ and  $\beta > h_2$, then $D_{\alpha \beta} = \left( \begin{matrix} X_{\alpha \beta} & Y_{\alpha \beta} \\ 0 & 0 \end{matrix} \right)$;

\item If $\alpha > h_1$  and $n_2 -h_2 < \beta \leq h_2$, then $D_{\alpha \beta} = \left( \begin{matrix} 0 & Y_{\alpha \beta} \\ 0 & Z_{\alpha \beta} \end{matrix} \right)$;

\item If $n_1 -h_1 < \alpha \leq h_1$ and  $n_2 -h_2 < \beta \leq h_2$, then $D_{\alpha \beta} = \left( \begin{matrix} X_{\alpha \beta} & Y_{\alpha \beta} \\ 0 & Z_{\alpha \beta} \end{matrix} \right)$.

\end{enumerate}

In each case $X_{\alpha \beta}, Y_{\alpha \beta}$ and $Z_{\alpha \beta}$ are arbitrary $l_i \times l_j$ matrices.

\end{assertion}

For example, for two blocks with $n_1 = n_2 = 3$, $l_1 = l_2 = 1$ and $h_1 = h_2 = 2$ the matrices from Assertion~\ref{A:PreserDecompBiLagr} have the following form:
\begin{equation} \label{Eq:StabLEx1} C = \left(\begin{matrix} C_{11} & C_{12} \\ C_{21} & C_{22} \end{matrix} \right), \qquad 
C_{ij} = \left(\begin{array}{cc|cc|cc} 
* & * & 0 & * & 0 & 0 \\
* & * & 0 & * & 0 & 0 \\
\hline
* & * & * & * & * & * \\
* & * & 0 & * & 0 & 0 \\
\hline
* & * & * & * & * & * \\
* & * & * & * & * & * \\
\end{array}\right).
\end{equation}
Let us demonstrate that $C_{11}$ has the required form. If $e_1, e_2, e_3, f_1, f_2, f_3$ is a standard basis of a $6\times 6$ Jordan block, then the corresponding bi-Lagrangian subspace in that block is \[\operatorname{Span}\left\{e_2, e_3, f_1\right\}.\] Then in the basis \[e_1, f_3, e_2, f_1, e_3, f_1\] the matrix $C_{11}$ has to take the form \eqref{Eq:StabLEx1}.

\begin{remark} Note that both the Lie algebra $\operatorname{aut}(V, \mathcal{P})$, described in Theorem~\ref{T:BiSymp_General_Jordan_Case_Mega}, and the stabilizer of $L$, described in Assertion~\ref{A:PreserDecompBiLagr} have the same block diagonal structure with conditions given by linear equations. Hence, it suffices to consider the cases $t=1$ and $t=2$, i.e. when all Jordan are equal, or when there are only two types of Jordan blocks. \end{remark}

\subsubsection{Dimension of orbits for bi-Lagrangian subspaces with type \texorpdfstring{$(h, n) \times l$}{(h, n) l times}} Let us consider the case, when all Jordan blocks are equal and all subspace $L_i$ have the same height.

\begin{assertion} \label{A:AutPresLOneBlock} Let $(V, \mathcal{P}) =  \bigoplus_{j=1}^{l} \mathcal{J}_{0, 2n} $ be a sum $l$ of Jordan $2n \times 2n$ blocks and let $e^i_j, f^i_j, i=1, \dots, l, j=1, \dots, l$ be a standard basis from the JK theorem. Consider the following bi-Lagrangian subspace with height $h$: \[L= \bigoplus_{i=1}^l \operatorname{Span} \left( e^i_{n-h+1}, \dots, e^i_n, \quad f^i_{n-h}, \dots, f^i_1\right). \] The element of the Lie algebra of bi-Poisson automorphisms $C \in \operatorname{aut}\left(V, A, B\right)$ that preserve $L$, i.e. $CL \subset L$, have the following form \[ C = \left(\begin{matrix} C_1 & & & \\ C_2 & C_1 & & \\ \vdots & \ddots & \ddots & \\ C_{n} & \dots &  C_2 & C_1 \end{matrix} \right),  \] where all $2l\times 2l$ matrices $C_i \in \operatorname{sp}(2l)$ and for $i \geq 2h - n$ these matrices have the form \begin{equation} \label{Eq:AutPresLOneBlockCond1} C_i = \left( \begin{matrix} X_{i} & Y_{i} \\ 0 & -X_i^T \end{matrix} \right), \qquad Y_i^T = Y_i.\end{equation} Here $X_i$ and $Y_i$ are $l\times l$ matrices. \end{assertion}

\begin{proof}[Proof of Assertion~\ref{A:AutPresLOneBlock}]  The proof is by direct calculation using Theorem~\ref{T:BiSymp_General_Jordan_Case_Mega}  and Assertion~\ref{A:PreserDecompBiLagr}. In the notations from Assertion~\ref{A:PreserDecompBiLagr} the matrices $D_{ij}$ on diagonals are equal and belong to $\operatorname{sp}(2l)$. In other words, if $i' +j' = i+j$, then $D_{i'j'} = D_{ij}$. If $i+j > n+1$, then $D_{ij} = 0$. Note that if $\alpha > h$, then $\beta \leq n-h$, therefore the only restriction is \begin{equation} \label{Eq:CondAutPrLOneBlock1} D_{\alpha \beta} = \left( \begin{matrix} X_{\alpha \beta} & Y_{\alpha \beta} \\ 0 & Z_{\alpha \beta} \end{matrix} \right), \qquad  n_1 -h_1 < \alpha \leq h_1, \qquad n_2 -h_2 < \beta \leq h_2.\end{equation} We get restrictions on the matrices $D_{\alpha \beta}$ in  the following ``triangle'': \[ \left(\begin{matrix} D_{h, n-h+1} & &  \\ \vdots & \ddots & \\  D_{n-h+1, n-h+1} & \dots  &  D_{n-h+1, h}  \end{matrix} \right). \] There are $2h - n$ different matrices in this triangle, namely, $C_1, \dots, C_{2h-n}$. They satisfy \eqref{Eq:CondAutPrLOneBlock1} and belong to $\operatorname{sp}(2l)$. Thus, they have the form~\eqref{Eq:AutPresLOneBlockCond1}. Assertion~\ref{A:AutPresLOneBlock} is proved. \end{proof}

\begin{corollary} \label{Cor:DimOLEqualHN} Under the conditions of Assertion~\ref{A:AutPresLOneBlock} the dimension of the $\operatorname{Aut}(V, \mathcal{P})$-orbit $O_L$ of $L$ is \[ \dim O_L = (2h-n)\frac{l(l+1)}{2}.\] \end{corollary}

\begin{proof}[Proof of Corollary~\ref{Cor:DimOLEqualHN}]   By Corollary~\ref{C:DimAutJordanCase} the dimension of the automorphism group is \[\dim \operatorname{Aut}\left(V, \mathcal{P}\right) = n l (2l+1) \] and by Assertion~\ref{A:AutPresLOneBlock} the dimension of stabilizer of $L$ is \[ \dim \operatorname{St}_L = n l (2l+1) - (2h-n) \frac{l(l+1)}{2}.\] The orbit dimension is $\dim O_L = \dim \operatorname{Aut}\left(V, \mathcal{P}\right) - \dim \operatorname{St}_L$. Corollary~\ref{Cor:DimOLEqualHN} is proved.  \end{proof}

\subsubsection{Conditions for blocks \texorpdfstring{$C_{ij}$}{Cij} with \texorpdfstring{$i < j$}{i<j}} This section proves that restrictions from  Theorem~\ref{T:BiSymp_General_Jordan_Case_Mega} only need to be applied to the blocks $C_{ij}$, $i \geq  j$ from Assertion~\ref{A:PreserDecompBiLagr}. The conditions for $i<j$ will automatically follow.

\begin{assertion} \label{A:DecomCondIJ} Under the conditions of Theorem~\ref{T:BiSymp_General_Jordan_Case_Mega}, assume that $C \in \operatorname{aut} (V, A, B)$ and that it satisfies the conditons of Assertion~\ref{A:PreserDecompBiLagr} for the matrices $C_{ij}$, where $i \geq j$. Then $C_{ij}$ with $i < j$ also satisfy the conditons from Assertion~\ref{A:PreserDecompBiLagr}. \end{assertion}

\begin{proof}[Proof of Assertion~\ref{A:DecomCondIJ}] The proof is by direct calculation. The blocks in $C_{ij}$ and $C_{ji}$ are linked by the condition \eqref{E:Cond_on_BiSymp_Jordan_Case}. It can be rewritten as follows: \[ C_{ji}^T = B_i C_{ij} B_j^{-1}, \qquad B_i = \left(
\begin{array}{cccc}  & & Q_{2l_i} \\ &  \udots &  \\ 
Q_{2l_i} &   &  \end{array} \right),\qquad B_j = \left(
\begin{array}{cccc}  & & Q_{2l_j} \\ &  \udots &  \\
Q_{2l_j} &  &  \end{array} \right). \] If $C_{ij}$ has the form \eqref{Eq:FormCpq}, then \[ C_{ji}^T = \left(\begin{matrix} \hat{D}_{1n_j}  & \dots & \hat{D}_{11} \\ \vdots & \ddots & \vdots \\\hat{D}_{n_i n_j}  & \dots & \hat{D}_{n_i 1} \\ \end{matrix}  \right), \qquad \hat{D}_{\alpha \beta} = Q_{2l_i} D_{\alpha \beta} Q_{2l_j}. \] If $D_{\alpha \beta} =\left(\begin{matrix} X & Y \\ Z & W \end{matrix} \right)$, then $ \hat{D}_{\alpha \beta}\left(\begin{matrix} -W & Z \\ Y & -X \end{matrix} \right)$. Now, it is easy to check that all $C_{ji}$ with $j<i$ satisfy conditions of  Assertion~\ref{A:PreserDecompBiLagr}. Assertion~\ref{A:DecomCondIJ} is proved.
\end{proof}

\subsubsection{Proof of Theorem~\ref{T:DimDecomp}}

\begin{proof}[Proof of Theorem~\ref{T:DimDecomp}] The proof is by direct calculation. It suffices to describe the stabilizer of $L$ in $\operatorname{aut}(V, \mathcal{P})$, i.e. to apply conditions of Assertion~\ref{A:PreserDecompBiLagr} to the matrices from Theorem~\ref{T:BiSymp_General_Jordan_Case_Mega}. For the diagonal blocks $C_{ii}$ from Assertion~\ref{A:PreserDecompBiLagr} we did it in Assertion~\ref{A:AutPresLOneBlock}. Also in Assertion~\ref{A:DecomCondIJ} we proved that it remains only to consider  the blocks $C_{ij}$ with $i> j$. Since $C \in \operatorname{aut}(V, \mathcal{P})$, by Theorem~\ref{T:BiSymp_General_Jordan_Case_Mega} the matrix $C_{ij}$ has the form \[\left(\begin{matrix} 0 \\ Y_{ij} \end{matrix} \right), \qquad Y_{ij} = \left(\begin{matrix} C^{i,j}_1 & & & \\ C^{i, j}_2 & C^{i,j}_1 &  & \\ \vdots & \ddots & \ddots  & \\ C^{i,j}_{n_j} & \dots & \dots  & C^{i,j}_1\\ \end{matrix}  \right). \] Now consider the blocks $D_{\alpha \beta}$ for the matrix $C_{ij}$ from  Assertion~\ref{A:PreserDecompBiLagr}. We see that the elements on diagonal are equal: \[ \alpha' + \beta'  = \alpha + \beta \qquad \Rightarrow \qquad D_{\alpha'\beta'} = D_{\alpha \beta}.\] And $D_{\alpha \beta} = 0$ if $\alpha + \beta > n_j +1$. Below we assume that $\alpha + \beta \leq n_j +1$. Assertion~\ref{A:PreserDecompBiLagr} gives us the following conditions on the elements $D_{\alpha \beta}$:

\begin{enumerate}

\item If $\alpha + \beta \geq (n_i -h_i +1 ) + (h_j + 1)$, then $D_{\alpha \beta} = \left( \begin{matrix} X_{\alpha \beta} & Y_{\alpha \beta} \\ 0 & 0 \end{matrix} \right)$;

\item If $\alpha + \beta \geq (h_i +1 ) + (n_j - h_j + 1)$, then $D_{\alpha \beta} = \left( \begin{matrix} 0 & Y_{\alpha \beta} \\ 0 & Z_{\alpha \beta} \end{matrix} \right)$;

\item If $\alpha + \beta \geq (n_i -h_i +1 ) + (n_j - h_j + 1)$, then $D_{\alpha \beta} = \left( \begin{matrix} X_{\alpha \beta} & Y_{\alpha \beta} \\ 0 & Z_{\alpha \beta} \end{matrix} \right)$.

\end{enumerate}

Here, if several equalities are satisfied, then we apply all the conditions. Note that $\alpha + \beta \leq n_j +1$, and $n_i \geq n_j$ thus the case \[  \alpha + \beta \geq (h_i +1 ) + (n_j - h_j + 1), \qquad \text{ and } \qquad \alpha + \beta \geq (n_i -h_i +1 ) + (h_j + 1)\] is not possible. Similarly, we did not list zero blocks from Assertion~\ref{A:PreserDecompBiLagr}, because $h_k \geq \frac{n_k}{2}$ and the case $\alpha + \beta \geq (h_i+1) + (h_j + 1)$ is also impossible. It is easy to see that there will be 

\begin{itemize}

\item $s_1 = \max(0, (n_j - h_j) - (n_i - h_i))$ blocks $C^{ij}_{k} = \left( \begin{matrix} * & * \\ 0 & 0 \end{matrix} \right)$;

\item $s_2 = \max(0, h_j - h_i)$ blocks $C^{ij}_{k}= \left( \begin{matrix} 0 & * \\ 0 & * \end{matrix} \right)$;

\item $\max\left(0, h_j - (n_i - h_i)\right)  - s_1 - s_2$ blocks $C^{ij}_{k} = \left( \begin{matrix} * & * \\ 0 & * \end{matrix} \right)$.

\end{itemize}

Note that at least one of $s_1$ and $s_2$ is equal to $0$, since $n_i \geq n_j$ and $h_s \geq \frac{n_s}{2}$. It is easy to check that each pair of block $C_{ij}, C_{ji}$ with $i \not = j$ increases the dimension of the orbit $O_L$ by $l_i l_j \Delta_{ij}$, where $D_{ij}$ is given by \eqref{Eq:DeltaDimDec}. By Assertion~\ref{A:AutPresLOneBlock} the diagonal blocks $C_{ii}$  increases the dimension of the orbit $O_L$ by $\frac{l_i(l_i+1)}{2} (2h-n_i)$. We get the required dimension~\eqref{Eq:DimDec}. Theorem~\ref{T:DimDecomp} is proved. \end{proof}

\section{Real case} \label{S:RealCase}

There exists a natural real analog of the Jordan–Kronecker theorem (see \cite{Thompson91} or \cite{Lancaster05}).

\begin{theorem} Any two skew-symmetric bilinear forms A and B on a real finite-dimensional vector space $V$ can be reduced simultaneously to block-diagonal form; besides, each block is either
a Kronecker block or a Jordan block with eigenvalue $\lambda \in \mathbb{R} \cup \left\{\infty \right\}$ or a real Jordan block with
complex eigenvalue $\lambda = \alpha + i \beta$:
{\scriptsize \begin{equation} \label{Eq:RealJordBlock}  A_i =\left(
\begin{array}{c|c}
  0 & \begin{matrix}
   \Lambda &E&        & \\
      & \Lambda & \ddots &     \\
      &        & \ddots & E  \\
      &        &        & \Lambda   \\
    \end{matrix} \\
  \hline
  \begin{matrix}
  \minus\Lambda  &        &   & \\
  \minus E   & \minus\Lambda &     &\\
      & \ddots & \ddots &  \\
      &        & \minus E   & \minus \Lambda \\
  \end{matrix} & 0
 \end{array}
 \right)
\quad  B_i= \left(
\begin{array}{c|c}
  0 & \begin{matrix}
    E & &        & \\
      & E &  &     \\
      &        & \ddots &   \\
      &        &        & E   \\
    \end{matrix} \\
  \hline
  \begin{matrix}
  \minus E  &        &   & \\
     & \minus E &     &\\
      &  & \ddots &  \\
      &        &    & \minus E  \\
  \end{matrix} & 0
 \end{array}
 \right)  \end{equation}}
Here $\Lambda$ and $E$ are the $2 \times 2$ matrices
\[ \Lambda =\left( \begin{matrix} \alpha & - \beta \\ \beta & \alpha \end{matrix} \right), \qquad 
 E = \left( \begin{matrix} 1 & 0 \\ 0 & 1 \end{matrix} \right).\]
\end{theorem}

In the real case in the Jordan-Kronecker decomposition we should "group together" subspaces corresponding to pairs of complex conjugate eigenvalues $\alpha_j \pm i \beta_j$: \begin{equation} \label{Eq:JKDecomIntroReal}
 (V, \mathcal{P}) = \bigoplus_{j=1}^{S_1}\left(\bigoplus_{k=1}^{N_j} \mathcal{J}^{\mathbb{R}}_{\lambda_j, 2n_{j,k}} \right) \oplus \bigoplus_{j=1}^{S_2}\left(\bigoplus_{k=1}^{M_j} \mathcal{J}^{\mathbb{R}}_{\alpha_j \pm i \beta_j, 4m_{j,k}} \right) \oplus  \bigoplus_{j=1}^q \mathcal{K}^{\mathbb{R}}_{2k_j+1}
\end{equation} We use superscripts to distinguish between real and complex vector spaces: $\mathbb{R}$ for real, $\mathbb{C}$ for complex. For instance, $\mathcal{J}_{\lambda, 2n}^{\mathbb{C}}$ denotes a complex Jordan block.  Next statement shows that real Jordan blocks $\mathcal{J}^{\mathbb{R}}_{\alpha \pm i \beta, 4m}$  are realifications of the complex Jordan blocks $\mathcal{J}^{\mathbb{C}}_{\alpha + i \beta, 2m}$.

\begin{lemma} \label{L:CompStrRealJord} Each real Jordan block $\mathcal{J}^{\mathbb{R}}_{\alpha \pm i \beta, 4m}$ admits a natural complex structure $J = \frac{S - \alpha}{\beta}$, where $S$ is the semisimple part of the recursion operator. Complexification w.r.t. $J$ transforms $\mathcal{J}^{\mathbb{R}}_{\alpha \pm i \beta, 4m}$ into the complex Jordan block $\mathcal{J}^{\mathbb{C}}_{\alpha + i \beta, 2m}$. \end{lemma}

Simply speaking, for the Jordan block \eqref{Eq:RealJordBlock} the recursion operator $P$ and the complex structure $J$ are { \scriptsize \[ P = \left(
\begin{array}{c|c}
  \begin{matrix}
  \Lambda  &        &   & \\
   E   & \Lambda &     &\\
      & \ddots & \ddots &  \\
      &        &  E   &  \Lambda \\
  \end{matrix} & 0 \\
  \hline
 0 &  \begin{matrix}
   \Lambda &E&        & \\
      & \Lambda & \ddots &     \\
      &        & \ddots & E  \\
      &        &        & \Lambda   \\
    \end{matrix} 
 \end{array}
 \right), \qquad J = \left(
\begin{array}{c|c}
  \begin{matrix}
  J_2  &        &   & \\
      & J_2 &     &\\
      &  & \ddots &  \\
      &        &     &  J_2 \\
  \end{matrix} & 0 \\
  \hline
 0 &  \begin{matrix}
   J_2 &&        & \\
      & J_2 &  &     \\
      &        & \ddots &   \\
      &        &        & J_2   \\
    \end{matrix} 
 \end{array}
 \right), \]} where $J_2 = \left(\begin{matrix} 0 & -1 \\ 1 & 0 \end{matrix}  \right)$. We are ready to describe the real bi-Lagrangian Grassmanians, denoted by $\operatorname{BLG}_{\mathbb{R}}(V, \mathcal{P})$.  For a complex manifold $M$,  we use $M_{\mathbb{R}}$ to denote its realification. 

\begin{theorem} \label{T:RealBiLagr} Let $\mathcal{P}$ be a pencil of $2$-forms on a real vector space $V$ with Jordan--Kronecker decomposition \eqref{Eq:JKDecomIntroReal}. Then the bi-Lagrangian Grassmanian is isomorphic to the direct product  \begin{equation} \label{Eq:DecomRealBLG} \operatorname{BLG}_{\mathbb{R}}\left(V, \mathcal{P} \right) \approx \prod_{j=1}^{S_1} \operatorname{BLG}_{\mathbb{R}} \left( \bigoplus_{k=1}^{N_j} \mathcal{J}^{\mathbb{R}}_{\lambda_j, 2n_{j,k}} \right) \times \prod_{j=1}^{S_2} \operatorname{BLG}_{\mathbb{R}} \left(\bigoplus_{k=1}^{M_j} \mathcal{J}^{\mathbb{R}}_{\alpha_j \pm i \beta_j, 4m_{j,k}} \right).\end{equation} Moreover, for the real Jordan blocks the real bi-Lagrangian Grassmanian is isomorphic to the realification of the corresponding complex bi-Lagrangian Grassmanian: \begin{equation} \label{Eq:BiLagRealJord} \operatorname{BLG}_{\mathbb{R}} \left(\bigoplus_{k=1}^{M_j} \mathcal{J}^{\mathbb{R}}_{\alpha_j \pm i \beta_j, 4m_{j,k}} \right) \approx \left(\operatorname{BLG} \left(\bigoplus_{k=1}^{M_j} \mathcal{J}^{\mathbb{C}}_{\alpha_j + i \beta_j, 2m_{j,k}} \right)\right)_{\mathbb{R}}. \end{equation} 
\end{theorem}

\begin{proof}[Proof of Theorem~\ref{T:RealBiLagr} ] Decomposition~\eqref{Eq:DecomRealBLG} is proved similarly to Theorems~
\ref{T:BiLagrKronPart} and \ref{T:JordaMultEigen}. By Lemma~\ref{L:CompStrRealJord} the sum $\mathcal{W}_j = \bigoplus_{k=1}^{M_j} \mathcal{J}^{\mathbb{R}}_{\alpha_j \pm i \beta_j, 4m_{j,k}}$ of Jordan blocks with the same eigenvalue $\alpha_j \pm i \beta_j$ possesses a natural complex structure $J$. Note that by Lemma~\ref{L:CompStrRealJord} $J$ is a polynomial of the recursion operator $P$ on $\mathcal{W}_j$. Each bi-Lagrangian subspace $L \subset \mathcal{W}_j$ is $P$-invariant and thus it is also $J$-invariant. Therefore, after complexification w.r.t. $J$ each bi-Lagrangian subspace remains bi-Lagrangian. This proves \eqref{Eq:BiLagRealJord}. 
Theorem~\ref{T:RealBiLagr}  is proved.  \end{proof}

All the results of the previous sections (excluding Theorem~\ref{T:JordaMultEigen} on eigenvalue decomposition) remain true for the bi-Lagrangian Grassmanians $\operatorname{BLG}_{\mathbb{R}} \left( \bigoplus_{k=1}^{N_j} \mathcal{J}^{\mathbb{R}}_{\lambda_j, 2n_{j,k}} \right)$.

\section{Question and Open Problems} \label{S:OpenProblems}

We present a collection of open questions related to the structure of bi-Lagrangian Grassmannians.  These questions serve as stepping stones for future research, guiding us towards a more comprehensive understanding of their structure and behavior.

\subsection{Automorphism orbits}

We start with problems about $\operatorname{Aut}(V,\mathcal{P})$-orbits.

\subsubsection{Number of orbits}

Section~\ref{SubS:InfOrb} shows that bi-Lagrangian Grassmanians $\operatorname{BLG}(V,\mathcal{P})$ may have an infinite number of  $\operatorname{Aut}(V,\mathcal{P})$-orbits. However, some special cases, like sums of equal Jordan blocks, have a finite number of automorphism orbits (see Theorem~\ref{T:BiLagrEqTypesNum}). This motivates the following question:

\begin{problem} \label{Prob:InfOrb} When the number of $\operatorname{Aut}(V, \mathcal{P})$-orbits of a bi-Lagrangian Grassmanian $\operatorname{BLG}(V, \mathcal{P})$ is infinite? A particularly interesting challenge is to find examples that:

\begin{enumerate}
    \item Minimize the dimension of $V$.
    \item Have a nilpotent recursion operator $P = B^{-1}A$ with the smallest possible height.
\end{enumerate} \end{problem}

A first step to tackling Problem~\ref{Prob:InfOrb} 
is analyzing low-dimensional bi-Lagrangian Grassmanians.

\begin{problem} Describe $\operatorname{BLG}(V, \mathcal{P})$ for the following cases: when $V$ is low-dimensional; when the recursion operator $P = B^{-1}A$ has small height.
\end{problem}

\begin{remark} It may be fruitful to compare the structure of bi-Lagrangian subspaces with that of invariant subspaces for nilpotent operators. The behavior of invariant subspaces is heavily dependent on the operator's height $n$, with a distinct shift in behavior at $n=6$. For $n < 6$, the number of indecomposable subspace types is finite, whereas for $n>6$, the problem becomes "wild" and defies complete classification. The case $n=6$ has been extensively explored. See e.g. \cite{Ringel08}, \cite{Ringel24}  for more details. \end{remark}

\subsubsection{Minimal orbits}
Having investigated top-dimensional orbits, we propose exploring minimal-dimensional ones, crucial for bi-integrability (see Section~\ref{SubS:BiIntProb}).

\begin{problem} \label{Prob:MinDimOrb} Classify minimal-dimensional $\operatorname{Aut}(V, \mathcal{P})$-orbits in $\operatorname{BLG}(V, \mathcal{P})$.
\end{problem}

\begin{remark} For $(V, \mathcal{P}) = \bigoplus_{i=1}^N \mathcal{J}_{0, 2n_i}$, candidate minimal-orbit bi-Lagrangian subspaces satisfy \[ U \subset L \subset U^{\perp}, \qquad U = \bigoplus_{i=1}^N \mathcal{J}^{\leq n_i /2}_{0, 2n_i}.\] Figure~\eqref{Eq:BiLagrMinOrb} visualizes such subspaces for For $\mathcal{J}_{0, 8} \oplus \mathcal{J}_{0, 6}$. The shaded area is $U$, union of shaded and marked cells is $U^{\perp}$:
 \begin{equation} \label{Eq:BiLagrMinOrb}
 \begin{tabular}{|c|c|c|c|} 
  \cline{1-2} &  & \multicolumn{1}{c}{} & \multicolumn{1}{c}{} \\
\hline    &   & &  \\ 
     \hline     {\cellcolor{gray!25} } & {\cellcolor{gray!25} }  & X & X  \\ 
   \hline   {\cellcolor{gray!25} } & {\cellcolor{gray!25} } &   {\cellcolor{gray!25} } & {\cellcolor{gray!25} }  \\ \hline
  \end{tabular} 
\end{equation}
\end{remark}

\subsubsection{Topology of orbits}

The next question concerns the topology of $\operatorname{Aut}(V,\mathcal{P})$-orbits. Theorems~\ref{T:EqualJordTopolOrb}, \ref{T:GenJordTopolMaxOrb} \ref{T:2DistJordTopolMaxOrb}, \ref{T:2DistJordTopolMaxOrbIndecTypeI}, \ref{T:Type2Special} and \ref{T:2DistJordTopolMaxOrbIndecTypeII} revealed a striking similarity: all studied orbits are $(\mathbb{K}^*)^{M_1}\times \mathbb{K}^{M_2}$-fiber bundles over homogeneous spaces. This shared structure naturally leads us to the next question:

\begin{problem} \label{Prob:AutHomog} Consider a sum of Jordan blocks \[(V, \mathcal{P}) = \bigoplus_{i=1}^t \left(\bigoplus_{j=1}^{l_i} \mathcal{J}_{0, 2n_i} \right), \qquad n_1 > \dots > n_t.\] Let $L \subset (V,\mathcal{P})$ be a bi-Lagrangian subspace. 

\begin{enumerate}

\item Does its orbit $O_L$ admit a structure of a $(\mathbb{K}^*)^{M_1}\times \mathbb{K}^{M_2}$-fibre over a homogeneous space for a product of symplectic groups:  \begin{equation} \label{Eq:HomogFibre} \pi: O_L  \xrightarrow{(\mathbb{K}^*)^{M_1}\times \mathbb{K}^{M_2}}  \left(\operatorname{Sp}(2l_1) \times \dots \times \operatorname{Sp}(2l_t)\right) / G.\end{equation} 

\item If the answer to the previous question is affirmative, can we describe these homogeneous spaces $\left(\operatorname{Sp}(2l_1) \times \operatorname{Sp}(2l_t)\right) / G$?

\end{enumerate}

\end{problem}

The first step to proving \eqref{Eq:HomogFibre} is examining the group of automorphisms.

\begin{claim} \label{Claim:Aut} Consider a sum of equal Jordan blocks $(V, \mathcal{P}) = \bigoplus_{i=1}^t \left(\bigoplus_{j=1}^{l_i} \mathcal{J}_{0, 2n} \right)$, where $n_1 > \dots > n_t$.  We claim that the group of bi-Poisson automorphisms $\operatorname{Aut}(V,\mathcal{P})$ admits a Levi decomposition\footnote{Although the Levi decomposition is guaranteed for complex automorphism groups \eqref{Eq:LeviAut} due to simple connectivity (see \cite[Theorem 3.18.13]{Varadarajan84} ), the real case requires further analysis. Here, we claim that the Levi decomposition exists even for real automorphism groups, despite the non-trivial fundamental group ($\pi_1\left(\operatorname{Sp}(2n, \mathbb{R})\right) = \mathbb{Z}$).} of the form:\begin{equation} \label{Eq:LeviAut} \operatorname{Aut}(V,\mathcal{P}) = R \rtimes S, \end{equation} where \begin{itemize}

    \item The semisimple subgroup is isomorphic to the product of symplectic groups: \begin{equation} \label{Eq:SemiSimpleAutGroup} S \approx \operatorname{Sp}(2l_1) \times \dots \times \operatorname{Sp}(2l_t).\end{equation}

    \item The radical $R$ is simply connected. This implies it is diffeomorphic to $\mathbb{K}^{\dim R}$ (see \cite[Theorem 3.18.11]{Varadarajan84}).
\end{itemize}
\end{claim}

\begin{remark} Elements of $\operatorname{Aut}(\bigoplus_{i=1}^t \left(\bigoplus_{j=1}^{l_i} \mathcal{J}_{0, 2n} \right))$ have the form \eqref{E:bsp_algebra_matrix}. The semisimple subgroup is readily identifiable within this form. The diagonal elements $C^{i,i}_{1} \in \operatorname{Sp}(2l_i)$, for $i=1\dots,t$, generate the semisimple subgroup \eqref{Eq:SemiSimpleAutGroup}. All that remains to prove is that the subgroup where $C^{i,i}_{1} = I_{2l_i}$, for $i=1\dots,t$, is solvable and simply connected. \end{remark}

\begin{problem} Prove/disprove Claim~\ref{Claim:Aut}.
\end{problem}

\subsubsection{Jet spaces and group structures}

The most peculiar examples of orbits within the bi-Lagrangian Grassmannian are the generic orbit for the sum of equal Jordan blocks and the orbits associated with Type II-S indecomposable subspaces, both of which exhibit the structure of jet spaces (Theorems~\ref{T:EqualJordTopolOrb} and \ref{T:Type2Special}). This phenomenon is likely attributed to the presence of the recursion operator $P$. A comprehensive description of all orbits sharing this property would be an intriguing area of future research.

\begin{problem} What $\operatorname{Aut}(V,\mathcal{P})$-orbits of $\operatorname{BLG}(V,\mathcal{P})$ are jet spaces? \end{problem}

Notably, Type II-S orbits also possess a group structure (Theorem~\ref{T:Type2Special}). It is unclear whether this is a specific instance or indicative of a deeper underlying phenomenon. 

\begin{problem} What $\operatorname{Aut}(V,\mathcal{P})$-orbits of $\operatorname{BLG}(V,\mathcal{P})$ admit a (natural) group structure? \end{problem}

\subsection{Global Structure of bi-Lagrangian Grassmanian}

The next group of question concern global properties of $\operatorname{BLG}(V,\mathcal{P})$.

\subsubsection{Smoothness}

The question about smoothness of bi-Lagrangian Grassmanians was previously posed in \cite[Problem 13]{BolsinovIzosimomKonyaevOshemkov12},   \cite[Section 2.2.2]{Rosemann15} and \cite[Problem 5.3]{BolsinovOpen17}.

\begin{problem} Find necessary and sufficient conditions for a bi-Lagrangian Grassmannian $\operatorname{BLG}(V, \mathcal{P})$ to be a smooth algebraic variety. If $\operatorname{BLG}(V, \mathcal{P})$  is not smooth, what types of singularities can it have?   \end{problem}

\begin{remark}A bi-Lagrangian Grassmanian $\operatorname{BLG}(V,\mathcal{P})$ is probably not smooth, except when all Jordan blocks are $2\times 2$-blocks, when it is diffeomorphic to a product of Lagrangian Grassmanians: \[\operatorname{BLG}\left(\bigoplus_{j=1}^{S}\left(\bigoplus_{k=1}^{N_j} \mathcal{J}_{\lambda_j, 2} \right) \oplus  \bigoplus_{j=1}^q \mathcal{K}_{2k_j+1}\right) \approx \prod_{j=1}^S \Lambda(N_{j}).  \] \end{remark}

\subsubsection{Orbit closures and irreducible components}

Like piecing together a puzzle, so far we've analyzed the individual components of automorphism orbits. Now, we aim to see the complete image.

\begin{problem} Describe the closure $\bar{O}_{\max}$ (in the standard topology) of the highest-dimensional $\operatorname{Aut}(V,\mathcal{P})$-orbit  in $\operatorname{BLG}(V,\mathcal{P})$. Does it coincide with the set of all semisimple bi-Lagrangian subspaces (cf. Section~\ref{SubS:Exam})? \end{problem}

\begin{remark} As noted in Remark~\ref{Rem:RedDisc}, a bi-Poisson reduction map is not continuous. While it maps the top-dimensional orbit to a subset of another top-dimensional orbit, it is not guaranteed that the closure of the original orbit is mapped to the closure of its image. \end{remark}

A more in-depth exploration of the bi-Lagrangian Grassmannian as an algebraic variety is warranted.

\begin{problem} Count the number of irreducible components of $\operatorname{BLG}(V,\mathcal{P})$ (cf. Remark~\ref{Rem:ClosureOrbitsJ2}). \end{problem}

\subsubsection{Cohomology of bi-Lagrangian Grassmanian}

To date, algebraic invariants of bi-Lagrangian Grassmanians is an enigma shrouded in mystery.

\begin{problem} Calculate standard algebraic invariants of bi-Lagrangian Grasmmanian $\operatorname{BLG}(V, \mathcal{P})$:

\begin{enumerate}

\item Its fundamental group $\pi_1(\operatorname{BLG}(V, \mathcal{P}))$.

\item Its cohomology ring $H^{*}(\operatorname{BLG}(V, \mathcal{P}))$.

\end{enumerate}
  \end{problem}

The cohomology ring for Grassmanians are well-known (see e.g. \cite{BerryTilton20} and the references therein).  For Lagrangian Grassmannians, refer to \cite{Coskun13}, \cite{Kolhatkar04} or \cite{Tamkavis05} and their references.

\subsection{Decomposition of Bi-Lagrangian subspaces}

\begin{problem} Is the decomposition of a bi-Lagrangian subspace into indecomposables unique? Formally, given two decompositions  \[ (V,\mathcal{P}, L) \approx \bigoplus_{i=1}^N (V_i,\mathcal{P}_i, L_i) \approx \bigoplus_{i=1}^{N'} (V'_i,\mathcal{P}'_i, L'_i), \]  with indecomposable components, does $N = N'$ and and exist a permutation $\pi$ with \[(V_{\pi(i)},\mathcal{P}_{\pi(i)}, L_{\pi(i)}) \approx (V'_i,\mathcal{P}'_i, L'_i), \qquad i=1,\dots, N. \] 
\end{problem}

\begin{remark} Invariant subspaces of nilpotent operators form an additive category with unique decompositions (Krull-Remak-Schmidt property, see e.g. \cite{Ringel08}). In contrast, bi-Lagrangian subspaces is not an additive category, forming only a monoid under direct sums.  \end{remark}

\subsection{New methods for integrating bi-Hamiltonian systems} \label{SubS:BiIntProb}

The connection between bi-Lagrangian subspaces and bi-integrable bi-Hamiltonian systems was briefly explored in Section~\ref{SubS:Motivation}. The fundamental challenge lies in resolving the Generalized Argument Shift Conjecture~\ref{Conj:GenArgConj}. To advance the field, we suggest:

\begin{problem} Develop new integrability techniques of bi-Hamiltonian systems using $\operatorname{Aut}(V, \mathcal{P})$-invariant subspaces.\end{problem}

Consider a pencil of compatible Poisson structures $\mathcal{P} = \left\{\mathcal{A} + \lambda \mathcal{B} \right\}$ on a manifold $M$. $\operatorname{Aut}(T^{*}_{x} M, \mathcal{P}_x)$-invariant bi-Lagrangian subspaces $L_x \subset (T^*_xM, \mathcal{P}_x)$ form an integrable bi-Lagrangian distribution $\mathcal{L} \subset (T^*M, \mathcal{P})$. The objective is to describe such distributions for some interesting bi-Hamiltonian systems.

\begin{remark} Standard integrals of bi-Hamiltonian systems \[v = \mathcal{A}_{\lambda}df_{\lambda}, \qquad \mathcal{A}_{\lambda} =\mathcal{A} + \lambda \mathcal{B}, \quad \lambda \in\bar{\mathbb{K}}\] include (local) Casimir functions of the Poisson brackets $\mathcal{A}_{\lambda}$ and the eigenvalues $\lambda_j(x)$. Integrating bi-Hamiltonian systems might use
$\operatorname{Aut}$-invariant bi-Lagrangian subspaces \textit{post Casimir/eigenvalue bi-Poisson reduction} (which can be done similar to \cite[Theorem 5.9]{Kozlov23JKRealization}). Note that differentials of eigenvalues $d\lambda_j(x)$ distinguish bi-Lagrangian subspaces within some some low-dimensional orbit. Hence, an analysis of small bi-Lagrangian orbits is imperative (Problem~\ref{Prob:MinDimOrb}). \end{remark}

Of particular interest are the semidirect sums  $\operatorname{gl}(kl) + \left(\mathbb{R}^{kl}\right)^k$, where $k \geq 2, l \geq 1$. As it was shown in \cite{Vorushilov19} their JK invariants consist of $k-1$ Jordan blocks for each of  $\frac{kl(l+1)}{2}$ eigenvalues. The sizes of these blocks are \[ \underbrace{2,\dots, 2}_{k-2}, 4.\] The case $l=1$ is covered in \cite{BolsZhang}. We are interested in the case $k=1$, where there is one Jordan $4\times 4$ blocks for each eigenvalue and by Theorem~\ref{T:InvarBilagr} there is an invariant bi-Lagrangian distribution.

\begin{problem} Let $\mathfrak{g} = \mathfrak{gl}(4)+(\mathbb{R}^4)^2$. For a generic pair $x,a \in \mathfrak{g}^*$ define a pair of matrices \[ A_x = \left(c^i_{jk} x_i\right), \qquad A_a = \left(c^i_{jk} x_i\right), \] where $c^i_{jk}$ are structural constants of $\mathfrak{g}$. The invariant bi-Lagrangian distribution $\mathcal{L} \subset T^*\mathfrak{g}^*$ has the form   \[ \mathcal{L}_x = \bigcap_{\lambda_j(x) \in \sigma(P_x)} \operatorname{Im} (P_x - \lambda_j(x)), \qquad P_x = A_a^{-1}A_x. \] Can the distribution\footnote{It is integrable by \cite[Corollary 2.5]{BolsinovNijenhuis} or \cite[Theorem 5]{Kozlov17}} $\mathcal{L}$ be represented by polynomials $f_1, \dots, f_6$ on $\mathfrak{g}^*$? Are there easy-to-use expressions for the $f_i$? Does the same hold for $\operatorname{gl}(2l) + \left(\mathbb{R}^{2l}\right)^2$? \end{problem}

To delve deeper, consider Lie algebras possessing specific JK invariants from \cite{Kozlov23JKRealization}. 

\subsection{Pencils of symmetric and skew matrices}

\cite{Thompson91} offers a canonical form for pencils $\mathcal{P} = \left\{A+\lambda B\right\}$ on a real or complex vector space $V$, where $A$ and $B$ are symmetric or skew matrices: \[A^T = \varepsilon_1 A, \qquad B^T = \varepsilon_2 B, \qquad \varepsilon_i \in \left\{-1, 1\right\}.\] Let $\operatorname{Aut}(V,A,B)$ denote the set of endomorphisms $C \in \operatorname{End}(V)$ such that \[A(Cu, Cv) = A(u,v), \qquad B(Cu, Cv) = B(u,v), \qquad \forall u, v\in V.\] This paper focused on skew-symmetric matrix pairs. The cases of "symmetric + symmetric" and "symmetric + skew-symmetric" merit further investigation.

\begin{problem} Let $A$ be a symmetric form and $B$ be either symmetric or skew-symmetric forms on $V$. 

\begin{enumerate}

\item Describe the automorphism group $\operatorname{Aut}(V,A,B)$ (or its Lie algebra).

\item Analyze the action of $\operatorname{Aut}(V, A, B)$ on the set of $k$-dimensional subspaces $V$. What are generic orbits\footnote{Determining the canonical form of the pencil $\mathcal{P}\bigr|_{W}$ restricted to a generic hyperspace $W \subset V$ is a separate challenge. Knowledge of generic covectors $\alpha \in V^*$ offers a potential solution.}? Are there invariant and fixed subspaces?

\end{enumerate}

\end{problem}

Let $B$ be a symplectic form on $V$, i.e. \[ B^T = - B, \qquad \operatorname{Ker} B = 0.\] We can replace a symmetric form $A$ with a skew-adjoint operator $P = B^{-1}A$:\[B(u, Pv) = - B(Pu, v), \qquad \forall u, v \in V.\]

In \cite{Lancaster94} it was proved that all (real) maximal $P$-invariant isotropic subspaces have the same dimension. 

\begin{problem} Investigate maximal $P$-invariant isotropic subspaces of a symplectic space, when $P$ is a skew-adjoint operator (i.e. $A^T = A$ for $A = BP$). Describe generic subspaces and $\operatorname{Aut}(V, A, B)$-orbit topology. \end{problem}

The general case may be overly intricate. As a first step,  consider orbits containing stable $P$-invariant $B$-Lagrangian subspaces, as characterized in \cite{RanRodman88} and \cite{RanRodman89}. Note that our investigation is not limited to symmetric and skew-symmetric matrices. For instance, the stability of $P$-invariant $B$-Lagrangian subspaces has been studied for a more general class of matrices $P$ and $B$, as exemplified by \cite{Mehl09} and related research.


\begin{thebibliography}{99}


\bibitem{Arnold67}  V.\,I.~Arnol'd, ``Characteristic class entering in quantization conditions'', \textit{Funktsional. Anal. i Prilozhen.}, \textbf{1}:1 (1967), 1--14; \textit{Funct. Anal. Appl.}, \textbf{1}:1 (1967), 1--13.

\bibitem{BerryTilton20} E.~Berry, S.~Tilton, ``The cohomology of real Grassmannians via Schubert stratifications'',
\texttt{arXiv:2021.07695v2 [math.AT]}

\bibitem{BolsinovIzosimomKonyaevOshemkov12} A.\,V.~Bolsinov, A.\,M.~Izosimov, A.\,Y.~Konyaev, A.\,A.~Oshemkov, ``Algebra and topology of integrable systems: research problems'', \textit{Trudi seminara po vectornomu i tenzornomu analizu}, \textbf{28} (2012), 119--191.


\bibitem{BolsinovTsonev17}  A.\,V.~Bolsinov, A.\,M.~Izosimov, D.\,M.~Tsonen, ``Finite-dimensional integrable systems: a collection of research problems'', \textit{Journal of Geometry and Physics}, \textbf{115} (2017), 2-15.


\bibitem{BolsinovNijenhuis} A.\,V.~Bolsinov, A.\,Yu.~Konyaev, V.\,S.~Matveev, ``Nijenhuis Geometry'', \textit{Advances in Mathematics}, 394 (2022), 108001.

\bibitem{BolsinovOpen17} A.\,V.~Bolsinov, V.\,S.~Matveev,  E.~Miranda, S.~Tabachnikov, ``Open problems, questions, and challenges in finite-dimensional integrable systems'', \textit{Phil. Trans. R. Soc. A}, \textbf{376} (2018), 20170430.


\bibitem{BolsZhang}
A.\,V.~Bolsinov, P.~Zhang, ``Jordan-Kronecker invariants of finite-dimensional Lie algebras'',
\textit{Transformation Groups}, \textbf{21}:1 (2016),  51--86.



\bibitem{Humphreysl75} J.~Humphreys, \textit{Linear algebraic groups}, Graduate Texts Math. 21, Springer Verlag, New York, (1975).

\bibitem{Bru91} R.~Bru, L.~Rodman, H.~Schneider, ``Extensions of Jordan bases for invariant subspaces of a matrix'', \textit{Linear
Algebra Appl.}, \textbf{150} (1991), 209--225.

\bibitem{Coskun13} I.~Coskun, \textit{Lectures in Warsaw Poland on homogeneous varieties, December 2013. Day 6-7: Isotropic Grassmannians, their basic geometry, Schubert varieties, the restriction problem Lecture 5}, \url{https://homepages.math.uic.edu/~coskun/poland-lec5.pdf} 


\bibitem{Coskun2014} I.~Coskun,  ``Symplectic restriction varieties and geometric branching rules ii'', \textit{Journal of Combinatorial Theory Series A} \textbf{125} (2014), 47--97. 


\bibitem{Dmytryshyn2013} A.\,R.~Dmytryshyn, B.~Kagström, V.\,V.~Sergeichuk, ``Skew-symmetric matrix pencils: Codimension counts and the solution of a pair of matrix equations'', \textit{Linear Algebra and Its Applications}, \textbf{438}:8 (2013), 3375--3396.

\bibitem{Domanov10} I.~Domanov, ``On invariant subspaces of matrices: A new proof of a theorem of Halmos'', \textit{Linear Algebra Appl.}, \textbf{433}: 11-12 (2010), 2255--2256.


\bibitem{Faouzi01} A.~Faouzi, ``On the orbit of invariant subspaces of linear operators in finite-dimensional spaces (new proof of a
Halmos result)'', \textit{Linear Algebra Appl.}, \textbf{329}:1-3 (2001), 171--174.


\bibitem{Ferrer96} J.~Ferrer, F.~Puerta, X.~Puerta, ``Geometric Characterization and Classification of Marked Subspaces'', \textit{Linear Algebra Appl.}, \textbf{235} (1996), 15--34.

\bibitem{Fulton91} W.~Fulton, J.~Harris, \textit{Representation Theory, a First Course}, Springer Verlag (1991).


\bibitem{Gantmacher88} F.\,R.~Gantmacher, \textit{Theory of Matrices}, Nauka, Moscow, 1988. [in Russian]

\bibitem{Gelfand79} I.\,M.~Gel’fand, I.\,Ya.~Dorfman, ``Hamiltonian operators and algebraic structures related to
them'', \textit{Functional Analysis and Its Applications}, \textbf{13}:4 (1979), 248--262.


\bibitem{Gohberg86} I.~Gohberg, P.~Lancaster, L.~Rodman, \textit{Invariant Subspace of Matrices with Applications}, John Wiley and Sons, New York, 1986.


\bibitem{Gurevich50} G.\,B.~Gurevich,  ``Canonization of a pair of bivectors'', \textit{Trudi seminara po vectornomu i tenzornomu analizu}, \textbf{8} (1950), 355--363.



\bibitem{Halmos71} P.\,R.~Halmos, ``Eigenvectors and adjoints'', \textit{Linear Algebra Appl.}, \textbf{4}:1 (1971), 11--15.

\bibitem{HM82} M.~Hazewinkel, C.\,F.~Martin, ``A short elementary proof of Grothendieck's theorem on algebraic vector bundles over the projective line'', \textit{J. Pure Appl. Algebra}, \textbf{25} (1982), 207--211.


\bibitem{HoffmanKunze} K.~Hoffman, R.~Kunze, \textit{Linear algebra}, Second Edition, Prentice-Hall, Inc., Englewood Cliffs, N.J. 1971


\bibitem{Kolar93} I. Kolář, P. Michor, J. Slovák, 
\textit{Natural operations in differential geometry}. 
Springer-Verlag, Berlin Heidelberg, 1993. 


\bibitem{Kozlov17} I.\,K.~Kozlov, ``Invariant foliations of nondegenerate bi-Hamiltonian structures'', \textit{Fundam. Prikl. Mat.}, \textbf{20}:3 (2015), 91--111; \textit{J. Math. Sci.}, \textbf{225}:4 (2017), 596--610.

\bibitem{Kozlov23JKRealization} I.\,K.~Kozlov,  ``Realization of Jordan-Kronecker invariants by Lie algebras'', {\tt  arXiv:2307.08642v1 [math.DG]}

\bibitem{Kolhatkar04} R.~Kolhatkar, \textit{Grassmann varieties}, Master’s thesis, McGill University, 2004 \url{https://www.math.mcgill.ca/goren/Students/KolhatkarThesis.pdf}

\bibitem{Lancaster05} P.~Lancaster, L.~Rodman, ``Canonical forms for symmetric/skew-symmetric real matrix pairs under strict equivalence and congruence'', \textit{Linear Algebra Appl.}, \textbf{406} (2005), 1--76.

\bibitem{Lancaster94} P.~Lancaster, L.~Rodman, ``Invariant Neutral Subspaces for Symmetric and Skew Real Matrix Pairs'', \textit{Canadian Journal of Mathematics}, \textbf{46}:3, (1994), 602--618.

\bibitem{LewisJet} A.\,D.~Lewis, \textit{MATH 942. Advanced Topics in Geometry and Topology (Analytic Vector Bundles). Fall/Winter 2011-2012}, \url{https://mast.queensu.ca/~andrew/teaching/math942/pdf/1chapter5.pdf} 

\bibitem{Magri78} F.~Magri, ``A simple model of the integrable Hamiltonian equation'', \textit{J. Math. Phys.}, \textbf{19}:5 (1978), 1156--1162.


\bibitem{Magri84} F.~Magri, C.~Morosi, ``A geometrical characterization of integrable Hamiltonian systems
through the theory of Poisson-Nijenhuis manifolds'', \textit{Quaderno} \textbf{S19} (1984) of the Department of Mathematics of the
University of Milan.

\bibitem{Malagon17} C.\,S.~Malag\'on, ``An elementary proof of the symplectic spectral theorem'', \textit{Asian-European Journal of Mathematics}, \textbf{12}:3 (2019), 1950033 

\bibitem{Manakov76} S.\,V.~Manakov, ``Note on the integration of Euler's equations of the dynamics of an n-dimensional
rigid body'', \textit{Funktsional. Anal. i Prilozhen.}, \textbf{10}:4 (1976), 93-94; \textit{Funct. Anal. Appl.}, \textbf{10}:4 (1976), 328-329

\bibitem{McDuffSalamon} D.~McDuff, D.~Salamon, \textit{Introduction to Symplectic Topology}, Oxford University Press, 1995.


\bibitem{McLean} M.~McLean, \textit{Differential Topology, Spring 2021 (MAT 566), Lecture 12}, \url{https://www.math.stonybrook.edu/~markmclean/MAT566/lecture12.pdf} 

\bibitem{Mehl09} Chr.~Mehl, V.~Mehrmann, A.\,C.\,M.~Ran, L.~Rodman, ``Perturbation analysis of Lagrangian invariant subspaces of symplectic matrices'', \textit{Linear and Multilinear Algebra}, \textbf{57} (2009), 141--184.

\bibitem{MishchenkoFomenko78EulerEquations} A.\,S.~Mishchenko, A.\,T.~Fomenko, ``Euler equations on finite--dimensional Lie groups'', \textit{Izv. Akad. Nauk SSSR Ser. Mat.}, \textbf{42}:2 (1978), 396--415.

\bibitem{Poor81} W.\,A.~Poor, \textit{Differential geometric structures}, McGraw-Hill Book Co., New York, 1981.

\bibitem{RanRodman88} A.\,C.\,M.~Ran, L.~Rodman, ``Stability of invariant Lagrangian subspaces I'',  \textit{Operator Theory: Advances and Applications (I. Gohberg ed.)}, \textbf{32} (1988), 181--218.

\bibitem{RanRodman89}  A.\,C.\,M.~Ran, L.~Rodman, ``Stability of invariant Lagrangian subspaces II'', \textit{Operator Theory: Advances and Applications (H.~Dym, S.~Goldberg, M.\,A.~Kaashoek, P.~Lancaster,
eds.)}, \textbf{40} (1989), 391--425.

\bibitem{RanRodman93} A.\,C.\,M.~Ran, L.~Rodman, J.\,E.~Rubin,
``Direct complements of invariant Lagrangian subspaces and minimal factorizations of real skew-symmetric rational matrix functions'',
\textit{Linear Algebra and its Applications} \textbf{180} (1993), 61--94. 

\bibitem{Ringel08}  C.\,M.~Ringel, M.~Schmidmeier, ``Invariant subspaces of nilpotent linear operators, I'', \textit{J. Reine Ang. Math.}, \textbf{614} (2008), 1–-52.

\bibitem{Ringel24}  C.\,M.~Ringel, M.~Schmidmeier, ``Invariant Subspaces of Nilpotent Operators. Level, Mean, and Colevel: The Triangle $\mathbb{T}(n)$'', \texttt{arXiv:2405.18592 [math.RT]}.


\bibitem{Rosemann15} S.~Rosemann, K.~Schöbel,  ``Open problems in the theory of finite-dimensional integrable systems and related fields'', \textit{Journal of Geometry and Physics}, \textbf{87} (2015), 396–-414. 


\bibitem{Reiman80} A.\,G.~Reiman, M.\,A.~Semenov-Tyan-Shanskii, ``A family of Hamiltonian structures, hierarchy
of hamiltonians, and reduction for first-order matrix differential operators'', \textit{Functional Analysis
and Its Applications}, \textbf{14}:2 (1980), 146–148.


\bibitem{Sadetov2004} S.\,T.~Sadetov, ``A proof of the  Mishchenko--Fomenko conjecture'', \textit{Doklady Math.}, \textbf{70}:1 (2004), 634–638.

\bibitem{Tamkavis05} H.~Tamvakis, ``Quantum cohomology of isotropic Grassmannians'', in \textit{Geometric methods in algebra and number theory}, volume 235 of \textit{Progr. Math.}, pages 311–338.
Birkh\"auser Boston, Boston, MA, 2005

\bibitem{Thompson91} R.\,C.~Thompson, ``Pencils of complex and real symmetric and
skew matrices'', \textit{Linear Algebra and its Applications},
\textbf{147} (1991), 323--371.


\bibitem{Varadarajan84} V.\,S.~Varadarajan, \textit{Lie groups, Lie algebras, and their representations}, Graduate Text in Mathematics. Springer-Verlag, New York, 1984.

\bibitem{Vizman13} C.~Vizman, ``The group structure for jet bundles over Lie groups'', \textit{Journal of Lie Theory}, \textbf{23}:3 (2013), 885--897.

\bibitem{Vorushilov19} K.\,S.~Vorushilov, ``Jordan–Kronecker invariants of semidirect sums of the form $\operatorname{sl}(n)+(\mathbb{R}^n)^k$ and $\operatorname{gl}(n)+(\mathbb{R}^n)^k$'', \textit{Fundam. Prikl. Mat.}, \textbf{22}:6 (2019), 3--18.

\bibitem{Pumei10} P.~Zhang,  \textit{Algebraic properties of compatible Poisson structures.} Preprint
(Loughborough University, no. 10--02). 2010.


\end{thebibliography}
\end{document}